\documentclass{amsart} 

\usepackage{float}
\usepackage{etoolbox}
\patchcmd{\subsection}{-.5em}{.5em}{}{}
\usepackage{graphicx}
\usepackage{fancyhdr}
\usepackage{makecell}

\graphicspath{{Images/}}
\usepackage{color}
\usepackage{lscape}
\usepackage{subfigure}
\newtheorem{theorem}{Theorem}[section]

\theoremstyle{definition}
\newtheorem{definition}[theorem]{Definition}
\newtheorem{corollary}[theorem]{Corollary}
\theoremstyle{remark}
\newtheorem{remark}[theorem]{{\bf{Remark}}}
\numberwithin{equation}{section}
\usepackage[]{hyperref}

\newtheorem{example}[theorem]{Example}

\usepackage{calc}
\usepackage{layout}

\makeatletter
\@namedef{subjclassname@2020}{\textup{2020} Mathematics Subject Classification}
\makeatother

\begin{document}


\title[Least p-Variances Theory]{Least p-Variances Theory}

\author[M. Masjed-Jamei]{by\\ \sc Mohammad Masjed-Jamei$ ^{1,2}$}
\thanks{$ ^{1} $Research Fellow of the Alexander von Humboldt Foundation, Germany,\\(mmjamei@gmail.com).}
\thanks{$ ^{2} $Faculty of Mathematics, K.N.Toosi University of Technology, Iran, (mmjamei@kntu.ac.ir).}

%

\date{}

\begin{abstract}
As a result of a rather long-time research started in 2016, this theory whose structure is based on a fixed variable and an algebraic inequality, improves and somehow generalizes the well-known least squares theory. In fact, the fixed variable has a fundamental role in constituting the least p-variances theory. In this sense, some new concepts such as p-covariances with respect to a fixed variable, p-correlation coefficient with respect to a fixed variable and p-uncorrelatedness with respect to a fixed variable are first defined in order to establish least p-variance approximations. Then, we obtain a specific system called p-covariances linear system and apply the p-uncorrelatedness condition on its elements to find a general representation for p-uncorrelated variables. Afterwards, we apply the concept of p-uncorrelatedness for continuous functions particularly for polynomial sequences and find some new sequences such as a generic two-parameter hypergeometric polynomial of $ _4{F_3} $ type that satisfy such a p-uncorrelatedness property. In the sequel, we obtain an upper bound for 1-covariances, an approximation for p-variances, an improvement for the approximate solutions of over-determined systems and an improvement for the Bessel inequality and Parseval identity. Finally, we generalize the notion of least p-variance approximations based on several fixed orthogonal variables. For a better comprehension, here we bring a list of the main contents of this theory including 15 sections.

\end{abstract}

\maketitle

\tableofcontents

\section{Introduction}\label{sec1}
Least squares theory is known as an essential tool for optimal approximations and regression analysis and has found extensive applications in mathematics, statistics, physics and engineering \cite{ref6, ref9, ref25, ref30}. Some important parts of this theory include: best linear unbiased estimator, Gauss-Markov theorem, moving least squares method and its improvements \cite{ref12, ref28}, Least-squares spectral analysis, orthogonal projections and least squares function and polynomial approximations \cite{ref7}.

It was officially discovered and published by A. M. Legendre \cite{ref13} in 1805 though it is also co-credited to C. F. Gauss who contributed significant theoretical advances to the method and may have previously used it in his work in 1795 \cite{ref29}.

The base of this theory is to minimize the sum of the squares of the residuals, which are the difference between observed values and the fitted values provided by an approximate linear (or nonlinear) model.

In many linear models \cite{ref27}, in order to reduce the influence of errors in the derived observations, one would like to use a greater number of samplings than the number of unknown parameters in the model, which leads to the so called overdetermined linear systems. In other words, let $ b \in {{\bf{R}}^m} $ be a given vector and $ A \in {{\bf{R}}^{m\times n}} $ for $ m>n $ a given matrix. The problem is to find a vector $ x \in {{\bf{R}}^n} $ such that $ Ax $ is the best approximation to $ b $.

There are many possible ways of defining the best solution, see \cite{ref2}. A very common choice is to let $ x $ be a solution of the minimization problem
\[\mathop {\min }\limits_x \,\,\left\| {Ax - b\,} \right\|_2^2,\]
where $ \parallel . \parallel_{2} $ denotes the Euclidean vector norm. Now, if one refers to $ r = b - Ax $ as the residual vector, a least squares solution in fact minimizes $ \left\| {r\,} \right\|_2^2 = \sum\limits_{i = 1}^m {r_i^2}  $.

In linear statistical models, one assumes that the vector $ b \in {{\bf{R}}^m} $ of observations is related to the unknown parameters vector $ x \in {{\bf{R}}^n} $ by a linear relation
\begin{equation}\label{eq1.1}
Ax = b + e,
\end{equation}
where $ A \in {{\bf{R}}^{m\times n}} $ is a predetermined matrix of rank $ n $ and $ e $ denotes a vector of random errors. In the standard linear model, the following conditions are known as the basic conditions of Gauss-Markov theorem:
\begin{equation}\label{eq1.2}
E(e) = 0\,\,\,\,\,\,\,{\rm{and}}\,\,\,\,\,\,{\mathop{\rm var}} (e) = {\sigma ^2}{I_n},
\end{equation}
i.e., the random errors $ e_{i} $ are uncorrelated and all have zero means with the same variance in which $ {\sigma ^2} $ is the true error variance.

In summary, the Gauss-Markov theorem expresses that if the linear model \eqref{eq1.1} is available, where $ e $ is a random vector with mean and variance given by \eqref{eq1.2}, then the optimal solution of \eqref{eq1.1} is the least squares estimator, obtained by minimizing the sum of squares $ \,\left\| {Ax - b - e\,} \right\|\,_2^2 $. Furthermore, $ E({s^2}) = {\sigma ^2} $, where
\[ {s^2} = \frac{1}{{m - n}}\left\| {b - Ax - Ae\,} \right\|_2^2. \]
Similar to other available approximate techniques, least squares method contains some constrains and limitations. For instance, in the above mentioned theorem, the conditions \eqref{eq1.2} are necessary.

In this research, we introduce a new fixed variable in the minimization problem to somehow miniaturize the primitive quantity associated with sampling errors. Let us begin with a basic identity which gives rise to an important algebraic inequality, too.

Let ${\bf{S}}  $ denote an inner product space. Knowing that the elements of $ {\bf{S}} $ satisfy four properties of linearity, symmetry, homogeneity and positivity, there is an identity in this space which is directly related to the definition of the projection of two arbitrary elements of $ {\bf{S}} $, i.e. 
\begin{equation}\label{eq1.3}
{\rm{pro}}{{\rm{j}}_{\,{x_2}}}{x_1}\, = \frac{{\left\langle {{x_1},{x_2}} \right\rangle }}{{\left\langle {{x_2},{x_2}} \right\rangle }}\,\,{x_2}\,,
\end{equation}
where $ \left\langle {{x_1},{x_2}} \right\rangle  $ indicates the inner product of $ x_{1} $ and $ x_{2} $ and $ \left\langle {{x_2},{x_2}} \right\rangle  = \left\| {{x_2}} \right\|{\,^2}\, \ge 0 $ denotes the norm square value of $ x_{2} $. In other words, suppose $ x,y,z $ are three specific elements of $ {\bf{S}} $ and $ p \in [0,1] $ is a free parameter. Noting \eqref{eq1.3}, the following identity holds true
\begin{equation}\label{eq1.4}
\left\langle {x - (1 - \sqrt {1 - p} )\,{\rm{pro}}{{\rm{j}}_{\,z}}x\,,\,\,y - (1 - \sqrt {1 - p} )\,{\rm{pro}}{{\rm{j}}_{\,z}}y} \right\rangle  = \left\langle {x,y} \right\rangle  - p\frac{{\left\langle {x,z} \right\rangle \left\langle {y,z} \right\rangle }}{{\left\langle {z,z} \right\rangle }},
\end{equation}
and for $ y=x $ it leads to the Schwarz inequality
\begin{equation}\label{eq1.5}
\left\langle {x,x} \right\rangle  - p\frac{{{{\left\langle {x,z} \right\rangle }^2}}}{{\left\langle {z,z} \right\rangle }} \ge 0\,,\,\,\,\,\,\,\,\forall p \in [0,1].
\end{equation}
This identity can also be presented in mathematical statistics. Essentially as the content of this research is about the least variance of an approximation based on a fixed variable, if we equivalently use the expected value symbol $ E(.) $ instead of the inner product symbol $\left\langle . \right\rangle   $, though there is no difference between them for use, we will get to a new definition of covariance notion as follows:

\begin{definition}\label{def1.1}
p-Covariances with respect to a fixed variable\\
Let $ X, Y $ and $ Z $ be three random variables and $ p \in [0,1] $. With reference to \eqref{eq1.4} we correspondingly define
\begin{align}\label{eq1.6}
{{\mathop{\rm cov}} _p}(X,Y;Z) &= E\left( {\left( {X - (1 - \sqrt {1 - p} )\frac{{E(XZ)}}{{E({Z^2})}}\,Z} \right)\left( {Y - (1 - \sqrt {1 - p} )\frac{{E(YZ)}}{{E({Z^2})}}\,Z} \right)} \right)\\
&= E(XY) - p\frac{{E(XZ)E(YZ)}}{{E({Z^2})}}\,,\nonumber
\end{align}
and call it "p-covariance of $ X $ and $ Y $ with respect to the fixed variable $ Z $".
\end{definition}
Note in \eqref{eq1.6} that 
\[\frac{{E(XZ)}}{{E({Z^2})}}\,Z = {\rm{pro}}{{\rm{j}}_{\,Z}}X\,,\]
and therefore e.g. for $ p=1 $,
\[{{\mathop{\rm cov}} _1}(X,Y;Z) = E\Big( {\left( {X - \,{\rm{pro}}{{\rm{j}}_{\,Z}}X} \right)\left( {Y - \,{\rm{pro}}{{\rm{j}}_{\,Z}}Y} \right)} \Big).\]
If $ p=1 $ and $ Z $ follows a uniform distribution, say $ Z = c^{*}  $ where $ c^{*}  \ne 0 $ is a constant value, then \eqref{eq1.6} will reduce to an ordinary covariance as
\[{{\mathop{\rm cov}} _1}\,(X,Y; c^{*} ) = E(XY) - E(X)E(Y) = {\mathop{\rm cov}} \,(X,Y)\,.\]
Also for $ p=0 $, it follows that the fixed variable $ Z $ has no effect on definition \eqref{eq1.6}.

For $ Y=X $ in \eqref{eq1.6}, an extended definition of the ordinary variance is concluded as
\begin{equation}\label{eq1.7}
{{\mathop{\rm var}} _p}\,(X;Z) = {{\mathop{\rm cov}} _p}(X,X;Z) = E({X^2}) - p\frac{{{E^2}(XZ)}}{{E({Z^2})}}\, \ge 0\,,
\end{equation}
where $ E({Z^2}) $ is to be always positive.

Moreover, for any fixed variable $ Z $ and $ p \in [0,1] $,
\begin{equation}\label{eq1.8}
0 \le {{\mathop{\rm var}} _1}\,(X;Z) \le {{\mathop{\rm var}} _p}\,(X;Z) \le {{\mathop{\rm var}} _0}\,(X;Z) = E({X^2}),
\end{equation}
which is equivalent to
\begin{multline*}
0 \le \left\langle {x - \frac{{\left\langle {x,z} \right\rangle }}{{\left\langle {z,z} \right\rangle }}\,\,z\,,\,\,x - \frac{{\left\langle {x,z} \right\rangle }}{{\left\langle {z,z} \right\rangle }}\,\,z} \right\rangle \\
\,\,\,\, \le \left\langle {x - (1 - \sqrt {1 - p} )\frac{{\left\langle {x,z} \right\rangle }}{{\left\langle {z,z} \right\rangle }}\,\,z\,,\,\,x - (1 - \sqrt {1 - p} )\frac{{\left\langle {x,z} \right\rangle }}{{\left\langle {z,z} \right\rangle }}\,\,z} \right\rangle  \le \left\langle {x,x} \right\rangle \,\,\,\,\,\,\,\,\,\forall z \in \mathbf{S},
\end{multline*}
in an inner product space. Note after simplification, the second part of the latter inequality is in the same form as \eqref{eq1.5}.

Although inequality \eqref{eq1.8} shows that the best option is $ p=1 $, we prefer to apply the parametric case $ p \in [0,1] $ throughout this paper. The reasons for making such a decision will be revealed in forthcoming sections (see Figure 1 and the illustrative example \ref{subsec2.4} in this regard).

The following properties hold true for definitions \eqref{eq1.6} and \eqref{eq1.7}:\\
\\
a1)$\qquad\qquad\qquad\qquad {{\mathop{\rm cov}} _p}(X,Y;Z) = {{\mathop{\rm cov}} _p}(Y,X;Z).$\\
a2)$ \qquad\qquad\qquad\qquad {{\mathop{\rm cov}} _p}\,(\alpha X,\beta Y;Z) = \alpha \beta \,\,{{\mathop{\rm cov}} _p}(X,Y;Z)\,\,\,\,\,\,\,\,\,\,\,(\alpha ,\beta  \in\mathbb{R} ). $\\
a3)$ \qquad\qquad\qquad\qquad {{\mathop{\rm cov}} _p}\,(X + \alpha ,Y + \beta ;Z) =$\\
 $\qquad\qquad{{\mathop{\rm cov}} _p}\,(X,Y;Z) + \alpha \,{{\mathop{\rm cov}} _p}\,(1,Y;Z) + \beta \,\,{{\mathop{\rm cov}} _p}\,(X,1;Z) + \alpha \beta \,\,{{\mathop{\rm cov}} _p}\,(1,1;Z).$\\
a4) $\qquad\qquad\qquad\qquad{{\mathop{\rm var}} _p}\,(X;Z) = 0\,\,{\rm{if}}\,p = 1\,\,{\text {and}}\,\,Z = c^{*} X\,\,(c^{*}  \ne 0).$\\
a5) $ \qquad {{\mathop{\rm cov}} _p}\,(X,Y;c^{*} X) = {{\mathop{\rm cov}} _p}\,(X,Y;c^{*} Y) = (1 - p)E(XY)\,\,\,\,\,\,(c^{*}  \ne 0). $\\
a6) $\quad {{\mathop{\rm cov}} _p}\,\left( {\sum\limits_{k = 0}^n {{c_k}{X_k}} ,{X_m};Z} \right) = \sum\limits_{k = 0}^n {{c_k}{{{\mathop{\rm cov}} }_p}\,({X_k},{X_m};Z)} \,\,\,\,\,\,\,(\{ {c_k}\} _{k = 0}^n \in\mathbb{R} ). $\\
and
\begin{equation}\label{eq1.9}
{\rm{var}}_{p}(\alpha X + \beta Y;Z) = {\alpha ^2}\,{{\mathop{\rm var}} _p}\,(X;Z) + {\beta ^2}{{\mathop{\rm var}} _p}\,(Y;Z) + 2\alpha \beta \,\,{{\mathop{\rm cov}} _p}\,(X,Y;Z).
\end{equation}

\begin{definition}
p-Correlation coefficient with respect to a fixed variable\\
Based on relations \eqref{eq1.6} and \eqref{eq1.7}, we can define a generalization of the Pearson correlation coefficient as 
\begin{equation}\label{eq1.10}
{\rho _p}\,(X,Y;Z) = \frac{{{{{\mathop{\rm cov}} }_p}\,(X,Y;Z)}}{{\sqrt {{{{\mathop{\rm var}} }_p}\,(X;Z){\rm var}_{p}(Y;Z)} }},
\end{equation}
and call it "p-correlation coefficient of $ X $ and $ Y $ with respect to the fixed variable $ Z $".
\end{definition}
It is straightforward to observe that
\[{\rho _p}(X,Y;Z) \in [ - 1,1],\]
because if the values
\[U = X - (1 - \sqrt {1 - p} )\frac{{E(XZ)}}{{E({Z^2})}}\,Z\,\,\,\,\,\,\,{\rm{and}}\,\,\,\,\,\,V = Y - (1 - \sqrt {1 - p} )\frac{{E(YZ)}}{{E({Z^2})}}\,Z,\]
are replaced in the Cauchy-Schwarz inequality \cite{ref18}
\begin{equation}\label{eq1.11}
{E^2}(UV) \le E({U^2})E({V^2}),
\end{equation} 
then
\[{\mathop{\rm cov}} _p^2(X,Y;Z) \le {{\mathop{\rm var}} _p}\,(X;Z)\,\,{\rm var}_{p}(Y;Z).\]
In this sense, note that
\[E(UZ) = \sqrt {1 - p} \,E(XZ)\,\,\,\,\,\,{\rm{and}}\,\,\,\,\,E(VZ) = \sqrt {1 - p} \,E(YZ),\]
which are equal to zero only for the case $ p=1 $.
\\
\noindent
{\textbf{Definition 1.2.1.}}\label{def1.2.1}
p-Normal standard variable with respect to a fixed variable\\
Noting the definitions \eqref{eq1.10} and \eqref{eq1.6}, since 
\[{\rho _p}\,(X,Y;Z) = E\left( {\left( {\frac{{X - (1 - \sqrt {1 - p} )\,{\rm{pro}}{{\rm{j}}_{\,Z}}X}}{{\sqrt {{{{\mathop{\rm var}} }_p}\,(X;Z)} }}} \right)\left( {\frac{{Y - (1 - \sqrt {1 - p} ){\rm{pro}}{{\rm{j}}_{\,Z}}Y}}{{\sqrt {{\rm var}_{p}(Y;Z)} }}} \right)} \right),\]
a p-Normal standard variable, say $ {N_p}(X;Z) $ can be defined with respect to the fixed variable $ Z $ as
\[{N_p}(X;Z) = \frac{{X - (1 - \sqrt {1 - p} )\,{\rm{pro}}{{\rm{j}}_{\,Z}}X}}{{\sqrt {{{{\mathop{\rm var}} }_p}\,(X;Z)} }}.\]
For instance, we have 
\[{N_1}(X;Z = c^{*} ) = \frac{{X - E(X)}}{{\sqrt {{\mathop{\rm var}} \,(X)} }}.\]

\begin{definition}
p-Uncorrelatedness with respect to a fixed variable\\
If in \eqref{eq1.10} $ {\rho _p}(X,Y;Z) = 0 $, which is equivalent to the condition
\[p\,E(XZ)E(YZ) = E(XY)E({Z^2}),\]
then we say that $ X $ and $ Y $ are p-uncorrelated with respect to the fixed variable $ Z $.
\end{definition}
Such a definition can be expressed in an inner product space, too. We say that two elements $ x,y \in {\bf{S}} $ are p-uncorrelated with respect to the fixed element $ z \in {\bf{S}} $ if
\begin{equation}\label{eq1.12}
\left( {x - (1 - \sqrt {1 - p} )\,\,{\rm{pro}}{{\rm{j}}_z}x} \right)\, \bot \,\,\left( {y - (1 - \sqrt {1 - p} )\,\,{\rm{pro}}{{\rm{j}}_z}y} \right),
\end{equation}
or equivalently
\[p\,\left\langle {x,z} \right\rangle \left\langle {y,z} \right\rangle  = \left\langle {x,y} \right\rangle \left\langle {z,z} \right\rangle .\]
As \eqref{eq1.12} shows, $ p=0 $ gives rise to the well-known orthogonality property. In summary, 0-uncorrelatedness gives the orthogonality notion, $ p=1 $ results in a complete uncorrelatedness and $ p \in (0,1) $ leads to an incomplete uncorrelatedness. In this paper, we apply the phrase "complete uncorrelated" instead of 1-uncorrelated.

The aforesaid definition can similarly be expressed in a probability space. We say that two events $ A $ and $ B $ are p-independent with respect to the event $ C $ if
\[p\,{P_r}(\left. {A\,} \right|C)\,{P_r}(\left. {B\,} \right|C) = {P_r}(A \cap B)\,,\]
which is equivalent to
\[p\,{P_r}(A \cap C){P_r}(B \cap C) = {P_r}(A \cap B)P_r^2(C).\]
Hence, e.g. 1-independent means a complete independent with respect to the event $ C $.

\section{Least p-variance approximations based on a fixed variable }\label{sec2}
Let $ \{ {X_k}\} _{k = 0}^n $ and $ Y $ be arbitrary random variables and consider the following approximation in the sequel
\begin{equation}\label{eq2.1}
Y \cong \sum\limits_{k = 0}^n {{c_k}{X_k}} \,,
\end{equation}
in which $ \{ {c_k}\} _{k = 0}^n $ are unknown coefficients to be appropriately determined.

According to \eqref{eq1.7}, the p-variance of the remaining term
\begin{equation}\label{eq2.2}
R({c_0},{c_1},...,{c_n}) = \sum\limits_{k = 0}^n {{c_k}{X_k}}  - Y\,,
\end{equation}
is defined with respect to the fixed variable $ Z $ as
\begin{multline}\label{eq2.3}
{{\mathop{\rm var}} _p}\,\left( {R({c_0},...,{c_n});Z} \right) = E\left( {{{\left( {R({c_0},...,{c_n}) - (1 - \sqrt {1 - p} )\frac{{E(Z\,R({c_0},...,{c_n}))}}{{E({Z^2})}}\,Z} \right)}^2}} \right)\\
 = E\left( {{{\left( {\sum\limits_{k = 0}^n {{c_k}\left( {{X_k} - (1 - \sqrt {1 - p} )\frac{{E(Z{X_k})}}{{E({Z^2})}}Z} \right)}  - (Y\, - (1 - \sqrt {1 - p} )\frac{{E(ZY)}}{{E({Z^2})}}Z)} \right)}^2}} \right),
\end{multline}
where $ \frac{{E(Z\,R({c_0},...,{c_n}))}}{{E({Z^2})}}\,Z $ shows the projection of the error term \eqref{eq2.2} on the fixed variable $ Z $ and $ \frac{{E(Z{X_k})}}{{E({Z^2})}}Z $ the projection of each elements on $ Z $. For the special case $ p=1 $, we have in fact considered \eqref{eq2.3} as
\[{{\mathop{\rm var}} _1}\,\left( {R({c_0},...,{c_n});Z} \right) = E\left( {{{\left( {R({c_0},...,{c_n}) - {\rm{pro}}{{\rm{j}}_z}R({c_0},...,{c_n})} \right)}^2}} \right).\]
We wish to find the unknown coefficients $ \{ {c_k}\} _{k = 0}^n $ in \eqref{eq2.3} such that $ {{\mathop{\rm var}} _p}\,\left( {R({c_0},...,{c_n});Z} \right) $ is minimized. In this direction, it is important to point out that, according to inequality \eqref{eq1.8}, the following inequalities always hold for any arbitrary variable $ Z $:
\begin{multline}\label{eq2.4}
0 \le {{\mathop{\rm var}} _1}\,\left( {R({c_0},...,{c_n});Z} \right) \le {{\mathop{\rm var}} _p}\,\left( {R({c_0},...,{c_n});Z} \right) \\
\le {{\mathop{\rm var}} _0}\,\left( {R({c_0},...,{c_n});Z} \right) = E\left( {{R^2}({c_0},...,{c_n})} \right).
\end{multline}
This means that inequalities \eqref{eq2.4} are valid for any free selection of $ \{ {c_k}\} _{k = 0}^n $ especially when they minimize the quantity \eqref{eq2.3}. In other words, we have
\[\mathop {\min }\limits_{{c_0},...,{c_n}} {{\mathop{\rm var}} _1}\left( {R({c_0},...,{c_n});\,\,Z} \right) \le \mathop {\min }\limits_{{c_0},...,{c_n}} {{\mathop{\rm var}} _p}\left( {R({c_0},...,{c_n});\,\,Z} \right) \le \mathop {\min }\limits_{{c_0},...,{c_n}} E\left( {{R^2}({c_0},...,{c_n})} \right),\]
which shows the superiority of the present theory with respect to the ordinary least squares theory (see examples \ref{subsec2.4} and \ref{ex13.1} in this regard).

To minimize \eqref{eq2.3}, for every $ j = 0,1,\,...\,,\,n $ we have
{\small
\begin{multline}\label{eq2.5}
\frac{{\partial {{{\mathop{\rm var}} }_p}\,(R({c_0},...,{c_n});Z)}}{{\partial {c_j}}} = 0\,\,\, \Rightarrow \\
2E\left( ({X_j} - (1 - \sqrt {1 - p} )\frac{{E({X_j}Z)}}{{E({Z^2})}}Z)\right.\\
\left.\left( {\sum\limits_{k = 0}^n {{c_k}\left( {{X_k} - (1 - \sqrt {1 - p} )\frac{{E({X_k}Z)}}{{E({Z^2})}}Z} \right)}  - (Y\, - (1 - \sqrt {1 - p} )\frac{{E(YZ)}}{{E({Z^2})}}Z)} \right) \right) = 0,
\end{multline}
}
\normalsize
leading to the linear system
\begin{align}\label{eq2.6}
&\left[ {\begin{array}{*{20}{c}}
{\begin{array}{*{20}{c}}
{{{{\mathop{\rm var}} }_p}\,({X_0};Z)}\\
{{{{\mathop{\rm cov}} }_p}\,({X_1},{X_0};Z)}\\
 \vdots \\
{{{{\mathop{\rm cov}} }_p}\,({X_n},{X_0};Z)}
\end{array}}&{\begin{array}{*{20}{c}}
{{{{\mathop{\rm cov}} }_p}\,({X_0},{X_1};Z)}\\
{{{{\mathop{\rm var}} }_p}\,({X_1};Z)}\\
 \vdots \\
{{{{\mathop{\rm cov}} }_p}\,({X_n},{X_1};Z)}
\end{array}}&{\begin{array}{*{20}{c}}
 \cdots \\
 \cdots \\
 \vdots \\
 \cdots 
\end{array}}&{\begin{array}{*{20}{c}}
{{{{\mathop{\rm cov}} }_p}\,({X_0},{X_n};Z)}\\
{{{{\mathop{\rm cov}} }_p}\,({X_1},{X_n};Z)}\\
 \vdots \\
{{{{\mathop{\rm var}} }_p}\,({X_n};Z)}
\end{array}}
\end{array}} \right]\left[ {\begin{array}{*{20}{c}}
{{c_0}}\\
{{c_1}}\\
 \vdots \\
{{c_n}}
\end{array}} \right] \\
&\qquad\qquad\qquad\qquad\qquad\qquad\qquad= \left[ {\begin{array}{*{20}{c}}
{{{{\mathop{\rm cov}} }_p}\,({X_0},Y;Z)}\\
{{{{\mathop{\rm cov}} }_p}\,({X_1},Y;Z)}\\
 \vdots \\
{{{{\mathop{\rm cov}} }_p}\,({X_n},Y;Z)}
\end{array}} \right].\nonumber
\end{align}
Only finding the analytical solution of the above system is sufficient to guarantee that $ {{\mathop{\rm var}} _p}\,\left( {R({c_0},...,{c_n});Z} \right) $ is minimized, because we automatically have 
\[\left\{ {\frac{{{\partial ^2}{{{\mathop{\rm var}} }_p}\,(R({c_0},...,{c_n});Z)}}{{\partial {c_j}^2}}} \right\}_{j = 0}^n \ge 0.\]
For instance, after solving (2.6) for $ n=1 $, the approximation \eqref{eq2.1} takes the form
\begin{multline*}
Y \cong \frac{{{{{\mathop{\rm var}} }_p}\,({X_1};Z){{{\mathop{\rm cov}} }_p}\,({X_0},Y;Z) - {{{\mathop{\rm cov}} }_p}\,({X_0},{X_1};Z){{{\mathop{\rm cov}} }_p}\,({X_1},Y;Z)}}{{{{{\mathop{\rm var}} }_p}\,({X_1};Z){{{\mathop{\rm var}} }_p}\,({X_0};Z) - {\mathop{\rm cov}} _p^2\,({X_0},{X_1};Z)}}\,{X_0}\\
+ \frac{{{{{\mathop{\rm var}} }_p}\,({X_0};Z){{{\mathop{\rm cov}} }_p}\,({X_1},Y;Z) - {{{\mathop{\rm cov}} }_p}\,({X_0},{X_1};Z){{{\mathop{\rm cov}} }_p}\,({X_0},Y;Z)}}{{{{{\mathop{\rm var}} }_p}\,({X_1};Z){{{\mathop{\rm var}} }_p}\,({X_0};Z) - {\mathop{\rm cov}} _p^2\,({X_0},{X_1};Z)}}\,{X_1}\,.
\end{multline*}
In general, two continuous and discrete cases can be considered for the system \eqref{eq2.6}, which we call it from now on "p-covariances linear system with respect to the fixed variable $ Z $".
\subsection{First case of p-covariances linear system} \label{subsec2.1}
If $ \left\{ {{X_k} = {\Phi _k}(x)} \right\}_{k = 0}^n $, $ Y = f(x) $ and $Z = z(x)$ are defined in a continuous space with a probability density function as
\[{P_r}\left( {X = x} \right) = \frac{{w(x)}}{{\int_{\,a}^b {w(x)\,dx} }},\]
for any arbitrary function $ w(x) $ positive on the interval $ [a,b] $, then the components of the linear system \eqref{eq2.6} appear as
\begin{multline}\label{eq2.7}
{{\mathop{\rm cov}} _p}\,\left( {{\Phi _i}(x),{\Phi _j}(x)\,;z(x)} \right) = \\
\frac{1}{{\int_{\,a}^b {w(x)\,dx} }}\left( {\int_{\,a}^b {w(x)\,{\Phi _i}(x)\,{\Phi _j}(x)\,dx}  - p\frac{{\int_{\,a}^b {w(x)\,{\Phi _i}(x)\,z(x)\,dx} \,\int_{\,a}^b {w(x)\,{\Phi _j}(x)\,z(x)\,dx} }}{{\,\int_{\,a}^b {w(x)\,{z^2}(x)\,dx} }}} \right),
\end{multline}
where $ \int_{\,a}^b {w(x)\,{z^2}(x)\,dx}  > 0 $.
\subsection{Second case of p-covariances linear system}\label{subsec2.2}
If the above-mentioned variables are defined on a counter set, say $ {A^*} = \{ {x_k}\} _{k = 0}^m $ with a discrete probability density function as
\[{P_r}\left( {X = x} \right) = \frac{{j(x)}}{{\sum\limits_{x \in {A^*}} {j(x)} }},\]
for any arbitrary function $ j(x) $ positive on $ x \in {A^*} $, then
\begin{multline}\label{eq2.8}
{{\mathop{\rm cov}} _p}\,\left( {{\Phi _i}(x),{\Phi _j}(x)\,;z(x)} \right) = \\
\frac{1}{{\sum\limits_{x \in {A^*}} {j(x)} }}\left( {\sum\limits_{x \in {A^*}} {j(x){\Phi _i}(x)\,{\Phi _j}(x)}  - p\frac{{\sum\limits_{x \in {A^*}} {j(x){\Phi _i}(x)\,z(x)} \,\sum\limits_{x \in {A^*}} {j(x){\Phi _j}(x)\,z(x)} }}{{\,\sum\limits_{x \in {A^*}} {j(x)\,{z^2}(x)} }}} \right),
\end{multline}
where $ \sum\limits_{x \in {A^*}} {j(x)\,{z^2}(x)}  > 0 $.

One of the most important cases of the approximation \eqref{eq2.1} is when $ \left\{ {{X_k} = {\Phi _k}(x) = {x^k}} \right\}_{k = 0}^n $ having a uniform distribution function, which leads to the Hilbert problem \cite{ref2, ref26} in a continuous space and to the polynomial type regression in a discrete space. In other words, if $ \left\{ {{X_k} = {x^k}} \right\}_{k = 0}^n $, $ Y = f(x) $, $ Z=z(x) $ and $ w(x)=1 $ for $ x\in [0,1] $ are substituted into \eqref{eq2.6}, then
\[{{\mathop{\rm cov}} _p}\,({x^i},{x^j};z(x)) = \frac{1}{{i + j + 1}} - p\frac{{\int_0^1 {\,{x^i}z(x)\,dx} \,\int_0^1 {\,{x^j}z(x)\,dx} }}{{\int_0^1 {\,{z^2}(x)\,dx} }}\,\,\,\,\,\,\,\,{\rm for}\,\,\,i,j = 0,1,...,n,\]
and
\[{{\mathop{\rm cov}} _p}\,({x^j},f(x);\,z(x)) = \int_0^1 {\,{x^j}f(x)\,dx}  - p\frac{{\int_0^1 {\,{x^j}z(x)\,dx} \,\int_0^1 {\,f(x)z(x)\,dx} }}{{\int_0^1 {\,{z^2}(x)\,dx} }}\,\,\,\,\,{\rm for}\,\,j = 0,1,...,n.\]
For $ p=0 $, we clearly encounter with the Hilbert problem.

Similarly, if in a discrete space $ \left\{ {{X_k} = {x^k}} \right\}_{k = 0}^n $, $ {A^*} = \{ {x_k}\} _{k = 1}^m $, $ Y = f(x) $ where $ f({x_k}) = {y_k} $, $ Z = z(x) $ and $ j(x)=1 $, the entries of the system \eqref{eq2.6} respectively take the forms
\[{{\mathop{\rm cov}} _p}\,({x^i},{x^j};z(x)) = \frac{1}{m}\sum\limits_{k = 1}^m {{{({x_k})}^{i + j}}} \, - \frac{p}{m}\frac{{\sum\limits_{k = 1}^m {{{({x_k})}^i}z({x_k})} \,\sum\limits_{k = 1}^m {{{({x_k})}^j}z({x_k})} }}{{\,\sum\limits_{k = 1}^m {{z^2}({x_k})} }}\,\,\,\,\,\,{\rm for}\,\,\,i,j = 0,1,...,n,\]
and
\[{{\mathop{\rm cov}} _p}\,({x^j},f(x);\,z(x)) = \frac{1}{m}\sum\limits_{k = 1}^m {{y_k}{{({x_k})}^j}} \, - \frac{p}{m}\frac{{\sum\limits_{k = 1}^m {{{({x_k})}^j}z({x_k})} \,\sum\limits_{k = 1}^m {{y_k}\,z({x_k})} }}{{\,\sum\limits_{k = 1}^m {{z^2}({x_k})} }}\,\,{\rm for}\,\,j = 0,1,...,n.\]
As a sample, the system corresponding to the linear regression $ Y = {c_0}{X_0} + {c_1}{X_1} = {c_0} + {c_1}x $ with respect to the fixed variable $ Z=z(x) $ appears as
\begin{multline*}
\left[ {\begin{array}{*{20}{c}}
{m\sum\limits_{k = 1}^m {{z^2}({x_k})}  - p{{\big(\sum\limits_{k = 1}^m {z({x_k})}\big )^2}},}&{\sum\limits_{k = 1}^m {{x_k}} \sum\limits_{k = 1}^m {{z^2}({x_k})}  - p\sum\limits_{k = 1}^m {z({x_k})} \sum\limits_{k = 1}^m {{x_k}z({x_k})} }\\
{\sum\limits_{k = 1}^m {{x_k}} \sum\limits_{k = 1}^m {{z^2}({x_k})}  - p\sum\limits_{k = 1}^m {z({x_k})} \sum\limits_{k = 1}^m {{x_k}z({x_k})} \,,}&{\sum\limits_{k = 1}^m {x_k^2} \sum\limits_{k = 1}^m {{z^2}({x_k})}  - p{{\big(\sum\limits_{k = 1}^m {{x_k}z({x_k})} \big)^2}}}
\end{array}} \right]\\
 \times\left[ {\begin{array}{*{20}{c}}
{{c_0}}\\
{{c_1}}
\end{array}} \right]
= \left[ {\begin{array}{*{20}{c}}
{\sum\limits_{k = 1}^m {{y_k}} \sum\limits_{k = 1}^m {{z^2}({x_k})}  - p\sum\limits_{k = 1}^m {z({x_k})} \sum\limits_{k = 1}^m {{y_k}z({x_k})} }\\
{\sum\limits_{k = 1}^m {{x_k}{y_k}} \sum\limits_{k = 1}^m {{z^2}({x_k})}  - p\sum\limits_{k = 1}^m {{x_k}z({x_k})} \sum\limits_{k = 1}^m {{y_k}z({x_k})} }
\end{array}} \right].
\end{multline*}
It is interesting to know that if $ z(x) = c^{*}  \ne 0 $ is substituted into the above partially system for $ p \ne 1 $, the output gives the well-known result
\[Y = \frac{{\sum\limits_{k = 1}^m {x_k^2} \sum\limits_{k = 1}^m {{y_k}}  - \sum\limits_{k = 1}^m {{x_k}} \sum\limits_{k = 1}^m {{x_k}{y_k}} }}{{m\sum\limits_{k = 1}^m {x_k^2}  - {{(\sum\limits_{k = 1}^m {{x_k}} )}^2}}} + \frac{{m\sum\limits_{k = 1}^m {{x_k}{y_k}}  - \sum\limits_{k = 1}^m {{x_k}} \sum\limits_{k = 1}^m {{y_k}} }}{{m\sum\limits_{k = 1}^m {x_k^2}  - {{(\sum\limits_{k = 1}^m {{x_k}} )}^2}}}x\,,\]
i.e. $ p \ne 1 $ has been automatically deleted. In this sense, if the only case $ p=1 $ is substituted into the above mentioned system for 
$ z(x) = c^{*}  \ne 0 $ as
\[\left[ {\begin{array}{*{20}{c}}
{0,}&0\\
{0\,,}&{m\sum\limits_{k = 1}^m {x_k^2}  - {{\big(\sum\limits_{k = 1}^m {{x_k}} \big)^2}}}
\end{array}} \right]\left[ {\begin{array}{*{20}{c}}
{{c_0}}\\
{{c_1}}
\end{array}} \right] = \left[ {\begin{array}{*{20}{c}}
0\\
{m\sum\limits_{k = 1}^m {{x_k}{y_k}}  - \sum\limits_{k = 1}^m {{x_k}} \sum\limits_{k = 1}^m {{y_k}} }
\end{array}} \right],\]
the output gives the result
\[Y = {c_0} + \frac{{m\sum\limits_{k = 1}^m {{x_k}{y_k}}  - \sum\limits_{k = 1}^m {{x_k}} \sum\limits_{k = 1}^m {{y_k}} }}{{m\sum\limits_{k = 1}^m {x_k^2}  - {{(\sum\limits_{k = 1}^m {{x_k}} )}^2}}}x\,,\]
where $ c_{0} $ is a free parameter.

In the next section, we show that why $ c_{0} $ became a free parameter when $ z(x) = c^{*}  \ne 0 $ and $ p=1 $.

\subsection{Some reducible cases of the p-covariances linear system}\label{subsec2.3}

\subsubsection{First case.}\label{subsubsec2.3.1}
Suppose there exists a specific distribution for the fixed variable $ Z $ such that
\begin{equation}\label{eq2.9}
E(Z{X_k}) = 0\qquad \text{for every}\qquad k = 0,1,...,n.
\end{equation}
Then, the p-covariances system \eqref{eq2.6} reduces to the well-known normal equations system \cite{ref2}
\begin{equation}\label{eq2.10}
\left[ {\begin{array}{*{20}{c}}
{\begin{array}{*{20}{c}}
{E\,(X_0^2)}\\
{E({X_1}{X_0})}\\
 \vdots \\
{E({X_n}{X_0})}
\end{array}}&{\begin{array}{*{20}{c}}
{E({X_0}{X_1})}\\
{E\,(X_1^2)}\\
 \vdots \\
{E({X_n}{X_1})}
\end{array}}&{\begin{array}{*{20}{c}}
 \cdots \\
 \cdots \\
 \vdots \\
 \cdots 
\end{array}}&{\begin{array}{*{20}{c}}
{E({X_0}{X_n})}\\
{E({X_1}{X_n})}\\
 \vdots \\
{E\,(X_n^2)}
\end{array}}
\end{array}} \right]\left[ {\begin{array}{*{20}{c}}
{{c_0}}\\
{{c_1}}\\
 \vdots \\
{{c_n}}
\end{array}} \right] = \left[ {\begin{array}{*{20}{c}}
{E({X_0}Y)}\\
{E({X_1}Y)}\\
 \vdots \\
{E({X_n}Y)}
\end{array}} \right],
\end{equation}
and
\begin{equation}\label{eq2.11}
\mathop {\min }\limits_{{c_0},...,{c_n}} {{\mathop{\rm var}} _p}\Big( { {R({c_0},...,{c_n});\,\,Z\,} \Big|\,\,\{ E(Z{X_k})\} _{k = 0}^n = 0\,} \Big) = \mathop {\min }\limits_{{c_0},...,{c_n}} E\left( {{R^2}({c_0},...,{c_n})} \right).
\end{equation}
Also, it is obvious for $ p=0 $ that
\[\mathop {\min }\limits_{{c_0},...,{c_n}} {{\mathop{\rm var}} _0}\left( {R({c_0},...,{c_n});\,\,Z} \right) = \mathop {\min }\limits_{{c_0},...,{c_n}} E\left( {{R^2}({c_0},...,{c_n})} \right).\]

\subsubsection{Second case.}\label{subsubsec2.3.2}
If $ Z $ follows a uniform distribution (say $ Z =c^{*}  \ne 0 $) and $ {X_0} = 1 $, the condition \eqref{eq2.9} is no longer valid for $ k=0 $ because $ c^{*} E(1) \ne 0 $. Therefore, noting that $ {{\mathop{\rm cov}} _p}\,(X,Y;c^{*} ) = E(XY) - p\,E(X)E(Y) $, the linear system \eqref{eq2.6} reduces to
\begin{multline}\label{eq2.12}
\left[ {\begin{array}{*{20}{c}}
{\begin{array}{*{20}{c}}
{1 - p}\\
{(1 - p)E({X_1})}\\
 \vdots \\
{(1 - p)E({X_n})}
\end{array}}&{\begin{array}{*{20}{c}}
{(1 - p)E({X_1})}\\
{E\,(X_1^2) - p\,{E^2}({X_1})}\\
 \vdots \\
{E({X_n}{X_1}) - pE({X_n})E({X_1})}
\end{array}}&{\begin{array}{*{20}{c}}
 \cdots \\
 \cdots \\
 \vdots \\
 \cdots 
\end{array}}&{\begin{array}{*{20}{c}}
{(1 - p)E({X_n})}\\
{E({X_1}{X_n}) - pE({X_1})E({X_n})}\\
 \vdots \\
{E\,(X_n^2) - p{E^2}({X_n})}
\end{array}}
\end{array}} \right]\\
\times \left[ {\begin{array}{*{20}{c}}
{{c_0}}\\
{{c_1}}\\
 \vdots \\
{{c_n}}
\end{array}} \right]
= \left[ {\begin{array}{*{20}{c}}
{(1 - p)E(Y)}\\
{E({X_1}Y) - pE({X_1})E(Y)}\\
 \vdots \\
{E({X_n}Y) - pE({X_n})E(Y)}
\end{array}} \right],
\end{multline}
Relation \eqref{eq2.12} shows that for $ p=1 $, $ c_{0} $ is a free parameter. For $ {X_0} = 1 $, the approximation \eqref{eq2.1} takes the simplified form
\[Y \cong {c_0} + \sum\limits_{k = 1}^n {{c_k}{X_k}} \,,\]
and replacing it in the normal system \eqref{eq2.10} yields
\begin{equation}\label{eq2.13}
\left[ {\begin{array}{*{20}{c}}
{\begin{array}{*{20}{c}}
1\\
{E({X_1})}\\
 \vdots \\
{E({X_n})}
\end{array}}&{\begin{array}{*{20}{c}}
{E({X_1})}\\
{E\,(X_1^2)}\\
 \vdots \\
{E({X_n}{X_1})}
\end{array}}&{\begin{array}{*{20}{c}}
 \cdots \\
 \cdots \\
 \vdots \\
 \cdots 
\end{array}}&{\begin{array}{*{20}{c}}
{E({X_n})}\\
{E({X_1}{X_n})}\\
 \vdots \\
{E\,(X_n^2)}
\end{array}}
\end{array}} \right]\left[ {\begin{array}{*{20}{c}}
{{c_0}}\\
{{c_1}}\\
 \vdots \\
{{c_n}}
\end{array}} \right] = \left[ {\begin{array}{*{20}{c}}
{E(Y)}\\
{E({X_1}Y)}\\
 \vdots \\
{E({X_n}Y)}
\end{array}} \right].
\end{equation}
Now, the important point is that after some elementary operations, the system \eqref{eq2.13} will exactly be transformed to the linear system \eqref{eq2.12}. This means that for any arbitrary $ p \in [0,1] $ we have
\begin{multline}\label{eq2.14}
\,\,\,\,\,\,\,\,\,\,\,\,\,\,\,\,\,\,\mathop {\min }\limits_{{c_1},...,{c_n}} {{\mathop{\rm var}} _p}\Big( { {R({c_0},...,{c_n});Z\,} \Big|\,\,Z =c^{*} \,\,{\rm{and}}\,\,{X_0} = 1} \Big)\\
 = \mathop {\min }\limits_{{c_1},...,{c_n}} E\left( {{R^2}({c_0},...,{c_n}) \, \Big|{X_0} = 1\,\,{\rm{and}}\,\,{c_0} = E(Y) - \sum\limits_{k = 1}^n {{c_k}E({X_k})} } \right).
\end{multline}
Such a result in \eqref{eq2.14} can similarly be proved for every $ Z = {X_k}$ where $ k = 1,2,..,n$, as the following example approves it.

\subsection{An illustrative example and the role of the fixed variable in it}\label{subsec2.4}
The previous sections \ref{subsec2.1}, \ref{subsec2.2} and \ref{subsec2.3} affirm that the fixed variable $ Z $ has a basic role in the theory of least p-variances. Here we present an illustrative example to show the importance of such a fixed function in constituting the initial approximation \eqref{eq2.1}. Let $ Y = \sqrt {1 - x}  $ be defined on $ [0,1] $ together with the probability density function $ w(x)=1 $ and the fixed function $ z(x) = {x^\lambda } $ where $ \lambda  >  - 1/2 $ because $ \int_{\,0}^1 {w(x)\,{z^2}(x)\,dx}  = \frac{1}{{2\lambda  + 1}} $. For the basis functions $ \left\{ {{X_k} = {x^k}} \right\}_{k = 0}^2 $, the initial approximation \eqref{eq2.1} takes the form
\begin{equation}\label{eq2.15}
\sqrt {1 - x}  \cong {c_0}(\lambda ,p) + {c_1}(\lambda ,p)\,x + {c_2}(\lambda ,p)\,{x^2} = {{\bf{A}}_2}(x;\lambda ,p),
\end{equation}
in which the unknown coefficients satisfy the following linear system according to \eqref{eq2.6}:
\begin{align}\label{eq2.16}
&\left[ {\begin{array}{*{20}{c}}
{1 - \frac{{2\lambda  + 1}}{{{{(\lambda  + 1)}^2}}}p,}&{\frac{1}{2} - \frac{{2\lambda  + 1}}{{(\lambda  + 1)(\lambda  + 2)}}p,}&{\frac{1}{3} - \frac{{2\lambda  + 1}}{{(\lambda  + 1)(\lambda  + 3)}}p}\\
{\frac{1}{2} - \frac{{2\lambda  + 1}}{{(\lambda  + 1)(\lambda  + 2)}}p,}&{\frac{1}{3} - \frac{{2\lambda  + 1}}{{{{(\lambda  + 2)}^2}}}p,}&{\frac{1}{4} - \frac{{2\lambda  + 1}}{{(\lambda  + 2)(\lambda  + 3)}}p}\\
{\frac{1}{3} - \frac{{2\lambda  + 1}}{{(\lambda  + 1)(\lambda  + 3)}}p,}&{\frac{1}{4} - \frac{{2\lambda  + 1}}{{(\lambda  + 2)(\lambda  + 3)}}p,}&{\frac{1}{5} - \frac{{2\lambda  + 1}}{{{{(\lambda  + 3)}^2}}}p}
\end{array}} \right]\left[ {\begin{array}{*{20}{c}}
{{c_0}(\lambda ,p)}\\
{{c_1}(\lambda ,p)}\\
{{c_2}(\lambda ,p)}
\end{array}} \right]\\
&\qquad\qquad\qquad\qquad = \left[ {\begin{array}{*{20}{c}}
{\frac{2}{3} - \frac{{2\lambda  + 1}}{{\lambda  + 1}}B(\lambda  + 1,\frac{3}{2})\,p}\\
{\frac{4}{{15}} - \frac{{2\lambda  + 1}}{{\lambda  + 2}}B(\lambda  + 1,\frac{3}{2})\,p}\\
{\frac{{16}}{{105}} - \frac{{2\lambda  + 1}}{{\lambda  + 3}}B(\lambda  + 1,\frac{3}{2})\,p}
\end{array}} \right],\nonumber
\end{align} 
where
\[B(a,b) = \int_{\,0}^1 {{x^{a - 1}}{{(1 - x)}^{b - 1}}dx}  = \frac{{\Gamma (a)\Gamma (b)}}{{\Gamma (a + b)}}\,\,\,\,\,\,\,(a,b > 0).\]
After solving \eqref{eq2.16}, $ {{\bf{A}}_2}(x;\lambda ,p) $ will be determined and since the remaining term of \eqref{eq2.15} is defined as
\[{{\bf{R}}_2}(x;\lambda ,p) = {{\bf{A}}_2}(x;\lambda ,p) - \sqrt {1 - x} ,\]
so
\[{{\mathop{\rm var}} _p}\left( {{{\bf{R}}_2}(x;\lambda ,p);{x^\lambda }} \right) = E\left( {{{\left( {{{\bf{R}}_2}(x;\lambda ,p)} \right)}^2}} \right) - p\,(2\lambda  + 1){E^2}\left( {{x^\lambda }{{\bf{R}}_2}(x;\lambda ,p)} \right).\]
However, there are two different cases for the parameter $ \lambda  >  - 1/2 $ which should be studied separately.

\subsubsection{First case.}
If $ \lambda  = 0,1,2 $, then the fixed functions $ z(x) = 1,x,{x^2} $ coincide respectively with the first, second and third basis functions in \eqref{eq2.15} and the system \eqref{eq2.16} is simplified for $ \lambda  = 0 $ as
\[\left[ {\begin{array}{*{20}{c}}
{1 - p,}&{\frac{1}{2} - \frac{1}{2}p,}&{\frac{1}{3} - \frac{1}{3}p}\\
{\frac{1}{2} - \frac{1}{2}p,}&{\frac{1}{3} - \frac{1}{4}p,}&{\frac{1}{4} - \frac{1}{6}p}\\
{\frac{1}{3} - \frac{1}{3}p,}&{\frac{1}{4} - \frac{1}{6}p,}&{\frac{1}{5} - \frac{1}{9}p}
\end{array}} \right]\left[ {\begin{array}{*{20}{c}}
{{c_0}(0,p)}\\
{{c_1}(0,p)}\\
{{c_2}(0,p)}
\end{array}} \right] = \left[ {\begin{array}{*{20}{c}}
{\frac{2}{3} - \frac{2}{3}\,p}\\
{\frac{4}{{15}} - \frac{1}{3}p}\\
{\frac{{16}}{{105}} - \frac{2}{9}\,p}
\end{array}} \right],\]
and for $ \lambda=1 $ as
\[\left[ {\begin{array}{*{20}{c}}
{1 - \frac{3}{4}p,}&{\frac{1}{2} - \frac{1}{2}p,}&{\frac{1}{3} - \frac{3}{8}p}\\
{\frac{1}{2} - \frac{1}{2}p,}&{\frac{1}{3} - \frac{1}{3}p,}&{\frac{1}{4} - \frac{1}{4}p}\\
{\frac{1}{3} - \frac{3}{8}p,}&{\frac{1}{4} - \frac{1}{4}p,}&{\frac{1}{5} - \frac{3}{{16}}p}
\end{array}} \right]\left[ {\begin{array}{*{20}{c}}
{{c_0}(1,p)}\\
{{c_1}(1,p)}\\
{{c_2}(1,p)}
\end{array}} \right] = \left[ {\begin{array}{*{20}{c}}
{\frac{2}{3} - \frac{2}{5}\,p}\\
{\frac{4}{{15}} - \frac{4}{{15}}p}\\
{\frac{{16}}{{105}} - \frac{1}{5}\,p}
\end{array}} \right],\]
and finally for $ \lambda=2 $ as
\[\left[ {\begin{array}{*{20}{c}}
{1 - \frac{5}{9}p,}&{\frac{1}{2} - \frac{5}{{12}}p,}&{\frac{1}{3} - \frac{1}{3}p}\\
{\frac{1}{2} - \frac{5}{{12}}p,}&{\frac{1}{3} - \frac{5}{{16}}p,}&{\frac{1}{4} - \frac{1}{4}p}\\
{\frac{1}{3} - \frac{1}{3}p,}&{\frac{1}{4} - \frac{1}{4}p,}&{\frac{1}{5} - \frac{1}{5}p}
\end{array}} \right]\left[ {\begin{array}{*{20}{c}}
{{c_0}(2,p)}\\
{{c_1}(2,p)}\\
{{c_2}(2,p)}
\end{array}} \right] = \left[ {\begin{array}{*{20}{c}}
{\frac{2}{3} - \frac{{16}}{{63}}\,p}\\
{\frac{4}{{15}} - \frac{4}{{21}}p}\\
{\frac{{16}}{{105}} - \frac{{16}}{{105}}\,p}
\end{array}} \right].\]
According to the result \eqref{eq2.14}, the solutions of all above systems must be the same and independent of $ p $, so that in the final form we have
\[{{\bf{A}}_2}(x;0,p) = {{\bf{A}}_2}(x;1,p) = {{\bf{A}}_2}(x;2,p) = {{\bf{A}}_2}(x;\lambda ,0) =  - \frac{4}{7}\,{x^2} - \frac{8}{{35}}x + \frac{{34}}{{35}}.\]
Note for $ p=1 $, after solving \eqref{eq2.16} we get
\[{{\bf{A}}_2}(x;0,1) =  - \frac{4}{7}\,{x^2} - \frac{8}{{35}}x + {c_0}(0,1),\]
in which $ {c_0}(0,1) = E\left( {\sqrt {1 - x} } \right) + \frac{4}{7}\,E({x^2}) + \frac{8}{{35}}E(x) = \frac{{34}}{{35}} $ according to \eqref{eq2.14} again.

In this sense,
\[E\left( {{{\bf{R}}_2}(x;0,1)} \right) = E\left( {{{\bf{A}}_2}(x;0,1) - \sqrt {1 - x} } \right) = 0,\]
implies that
\begin{multline*}
{{\mathop{\rm var}} _1}\left( {{{\bf{R}}_2}(x;0,1);z(x) = 1} \right) = E\left( {{{\left( {{{\bf{R}}_2}(x;0,1)} \right)}^2}} \right)\\
 = {\int_{\,0}^1 {\left( { - \frac{4}{7}\,{x^2} - \frac{8}{{35}}x + \frac{{34}}{{35}} - \sqrt {1 - x} } \right)} ^2}dx = \frac{1}{{2450}} \cong 0.000408.
\end{multline*}

\subsubsection{Second case.}
Any assumption other than $ \lambda=0,1,2 $ causes the system \eqref{eq2.16} to have a unique solution. For example, $ \lambda  = 1/2 $ simplifies \eqref{eq2.16} in the form
\begin{equation}\label{eq2.17}
\left[ {\begin{array}{*{20}{c}}
{1 - \frac{8}{9}p,}&{\frac{1}{2} - \frac{8}{{15}}p,}&{\frac{1}{3} - \frac{8}{{21}}p}\\
{\frac{1}{2} - \frac{1}{2}p,}&{\frac{1}{3} - \frac{8}{{25}}p,}&{\frac{1}{4} - \frac{8}{{35}}p}\\
{\frac{1}{3} - \frac{1}{3}p,}&{\frac{1}{4} - \frac{8}{{35}}p,}&{\frac{1}{5} - \frac{8}{{49}}p}
\end{array}} \right]\left[ {\begin{array}{*{20}{c}}
{{c_0}(\frac{1}{2},p)}\\
{{c_1}(\frac{1}{2},p)}\\
{{c_2}(\frac{1}{2},p)}
\end{array}} \right] = \left[ {\begin{array}{*{20}{c}}
{\frac{2}{3} - \frac{1}{6}\,\pi p}\\
{\frac{4}{{15}} - \frac{1}{{10}}\pi p}\\
{\frac{{16}}{{105}} - \frac{1}{{14}}\,\pi p}
\end{array}} \right].
\end{equation}
After solving the system \eqref{eq2.17} we obtain
\[{{\bf{A}}_2}(x;\frac{1}{2},p) =  - \frac{7}{3}\,\frac{{(75\pi  + 64)p - 300}}{{1224p - 1225}}{x^2} + 20\frac{{(21\pi  - 80)p + 14}}{{1224p - 1225}}x + \frac{1}{2}\frac{{(105\pi  + 2048)p - 2380}}{{1224p - 1225}}.\]
For the special case $ p=1 $, i.e.
\[{{\bf{A}}_2}(x;\frac{1}{2},1) =  - (\frac{{1652}}{3} - 175\,\pi )\,{x^2} + (1320 - 420\,\pi )\,x + 166 - \frac{{105}}{2}\,\pi ,\]
we have
\[E\left( {{x^{\frac{1}{2}}}{{\bf{R}}_2}(x;\frac{1}{2},1)} \right) =  - \frac{2}{7}(\frac{{1652}}{3} - 175\,\pi )\, + \frac{2}{5}(1320 - 420\,\pi )\, + \frac{2}{3}(166 - \frac{{105}}{2}\,\pi ) - \frac{\pi }{8},\]
and consequently
\[{{\mathop{\rm var}} _1}\left( {{{\bf{R}}_2}(x;\frac{1}{2},1);{x^{\frac{1}{2}}}} \right) = E\left( {{{\Big({{\bf{R}}_2}(x;\frac{1}{2},1)\Big)^2}}} \right) - 2{E^2}\left( {{x^{\frac{1}{2}}}\,{{\bf{R}}_2}(x;\frac{1}{2},1)} \right) \cong 0.000388.\]
As we observe
\[{{\mathop{\rm var}} _1}\left( {{{\bf{R}}_2}(x;\frac{1}{2},1);{x^{\frac{1}{2}}}} \right) < {{\mathop{\rm var}} _1}\left( {{{\bf{R}}_2}(x;0,1);1} \right) = E\left( {{{\left( {{{\bf{R}}_2}(x;0,1)} \right)}^2}} \right),\]
which clearly shows the role of the fixed function in the obtained approximations.

\subsection{How to choose a suitable option for the fixed variable: A geometric interpretation}\label{subsec2.5}
Although inequality \eqref{eq2.16} holds for any arbitrary variable $ Z $, our wish is to determine some conditions to find a suitable option for the fixed variable $ Z $ such that
\begin{equation}\label{eq2.18}
{{\mathop{\rm var}} _1}\,(X;Z) \ll E({X^2}),
\end{equation}
and/or
\begin{equation}\label{eq2.19}
\frac{{{E^2}(XZ)}}{{E({Z^2})}} \gg 0\,.
\end{equation}
Various states can be considered for the above-mentioned goal. For example, replacing $ Z =c^{*}  \ne 0 $ in \eqref{eq2.19} yields $ {E^2}(X) \gg 0\, $ which means if the magnitude of the mean value $ E(X) \ne 0 $ is very big, we expect that the presented theory based on the fixed variable $ Z =c^{*}  \ne 0 $ acts much better than the ordinary least squares theory.

Another approach is to directly minimize the value $ {{\mathop{\rm var}} _1}\,(X;Z) $ in \eqref{eq2.18} such that
\[X - \frac{{E(XZ)}}{{E({Z^2})}}\,Z \to 0\,.\]
The following figure describes our aim in a vector space. 
\begin{figure}[ht]
\centering
{\resizebox*{12cm}{!}{\includegraphics{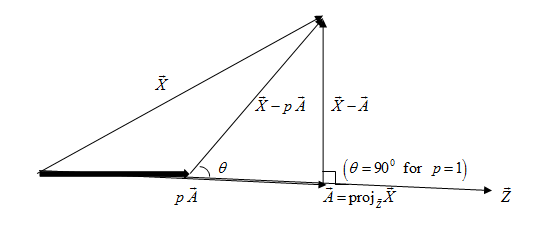}}}
\caption{Least p-variances based on the fixed vector $ \vec Z $ and the role of $ p $ from 0 to 1} \label{fig1}
\end{figure}

In this figure, $ \left\| {\vec X} \right\|_{\,2}^{\,2} = \sum\limits_{k = 1}^m {x_k^2}  $ leads to the same as ordinary least squares for $ p=0 $, $ \left\| {\vec X - \vec A} \right\|_{\,2}^{\,2} = \left\| {\vec X - {\rm{pro}}{{\rm{j}}_{\,\vec Z}}\vec X\,} \right\|_{\,2}^{\,2} $ leads to the complete type of least variances based on $ \vec{Z} $ and $ p=1 $ and finally $ \left\| {\vec X - p\,\vec A} \right\|_{\,2}^{\,2} $ leads to the least p-variances based on $ \vec{Z} $ and $ p \in [0,1] $.

\section{p-uncorrelatedness condition on the p-covariances linear system and its consequences}\label{sec3}
Since the coefficients matrix of the system \eqref{eq2.6} is symmetric and all its diagonal entries are positive, the aforesaid system is solvable and its unknown coefficients are computable. However, as we observed in the previous example, analytically solving such linear systems is difficult from computational point of view. To resolve this problem, we can apply the condition
\begin{equation}\label{eq3.1}
{{\mathop{\rm cov}} _p}\,({X_i},{X_j};Z) = {{\mathop{\rm var}} _p}\,({X_j};Z)\,{\delta _{i,j}}\,\,\,{\rm{for}}\,{\rm{every}}\,\,\,i,j = 0,1,...,n\,\,\,{\rm{with}}\,\,\,{\delta _{i,j}} = \left\{ \begin{array}{l}
0\,\,\,\,\,(i \ne j),\\
1\,\,\,\,\,\,(i = j),
\end{array} \right.
\end{equation}
on the elements of the system \eqref{eq2.6} to easily obtain the unknown coefficients in the form
\[{c_k} = \frac{{{{{\mathop{\rm cov}} }_p}\,({X_k},Y;Z)}}{{{{{\mathop{\rm var}} }_p}\,({X_k};Z)}}.\]
In this case
\begin{equation}\label{eq3.2}
Y \cong \sum\limits_{k = 0}^n {\frac{{{{{\mathop{\rm cov}} }_p}({X_k},Y;Z)}}{{{{{\mathop{\rm var}} }_p}\,({X_k};Z)}}{X_k}} \,,
\end{equation}
would be the best approximation in the sense of least p-variance of the error with respect to the fixed variable $ Z $.
\begin{theorem}\label{thm3.1}
Any finite set of random variables satisfying the p-uncorrelatedness condition \eqref{eq3.1} is linearly independent.
\end{theorem}
\begin{proof}
Assume that $ \sum\limits_{k = 0}^n {{a_k}{X_k}}  = 0 $ where $ \{ {a_k}\} _{k = 0}^n $ are not all zero and $ \{ {X_k}\} _{k = 0}^n $ are p-uncorrelated variables satisfying \eqref{eq3.1}. Then, applying $ {{\mathop{\rm cov}} _p}\,(.\,,{X_j};Z) $ on both sides of this assumption yields 
\[0 = {{\mathop{\rm cov}} _p}\left( {\sum\limits_{k = 0}^n {{a_k}{X_k}} ,{X_j};Z} \right) = \sum\limits_{k = 0}^n {{a_k}{{{\mathop{\rm cov}} }_p}\,({X_k},{X_j};Z)}  = {a_j}{{\mathop{\rm var}} _p}\,({X_j};Z),\]
which implies that $ {a_j} = 0 $ for $ j=0,1,\ldots,n $, i.e. a contradiction.
\end{proof}

The question is now how to find linear combinations that are p-uncorrelated with respect to the variable $ Z $. There is a basic theorem in this direction similar to the Gram-Schmidt orthogonalization theorem \cite{ref5}.

\begin{theorem}\label{thm3.2}
Let $ \{ {V_k}\} _{k = 0}^{n} $ be a finite or infinite sequence of random variables such that any finite number of elements $ \{ {V_k}\} _{k = 0}^{n} $ are linearly independent. One can find constants $ \{ {a_{i,j}}\}  $ such that the elements  
\begin{equation}\label{eq3.3}
\begin{array}{l}
{X_0} = {V_0},\\
{X_1} = {V_1} + {a_{12}}{V_0},\\
{X_2} = {V_2} + {a_{22}}{V_1} + {a_{23}}{V_0},\\
\,\,\,\,\,\,\,\,\,\,\,\,\,\,\,\,\,\,\,\,\,\,\,\,\, \vdots \\
{X_n} = {V_n} + {a_{n2}}{V_{n - 1}} + ... + {a_{n,n + 1}}{V_0},
\end{array}
\end{equation}
are mutually p-uncorrelated with respect to the fixed variable $ Z $
\end{theorem}
\begin{proof}
Let us set recursively 
\begin{equation}\label{eq3.4}
\begin{array}{l}
{X_0} = {V_0},\\
{X_1} = {V_1} - \frac{{{{{\mathop{\rm cov}} }_p}\,({X_0},{V_1};Z)}}{{{{{\mathop{\rm var}} }_p}\,({X_0};Z)}}{X_0},\\
\,\,\,\,\,\,\,\,\,\,\,\,\,\,\,\,\,\,\,\,\,\,\,\,\, \vdots \\
{X_n} = {V_n} - \sum\limits_{k = 0}^{n - 1} {\frac{{{{{\mathop{\rm cov}} }_p}\,({X_k},{V_n};Z)}}{{{{{\mathop{\rm var}} }_p}\,({X_k};Z)}}{X_k}} .
\end{array}
\end{equation}
As $ X_{n} $ is a linear combination of $ \{ {V_k}\} _{k = 0}^n $ in \eqref{eq3.4}, it is enough to show that $ X_{n} $ is p-uncorrelated to $ \{ {X_j}\} _{j = 0}^{n - 1} $ with respect to the variable $ Z $. For $ j=0,1,\ldots,n-1 $ we have
\begin{align*}
{{\mathop{\rm cov}} _p}\,({X_n},{X_j};Z) &= {{\mathop{\rm cov}} _p}\,\left( {{V_n} - \sum\limits_{k = 0}^{n - 1} {\frac{{{{{\mathop{\rm cov}} }_p}\,({X_k},{V_n};Z)}}{{{{{\mathop{\rm var}} }_p}\,({X_k};Z)}}{X_k}} ,{X_j};Z} \right)\\
& = {{\mathop{\rm cov}} _p}\,({V_n},{X_j};Z) - \sum\limits_{k = 0}^{n - 1} {\frac{{{{{\mathop{\rm cov}} }_p}\,({X_k},{V_n};Z)}}{{{{{\mathop{\rm var}} }_p}\,({X_k};Z)}}{{{\mathop{\rm cov}} }_p}\,({X_k},{X_j};Z)} \\
&= {{\mathop{\rm cov}} _p}\,({V_n},{X_j};Z) - {{\mathop{\rm cov}} _p}\,({X_j},{V_n};Z) = 0.
\end{align*}
\end{proof}
In fact, theorem \ref{thm3.2} is a generalization of Gram-Schmidt orthogonalization process when $ p=0 $ or the condition $ E(Z{V_k}) = 0 $ holds for every $ k = 0,1,...,n $.
\begin{theorem}\label{thm3.3}
The reverse proposition of the previous theorem is: There are constants $ \{ {b_{i,j}}\}  $  such that 
\begin{equation}\label{eq3.5}
\begin{array}{l}
{V_0} = {X_0},\\
{V_1} = {X_1} + {b_{12}}{X_0},\\
{V_2} = {X_2} + {b_{22}}{X_1} + {b_{23}}{X_0},\\
\,\,\,\,\,\,\,\,\,\,\,\,\,\,\,\,\,\,\,\,\,\,\,\,\, \vdots \\
{V_n} = {X_n} + {b_{n2}}{X_{n - 1}} + ... + {b_{n,n + 1}}{X_0},
\end{array}
\end{equation}
and
\[{{\mathop{\rm cov}} _p}\,({X_n},{V_k};Z) = 0\,\,\,\,\,\,\,\,{\rm{for}}\,\,\,k = 0,1,...,n - 1,\]
provided that
\begin{equation}\label{eq3.6}
{{\mathop{\rm cov}} _p}\,({X_i},{X_j};Z) = 0\,\,\,\,\,\,\,{\rm{for}}\,\,\,\,\,i \ne j.
\end{equation}
\end{theorem}

\begin{proof}
By virtue of the fact that the system \eqref{eq3.3} is invertible, the general element of \eqref{eq3.5} in the form
\[{V_k} = {X_k} + {b_{k2}}{X_{k - 1}} + ... + {b_{k,k + 1}}{X_0},\]
gives
\begin{align*}
{{\mathop{\rm cov}} _p}\,({X_n},{V_k};Z) &= {{\mathop{\rm cov}} _p}\,({X_n},{X_k} + \sum\limits_{j = 2}^{k + 1} {{b_{k,j}}{X_{k + 1 - j}}} ;Z)\\
&= {{\mathop{\rm cov}} _p}\,({X_n},{X_k};Z) + \sum\limits_{j = 2}^{k + 1} {{b_{k,j}}{{{\mathop{\rm cov}} }_p}\,({X_n},{X_{k + 1 - j}};Z)}  = 0,
\end{align*}
which is valid for every $ k=0,1,\ldots, n-1 $ according to the condition \eqref{eq3.6}.
\end{proof}

\subsection{A general representation for p-uncorrelated variables}\label{subsec3.4}
Both previous theorems \ref{thm3.2} and \ref{thm3.3} help us now obtain a general representation for p-uncorrelated variables with respect to the variable $ Z $. Assume that $ \{ {V_k}\} _{k = 0}^n $ is an arbitrary independent sequence of random variables and $ \{ {X_k}\} _{k = 0}^n $ satisfy the condition \eqref{eq3.1} as before. Noting the last row of \eqref{eq3.4} and this fact from \eqref{eq3.3} that $ {X_n} = \sum\limits_{k = 0}^n {{a_{n,n + 1 - k}}\,{V_k}}  $ where $ {a_{n,1}} = 1 $, for $ m \le n $ we have
\begin{align}\label{eq3.7}
&{{\mathop{\rm cov}} _p}\,({X_n},{X_m};Z) = {{\mathop{\rm cov}} _p}\left( {{X_n},{V_m} - \sum\limits_{j = 0}^{m - 1} {\frac{{{{{\mathop{\rm cov}} }_p}\,({X_j},{V_m};Z)}}{{{{{\mathop{\rm var}} }_p}\,({X_j};Z)}}{X_j}} ;Z} \right)\\
&\qquad\qquad= {{\mathop{\rm cov}} _p}\,({X_n},{V_m};Z) - \sum\limits_{j = 0}^{m - 1} {\frac{{{{{\mathop{\rm cov}} }_p}\,({X_j},{V_m};Z)}}{{{{{\mathop{\rm var}} }_p}\,({X_j};Z)}}\,{{{\mathop{\rm cov}} }_p}\left( {{X_n},{X_j};Z} \right)} \nonumber\\
&\qquad\qquad= {{\mathop{\rm cov}} _p}\,({V_m},{X_n};Z) = {{\mathop{\rm cov}} _p}\,({V_m},\sum\limits_{k = 0}^n {{a_{n,n + 1 - k}}\,{V_k}} ;Z)\nonumber\\
&\qquad\qquad= \sum\limits_{k = 0}^n {{a_{n,n + 1 - k}}\,{{{\mathop{\rm cov}} }_p}\,\left( {{V_m},{V_k};Z} \right)}  = {{\mathop{\rm var}} _p}\,({X_n};Z)\,{\delta _{m,n}}.\nonumber
\end{align}
For $ m = 0,1,...,n $, relation \eqref{eq3.7} leads eventually to the linear system
\begin{align}\label{eq3.8}
&\left[ {\begin{array}{*{20}{c}}
{\begin{array}{*{20}{c}}
{{{{\mathop{\rm var}} }_p}\,({V_0};Z)}\\
{{{{\mathop{\rm cov}} }_p}\,({V_1},{V_0};Z)}\\
 \vdots \\
{{{{\mathop{\rm cov}} }_p}\,({V_n},{V_0};Z)}
\end{array}}&{\begin{array}{*{20}{c}}
{{{{\mathop{\rm cov}} }_p}\,({V_0},{V_1};Z)}\\
{{{{\mathop{\rm var}} }_p}\,({V_1};Z)}\\
 \vdots \\
{{{{\mathop{\rm cov}} }_p}\,({V_n},{V_1};Z)}
\end{array}}&{\begin{array}{*{20}{c}}
 \cdots \\
 \cdots \\
 \vdots \\
 \cdots 
\end{array}}&{\begin{array}{*{20}{c}}
{{{{\mathop{\rm cov}} }_p}\,({V_0},{V_n};Z)}\\
{{{{\mathop{\rm cov}} }_p}\,({V_1},{V_n};Z)}\\
 \vdots \\
{{{{\mathop{\rm var}} }_p}\,({V_n};Z)}
\end{array}}
\end{array}} \right]\left[ {\begin{array}{*{20}{c}}
{{a_{n,n + 1}}}\\
{{a_{n,n}}}\\
 \vdots \\
{{a_{n,1}}}
\end{array}} \right] \\
&\qquad\qquad\qquad\qquad\qquad\qquad= \left[ {\begin{array}{*{20}{c}}
0\\
0\\
 \vdots \\
{{{{\mathop{\rm var}} }_p}\,({X_n};Z)}
\end{array}} \right].\nonumber
\end{align}
If the determinant
\begin{equation}\label{eq3.9}
\Delta _n^{(p)}\left( {\{ {V_k}\} _{k = 0}^n;Z} \right) = \left| {\begin{array}{*{20}{c}}
{\begin{array}{*{20}{c}}
{{{{\mathop{\rm var}} }_p}\,({V_0};Z)}\\
{{{{\mathop{\rm cov}} }_p}\,({V_1},{V_0};Z)}\\
 \vdots \\
{{{{\mathop{\rm cov}} }_p}\,({V_n},{V_0};Z)}
\end{array}}&{\begin{array}{*{20}{c}}
{{{{\mathop{\rm cov}} }_p}\,({V_0},{V_1};Z)}\\
{{{{\mathop{\rm var}} }_p}\,({V_1};Z)}\\
 \vdots \\
{{{{\mathop{\rm cov}} }_p}\,({V_n},{V_1};Z)}
\end{array}}&{\begin{array}{*{20}{c}}
 \cdots \\
 \cdots \\
 \vdots \\
 \cdots 
\end{array}}&{\begin{array}{*{20}{c}}
{{{{\mathop{\rm cov}} }_p}\,({V_0},{V_n};Z)}\\
{{{{\mathop{\rm cov}} }_p}\,({V_1},{V_n};Z)}\\
 \vdots \\
{{{{\mathop{\rm var}} }_p}\,({V_n};Z)}
\end{array}}
\end{array}} \right|,
\end{equation}
is defined with $ \Delta _{ - 1}^{(p)}(.) = 1 $, the first result of solving the system \eqref{eq3.8} is that the value $ {{\mathop{\rm var}} _p}\,({X_n};Z) $ can be computed in terms of determinant \eqref{eq3.9} as
\begin{equation}\label{eq3.10}
{{\mathop{\rm var}} _p}\,({X_n};Z) = \frac{{\Delta _n^{(p)}\left( {\{ {V_k}\} _{k = 0}^n;Z} \right)}}{{\Delta _{n - 1}^{(p)}\left( {\{ {V_k}\} _{k = 0}^{n - 1};Z} \right)}}\,.
\end{equation}
On the other side, it follows from \eqref{eq3.10} as a first order equation that
\begin{equation}\label{eq3.11}
\Delta _n^{(p)}\left( {\{ {V_k}\} _{k = 0}^n;Z} \right) = \prod\limits_{k = 0}^n {{{{\mathop{\rm var}} }_p}\,({X_k};Z)} \,\, \ge 0\,.
\end{equation}
The second result of solving \eqref{eq3.8} is to derive a general representation for $ X_{n} $ as
\begin{multline}\label{eq3.12}
\Delta _{n - 1}^{(p)}\left( {\{ {V_k}\} _{k = 0}^{n - 1};Z} \right){X_n}\\
 = \left| \begin{array}{l}
\begin{array}{*{20}{c}}
{\begin{array}{*{20}{c}}
{{{{\mathop{\rm var}} }_p}\,({V_0};Z)}\\
{{{{\mathop{\rm cov}} }_p}\,({V_1},{V_0};Z)}\\
 \vdots \\
{{{{\mathop{\rm cov}} }_p}\,({V_{n - 1}},{V_0};Z)}
\end{array}}&{\begin{array}{*{20}{c}}
{{{{\mathop{\rm cov}} }_p}\,({V_0},{V_1};Z)}\\
{{{{\mathop{\rm var}} }_p}\,({V_1};Z)}\\
 \vdots \\
{{{{\mathop{\rm cov}} }_p}\,({V_{n - 1}},{V_1};Z)}
\end{array}}&{\begin{array}{*{20}{c}}
 \cdots \\
 \cdots \\
 \vdots \\
 \cdots 
\end{array}}&{\begin{array}{*{20}{c}}
{{{{\mathop{\rm cov}} }_p}\,({V_0},{V_n};Z)}\\
{{{{\mathop{\rm cov}} }_p}\,({V_1},{V_n};Z)}\\
 \vdots \\
{{{{\mathop{\rm cov}} }_p}\,({V_{n - 1}},{V_n};Z)}
\end{array}}
\end{array}\\
\qquad\qquad{V_0}\qquad\qquad\qquad\qquad\,\,{V_1}\qquad\qquad\qquad\cdots \qquad\qquad\qquad{V_n}
\end{array} \right|\,,
\end{multline}
where we have exploited \eqref{eq3.10} to derive it.

Inequality \eqref{eq3.11} shows that the determinant of the coefficients matrix of the system \eqref{eq2.6} is always non-negative. Moreover, if $ p=0 $ or $ E(Z{X_k}) = 0 $ for every $ k=0,1,\ldots, n $ in \eqref{eq3.12}, it leads to the Gram-Schmidt orthogonalization process. For instance, the expanded forms of \eqref{eq3.12} for $ n=0,1,2 $ are 
\begin{align*}
{X_0} &= {V_0}, \nonumber\\
{X_1} &= {V_1} - \frac{{{{{\mathop{\rm cov}} }_p}\,({V_0},{V_1};Z)}}{{{{{\mathop{\rm var}} }_p}\,({V_0};Z)}}{V_0}, \nonumber\\
{X_2} &= {V_2} - \frac{{{{{\mathop{\rm var}} }_p}\,({V_0};Z){{{\mathop{\rm cov}} }_p}\,({V_1},{V_2};Z) - {{{\mathop{\rm cov}} }_p}\,({V_0},{V_1};Z){{{\mathop{\rm cov}} }_p}\,({V_0},{V_2};Z)}}{{{{{\mathop{\rm var}} }_p}\,({V_0};Z){{{\mathop{\rm var}} }_p}\,({V_1};Z) - {\mathop{\rm cov}} _p^2({V_0},{V_1};Z)}}\,\,{V_1} \nonumber\\
& \,\,\,\,\,- \frac{{{{{\mathop{\rm var}} }_p}\,({V_1};Z){{{\mathop{\rm cov}} }_p}\,({V_0},{V_2};Z) - {{{\mathop{\rm cov}} }_p}\,({V_0},{V_1};Z){{{\mathop{\rm cov}} }_p}\,({V_1},{V_2};Z)}}{{{{{\mathop{\rm var}} }_p}\,({V_0};Z){{{\mathop{\rm var}} }_p}\,({V_1};Z) - {\mathop{\rm cov}} _p^2({V_0},{V_1};Z)}}\,\,{V_0}\,,\label{eq3.13}
\end{align*}
satisfying the conditions
\[{{\mathop{\rm cov}} _p}\,({X_0},{X_1};Z) = {{\mathop{\rm cov}} _p}\,({X_0},{X_2};Z) = {{\mathop{\rm cov}} _p}\,({X_1},{X_2};Z) = 0.\]
Note although representation \eqref{eq3.12} is computationally slower than the recursive algorithm described in theorem \ref{thm3.2}, it is however of theoretical interest as we have used it to find new sequences of p-uncorrelated functions (see examples of sections \ref{sec6}, \ref{sec7}, \ref{sec8} and \ref{sec9}).

\section{p-uncorrelated expansions with respect to a fixed variable}\label{sec4}
Let $ \{ {X_k}\} _{k = 0}^n $ be a sequence of p-uncorrelated variables with respect to the variable $ Z $. If $ n\rightarrow\infty $ in approximation \eqref{eq3.2}, the series
\[\sum\limits_{k = 0}^\infty  {\frac{{{{{\mathop{\rm cov}} }_p}\,({X_k},Y;Z)}}{{{{{\mathop{\rm var}} }_p}\,({X_k};Z)}}{X_k}}, \]
is called a p-uncorrelated expansion with respect to $ Z $ for $ Y $ and then we write
\begin{equation}\label{eq4.1}
Y \sim \sum\limits_{k = 0}^\infty  {\frac{{{{{\mathop{\rm cov}} }_p}\,({X_k},Y;Z)}}{{{{{\mathop{\rm var}} }_p}\,({X_k};Z)}}{X_k}} \,.
\end{equation}
In the case where $ n $ is finite, there is a separate theorem as follows.

\begin{theorem}\label{thm4.1}
Let $ \{ {V_k}\} _{k = 0}^n $ be linearly independent variables and $ \{ {X_k}\} _{k = 0}^n $ be their corresponding p-uncorrelated elements generated by $ \{ {V_k}\} _{k = 0}^n $ in theorem \ref{thm3.2}. If $ \sum\limits_{k = 0}^n {{a_k}{V_k}}  = {W_n} $ then
\begin{equation}\label{eq4.2}
{W_n} = \sum\limits_{k = 0}^n {\frac{{{{{\mathop{\rm cov}} }_p}\,({X_k},{W_n};Z)}}{{{{{\mathop{\rm var}} }_p}\,({X_k};Z)}}{X_k}} \,.
\end{equation}
\end{theorem}
\begin{proof}
First from theorem \ref{thm3.3} and relation \eqref{eq3.5}, we have
\begin{align}\label{eq4.3}
&{W_n} = \sum\limits_{k = 0}^n {{a_k}{V_k}}  = {a_0}({X_0}) + {a_1}({X_1} + {b_{12}}{X_0}) + ... + {a_n}({X_n} + {b_{n2}}{X_{n - 1}} + ... + {b_{n,n + 1}}{X_0})\\
&\qquad\qquad \qquad\qquad \qquad\qquad \qquad\qquad = {c_0}{X_0} + {c_1}{X_1} + ... + {c_n}{X_n}.\nonumber
\end{align}
Now, for $ k=0,1,\ldots,n $ apply $ {{\mathop{\rm cov}} _p}(.\,,{X_k};Z) $ on both sides of \eqref{eq4.3} to get
\[{{\mathop{\rm cov}} _p}({W_n},{X_k};Z) = {{\mathop{\rm cov}} _p}\,(\sum\limits_{j = 0}^n {{c_j}{X_j}} ,{X_k};Z) = \sum\limits_{j = 0}^n {{c_j}{{{\mathop{\rm cov}} }_p}\,({X_j},{X_k};Z)}  = {c_k}{{\mathop{\rm var}} _p}({X_k};Z),\]
which yields
\[{c_k} = \frac{{{{{\mathop{\rm cov}} }_p}\,({W_n},{X_k};Z)}}{{{{{\mathop{\rm var}} }_p}\,({X_k};Z)}},\]
and \eqref{eq4.2} is therefore proved.
\end{proof}

\begin{remark}\label{rem4.1.1}
There is an important point in theorem \ref{thm4.1}. Let us reconsider $ n+1 $ mutually p-uncorrelated elements $ \{ {X_k}\} _{k = 0}^n $ satisfying the condition
\begin{equation}\label{eq4.4}
{{\mathop{\rm cov}} _p}\,({X_i},{X_j};Z) = {{\mathop{\rm var}} _p}\,({X_j};Z)\,{\delta _{i,j}}.
\end{equation}
Since
\[{{\mathop{\rm cov}} _p}\,({X_i},{X_j};Z) = E\left( {\left( {{X_i} - (1 - \sqrt {1 - p} )\frac{{E({X_i}Z)}}{{E({Z^2})}}\,Z} \right)\left( {{X_j} - (1 - \sqrt {1 - p} )\frac{{E({X_j}Z)}}{{E({Z^2})}}\,Z} \right)} \right),\]
by defining the variable
\begin{equation}\label{eq4.5}
{T_{j,p}}(Z) = {X_j} - (1 - \sqrt {1 - p} )\frac{{E({X_j}Z)}}{{E({Z^2})}}\,Z = {X_j} - (1 - \sqrt {1 - p} )\,{\rm{pro}}{{\rm{j}}_{\,Z}}{X_j},
\end{equation}
we observe in (4.4) that
\begin{equation}\label{eq4.6}
E\Big( {{T_{i,p}}(Z){T_{j,p}}(Z)} \Big) = E\Big( {T_{j,p}^2(Z)} \Big){\delta _{i,j}}.
\end{equation}
Now notice that the orthogonal sequence $ \left\{ {{T_{j,p}}(Z)} \right\}_{j = 0}^n $ in \eqref{eq4.5} is generated only if we have previously created the uncorrelated sequence $ \left\{ {{X_j}} \right\}_{j = 0}^n $ because of the term $ \frac{{E({X_j}Z)}}{{E({Z^2})}}\,Z $ otherwise generating such orthogonal sequences is impossible. Also, note that the original form of these two sequences are completely different, so their corresponding approximations are different such that the approximation corresponding to the orthogonal base $ \left\{ {{T_{j,p}}(Z)} \right\}_{j = 0}^n $ is optimized in the sense of ordinary least squares while the approximation corresponding to the uncorrelated base $ \left\{ {{X_j}} \right\}_{j = 0}^n $ is optimized in the sense of least p-variances.

Referring to \eqref{eq4.5} and \eqref{eq4.6}, let us now consider the following orthogonal approximation sum for the same variable $ W_{n} $ as stated in \eqref{eq4.2},
\begin{equation}\label{eq4.7}
{S_n} = \sum\limits_{j = 0}^n {\frac{{E\left( {{W_n}\,{T_{j,p}}(Z)} \right)}}{{E\left( {T_{j,p}^2(Z)} \right)}}\,{T_{j,p}}(Z)} \,.
\end{equation}
Since in \eqref{eq4.7},
\[E\left( {{W_n}\,{T_{j,p}}(Z)} \right) = E\left( {{W_n}\,\left( {{X_j} - (1 - \sqrt {1 - p} )\frac{{E({X_j}Z)}}{{E({Z^2})}}\,Z} \right)} \right) = {{\mathop{\rm cov}} _p}\,({X_j},{W_n};Z),\]
and
\[E\left( {T_{j,p}^2(Z)} \right) = E\left( {{{\left( {{X_j} - (1 - \sqrt {1 - p} )\frac{{E({X_j}Z)}}{{E({Z^2})}}\,Z} \right)}^2}} \right) = {{\mathop{\rm var}} _p}\,({X_j};Z),\]
the sum is simplified as
\begin{align}\label{eq4.8}
{S_n}& = \sum\limits_{j = 0}^n {\frac{{{{{\mathop{\rm cov}} }_p}\,({X_j},{W_n};Z)}}{{{{{\mathop{\rm var}} }_p}\,({X_j};Z)}}\,\left( {{X_j} - (1 - \sqrt {1 - p} )\frac{{E({X_j}Z)}}{{E({Z^2})}}\,Z} \right)} \,\\
&= {W_n} - (1 - \sqrt {1 - p} )\frac{Z}{{E({Z^2})}}\,\sum\limits_{j = 0}^n {\frac{{{{{\mathop{\rm cov}} }_p}\,({X_j},{W_n};Z)\,E({X_j}Z)}}{{{{{\mathop{\rm var}} }_p}\,({X_j};Z)}}} .\nonumber
\end{align}
The main point is here: If in \eqref{eq4.8}, $ p=0 $ or $ \left\{ {E({X_j}Z) = 0} \right\}_{j = 0}^n $ (see section \ref{subsubsec2.3.1}) and/or the limit condition
\begin{equation}\label{eq4.9}
\sum\limits_{j = 0}^\infty  {\frac{{{{{\mathop{\rm cov}} }_p}\,({X_j},{W_n};Z)\,E({X_j}Z)}}{{{{{\mathop{\rm var}} }_p}\,({X_j};Z)}}}  = 0,
\end{equation}
is satisfied, then $ \mathop {\lim }\limits_{n \to \infty } \,\,{S_n} = \mathop {\lim }\limits_{n \to \infty } \,\,{W_n} = {W^*} $ and subsequently 
\begin{equation}\label{eq4.10}
\sum\limits_{j = 0}^\infty  {\frac{{{{{\mathop{\rm cov}} }_p}\,({X_j},{W^*};Z)}}{{{{{\mathop{\rm var}} }_p}\,({X_j};Z)}}{X_j}} \, = \sum\limits_{j = 0}^\infty  {\frac{{E\left( {{W^*}{T_{j,p}}(Z)} \right)}}{{E\left( {T_{j,p}^2(Z)} \right)}}\,{T_{j,p}}(Z)} \,.
\end{equation}
This means that the two aforesaid expansions in \eqref{eq4.10} coincide each other only if the condition \eqref{eq4.9} holds. However, in many cases, this condition will not occur independently as $ W_{n} $ is distinct for each choice especially in continuous spaces, see  corollary \ref{coro9.1.2} in this regard. We finally add in a general case that
\begin{align*}
\sum\limits_{j = 0}^n {{C_j}\,{T_{j,p}}(Z)} &= \sum\limits_{j = 0}^n {{C_j}\,\left( {{X_j} - (1 - \sqrt {1 - p} )\frac{{E({X_j}Z)}}{{E({Z^2})}}\,Z} \right)} \\
& = \sum\limits_{j = 0}^n {{C_j}{X_j}}  - (1 - \sqrt {1 - p} )\frac{Z}{{E({Z^2})}}\,E\left( {Z\sum\limits_{j = 0}^n {{C_j}{X_j}} } \right) \\
& = \sum\limits_{j = 0}^n {{C_j}{X_j}}  - (1 - \sqrt {1 - p} )\,{\rm{pro}}{{\rm{j}}_{\,Z}}\left( {\sum\limits_{j = 0}^n {{C_j}{X_j}} } \right)\,.
\end{align*}

\end{remark}

\subsection{A biorthogonality property}\label{subsec4.2}
Let $ \{ {X_k}\} _{k = 0} $ be a sequence of complete uncorrelated variables satisfying the condition
\[{{\mathop{\rm cov}} _1}\,({X_k},{X_j};Z) = {{\mathop{\rm var}} _1}\,({X_j};Z)\,{\delta _{k,j}}.\]
Since
\[ E\left( {{X_k}\left( {{X_j} - \frac{{E({X_j}\,Z)}}{{E({Z^2})}}\,Z} \right)} \right)={{\mathop{\rm cov}} _1}\,({X_k},{X_j};Z),\]
so the relation
\[E\left( {{X_k}\left( {{X_j} - \frac{{E({X_j}\,Z)}}{{E({Z^2})}}\,Z} \right)} \right) = {{\mathop{\rm var}} _1}\,({X_j};Z)\,{\delta _{k,j}}.\]
implies that the sequence $ \{ {X_k}\} _{k = 0} $ is biorthogonal with respect to the sequence $ {\left\{ {{X_j} - \frac{{E({X_j}\,Z)}}{{E({Z^2})}}\,Z} \right\}_{j = 0}} $. This means that every complete uncorrelated sequence can be biorthogonal but its reverse proposition is not true. See also corollary \ref{coro8.1}. To study the topic of biorthogonality we refer the readers to \cite{ref3, ref8}

\begin{theorem}\label{thm4.3}
Let $ \{ {X_k}\} _{k = 0} $ be p-uncorrelated variables satisfying the condition (4.4) and let $ Y $ be arbitrary. Then
\begin{equation}\label{eq4.11}
{{\mathop{\rm var}} _p}\,\left( {Y - \sum\limits_{k = 0}^n {\frac{{{{{\mathop{\rm cov}} }_p}\,({X_k},Y;Z)}}{{{{{\mathop{\rm var}} }_p}\,({X_k};Z)}}{X_k}} \,;Z} \right) \le {{\mathop{\rm var}} _p}\,\left( {Y - \sum\limits_{k = 0}^n {{\alpha _k}{X_k}} \,;Z} \right),
\end{equation}
for any selection of constants $ \{ {\alpha _k}\} _{k = 0}^n $.
\end{theorem}
\begin{proof}
According to \eqref{eq1.9}, we first have
\begin{align}\label{eq4.12}
&{{\mathop{\rm var}} _p}\,\left( {Y - \sum\limits_{k = 0}^n {{\alpha _k}{X_k}} \,;Z} \right) = {{\mathop{\rm var}} _p}\,\left( {Y\,;Z} \right) + {{\mathop{\rm var}} _p}\,\left( {\sum\limits_{k = 0}^n {{\alpha _k}{X_k}} \,;Z} \right) - 2\,{{\mathop{\rm cov}} _p}\,\left( {\sum\limits_{k = 0}^n {{\alpha _k}{X_k}} ,Y\,;Z} \right)\\
&\qquad\quad = {{\mathop{\rm var}} _p}\,\left( {Y\,;Z} \right) + \sum\limits_{k = 0}^n {\alpha _k^2\,{{{\mathop{\rm var}} }_p}\,\left( {\,{X_k};Z} \right)}  - 2\sum\limits_{k = 0}^n {{\alpha _k}{{{\mathop{\rm cov}} }_p}\,\left( {{X_k},Y\,;Z} \right)} \nonumber \\
&\qquad\quad = {{\mathop{\rm var}} _p}\,\left( {Y\,;Z} \right) + \sum\limits_{k = 0}^n {{{\left( {{\alpha _k}\sqrt {{{{\mathop{\rm var}} }_p}\,\left( {\,{X_k};Z} \right)}  - \frac{{{{{\mathop{\rm cov}} }_p}\,({X_k},Y;Z)}}{{\sqrt {{{{\mathop{\rm var}} }_p}\,\left( {\,{X_k};Z} \right)} }}} \right)}^2}}  - \sum\limits_{k = 0}^n {\frac{{{\mathop{\rm cov}} _p^2\,({X_k},Y;Z)}}{{{{{\mathop{\rm var}} }_p}\,({X_k};Z)}}} ,\nonumber
\end{align}
in which the following identity has been used
\[{{\mathop{\rm var}} _p}\,\left( {\sum\limits_{k = 0}^n {{a_k}{X_k}} \,;Z} \right) = \sum\limits_{k = 0}^n {a_k^2\,{{{\mathop{\rm var}} }_p}\,\left( {\,{X_k};Z} \right)}  \Leftrightarrow {{\mathop{\rm cov}} _p}\,\left( {{X_i},{X_j}\,;Z} \right) = 0\,\,\,{\rm{for}}\,\,i \ne j.\]
Since $ {{\mathop{\rm var}} _p}\,\left( {Y\,;Z} \right) - \sum\limits_{k = 0}^n {\frac{{{\mathop{\rm cov}} _p^2\,({X_k},Y;Z)}}{{{{{\mathop{\rm var}} }_p}\,({X_k};Z)}}}  $ is independent of $ \{ {\alpha _k}\} _{k = 0}^n $ in the last equality of \eqref{eq4.12}, the minimum of $ {{\mathop{\rm var}} _p}\,\left( {Y - \sum\limits_{k = 0}^n {{\alpha _k}{X_k}} \,;Z} \right) $ is therefore achieved only when
\begin{equation}\label{eq4.13}
{\alpha _k}\sqrt {{{{\mathop{\rm var}} }_p}\,\left( {\,{X_k};Z} \right)}  - \frac{{{{{\mathop{\rm cov}} }_p}\,({X_k},Y;Z)}}{{\sqrt {{{{\mathop{\rm var}} }_p}\,\left( {\,{X_k};Z} \right)} }} = 0,
\end{equation}
which results in the left side of inequality \eqref{eq4.11}.
\end{proof}
With reference to \eqref{eq1.8}, inequality \eqref{eq4.11} can be completed as follows
\begin{equation}\label{eq4.14}
{{\mathop{\rm var}} _p}\,\left( {Y - \sum\limits_{k = 0}^n {\frac{{{{{\mathop{\rm cov}} }_p}\,({X_k},Y;Z)}}{{{{{\mathop{\rm var}} }_p}\,({X_k};Z)}}{X_k}} \,;Z} \right) \le {{\mathop{\rm var}} _p}\,\left( {Y - \sum\limits_{k = 0}^n {{\alpha _k}{X_k}} \,;Z} \right) \le E\,\left( {{{\Big(Y - \sum\limits_{k = 0}^n {{\alpha _k}{X_k}} \Big)^2}}} \right),
\end{equation}
where $ \{ {\alpha _k}\} _{k = 0}^n $ are arbitrary.

Again, inequality \eqref{eq4.14} shows the superiority of p-uncorrelated approximations \eqref{eq3.2} with respect to any approximation that is made by the least square method (in particular orthogonal approximations). See example \ref{ex13.1} in this regard.

\begin{corollary}\label{coro4.4}
Let $ \{ {V_k}\} _{k = 0}^n $ be an independent set of random variables.

The problem of finding that linear combination of $ {V_0},{V_1},...,{V_n} $ which minimizes 
\[ {{\mathop{\rm var}} _p}\,\left( {Y - \sum\limits_{k = 0}^n {{\beta _k}{V_k}} \,;Z} \right) ,\]
for any selection of constants $ \{ {\beta _k}\} _{k = 0}^n $ is solved by $ \sum\limits_{k = 0}^n {\frac{{{{{\mathop{\rm cov}} }_p}\,({X_k},Y;Z)}}{{{{{\mathop{\rm var}} }_p}\,({X_k};Z)}}{X_k}} $ where $ \{ {X_k}\} _{k = 0}^n $ are mutually p-uncorrelated with respect to the variable $ Z $. This corollary tells us that every least p-variance problem with respect to a fixed variable is solved by an appropriate truncated p-uncorrelated expansion of type \eqref{eq4.1}.
\end{corollary}

\begin{corollary}\label{coro4.5}
Substituting \eqref{eq4.13} into \eqref{eq4.12} gives
\begin{multline}\label{eq4.15}
{{\mathop{\rm var}} _p}\,\left( {Y - \sum\limits_{k = 0}^n {\frac{{{{{\mathop{\rm cov}} }_p}\,({X_k},Y;Z)}}{{{{{\mathop{\rm var}} }_p}\,({X_k};Z)}}{X_k}} \,;Z} \right) = {{\mathop{\rm var}} _p}\,(Y;Z) - \sum\limits_{k = 0}^n {\frac{{{\mathop{\rm cov}} _p^2({X_k},Y;Z)}}{{{{{\mathop{\rm var}} }_p}\,({X_k};Z)}}} \\
= {{\mathop{\rm var}} _p}\,(Y;Z)\left( {1 - \sum\limits_{k = 0}^n {\rho _p^2\,({X_k},Y;Z)} } \right).
\end{multline}
On the other side, the positivity of a p-variance implies that
\[\sum\limits_{k = 0}^n {\rho _p^2\,({X_k},Y;Z)}  \le 1,\]
where $ \{ {X_k}\} _{k = 0}^n $ satisfy \eqref{eq4.4}. Also, if $ \{ {X_k}\} _{k = 0}^{\infty} $ is an infinite sequence of p-uncorrelated variables, then the latter inequality changes to
\begin{equation}\label{eq4.16}
\sum\limits_{k = 0}^\infty  {\rho _p^2\,({X_k},Y;Z)}  \le 1\,\,\,\,{\rm{or}}\,\,\,\,\,\sum\limits_{k = 0}^\infty  {\frac{{{\mathop{\rm cov}} _p^2({X_k},Y;Z)}}{{{{{\mathop{\rm var}} }_p}\,({X_k};Z)}}}  \le {{\mathop{\rm var}} _p}\,(Y;Z).
\end{equation}
Finally, the convergence of the series in \eqref{eq4.16} causes that
\[\mathop {\lim }\limits_{k \to \infty } \,\,\rho _p^2\,({X_k},Y;Z) = 0\,\,\,\,\,\,{\rm{or}}\,\,\,\,\,\mathop {\lim }\limits_{k \to \infty } \,\,\frac{{{\mathop{\rm cov}} _p^2({X_k},Y;Z)}}{{{{{\mathop{\rm var}} }_p}\,({X_k};Z)}} = 0\,.\]
\end{corollary}
\begin{definition}\label{def4.6}
Convergence in p-variance with respect to a fixed variable\\
Reconsider the approximation \eqref{eq2.1} and constitute the p-variance of its error just like relation \eqref{eq2.3}. If 
\begin{equation}\label{eq4.17}
\mathop {\lim }\limits_{n \to \infty } \,\,{{\mathop{\rm var}} _p}\,\left( {R({c_0},...,{c_n});Z} \right) = 0,
\end{equation}
we call the limit relation \eqref{eq4.17} ``Convergence in p-variance with respect to the fixed variable $ Z $". In this sense, by noting the identity \eqref{eq4.15} if 
\begin{multline*}
0 = \mathop {\lim }\limits_{n \to \infty } \,\,{{\mathop{\rm var}} _p}\,\left( {Y - \sum\limits_{k = 0}^n {\frac{{{{{\mathop{\rm cov}} }_p}\,({X_k},Y;Z)}}{{{{{\mathop{\rm var}} }_p}\,({X_k};Z)}}{X_k}} \,;Z} \right)\\
\,\,\,\, = {{\mathop{\rm var}} _p}\,(Y;Z)\,\,\mathop {\lim }\limits_{n \to \infty } \left( {1 - \sum\limits_{k = 0}^n {\rho _p^2\,({X_k},Y;Z)} } \right),
\end{multline*}
then
\[\sum\limits_{k = 0}^\infty  {\rho _p^2\,({X_k},Y;Z)}  = 1\,\,\,\,\,{\rm{or}}\,\,\,\,\,\sum\limits_{k = 0}^\infty  {\frac{{{\mathop{\rm cov}} _p^2\,({X_k},Y;Z)}}{{{{{\mathop{\rm var}} }_p}\,({X_k};Z)}}}  = {{\mathop{\rm var}} _p}\,(Y;Z)\,,\]
\end{definition}
provided that $ \{ {X_k}\} _{k = 0} $ are mutually p-uncorrelated.

\begin{theorem}\label{thm4.7}
Let $ \{ {V_k}\} _{k = 0}^{n} $ be linearly independent variables and $ \{ {X_k}\} _{k = 0}^{n} $ be their corresponding p-uncorrelated elements generated by $ \{ {V_k}\} _{k = 0}^{n} $ in theorem \ref{thm3.2}. Then, for any selection of constants $ \{ {\lambda_k}\} _{k = 0}^{n-1} $ we have
\begin{equation}\label{eq4.18}
{{\mathop{\rm var}} _p}\,\left( {{X_n}\,;Z} \right) \le {{\mathop{\rm var}} _p}\,\left( {{V_n} + \sum\limits_{k = 0}^{n - 1} {{\lambda _k}{V_k}} \,;Z} \right).
\end{equation}
\end{theorem}
\begin{proof}
First suppose in \eqref{eq4.11} that $ n \to n - 1 $ and then take $ Y=V_{n} $. Noting the last equality of \eqref{eq3.4} gives
\begin{equation}\label{eq4.19}
{{\mathop{\rm var}} _p}\,\left( {{V_n} - \sum\limits_{k = 0}^{n - 1} {\frac{{{{{\mathop{\rm cov}} }_p}\,({X_k},{V_n};Z)}}{{{{{\mathop{\rm var}} }_p}\,({X_k};Z)}}{X_k}} \,;Z} \right) = {{\mathop{\rm var}} _p}\,\left( {{X_n}\,;Z} \right) \le {{\mathop{\rm var}} _p}\,\left( {{V_n} - \sum\limits_{k = 0}^{n - 1} {{\alpha _k}{X_k}} \,;Z} \right).
\end{equation}
On the other hand, according to relations \eqref{eq3.3}, there exists an arbitrary sequence, say $ \{ {\lambda_k}\} _{k = 0}^{n-1} $, such that $  - \sum\limits_{k = 0}^{n - 1} {{\alpha _k}{X_k}}  = \sum\limits_{k = 0}^{n - 1} {{\lambda _k}{V_k}}  $ in \eqref{eq4.19} which completes the proof of \eqref{eq4.18}.
\end{proof}
\begin{corollary}\label{coro4.8}
Let $ \{ {X_k}\} _{k = 0}^{n} $ be p-uncorrelated variables with respect to the fixed variable $ Z $. Then, for any variable $ Y $ and every $ j=0,1,\ldots, n $ we have
\begin{multline*}
{{\mathop{\rm cov}} _p}\,\left( {Y - \sum\limits_{k = 0}^n {\frac{{{{{\mathop{\rm cov}} }_p}\,({X_k},Y;Z)}}{{{{{\mathop{\rm var}} }_p}\,({X_k};Z)}}{X_k}} ,{X_j}\,;Z} \right)\\
 = {{\mathop{\rm cov}} _p}\,\left( {Y,{X_j}\,;Z} \right) - \sum\limits_{k = 0}^n {\frac{{{{{\mathop{\rm cov}} }_p}\,({X_k},Y;Z)}}{{{{{\mathop{\rm var}} }_p}\,({X_k};Z)}}{{{\mathop{\rm cov}} }_p}\left( {{X_k},{X_j}\,;Z} \right)} \\
 = {{\mathop{\rm cov}} _p}\,\left( {Y,{X_j}\,;Z} \right) - {{\mathop{\rm cov}} _p}\,\left( {{X_j},Y\,;Z} \right) = 0.
\end{multline*}
\end{corollary}
It is now a good position to practically enter the p-uncorrelatedness topic and introduce a sequence of functions in continuous and discrete spaces.

\section{p-uncorrelated functions with respect to a fixed function}\label{sec5}
Let $ \left\{ {{\Phi _k}(x)} \right\}_{k = 0} $ be a sequence of continuous functions defined on $ [a,b] $ and $ w(x) $ be positive on this interval. Also let $ z(x) $ be a continuous function such that $ \int_{\,a}^b {w(x)\,{z^2}(x)\,dx}  > 0 $. Referring to \eqref{eq2.7}, we say that $ \left\{ {{\Phi _k}(x)} \right\}_{k = 0}^\infty  $ are (weighted) p-uncorrelated functions with respect to the fixed function $ z(x) $ if they satisfy the condition
\begin{multline}\label{eq5.1}
\int_{\,a}^b {w(x)\,{\Phi _m}(x)\,{\Phi _n}(x)\,dx}  - p\frac{{\int_{\,a}^b {w(x)\,{\Phi _m}(x)\,z(x)\,dx} \,\int_{\,a}^b {w(x)\,{\Phi _n}(x)\,z(x)\,dx} }}{{\,\int_{\,a}^b {w(x)\,{z^2}(x)\,dx} }}\\
 = \left( {\int_{\,a}^b {w(x)\,\Phi _n^2(x)\,dx}  - p\frac{{{{(\int_{\,a}^b {w(x)\,{\Phi _n}(x)\,z(x)\,dx} )}^2}\,}}{{\,\int_{\,a}^b {w(x)\,{z^2}(x)\,dx} }}} \right)\,{\delta _{m,n}}.
\end{multline}
Especially when $ w(x)z(x) = v(x) $, \eqref{eq5.1} becomes
\begin{multline*}
\int_{\,a}^b {w(x)\,{\Phi _m}(x)\,{\Phi _n}(x)\,dx}  - \lambda _p^2\int_{\,a}^b {v(x)\,{\Phi _m}(x)\,dx} \,\int_{\,a}^b {v(x)\,{\Phi _n}(x)\,dx} \\
 = \left( {\int_{\,a}^b {w(x)\,\Phi _n^2(x)\,dx}  - \lambda _p^2{{\Big(\int_{\,a}^b {v(x)\,{\Phi _n}(x)\,dx}\Big )^2}}} \right)\,{\delta _{m,n}},
\end{multline*}
where $ \lambda _p^2 = p/\int_{\,a}^b {\,\frac{{{v^2}(x)}}{{w(x)}}\,dx}  $ (e.g. see relation \eqref{eq9.9}).

Noting \eqref{eq2.8}, such a definition in \eqref{eq5.1} can similarly be considered for discrete functions as
\begin{multline*}
\sum\limits_{x \in {A^*}} {j(x){\Phi _m}(x)\,{\Phi _n}(x)}  - p\frac{{\sum\limits_{x \in {A^*}} {j(x){\Phi _m}(x)\,z(x)} \,\sum\limits_{x \in {A^*}} {j(x){\Phi _n}(x)\,z(x)} }}{{\,\sum\limits_{x \in {A^*}} {j(x)\,{z^2}(x)} }}\\
 = \left( {\sum\limits_{x \in {A^*}} {j(x)\Phi _n^2(x)\,}  - p\frac{{{{(\sum\limits_{x \in {A^*}} {j(x){\Phi _n}(x)\,z(x)} )}^2}\,}}{{\,\sum\limits_{x \in {A^*}} {j(x)\,{z^2}(x)} }}} \right)\,{\delta _{m,n}}\,.
\end{multline*}

\begin{remark}\label{rem5.1}
Relation \eqref{eq5.1} clarifies that, after deriving p-uncorrelated functions, one will be able to define a sequence of orthogonal functions in the form
\begin{equation}\label{eq5.2}
{\Phi _n}(x) - (1 - \sqrt {1 - p} )\frac{{\,\int_{\,a}^b {w(x)\,{\Phi _n}(x)\,z(x)\,dx} }}{{\,\int_{\,a}^b {w(x)\,{z^2}(x)\,dx} }}\,z(x) = {G_n}(x;p)\,,
\end{equation}
having the orthogonality property
\begin{align*}
\int_{\,a}^b {w(x)\,{G_m}(x;p){G_n}(x;p)\,dx}  &= \left( {\int_{\,a}^b {w(x)\,G_n^2(x;p)\,dx} } \right)\,{\delta _{m,n}} \\
&= \left( {\int_{\,a}^b {w(x)\,dx} } \right)\Big( {{{{\mathop{\rm var}} }_p}({\Phi _n}(x);\,z(x))} \Big){\delta _{m,n}},
\end{align*}
whereas its reverse proposition is not feasible when $ \int_{\,a}^b {w(x)\,{\Phi _n}(x)\,z(x)\,dx}  \ne 0 $ or $ p \ne 0 $ in \eqref{eq5.2}. This means that from orthogonal sequences, one cannot derive p-uncorrelated functions as they are somehow an extension of orthogonal functions for $ p=0 $, while from p-uncorrelated functions the aforesaid aim is always possible by applying \eqref{eq5.2}. The correctness of such corollaries has been approved by various illustrative examples in the next sections.
\end{remark}

There are essentially two classifications for deriving (weighted) p-uncorrelated functions with respect to a fixed function.
\\
\\
\noindent
{\textbf{5.1.1. First classification.}} If the weight function in \eqref{eq5.1} is non-classical, i.e. not belonging to the Pearson distributions family and its symmetric analogue \cite{ref15, ref21}, the procedure to derive p-uncorrelated functions will lead to the same as generalized Gram-Schmidt process presented in theorem \ref{thm3.2} or equivalently in the representation \eqref{eq3.12}. In this sense, note that there are ten classical sequences of real polynomials \cite{ref21} orthogonal with respect to the Pearson distributions family 
\[W\,\left( {\left. {\begin{array}{*{20}{c}}
{\begin{array}{*{20}{c}}
{d,}&e
\end{array}}\\
{a,\,\,\,b,\,\,\,c}
\end{array}} \right|\,x} \right) = \exp \left( {\int {\frac{{dx + e}}{{a{x^2} + bx + c}}dx} } \right)\,\,\,\,\,\,\,\,\,\,\,(a,b,c,d,e \in\mathbb{R}),\]
and its symmetric analogue \cite{ref15}
\[{W^*}\left( {\left. {\begin{array}{*{20}{c}}
{r,}&s\\
{p,}&q
\end{array}\,} \right|\,x} \right) = \exp \left( {\int {\frac{{r\,{x^2} + s}}{{x(p{x^2} + q)}}} \,dx} \right)\,\,\,\,\,\,\,\,\,\,(p,q,r,s \in\mathbb{R}).\]
Five of them are infinitely orthogonal with respect to particular cases of the two above-mentioned distributions and five other ones are finitely orthogonal which are limited to some parametric constraints. For other applications of orthogonal polynomials, see e.g. \cite{ref1, ref22, ref23, ref24}.The following table shows the main characteristics of these sequences.
\vspace{-0.3cm}
\begin{table}[H] 
\begin{center}
\caption{\label{tab1} Characteristics of ten sequences of classical orthogonal polynomials } 
\makebox[0cm]{
\begin{tabular}{| c | c | l | } 
\hline 
{Symbol} & {Weight Function} & \makecell{{Kind, Interval} \\ {\&} \\{Parameters Constraint}}\\ 
\hline 
$ P_n^{(u,v)}(x) $ & $ W\left( {\left. {\begin{array}{*{20}{c}}
{\begin{array}{*{20}{c}}
{ - u - v,}&{ - u + v}
\end{array}}\\
{\begin{array}{*{20}{c}}
{ - 1,}&{0,}&1
\end{array}}
\end{array}\,} \right|\,x} \right) = {(1 - x)^u}{(1 + x)^v} $ & 
\makecell{{Infinite}, $[-1,1]$  \\
$ u > - 1\,,v > - 1$}\\ 
\hline 
$ L_n^{(u)}(x) $ & $ W\left( {\left. {\begin{array}{*{20}{c}}
{\begin{array}{*{20}{c}}
{ - 1,}&u
\end{array}}\\
{\begin{array}{*{20}{c}}
{0,}&{1,}&0
\end{array}}
\end{array}\,} \right|\,x} \right) = {x^u}\exp ( - x) $ & \makecell{$ 
\text{Infinite}, [0,\infty)$ \\ $ u > - 1 $} \\ 
\hline 
$ {H_n}(x) $ & $ W\left( {\left. {\begin{array}{*{20}{c}}
{\begin{array}{*{20}{c}}
{ - 2,}&0
\end{array}}\\
{\begin{array}{*{20}{c}}
{0,}&{0,}&1
\end{array}}
\end{array}\,} \right|\,x} \right) = \exp ( - {x^2}) $ & \makecell{ $ \text{Infinite}, (-\infty,\infty) $ \\ - }\\ 
\hline 
$ J_n^{(u,v)}(x) $ & $ 
W\left( {\left. {\begin{array}{*{20}{c}}
{\begin{array}{*{20}{c}}
{ - 2u,}&v
\end{array}}\\
{\begin{array}{*{20}{c}}
{1,}&{0,}&1
\end{array}}
\end{array}\,} \right|\,x} \right) = 
{(1 + {x^2})^{ - u}}\,\exp (v\,\arctan x)
$ & \makecell{ $ \text{Finite}, (-\infty,\infty)$\\  
$\max n < (u - 1)/2 $} \\ 
\hline 
$ M_n^{(u,v)}(x) $ & $ W\left( {\left. {\begin{array}{*{20}{c}}
{\begin{array}{*{20}{c}}
{ - u,}&v
\end{array}}\\
{\begin{array}{*{20}{c}}
{1,}&{1,}&0
\end{array}}
\end{array}\,} \right|\,x} \right) = {x^v}{(x + 1)^{ - (u + v)}} $ & \makecell{$ \text{Finite}, [0,\infty)$ \\ 
$ \begin{array}{l} 
\max n < (u - 1)/2,\\ 
v > - 1 
\end{array} $} \\ 
\hline 
$ N_n^{(u)}(x) $ & $ W\left( {\left. {\begin{array}{*{20}{c}}
{\begin{array}{*{20}{c}}
{ - u,}&1
\end{array}}\\
{\begin{array}{*{20}{c}}
{1,}&{0,}&0
\end{array}}
\end{array}\,} \right|\,x} \right) = {x^{ - u}}\exp ( - 1/x) $ & \makecell{ $ 
\text{Finite}, [0,\infty) $ \\ $\max \,n < (u - 1)/2 $} \\ 
\hline 
$ {S_n}\left( {\left. {\begin{array}{*{20}{c}} 
{ - 2u - 2v - 2,}&{2u}\\ 
{ - 1,}&1 
\end{array}\,} \right|\,x} \right) $ & $ {W^*}\left( {\left. {\begin{array}{*{20}{c}}
{ - 2u - 2v,}&{2u}\\
{ - 1,}&1
\end{array}\,} \right|\,x} \right) = {x^{2u}}{(1 - {x^2})^v}$ & \makecell{$ \text{Infinite}, [-1,1] $ \\
$ u > - 1/2\,,\,\,v > - 1 $} \\ 
\hline 
$ {S_n}\left( {\left. {\begin{array}{*{20}{c}} 
{ - 2,}&{2u}\\ 
{0,}&1 
\end{array}\,} \right|\,x} \right) $ & $ {W^*}\left( {\left. {\begin{array}{*{20}{c}}
{ - 2,}&{2u}\\
{0,}&1
\end{array}\,} \right|\,x} \right) = {x^{2u}}\exp ( - {x^2}) $ & \makecell{ $ 
\text{Infinite}, (-\infty,\infty) $\\ 
 $ u > - 1/2 $} \\ 
\hline 
$ {S_n}\left( {\left. {\begin{array}{*{20}{c}} 
{ - 2u - 2v + 2,}&{ - 2u}\\ 
{1,}&1 
\end{array}\,} \right|\,x} \right) $ & $ {W^*}\left( {\left. {\begin{array}{*{20}{c}}
{ - 2u - 2v,}&{ - 2u}\\
{1,}&1
\end{array}} \right|\,x} \right) = {x^{ - 2u}}{(1 + {x^2})^{ - v}} $ & \makecell{$ 
\text{Finite}, (-\infty,\infty) $ \\ 
$\begin{array}{l} 
\,n < u + v - 1/2,\\ 
u < 1/2\,,\,\,v > 0 
\end{array} $} \\ 
\hline 
$ {S_n}\left( {\left. {\begin{array}{*{20}{c}} 
{ - 2u + 2,}&2\\ 
{1,}&0 
\end{array}\,} \right|\,x} \right) $ & $ {W^*}\left( {\left. {\begin{array}{*{20}{c}}
{ - 2u,}&2\\
{1,}&0
\end{array}\,} \right|\,x} \right) = {x^{ - 2u}}{\rm{exp}}( - 1/{x^2}) 
$ & \makecell{$ \text{Finite}, (-\infty,\infty) $ \\ $\max \,n < u - 1/2 $}\\ 
\hline 
\end{tabular} 
}
\end{center} 
\end{table} 
where
\[{S_n}\left( {\left. {\begin{array}{*{20}{c}}
{r,}&s\\
{p,}&q
\end{array}\,} \right|\,x} \right) = \sum\limits_{k = 0}^{[n/2]} {\,\left( {\begin{array}{*{20}{c}}
{[n/2]}\\
k
\end{array}} \right)\prod\limits_{i = 0}^{[n/2] - (k + 1)} {\frac{{\left( {2i + {{( - 1)}^{n + 1}} + 2\,[n/2]} \right)\,p + r}}{{\left( {2i + {{( - 1)}^{n + 1}} + 2} \right)\,\,q + s}}} \,\,{x^{n - 2k}} = {S_n}(x),} \]
is a four-parametric sequence of symmetric orthogonal polynomials \cite{ref15} that satisfies the second order differential equation
\[{x^2}(p{x^2} + q)\,{S''_n}(x) + x(r{x^2} + s)\,{S'_n}(x) - \left( {n(r + (n - 1)p){x^2} + (1 - {{( - 1)}^n})\,s/2} \right)\,{S_n}(x) = 0.\]
\\
\noindent
{\textbf{5.1.2. Second classification.}} If the weight function belongs to the family of distributions presented in table \ref{tab1}, there is an analogue theorem below for p-uncorrelated functions similar to the well-known Sturm-Liouville theorem for orthogonal functions \cite{ref21}.

\begin{theorem}\label{thm5.2}
Let $ a(x),\,b(x),\,u(x) $ and $ v(x) $ be four given functions continuous on the interval $ [a,b] $ and $ p\in[0,1] $. If
\[w(x) = u(x)\,\exp \left( {\int {\frac{{b(x) - a'(x)}}{{a(x)}}dx} } \right),\]
is a positive function on $ [a,b] $ and the sequence $ \left\{ {{\Phi _n}(x)} \right\}_{n = 0}^\infty  $ with the corresponding eigenvalues $ \{ {\lambda _n}\}  $ satisfies an integro-differential equation of the form
\begin{multline}\label{eq5.3}
a(x)\,\Phi ''_{n}(x) + b(x)\,\Phi '_{n}(x) + \left( {{\lambda _n}\,u(x) + v(x)} \right){\Phi _n}(x)\\
= p\left( {{\lambda _n} + {A_z}} \right)\,u(x)\,z(x)\frac{{\,\int_{\,a}^b {w(x)\,{\Phi _n}(x)\,z(x)\,dx} }}{{\,\int_{\,a}^b {w(x)\,{z^2}(x)\,dx} }},
\end{multline}
such that the function $ z(x) $ satisfies the equation
\begin{equation}\label{eq5.4}
\frac{{a(x)\,z''(x) + b(x)\,z'(x) + v(x)\,z(x)}}{{u(x)\,z(x)}} = {A_z},
\end{equation}
where $ A_{z} $ is a constant value, and also
\begin{equation}\label{eq5.5}
\begin{array}{l}
{\alpha _1}\,{\Phi _n}(a) + {\beta _1}\Phi '_{n}(a) = 0\,,\\[2mm]
{\alpha _2}\,{\Phi _n}(b) + {\beta _2}\Phi '_{n}(b) = 0\,,
\end{array}
\end{equation}
are two initial conditions where $ {\alpha _1}\,,{\beta _1},{\alpha _2} $ and $ {\beta _2} $ are given constants, then $ \left\{ {{\Phi _n}(x)} \right\}_{n = 0}^\infty  $ are (weighted) p-uncorrelated functions with respect to the fixed function $ z(x) $.
\end{theorem}
\begin{proof}
For the sake of simplicity, we assume in \eqref{eq5.3} that $ \frac{{\,\int_{\,a}^b {w(x)\,{\Phi _n}(x)\,z(x)\,dx} }}{{\,\int_{\,a}^b {w(x)\,{z^2}(x)\,dx} }} = \beta _n^* $. Then, equation \eqref{eq5.3} can be written in the self-adjoint form
\begin{equation}\label{eq5.6}
{\left( {r(x)\,\Phi '_{n}(x)} \right)^\prime } + \left( {{\lambda _n}\,w(x) + q(x)} \right){\Phi _n}(x) = p\left( {{\lambda _n} + {A_z}} \right)\,\beta _n^*\,w(x)\,z(x),
\end{equation}
and for the index $ m $ as
\begin{equation}\label{eq5.7}
{\left( {r(x)\,\Phi '_{m}(x)} \right)^\prime } + \left( {{\lambda _m}\,w(x) + q(x)} \right){\Phi _m}(x) = p\left( {{\lambda _m} + {A_z}} \right)\,\beta _m^*\,w(x)\,z(x),
\end{equation}
in which
\[r(x) = a(x)\,\exp \left( {\int {\frac{{b(x) - a'(x)}}{{a(x)}}dx} } \right),\]
and
\[q(x) = v(x)\,\exp \left( {\int {\frac{{b(x) - a'(x)}}{{a(x)}}dx} } \right).\]
On multiplying by $ {\Phi _m}(x) $ and $ {\Phi _n}(x) $ in respectively \eqref{eq5.6} and \eqref{eq5.7} and subtracting we get
\begin{multline*}
\left[ {r(x)\left( {{\Phi _m}(x)\Phi '_{n}(x) - {\Phi _n}(x)\Phi '_{m}(x)} \right)} \right]_{\,a}^{\,b} + \left( {{\lambda _n} - {\lambda _m}} \right)\int_{\,a}^b {w(x)\,{\Phi _n}(x)\,{\Phi _m}(x)\,dx} \\
 = p\left( {{\lambda _n} + {A_z}} \right)\beta _n^*\int_{\,a}^b {w(x)\,{\Phi _m}(x)\,z(x)\,dx}  - p\left( {{\lambda _m} + {A_z}} \right)\beta _m^*\int_{\,a}^b {w(x)\,{\Phi _n}(x)\,z(x)\,dx} \\
 = p\left( {{\lambda _n} - {\lambda _m}} \right)\frac{{\,\int_{\,a}^b {w(x)\,{\Phi _n}(x)\,z(x)\,dx} \int_{\,a}^b {w(x)\,{\Phi _m}(x)\,z(x)\,dx} }}{{\,\int_{\,a}^b {w(x)\,{z^2}(x)\,dx} }},
\end{multline*}
which results in
\begin{multline*}
\left( {{\lambda _n} - {\lambda _m}} \right)\left( {\int_{\,a}^b {w(x)\,{\Phi _n}(x)\,{\Phi _m}(x)\,dx}  - p\frac{{\,\int_{\,a}^b {w(x)\,{\Phi _n}(x)\,z(x)\,dx} \int_{\,a}^b {w(x)\,{\Phi _m}(x)\,z(x)\,dx} }}{{\,\int_{\,a}^b {w(x)\,{z^2}(x)\,dx} }}} \right)\\
=  - \left[ {r(x)\,\left( {{\Phi _m}(x)\Phi '_{n}(x) - {\Phi _n}(x)\Phi '_{m}(x)} \right)} \right]_{\,a}^{\,b} = 0\,\,\,\,\, \Leftrightarrow \,\,\,\,n \ne m.
\end{multline*}
\end{proof}
It is important to mention that the integro-differential equation \eqref{eq5.3} can be transformed to a fourth order equation using \eqref{eq5.4}. First of all, without loss of generality, we can assume in \eqref{eq5.3} that $ u(x)=1 $, because it is just enough to divide both side of \eqref{eq5.3} by $ u(x) $. In this case, substituting
\[z(x) = \frac{1}{{p\left( {{\lambda _n} + {A_z}} \right)\beta _n^*}}\left( {a(x)\,\Phi ''_{n}(x) + b(x)\,\Phi '_{n}(x) + \left( {{\lambda _n}\, + v(x)} \right){\Phi _n}(x)} \right)\]
into the equation
\[a(x)\,z''(x) + b(x)\,z'(x) + \left( {v(x) - {A_z}} \right)z(x) = 0,\]
yields
\begin{multline*}
{a^2}(x)\,\Phi _n^{(4)}(x) + 2a(x)\left( {a'(x) + b(x)} \right)\Phi _n^{(3)}(x) + \\
\left( {a(x)\left( {a''(x) + 2b'(x) + 2v(x) + {\lambda _n} - {A_z}} \right) + b(x)\left( {a'(x) + b(x)} \right)} \right)\Phi ''_{n}(x)\\
 + \left( {a(x)\left( {b''(x) + 2v'(x)} \right) + b(x)\left( {b'(x) + 2v(x) + {\lambda _n} - {A_z}} \right)} \right)\Phi '_{n}(x)\\
 + \left( {a(x)v''(x) + b(x)v'(x) + {v^2}(x) + ({\lambda _n} - {A_z})v(x) - {A_z}{\lambda _n}} \right){\Phi _n}(x) = 0,
\end{multline*}
having four independent basis solutions.
\\
\\
\noindent
{\textbf{5.2.1. An extension of the previous theorem.} This turn, assume in the right side of \eqref{eq5.3} that the function $ u(x)z(x) =h(x)$ satisifies a second order differential equation of the form
\[p(x)\,h''(x) + q(x)\,h'(x) + r(x)\,h(x) = 0,\]
instead of satisfying the equation \eqref{eq5.4} where $ p(x),q(x) $ and $ r(x) $ are all known functions. Then, equation \eqref{eq5.3} will be transformed to the fourth order equation
\begin{multline*}
p(x){\left( {a(x)\,\Phi ''_{n}(x) + b(x)\,\Phi '_{n}(x) + \left( {{\lambda _n}\,u(x) + v(x)} \right){\Phi _n}(x)} \right)^{\prime \prime }}\\
 + q(x)\,{\left( {a(x)\,\Phi ''_{n}(x) + b(x)\,\Phi '_{n}(x) + \left( {{\lambda _n}\,u(x) + v(x)} \right){\Phi _n}(x)} \right)^\prime }\\
 + r(x)\,\left( {a(x)\,\Phi ''_{n}(x) + b(x)\,\Phi '_{n}(x) + \left( {{\lambda _n}\,u(x) + v(x)} \right){\Phi _n}(x)} \right) = 0,
\end{multline*}
which is equivalent to
\begin{multline}\label{eq5.8}
p(x)a(x)\,\Phi _n^{(4)}(x) + \Big( {p(x)\left( {2a'(x) + b(x)} \right) + q(x)\,a(x)} \Big)\Phi _n^{(3)}(x)\\
 + \Big( {p(x)u(x){\lambda _n} + p(x)\left( {a''(x) + 2b'(x) + v(x)} \right) + q(x)\left( {a'(x) + b(x)} \right) + r(x)a(x)} \Big)\Phi ''_{n}(x)\\
 + \Big( {\left( {2p(x)u'(x) + q(x)u(x)} \right){\lambda _n} + p(x)\left( {b''(x) + 2v'(x)} \right) + q(x)\left( {b'(x) + v(x)} \right) + r(x)b(x)} \Big)\Phi '_{n}(x)\\
 + \Big( {\left( {p(x)u''(x) + q(x)u'(x) + r(x)u(x)} \right){\lambda _n} + p(x)v''(x) + q(x)v'(x) + r(x)v(x)} \Big){\Phi _n}(x) = 0.
\end{multline}
Noting the above-mentioned comments, equation \eqref{eq5.3} can now be generalized as
\[a(x)\,{\Phi ''_n}(x) + b(x)\,{\Phi '_n}(x) + \left( {{\lambda _n}\,u(x) + v(x)} \right)\,{\Phi _n}(x) = {R_n}(x),\]
whose self-adjoint form
\[{\left( {r(x)\,\Phi '_{n}(x)} \right)^\prime } + \left( {{\lambda _n}\,w(x) + q(x)} \right)\,{\Phi _n}(x) = \frac{{w(x)}}{{u(x)}}\,{R_n}(x),\]
yields
\begin{multline}\label{eq5.9}
\left[ {r(x)\,\left( {{\Phi _m}(x){\Phi '_n}(x) - {\Phi _n}(x){\Phi '_m}(x)} \right)} \right]_{\,a}^{\,b} + \left( {{\lambda _n} - {\lambda _m}} \right)\int_{\,a}^b {w(x)\,{\Phi _n}(x)\,{\Phi _m}(x)\,dx} \\
 = \int_{\,a}^b {\frac{{w(x)}}{{u(x)}}\,\left( {{R_n}(x)\,{\Phi _m}(x) - {R_m}(x)\,{\Phi _n}(x)} \right)\,dx} \,.
\end{multline}
Now, if in \eqref{eq5.9} we set
\[{R_n}(x) = p\left( {{\lambda _n} + {A_z}} \right)\frac{{\,\int_{\,a}^b {w(x)\,z(x)\,{\Phi _n}(x)\,dx} }}{{\,\int_{\,a}^b {w(x)\,{z^2}(x)\,dx} }}u(x)\,z(x),\]
we directly get to theorem \ref{thm5.2}. See also section \ref{subsubsec9.1.3} and relations \eqref{eq9.22} to \eqref{eq9.23}.
\\
\\
\noindent
{\textbf{5.2.2. Remark.}} The previous section shows that equation \eqref{eq5.3} is in fact a non-homogenous second order differential equation whose general solution for $ p=1 $ reads as
\begin{equation}\label{eq5.10}
{\Phi _n}(x) = {c_1}{y_{1,n}}(x) + {c_2}{y_{2,n}}(x) + {\lambda ^*}z(x)\,\,\,\,\,\,\,\,\,({c_1},{c_2},{\lambda ^*} \in\mathbb{R}),
\end{equation}
where $ {c_1}{y_{1,n}}(x) + {c_2}{y_{2,n}}(x) $ is the solution of the homogenous equation
\[a(x)\,{\Phi ''_n}(x) + b(x)\,{\Phi '_n}(x) + \left( {{\lambda _n}\,p(x) + q(x)} \right)\,{\Phi _n}(x) = 0,\]
with two independent solutions $ {y_{1,n}}(x) $ and $ {y_{2,n}}(x) $ respectively.

Hence, substituting (5.10) into (5.3) for $ p=1 $ simplifies the general solution as
\begin{equation}\label{eq5.11}
{\Phi _n}(x) = {c_1}\left( {{y_{1,n}}(x) - \frac{{\,\int_{\,a}^b {w(x)\,z(x)\,{y_{1,n}}(x)\,dx} }}{{\,\int_{\,a}^b {w(x)\,z(x)\,{y_{2,n}}(x)\,dx} }}{y_{2,n}}(x)} \right) + {\lambda ^*}\,z(x),
\end{equation}
provided that $\int_{\,a}^b {w(x)\,z(x)\,{y_{2,n}}(x)\,dx}  \ne 0\,\,\,\,\,(\forall n \notin {{\mathbb{Z}}^ + })$.

On the other hand, since in \eqref{eq5.11},
\[\int_{\,a}^b {{\Phi _n}(x)w(x)\,z(x)\,dx}  = {\lambda ^*}\int_{\,a}^b {w(x)\,{z^2}(x)\,dx} ,\]
the uncorrelated condition (5.1) for $ p=1 $ is simplified as
\begin{multline*}
\int_{\,a}^b {w(x)\,{\Phi _n}(x){\Phi _m}(x)\,dx}  - {\lambda ^*}^2\int_{\,a}^b {w(x)\,{z^2}(x)\,dx} \\
= \left( {\int_{\,a}^b {w(x)\,\Phi _n^2(x)\,dx}  - {\lambda ^*}^2\int_{\,a}^b {w(x)\,{z^2}(x)\,dx} } \right)\,{\delta _{m,n}}.
\end{multline*}
In this sense, note that
\begin{multline*}
\frac{1}{{c_1^2}}\left( {\int_{\,a}^b {w(x)\,\Phi _n^2(x)\,dx}  - {\lambda ^*}^2\int_{\,a}^b {w(x)\,{z^2}(x)\,dx} } \right) = \int_{\,a}^b {w(x)\,y_{1,n}^2(x)\,dx} \\
 + \frac{{\,{{(\int_{\,a}^b {w(x)\,z(x)\,{y_{1,n}}(x)\,dx} )}^2}}}{{\,{{(\int_{\,a}^b {w(x)\,z(x)\,{y_{2,n}}(x)\,dx} )}^2}}}\int_{\,a}^b {w(x)\,y_{2,n}^2(x)\,dx}\\
   - 2\frac{{\,\int_{\,a}^b {w(x)\,z(x)\,{y_{1,n}}(x)\,dx} }}{{\,\int_{\,a}^b {w(x)\,z(x)\,{y_{2,n}}(x)\,dx} }}\int_{\,a}^b {w(x)\,{y_{1,n}}(x){y_{2,n}}(x)\,dx} ,
\end{multline*}
i.e. the aforementioned 1-variance value is always independent of $ \lambda^{*} $.
\noindent
\begin{example}\label{ex5.3}
Let us consider the following integro-differential equation 
\begin{equation}\label{eq5.12}
{\Phi ''_n}(x) + {n^2}\,{\Phi _n}(x) = ({n^2} + A)\,z(x)\,\frac{{\int_{\,0}^\pi  {{\Phi _n}(x)\,z(x)\,dx} }}{{\int_{\,0}^\pi  {{z^2}(x)\,dx} }}\,,
\end{equation}
\end{example}
in which $ z(x) $ satisfies the equation
\begin{equation}\label{eq5.13}
\frac{{z''(x)}}{{z(x)}} = A.
\end{equation}
Noting \eqref{eq5.8}, the equivalent form of \eqref{eq5.12} together with \eqref{eq5.13} is as
\[\,\Phi _n^{(4)}(x) + \left( {{n^2} - A} \right){\Phi ''_n}(x) - A{n^2}{\Phi _n}(x) = ({D^2} + {n^2})({D^2} - A){\Phi _n}(x) = 0,\]
having the general solution
\begin{equation}\label{eq5.14}
{\Phi _n}(x) = {c_1}\sin nx + {c_2}\cos nx + {\lambda ^*}z(x)\,\,\,\,\,\,{\rm{where}}\,\,\,\,\,\,({D^2} - A)\,z(x) = 0.
\end{equation}
According to \eqref{eq5.11}, the solution \eqref{eq5.14} is simplified as
\begin{equation}\label{eq5.15}
{\Phi _n}(x) = {c_1}\left( {\sin nx - \frac{{\,\int_{\,0}^\pi  {z(x)\sin nx\,dx} }}{{\,\int_{\,0}^\pi  {z(x)\cos nx\,dx} }}\cos nx} \right) + {\lambda ^*}\,z(x),
\end{equation}
satisfying the relation
\[\int_{\,0}^\pi  {{\Phi _n}(x){\Phi _m}(x)\,dx}  - {\lambda ^*}^2\int_{\,0}^\pi  {{z^2}(x)\,dx}  = c_1^2\frac{\pi }{2}\left( {1 + \frac{{\,{{(\int_{\,0}^\pi  {z(x)\sin nx\,dx} )}^2}}}{{\,{{(\int_{\,0}^\pi  {z(x)\cos nx\,dx} )}^2}}}} \right)\,{\delta _{m,n}}\,.\]
Notice that the two sub-sequences
\[\,{\Phi _{2n}}(x) = {c_1}\left( {\sin 2nx - \frac{{\,\int_{\,0}^\pi  {z(x)\sin 2nx\,dx} }}{{\,\int_{\,0}^\pi  {z(x)\cos 2nx\,dx} }}\cos 2nx} \right) + {\lambda ^*}\,z(x),\]
and
\[\,{\Phi _{2n + 1}}(x) = {c_1}\left( {\sin (2n + 1)x - \frac{{\,\int_{\,0}^\pi  {z(x)\sin (2n + 1)x\,dx} }}{{\,\int_{\,0}^\pi  {z(x)\cos (2n + 1)x\,dx} }}\cos (2n + 1)x} \right) + {\lambda ^*}\,z(x),\]
in \eqref{eq5.15} are automatically uncorrelated (without any constraint) with respect to every fixed function $ z(x) $ that satisfies \eqref{eq5.13}.

Totally, three cases can occur for the constant value $ A $ in \eqref{eq5.13}.\\
i) If $ A = {\alpha ^2}\,\,\,(\alpha  \ne 0) $, the solution of equation \eqref{eq5.13} is as $ z(x) = {a_1}\,{e^{\alpha x}} + {a_2}\,{e^{ - \alpha x}} $ and subsequently \eqref{eq5.14} takes the form
\[\,{\Phi _n}(x) = {c_1}\sin nx + {c_2}\cos nx + {a_1}\,{e^{\alpha x}} + {a_2}\,{e^{ - \alpha x}}.\]
ii) If $ A =- {\alpha ^2}\,\,\,(\alpha  \ne 0) $ then $ z(x) = {b_1}\,\cos (\alpha x) + {b_2}\,\sin (\alpha x) $ and
\[\,{\Phi _n}(x) = {c_1}\sin nx + {c_2}\cos nx + {b_1}\,\cos (\alpha x) + {b_2}\,\sin (\alpha x).\]
iii) If $ A=0 $ then $ z(x) = {c_3}\,x + {c_4} $ and
\[\,{\Phi _n}(x) = {c_1}\sin nx + {c_2}\cos nx + {c_3}\,x + {c_4}.\]
As a particular sample, if $ z(x) = {e^x} $ is chosen from the above mentioned options, the general solution of the equation
\[\,{\Phi ''_n}(x) + {n^2}\,{\Phi _n}(x) = \frac{{2({n^2} + 1)}}{{{e^{2\pi }} - 1}}{e^x}\,\int_{\,0}^\pi  {{\Phi _n}(x)\,{e^x}dx} \,,\]
is, according to \eqref{eq5.15}, as
\begin{equation}\label{eq5.16}
{\Phi _n}(x) = {c_1}\left( {\sin nx - \frac{{\,\int_{\,0}^\pi  {{e^x}\sin nx\,dx} }}{{\,\int_{\,0}^\pi  {{e^x}\cos nx\,dx} }}\cos nx} \right) + {\lambda ^*}{e^x}.
\end{equation}
On the other hand,
\[\int_{\,0}^\pi  {{e^x}\sin nx\,dx}  =  - \frac{n}{{{n^2} + 1}}\,({( - 1)^n}{e^\pi } - 1)\,\,\,\,{\text {and}}\,\,\,\,\int_{\,0}^\pi  {{e^x}\cos nx\,dx}  = \frac{1}{{{n^2} + 1}}\,({( - 1)^n}{e^\pi } - 1),\]
imply that \eqref{eq5.16} takes the final form
\begin{equation}\label{eq5.17}
{\Phi _n}(x) = {c_1}\left( {\sin nx + n\cos nx} \right) + {\lambda ^*}{e^x},
\end{equation}
where $ {c_1},{\lambda ^*} \in\mathbb{R} $. Note that since in \eqref{eq5.17},
\[\,{\Phi _n}(0) = {\Phi '_n}(0)\,\,\,\,\,{\rm{and}}\,\,\,\,\,\,{\Phi _n}(\pi ) = {\Phi '_n}(\pi ),\]
so
\[\left[ {\,{\Phi _m}(x){\Phi '_n}(x) - {\Phi _n}(x){\Phi '_m}(x)} \right]_{\,0}^{\,\pi } = 0\,.\]
However, for other choices of $ z(x) $ we may have to determine the values $ c_{1} $ and $ \lambda^{*} $ uniquely. Also, for $ n=m $ we obtain 
\[\pi {{\mathop{\rm var}} _1}({\Phi _n}(x);{e^x}) = \int_{\,0}^\pi  {\Phi _n^2(x)\,dx}  - {\lambda ^*}^2\frac{{{e^{2\pi }} - 1}}{2} = c_1^2\frac{\pi }{2}({n^2} + 1)\,\,\,\,(\forall n \in\mathbb{N}).\]
Finally, according to remark 5.1, the sequence
\[{\Phi _n}(x) - {\lambda ^*}\,{e^x} = {c_1}\left( {\sin nx + n\cos nx} \right)\,,\]
is orthogonal with respect to the constant weight function on $ [0,\pi] $, i.e.
\begin{multline*}
\int_{\,0}^\pi  {\Big( {\sin (n + 1)x + (n + 1)\cos (n + 1)x} \Big)\Big( {\sin (m + 1)x + (m + 1)\cos (m + 1)x} \Big)\,dx}\\
  = \frac{\pi }{2}({n^2} + 2n + 2)\,{\delta _{n,m}},
\end{multline*}
in which
\[\sin (n + 1)x + (n + 1)\cos (n + 1)x = \sqrt {{n^2} + 2n + 2} \,\,\sin \Big( {(n + 1)x + \arctan \,(n + 1)} \Big),\]
having the explicit roots
\[{x_k} = \frac{{k\pi  - \arctan \,(n + 1)}}{{n + 1}}\,\,\,\,\,{\rm{for}}\,\,\,\,\,k = 1,2,...,n + 1.\]

\section{p-uncorrelated polynomials sequence (p-UPS)}\label{sec6}
Since the most important part of p-uncorrelated functions (introduced in the previous section) is the class of p-uncorrelated polynomials, we prefer to study them in a separate section and consider their monic form (without loss of generality) in a continuous space. With reference to theorem \ref{thm3.2} (or representation \eqref{eq3.12}) and using the initial monomial data $ \lbrace V_{k}=x^{k}\rbrace_{k=0}^{n} $, a monic type of p-UPS with respect to $ z(x) $, say $ {{\bar P}_n}\left( {x,p;\,z(x)} \right) $, can be generated, so that
\begin{equation}\label{eq6.1}
{{\mathop{\rm cov}} _p}\,\left( {{{\bar P}_n}\left( {x,p;\,z(x)} \right),{{\bar P}_m}\left( {x,p;\,z(x)} \right);z(x)} \right) = 0 \Leftrightarrow n \ne m.
\end{equation}
In a weighted continuous space, the condition \eqref{eq6.1} is equivalent to
\begin{multline}\label{eq6.2}
\int_{\,a}^b {w(x)\,{{\bar P}_n}\left( {x,p;\,z(x)} \right)\,{{\bar P}_m}\left( {x,p;\,z(x)} \right)\,dx}\\
  - p\frac{{\int_{\,a}^b {w(x)z(x)\,{{\bar P}_n}\left( {x,p;\,z(x)} \right)\,dx} \,\int_{\,a}^b {w(x)z(x)\,{{\bar P}_m}\left( {x,p;\,z(x)} \right)\,dx} }}{{\,\int_{\,a}^b {w(x)\,{z^2}(x)\,dx} }}\\
= \left( {\int_{\,a}^b {w(x)\,\bar P_n^{\,2}\left( {x,p;\,z(x)} \right)\,dx}  - p\frac{{{{\left( {\int_{\,a}^b {w(x)z(x)\,{{\bar P}_n}\left( {x,p;\,z(x)} \right)\,dx} } \right)}^2}}}{{\,\int_{\,a}^b {w(x)\,{z^2}(x)\,dx} }}} \right)\,{\delta _{n,m}},
\end{multline}
where $ w(x)>0 $ and $ z(x) $ are two continuous functions such that $ \int_{\,a}^b {w(x)\,{z^2}(x)\,dx}  > 0 $.

The existence of relation \eqref{eq6.2} is when the integrals $ \int_{\,a}^b {w(x)\,{x^n}z(x)\,dx}  $ all exist and are bounded for every $ n = 0,1,... $. In this section, we deal with such spaces of type \eqref{eq6.2} as the role of $ z(x)$ would be undeniable.

Let us recall that a hypergeometric polynomial as a particular case of the generalized hypergeometric series
\[{}_p{F_q}\left( {\left. {\begin{array}{*{20}{c}}
{{a_1},\,\,\,...\,\,,{a_p}}\\
{{b_1},\,\,\,...\,\,\,,{b_q}}
\end{array}} \right|\,x} \right) = \sum\limits_{k = 0}^\infty  {\frac{{{{({a_1})}_k}...{{({a_p})}_k}}}{{{{({b_1})}_k}...{{({b_q})}_k}}}\,\frac{{{x^k}}}{{k!}}} \,,\]
where $ {b_1},...,{b_q} \ne 0, - 1, - 2,... $ and 
\[{(a)_k} = a(a + 1)...(a + k - 1)\,,\]
appears when one of the values $ \{ {a_k}\} _{k = 1}^p $ is a negative integer.

Also, we recall that the aforesaid series is convergent for all $  x $ if $ p < q + 1 $, convergent for $ x \in ( - 1,1) $ if $ p=q+1 $ and it is divergent for all $ x\ne 0 $ if $ p > q + 1 $, while if one of $ \{ {a_k}\} _{k = 1}^p $ is a negative integer, it will reduce to a hypergeometric polynomial which is convergent for all real $ x $.

\begin{theorem}\label{thm6.1}
Let $ \{ {\bar P_n}(x,p;\,z(x))\}  $ be a monic p-UPS with respect to the fixed function $ z(x) $ and $ {\bar Q_m}(x) $ be an arbitrary monic polynomial of degree $ m\leq n $. Then
\[{{\mathop{\rm cov}} _p}\,\left( {{{\bar Q}_m}(x),{{\bar P}_n}(x,p;\,z(x));z(x)} \right) = {{\mathop{\rm var}} _p}\left( {{{\bar P}_n}(x,p;\,z(x));z(x)} \right)\,\,{\delta _{m,n}}.\]
\end{theorem}
\begin{proof}
Since ${{\bar P}_n}\left( {x,p;\,z(x)} \right)$ is a p-UPS, its elements are linearly independent according to theorem \ref{thm3.1}. Hence $ \{ {\bar P_k}(x,p;\,z(x))\} _{k = 0}^m $ is a basis for all polynomials of degree at most $ m $ and, for any arbitrary monic polynomial of degree $ m $, say $ {\bar Q_m}(x) $, there exist constants $ \{ {b_k}\}  $ such that
\[{\bar Q_m}(x) = \sum\limits_{k = 0}^m {{b_k}{{\bar P}_k}(x,p;\,z(x))} \,\,\,\,\,{\rm{with}}\,\,\,\,{b_m} = 1.\]
By the linearity of p-covariances with respect to $ z(x) $ and using \eqref{eq6.1} we get
\begin{align}\label{eq6.3}
{{\mathop{\rm cov}} _p}\,\left( {{{\bar Q}_m}(x),{{\bar P}_n}(x,p;\,z(x));z(x)} \right) &= \sum\limits_{k = 0}^m {{b_k}{{{\mathop{\rm cov}} }_p}\,\left( {{{\bar P}_k}(x,p;\,z(x)),{{\bar P}_n}(x,p;\,z(x));z(x)} \right)} \,\\
 &= \left\{ \begin{array}{l}
\qquad\qquad\qquad 0\,\,\,\,\,\,\,\,\,\,\,\,\,\,\,\,\,\,\,\,\,\,\,\,\,\,\,\,\,\,\,\,\,\,\,\,\,\,\,\,\,\,\,\,\,\,{\rm{if}}\,\,\,m < n,\\
{{\mathop{\rm var}} _p}\left( {{{\bar P}_n}(x,p;\,z(x));z(x)} \right)\,\,\,\,\,\qquad{\rm{if}}\,\,\,m = n.
\end{array} \right.\nonumber
\end{align}
\end{proof}
For the monomial basis $ \left\{ {{V_k} = {x^k}} \right\}_{k = 0}^n $ the elements of the main determinant \eqref{eq3.9}, i.e.
\begin{equation}\label{eq6.4}
\Delta _n^{(p)}\left( {\{ {x^k}\} _{k = 0}^n;z(x)} \right) = \left| {\begin{array}{*{20}{c}}
{\begin{array}{*{20}{c}}
{{{{\mathop{\rm var}} }_p}\,(1;z(x))}\\
{{{{\mathop{\rm cov}} }_p}\,(x,1;z(x))}\\
 \vdots \\
{{{{\mathop{\rm cov}} }_p}\,({x^n},1;z(x))}
\end{array}}&{\begin{array}{*{20}{c}}
{{{{\mathop{\rm cov}} }_p}\,(1,x;z(x))}\\
{{{{\mathop{\rm var}} }_p}\,(x;z(x))}\\
 \vdots \\
{{{{\mathop{\rm cov}} }_p}\,({x^n},x;z(x))}
\end{array}}&{\begin{array}{*{20}{c}}
 \cdots \\
 \cdots \\
 \vdots \\
 \cdots 
\end{array}}&{\begin{array}{*{20}{c}}
{{{{\mathop{\rm cov}} }_p}\,(1,{x^n};z(x))}\\
{{{{\mathop{\rm cov}} }_p}\,(x,{x^n};z(x))}\\
 \vdots \\
{{{{\mathop{\rm var}} }_p}\,({x^n};z(x))}
\end{array}}
\end{array}} \right|,
\end{equation}
can be expressed in terms of a special case of moments which are related to $ z(x) $ directly. In other words, if we define
\begin{equation}\label{eq6.5}
{\mu _n}(z(x)) = (\int_{\,a}^b {w(x)\,dx} \,)\,E\left( {{x^n}z(x)} \right) = \int_{\,a}^b {w(x)\,{x^n}z(x)\,dx} \,,
\end{equation}
then the entries of \eqref{eq6.4} can be represented in terms of definition \eqref{eq6.5} as follows
\[{{\mathop{\rm cov}} _p}\,\left( {{x^i},{x^j};z(x)} \right) = \frac{{{\mu _{i + j}}(1)}}{{{\mu _0}(1)}} - p\frac{{{\mu _i}(z(x)){\mu _j}(z(x))}}{{{\mu _0}(1)\,{\mu _0}({z^2}(x))}}\,.\]

\begin{theorem}\label{thm6.2}
A necessary condition for the existence of a monic p-UPS with respect to the function $ z(x) $ is that $ \Delta _n^{(p)}\left( {\{ {x^k}\} _{k = 0}^n;z(x)} \right) \ne 0 $. 
\end{theorem}
\begin{proof}
Suppose ${\bar P_n}(x,p;\,z(x)) = \sum\limits_{k = 0}^n {c_k^{(n)}(p)\,{x^k}} $ is a monic p-UPS with $ c_n^{(n)}(p) = 1 $. Referring to \eqref{eq6.3}, for $ m\leq n $ and $ {\bar Q_m}(x) = {x^m} $ we observe that the relations
\[{{\mathop{\rm cov}} _p}\,\left( {{x^m},{{\bar P}_n}(x,p;\,z(x))\,;z(x)} \right) = \sum\limits_{k = 0}^n {c_k^{(n)}(p)\,{{{\mathop{\rm cov}} }_p}\left( {{x^m},{x^k};z(x)} \right)} = {{\mathop{\rm var}} _p}\left( {{{\bar P}_n}(x,p;\,z(x));z(x)} \right){\delta _{m,n}},\]
will at last lead to the following solvable linear system
\begin{align*}
& \left[ {\begin{array}{*{20}{c}}
{\begin{array}{*{20}{c}}
{{{{\mathop{\rm var}} }_p}\,(1;z(x))}\\
{{{{\mathop{\rm cov}} }_p}\,(x,1;z(x))}\\
 \vdots \\
{{{{\mathop{\rm cov}} }_p}\,({x^n},1;z(x))}
\end{array}}&{\begin{array}{*{20}{c}}
{{{{\mathop{\rm cov}} }_p}\,(1,x;z(x))}\\
{{{{\mathop{\rm var}} }_p}\,(x;z(x))}\\
 \vdots \\
{{{{\mathop{\rm cov}} }_p}\,({x^n},x;z(x))}
\end{array}}&{\begin{array}{*{20}{c}}
 \cdots \\
 \cdots \\
 \vdots \\
 \cdots 
\end{array}}&{\begin{array}{*{20}{c}}
{{{{\mathop{\rm cov}} }_p}\,(1,{x^n};z(x))}\\
{{{{\mathop{\rm cov}} }_p}\,(x,{x^n};z(x))}\\
 \vdots \\
{{{{\mathop{\rm var}} }_p}\,({x^n};z(x))}
\end{array}}
\end{array}} \right]\left[ {\begin{array}{*{20}{c}}
{c_0^{(n)}(p)}\\
{c_1^{(n)}(p)}\\
 \vdots \\
{c_n^{(n)}(p)}
\end{array}} \right] \\
& \qquad\qquad\qquad \qquad\qquad\qquad = \left[ {\begin{array}{*{20}{c}}
0\\
0\\
 \vdots \\
{{{{\mathop{\rm var}} }_p}\left( {{{\bar P}_n}(x,p;\,z(x));z(x)} \right)}
\end{array}} \right],
\end{align*}
which obviously implies that $ \Delta _n^{(p)}\left( {\{ {x^k}\} _{k = 0}^n;z(x)} \right) \ne 0 $.
\end{proof}
One of the consequences of solving the above system is that
\begin{equation}\label{eq6.6}
{{\mathop{\rm var}} _p}\left( {{{\bar P}_n}(x,p;\,z(x));z(x)} \right) = \frac{{\Delta _n^{(p)}\left( {\{ {x^k}\} _{k = 0}^n;z(x)} \right)}}{{\Delta _{n - 1}^{(p)}\left( {\{ {x^k}\} _{k = 0}^n;z(x)} \right)}}\,\,\,\,\,\,\,\,(n \ge 1),
\end{equation}
which is valid for $ n=0 $ if we set $ \Delta _{ - 1}^{(p)}(.) = 1 $.

Another consequence is a general representation for the non-monic case of uncorrelated polynomials as
\begin{equation}\label{eq6.7}
{P_n}\left( {x,p;\,z(x)} \right)=
\begin{vmatrix}
{{{{\mathop{\rm var}} }_p}\,(1;z(x))} & {{{{\mathop{\rm cov}} }_p}\,(1,x;z(x))} & \ldots & {{{{\mathop{\rm cov}} }_p}\,(1,{x^n};z(x))}\\
{{{{\mathop{\rm cov}} }_p}\,(x,1;z(x))} &  {{{{\mathop{\rm var}} }_p}\,(x;z(x))} & \ldots & {{{{\mathop{\rm cov}} }_p}\,(x,{x^n};z(x))}\\
\vdots & \vdots & \vdots & \vdots \\
{{{{\mathop{\rm cov}} }_p}\,({x^{n - 1}},1;z(x))} & {{{{\mathop{\rm cov}} }_p}\,({x^{n - 1}},x;z(x))} & \ldots & {{{{\mathop{\rm cov}} }_p}\,({x^{n - 1}},{x^n};z(x))}\\
1 & x & \ldots & x^{n} 
\end{vmatrix}
\end{equation}

\begin{remark}\label{rem6.3}
For $ p=1 $ and $ z(x) = {x^\lambda } $, since
\[{{\mathop{\rm cov}} _1}({x^i},{x^j};{x^\lambda }) = 0\,\,\, \Leftrightarrow \,\,\,\lambda  = i\,\,\,{\rm{or}}\,\,\,\lambda  = j,\]
the rows of determinant (6.7) show that $ z(x) $ should conditionally be chosen as a member of the monomial basis $ {\{ {x^k}\} _{k = 0}} $. For example, if $ \lambda=5 $, the finite polynomial set $ \{ {P_n}(x,1;\,{x^5})\} _{n = 0}^5 $ is finitely 1-uncorrelated with respect to $ z(x) = {x^5} $.
The determinant \eqref{eq6.7} also shows that for the monic type of the polynomials we always have $ {\bar P_n}(x,1;\,{x^n}) = {x^n} $, e.g. $ {\bar P_5}(x,1;\,{x^5}) = {x^5} $.
\end{remark}
\begin{theorem}\label{thm6.4}
Let $ \{ {P_n}(x,p;\,z(x))\}  $ be a (weighted) p-UPS with respect to the function $ z(x) $. For any arbitrary polynomial $ {Q_n}(x) = \sum\limits_{k = 0}^n {{q_k}{x^k}}  $ of degree $ n $, we have
\begin{align*}
{{\mathop{\rm cov}} _p}\,\left( {{Q_n}(x),{P_n}\left( {x,p;\,z(x)} \right);z(x)} \right) &= {q_n}{{\mathop{\rm cov}} _p}\,\left( {{x^n},{P_n}\left( {x,p;\,z(x)} \right);z(x)} \right)\\
& = {q_n}\,c_n^{(n)}(p)\frac{{\Delta _n^{(p)}\left( {\{ {x^k}\} _{k = 0}^n;z(x)} \right)}}{{\Delta _{n - 1}^{(p)}\left( {\{ {x^k}\} _{k = 0}^n;z(x)} \right)}},
\end{align*}
where $ c_n^{(n)}(p) $ denotes the leading coefficient of $ {P_n}(x,p;\,z(x)) $.
\end{theorem}
\begin{proof}
Suppose $ {Q_n}(x) = {q_n}{x^n} + Q_{n - 1}^*(x) $ in which $ Q_{n - 1}^*(x) $ is a polynomial of degree at most $ n-1 $. In this case,
\begin{align*}
&{{\mathop{\rm cov}} _p}\,\left( {{Q_n}(x),{P_n}\left( {x,p;\,z(x)} \right);z(x)} \right)\\
&\qquad = {q_n}{{\mathop{\rm cov}} _p}\,\left( {{x^n},{P_n}\left( {x,p;\,z(x)} \right);z(x)} \right) + {{\mathop{\rm cov}} _p}\,\left( {Q_{n - 1}^*(x),{P_n}\left( {x,p;\,z(x)} \right);z(x)} \right)\\
 &\qquad = {q_n}{{\mathop{\rm cov}} _p}\,\left( {{x^n},{P_n}\left( {x,p;\,z(x)} \right);z(x)} \right),
\end{align*}
and the proof is completed by noting \eqref{eq6.6} and theorem \ref{thm6.2}.
\end{proof}

\subsection{A generic recurrence relation for p-UPS}\label{subsec6.5}
Let $ {\bar Q_m}(x) $ be an arbitrary monic polynomial of degree $ m $ and let $ \{ {\bar P_n}(x,p;\,z(x))\}  $ be a monic p-UPS with respect to the function $ z(x) $. According to theorem \ref{thm6.1} and relation \eqref{eq6.3}, we have
\begin{equation}\label{eq6.8}
{\bar Q_m}(x) = \sum\limits_{k = 0}^m {\frac{{{{{\mathop{\rm cov}} }_p}\,\left( {{{\bar Q}_m}(x),{{\bar P}_k}(x,p;\,z(x));z(x)} \right)}}{{{{{\mathop{\rm var}} }_p}\left( {{{\bar P}_k}(x,p;\,z(x));z(x)} \right)}}{{\bar P}_k}(x,p;\,z(x))} \,.
\end{equation} 
Now, replacing $ m=n+1 $ and $ {\bar Q_{n + 1}}(x) = x\,{\bar P_n}\left( {x,p;\,z(x)} \right) $ in \eqref{eq6.8} as
\begin{equation}\label{eq6.9}
x\,{\bar P_n}(x,p;\,z(x)) = \sum\limits_{k = 0}^{n + 1} {\frac{{{{{\mathop{\rm cov}} }_p}\,\left( {x\,{{\bar P}_n}(x,p;\,z(x)),{{\bar P}_k}(x,p;\,z(x));z(x)} \right)}}{{{{{\mathop{\rm var}} }_p}\left( {{{\bar P}_k}(x,p;\,z(x));z(x)} \right)}}{{\bar P}_k}(x,p;\,z(x))} \,,
\end{equation}
gives a $ n+2 $ term recurrence relation for the sequence $ \{ {\bar P_n}(x,p;\,z(x))\}  $.

In general, since
\begin{multline*}
{{\mathop{\rm cov}} _p}\,\left( {f(x)h(x),g(x);z(x)} \right) = {{\mathop{\rm cov}} _p}\,\left( {f(x),h(x)g(x);z(x)} \right)\\
 + p\frac{{E\left( {f(x)z(x)} \right)E\left( {h(x)g(x)z(x)} \right) - E\left( {g(x)z(x)} \right)E\left( {f(x)h(x)z(x)} \right)}}{{E\left( {{z^2}(x)} \right)}},
\end{multline*}
the coefficients of \eqref{eq6.9} can be changed in the form
\begin{multline*}
{{\mathop{\rm cov}} _p}\,\left( {x\,{{\bar P}_n}(x,p;\,z(x)),{{\bar P}_k}(x,p;\,z(x));z(x)} \right) = {{\mathop{\rm cov}} _p}\,\left( {{{\bar P}_n}(x,p;\,z(x)),x\,{{\bar P}_k}(x,p;\,z(x));z(x)} \right)\\
 + p\frac{{E\left( {x\,{{\bar P}_k}(x,p;\,z(x))z(x)} \right)E\left( {{{\bar P}_n}(x,p;\,z(x))z(x)} \right) - E\left( {x\,{{\bar P}_n}(x,p;\,z(x))z(x)} \right)E\left( {{{\bar P}_k}(x,p;\,z(x))z(x)} \right)}}{{E\left( {{z^2}(x)} \right)}}.
\end{multline*}
Note in the above formula that
\[{{\mathop{\rm cov}} _p}\,\left( {{{\bar P}_n}(x,p;\,z(x)),x\,{{\bar P}_k}(x,p;\,z(x));z(x)} \right) = 0 \Leftrightarrow k + 1 < n.\]
Accordingly, if $ p=0 $ or $ E\left( {z(x){{\bar P}_k}(x,p;z(x))} \right) = 0 $ in \eqref{eq6.9}, it will reduce to the same as celebrated three term recurrence relation of monic orthogonal polynomials \cite{ref4}. See also relations \eqref{eq8.8} and \eqref{eq8.9} in this regard.

\subsection{A complete uncorrelated sequence of hypergeometric polynomials of $ _4{F_3} $ type}
In this section, we introduce a complete uncorrelated hypergeometric polynomial that reveals the importance of the role of a non-constant fixed function. Let $ w(x) = 1 $ be defined on $ [0,1] $ and $ z(x)=x^{r} $ for $ r >  - 1/2 $ as $ \int_{\,0}^1 {w(x){z^2}(x)dx}  = \frac{1}{{2r + 1}} $. Then the components of determinant \eqref{eq6.7} are computed as
\begin{multline}\label{eq6.10}
{{\mathop{\rm cov}} _p}\,\left( {{x^i},{x^j};{x^r}} \right) = \frac{{(r + i + 1)(r + j + 1) - p(2r + 1)(i + j + 1)}}{{(r + i + 1)(r + j + 1)(i + j + 1)}}\\
{\rm{for}}\,\,\,i = 0,1,...,n - 1\,\,{\rm{and}}\,j = 0,1,...,n.
\end{multline}
A few samples of the monic type of the corresponding p-uncorrelated polynomials are
\begin{align*}
{{\bar P}_0}(x,p;\,{x^r}) &= 1,\\
{{\bar P}_1}(x,p;\,{x^r}) &= x - \frac{{r + 1}}{{2(r + 2)}}\frac{{(r + 1)(r + 2) - 2p(2r + 1)}}{{{{(r + 1)}^2} - p(2r + 1)}},\\
{{\bar P}_2}(x,p;\,{x^r}) &= {x^2} - \frac{{r + 2}}{{r + 3}}\frac{{{{(r + 1)}^2}(r + 2)(r + 3) - p(2r + 1)(5{r^2} + 5r + 6)}}{{{{(r + 1)}^2}{{(r + 2)}^2} - 4p(2r + 1)({r^2} + r + 1)}}\,\,x\\
&\quad + \frac{{r + 1}}{{6(r + 3)}}\frac{{(r + 1){{(r + 2)}^2}(r + 3) - 6p(2r + 1)({r^2} + r + 2)}}{{{{(r + 1)}^2}{{(r + 2)}^2} - 4p(2r + 1)({r^2} + r + 1)}}\,,
\end{align*}
satisfying the uncorrelatedness condition
\[\int_{\,0}^1 {{{\bar P}_n}(x,p;\,{x^r}){{\bar P}_m}(x,p;\,{x^r})\,dx}  = p(2r + 1)\int_{\,0}^1 {{x^r}{{\bar P}_n}(x,p;\,{x^r})dx} \,\int_{\,0}^1 {{x^r}{{\bar P}_m}(x,p;\,{x^r})dx}  \Leftrightarrow n \ne m.\]
For $ p=0 $, the above samples are in fact the monic type of shifted Legendre polynomials 
\[{P_n}(x,0;\,{x^r}) = {P_{n, + }}(x) = {}_2{F_1}\left( {\left. {\begin{array}{*{20}{c}}
{ - n,\,\,n + 1}\\
1
\end{array}\,} \right|\,x} \right) = \sum\limits_{k = 0}^n {\frac{{{{( - n)}_k}{{(n + 1)}_k}}}{{{{(1)}_k}}}\,\frac{{{x^k}}}{{k!}}} \,,\]
having the orthogonal property \cite{ref4}
\[\int_{\,0}^1 {{P_{n, + }}(x){P_{m, + }}(x)\,dx}  = \frac{1}{{2n + 1}}{\delta _{n,m}}\,,\]
while for the optimized case $ p=1 $, relation \eqref{eq6.10} is first simplified as
\begin{multline}\label{eq6.11}
{{\mathop{\rm cov}} _1}\left( {{x^i},{x^j};{x^r}} \right) = \frac{{(r - i)(r - j)}}{{(r + i + 1)(r + j + 1)(i + j + 1)}}\\
{\rm{for}}\,\,\,i = 0,1,...,n - 1\,\,{\rm{and}}\,j = 0,1,...,n.
\end{multline}
If \eqref{eq6.11} is directly substituted into \eqref{eq6.7}, with the aid of advanced mathematical software, see e.g. \cite{ref10, ref11} we can find the explicit form of the polynomials as
\begin{equation}\label{eq6.12}
{P_n}(x,1;\,{x^r}) = {}_4{F_3}\left( {\left. {\begin{array}{*{20}{c}}
{ - n,\,\,n + 1,\, - r,\,\,r + 2}\\
{1,\,\, - r + 1,\,\,r + 1}
\end{array}\,} \right|\,x} \right)\,,
\end{equation}
satisfying the complete uncorrelatedness condition
\begin{multline}\label{eq6.13}
\int_{\,0}^1 {{P_n}(x,1;\,{x^r}){P_m}(x,1;\,{x^r})\,dx}  - (2r + 1)\int_{\,0}^1 {{x^r}{P_n}(x,1;\,{x^r})dx} \,\int_{\,0}^1 {{x^r}{P_m}(x,1;\,{x^r})dx} \\
 = \left( {\int_{\,0}^1 {P_n^2(x,1;\,{x^r})\,dx}  - (2r + 1){{\left( {\int_{\,0}^1 {{x^r}{P_n}(x,1;\,{x^r})dx} } \right)}^2}} \right){\delta _{m,n}},
\end{multline}
where
\begin{align}\label{eq6.14}
\int_{\,0}^1 {{x^r}{P_n}(x,1;\,{x^r})dx}  &= \sum\limits_{k = 0}^n {\frac{{{{( - n)}_k}{{(n + 1)}_k}{{( - r)}_k}{{(r + 2)}_k}}}{{{{(1)}_k}{{( - r + 1)}_k}{{(r + 1)}_k}k!\,\,(r + 1 + k)}}\,} \\
& = \frac{1}{{r + 1}}{}_3{F_2}\left( {\left. {\begin{array}{*{20}{c}}
{ - n,\,\,n + 1,\, - r}\\
{1,\,\, - r + 1}
\end{array}\,} \right|\,\,1} \right) = \frac{1}{{r + 1}}\frac{{{{(1 + r)}_n}}}{{{{(1 - r)}_n}}},\nonumber
\end{align}
and
\begin{multline*}
\int_{\,0}^1 {{P_n}(x,1;\,{x^r}){P_m}(x,1;\,{x^r})\,dx}  \\
= \sum\limits_{k = 0}^n {\frac{{{{( - n)}_k}{{(n + 1)}_k}{{( - r)}_k}{{(r + 2)}_k}}}{{{{(2)}_k}{{( - r + 1)}_k}{{(r + 1)}_k}k!\,}}{}_5{F_4}\left( {\left. {\begin{array}{*{20}{c}}
{ - m,\,\,m + 1,\, - r,\,\,r + 2,\,k + 1}\\
{1,\,\, - r + 1,\,\,r + 1,\,\,k + 2}
\end{array}\,} \right|\,\,1} \right)\,} ,
\end{multline*}
which simplifies the left side of \eqref{eq6.13} as
\begin{align}\label{eq6.15}
&\int_{\,0}^1 {{P_n}(x,1;\,{x^r}){P_m}(x,1;\,{x^r})\,dx}  - (2r + 1)\int_{\,0}^1 {{x^r}{P_n}(x,1;\,{x^r})dx} \,\int_{\,0}^1 {{x^r}{P_m}(x,1;\,{x^r})dx} \\
&\qquad = \sum\limits_{k = 0}^n {\frac{{{{( - n)}_k}{{(n + 1)}_k}{{( - r)}_k}{{(r + 2)}_k}}}{{{{(2)}_k}{{( - r + 1)}_k}{{(r + 1)}_k}k!\,}}{}_5{F_4}\left( {\left. {\begin{array}{*{20}{c}}
{ - m,\,\,m + 1,\, - r,\,\,r + 2,\,k + 1}\\
{1,\,\, - r + 1,\,\,r + 1,\,\,k + 2}
\end{array}\,} \right|\,\,1} \right)}\nonumber \\
&\qquad\quad - \frac{{2r + 1}}{{{{(r + 1)}^2}}}\frac{{{{(1 + r)}_n}{{(1 + r)}_m}}}{{{{(1 - r)}_n}{{(1 - r)}_m}}}.\nonumber
\end{align}
Of course, the result of \eqref{eq6.14} has been derived by the following technique. Since
\[H(n) = {}_3{F_2}\left( {\left. {\begin{array}{*{20}{c}}
{ - n,\,\,n + 1,\, - r}\\
{1,\,\, - r + 1}
\end{array}\,} \right|\,\,1} \right),\]
satisfies the first order recurrence equation
\[H(n + 1) = \frac{{n + 1 + r}}{{n + 1 - r}}H(n),\]
so
\[{}_3{F_2}\left( {\left. {\begin{array}{*{20}{c}}
{ - n,\,\,n + 1,\, - r}\\
{1,\,\, - r + 1}
\end{array}\,} \right|\,\,1} \right) = \frac{{{{(1 + r)}_n}}}{{{{(1 - r)}_n}}}.\]
This technique can similarly be used to obtain $ {}_5{F_4}(.) $ in \eqref{eq6.15}. First, we can find out that the sequence
\[S(m) = {}_5{F_4}\left( {\left. {\begin{array}{*{20}{c}}
{ - m,\,\,m + 1,\, - r,\,\,r + 2,\,k + 1}\\
{1,\,\, - r + 1,\,\,r + 1,\,\,k + 2}
\end{array}\,} \right|\,\,1} \right),\]
satisfies the second order recurrence equation
\begin{multline*}
(m + 1)(m + 2 - r)(m + 3 + k)\,S(m + 2)\\
 - (2m + 3)\left( {(m + 1)(m + 2) + r(k + 1)} \right)\,S(m + 1) + (m + 2)(m + 1 + r)(m - k)\,S(m) = 0,
\end{multline*}
having two independent solutions
\[{S_1}(m) = \frac{{(m + r){{(r)}_m}}}{{(m - r){{( - r)}_m}}}\,\,\,\,\,\,{\rm{and}}\,\,\,\,{S_2}(m) = \frac{{{{( - k)}_m}}}{{(k + m + 1)(k + m){{(k)}_m}}}.\]
It is clear that
\begin{equation}\label{eq6.16}
S(m) = A\,{S_1}(m) + B\,{S_2}(m).
\end{equation}
To compute the unknown coefficients $ A $ and $ B $, it is sufficient to replace $ m=0,1 $ respectively in \eqref{eq6.16} to finally get
\begin{multline}\label{eq6.17}
{}_5{F_4}\left( {\left. {\begin{array}{*{20}{c}}
{ - m,\,\,m + 1,\, - r,\,\,r + 2,\,k + 1}\\
{1,\,\, - r + 1,\,\,r + 1,\,\,k + 2}
\end{array}\,} \right|\,\,1} \right)\\
 =  - \frac{{k + 1}}{{(k + r + 1)(r + 1)}}\left( {\frac{{(2r + 1)(m + r){{(r)}_m}}}{{(m - r){{( - r)}_m}}} + \,\frac{{rk(k - r){{( - k)}_m}}}{{(k + m + 1)(k + m){{(k)}_m}}}} \right).
\end{multline}
Thus, in order to compute
\[{{\mathop{\rm var}} _1}\left( {{P_n}(x,1;\,{x^r});\,{x^r}} \right) = \int_{\,0}^1 {P_n^2(x,1;\,{x^r})\,dx}  - (2r + 1){\left( {\int_{\,0}^1 {{x^r}{P_n}(x,1;\,{x^r})dx} } \right)^2},\]
we first suppose in \eqref{eq6.15} that $ m=n $ and then refer to \eqref{eq6.17} to arrive at
\begin{multline}\label{eq6.18}
{S^*} = \sum\limits_{k = 0}^n {\frac{{{{( - n)}_k}{{(n + 1)}_k}{{( - r)}_k}{{(r + 2)}_k}}}{{{{(2)}_k}{{( - r + 1)}_k}{{(r + 1)}_k}k!\,}}{}_5{F_4}\left( {\left. {\begin{array}{*{20}{c}}
{ - n,\,\,n + 1,\, - r,\,\,r + 2,\,k + 1}\\
{1,\,\, - r + 1,\,\,r + 1,\,\,k + 2}
\end{array}\,} \right|\,\,1} \right)} \\
 =  - \frac{{(2r + 1)(n + r){{(r)}_n}}}{{(r + 1)(n - r){{( - r)}_n}}}\sum\limits_{k = 0}^n {\frac{{{{( - n)}_k}{{(n + 1)}_k}{{( - r)}_k}{{(r + 2)}_k}}}{{{{(2)}_k}{{( - r + 1)}_k}{{(r + 1)}_k}k!\,}}\frac{{k + 1}}{{k + r + 1}}} \\
 - \frac{r}{{r + 1}}\sum\limits_{k = 0}^n {\frac{{{{( - n)}_k}{{(n + 1)}_k}{{( - r)}_k}{{(r + 2)}_k}}}{{{{(2)}_k}{{( - r + 1)}_k}{{(r + 1)}_k}k!\,}}\frac{{k + 1}}{{k + r + 1}}\frac{{k(k - r){{( - k)}_n}}}{{(k + n + 1)(k + n){{(k)}_n}}}} .
\end{multline}
Since in general
\[\frac{{{{(a)}_k}}}{{{{(a + 1)}_k}}} = \frac{a}{{a + k}},\]
relation \eqref{eq6.18} is simplified as
\begin{align}\label{eq6.19}
{S^*} &=  - \frac{{(2r + 1)(n + r){{(r)}_n}}}{{{{(r + 1)}^2}(n - r){{( - r)}_n}\,}}\frac{{{{(1 + r)}_n}}}{{{{(1 - r)}_n}\,}} + \frac{{{r^2}}}{{{{(r + 1)}^2}(n + 1)!\,}}\sum\limits_{k = 0}^n {\frac{{{{( - k)}_n}{{(n + 1)}_k}{{( - n)}_k}}}{{{{(1)}_k}{{(n + 2)}_k}}}} \\
& =  - \frac{{(2r + 1)(n + r){{(r)}_n}}}{{{{(r + 1)}^2}(n - r){{( - r)}_n}\,}}\frac{{{{(1 + r)}_n}}}{{{{(1 - r)}_n}\,}} + \frac{{{r^2}}}{{{{(r + 1)}^2}(n + 1)!\,}}\frac{{n!\,\,{{(n + 1)}_n}}}{{{{(n + 2)}_n}}}.\nonumber
\end{align}
Noting that $ {( - k)_n} = 0 $ for any $ k<n $ in \eqref{eq6.19}, the following final result will be derived.
\begin{corollary}\label{coro6.6.1}
If $ r >  - 1/2$ and $ r \notin {{\mathbb{Z}}^ + } $, then
\[\int_{\,0}^1 {P_n^2(x,1;\,{x^r})\,dx}  - (2r + 1){\left( {\int_{\,0}^1 {{x^r}{P_n}(x,1;\,{x^r})dx} } \right)^2} = \frac{{{r^2}}}{{{{(r + 1)}^2}}}\frac{1}{{2n + 1}}.\]
Moreover, \eqref{eq6.15} is simplified as
\begin{multline*}
\int_{\,0}^1 {{P_n}(x,1;\,{x^r}){P_m}(x,1;\,{x^r})\,dx}  - (2r + 1)\int_{\,0}^1 {{x^r}{P_n}(x,1;\,{x^r})dx} \,\int_{\,0}^1 {{x^r}{P_m}(x,1;\,{x^r})dx} \\
 = \frac{{{r^2}}}{{{{(r + 1)}^2}(m + 1)!\,}}\sum\limits_{k = 0}^n {\frac{{{{(n + 1)}_k}{{( - n)}_k}}}{{{{(m + 2)}_k}}}\frac{{{{( - k)}_m}}}{{k!}}}  = 0 \Leftrightarrow m \ne n,
\end{multline*}
which means
\[\sum\limits_{k = m}^n {\frac{{{{(n + 1)}_k}{{( - n)}_k}}}{{{{(m + 2)}_k}}}\frac{{{{( - k)}_m}}}{{k!}}}  = \frac{{(n + 1)!}}{{2n + 1}}{\delta _{n,m}}.\]
\end{corollary}
\begin{remark}\label{rem6.6.2}
It is interesting to know that
\begin{align*}
\mathop {\lim }\limits_{r \to \infty } {P_n}(x,1;\,{x^r}) &= \mathop {\lim }\limits_{r \to \infty } {}_4{F_3}\left( {\left. {\begin{array}{*{20}{c}}
{ - n,\,\,n + 1,\, - r,\,\,r + 2}\\
{1,\,\, - r + 1,\,\,r + 1}
\end{array}\,} \right|\,x} \right)\\
&= {}_2{F_1}\left( {\left. {\begin{array}{*{20}{c}}
{ - n,\,\,n + 1}\\
1
\end{array}\,} \right|\,x} \right) = {P_{n, + }}(x),
\end{align*}
and from corollary \ref{coro6.6.1} we subsequently conclude that 
\begin{multline*}
\mathop {\lim }\limits_{r \to \infty } \frac{{{r^2}}}{{{{(r + 1)}^2}}}\frac{1}{{2n + 1}} = \mathop {\lim }\limits_{r \to \infty } \left( {\int_{\,0}^1 {P_n^2(x,1;\,{x^r})\,dx}  - (2r + 1){{\left( {\int_{\,0}^1 {{x^r}{P_n}(x,1;\,{x^r})dx} } \right)}^2}} \right)\\
 = \int_{\,0}^1 {\mathop {\lim }\limits_{r \to \infty } P_n^2(x,1;\,{x^r})\,dx}  - \mathop {\lim }\limits_{r \to \infty } \,(2r + 1){\left( {\frac{1}{{r + 1}}\frac{{{{(1 + r)}_n}}}{{{{(1 - r)}_n}}}} \right)^2} = \int_{\,0}^1 {P_{n, + }^2(x)\,dx}  = \frac{1}{{2n + 1}}.
\end{multline*}
\end{remark}

\begin{remark}\label{rem6.6.3}
Instead of $ z(x) = {x^r} $, if we take $ z(x) = {(1 - x)^r} $ defined on $ [0,1] $, the condition \eqref{eq6.13} changes to
\begin{align*}
&\int_{\,0}^1 {{P_n}\left( {1 - x,1;\,{{(1 - x)}^r}} \right){P_m}\left( {1 - x,1;\,{{(1 - x)}^r}} \right)\,dx} \\
& - (2r + 1)\int_{\,0}^1 {{{(1 - x)}^r}{P_n}\left( {1 - x,1;\,{{(1 - x)}^r}} \right)dx} \,\int_{\,0}^1 {{{(1 - x)}^r}{P_m}\left( {1 - x,1;\,{{(1 - x)}^r}} \right)dx}  \\
&\qquad\qquad\qquad \qquad\qquad \qquad= \frac{{{r^2}}}{{{{(r + 1)}^2}}}\frac{1}{{2n + 1}}{\delta _{m,n}},
\end{align*}
in which
\begin{multline*}
{P_n}\left( {1 - x,1;\,{{(1 - x)}^r}} \right) = {}_4{F_3}\left( {\left. {\begin{array}{*{20}{c}}
{ - n,\,\,n + 1,\, - r,\,\,r + 2}\\
{1,\,\, - r + 1,\,\,r + 1}
\end{array}\,} \right|\,1 - x} \right)\\
 = \sum\limits_{k = 0}^n {\frac{{{{( - n)}_k}{{(\,n + 1)}_k}{{( - r)}_k}{{(\,r + 2)}_k}}}{{{{(1)}_k}{{( - r + 1)}_k}{{(\,r + 1)}_k}}}{}_4{F_3}\left( {\left. {\begin{array}{*{20}{c}}
{ - n + k,\,\,n + 1 + k,\, - r + k,\,\,r + 2 + k}\\
{1 + k,\,\, - r + 1 + k,\,\,r + 1 + k}
\end{array}\,} \right|\,1} \right)\,\frac{{{{( - x)}^k}}}{{k!}}} .
\end{multline*}
Using corollary \ref{coro6.6.1}, we can now establish an optimized approximation (or expansion) of polynomial type for any expandable function, say $ f(x) $, as follows
\[f(x) \cong \frac{{{{(r + 1)}^2}}}{{{r^2}}}\sum\limits_{k = 0}^{n \to \infty } {\,(2k + 1){{{\mathop{\rm cov}} }_1}\left( {f(x),{P_k}(x,1;\,{x^r});\,{x^r}} \right){P_k}(x,1;\,{x^r})} ,\]
whose error 1-variance is obviously minimized with respect to the fixed function $ z(x)=x^{r} $ on $ [0,1] $.

Furthermore, if $ f(x) $ is of a polynomial type, the approximate operator in the above relation will change to an equality according to finite type theorem \ref{thm4.1} and its important consequences in remark \ref{rem4.1.1}. For example, since
\begin{align*}
{P_0}(x,1;\,{x^r}) &= 1,\,\\
{P_1}(x,1;\,{x^r}) &=  - 2\frac{{r(r + 2)}}{{{r^2} - 1}}x + 1,\,
\end{align*}
and
\[{P_2}(x,1;\,{x^r}) = 6\frac{{r(r + 3)}}{{(r - 2)(r + 1)}}{x^2} - 6\frac{{r(r + 2)}}{{{r^2} - 1}}\,\,x + 1\,,\]
every arbitrary polynomial of degree $ 2 $, say $ a{x^2} + bx + c $, can be expanded in terms of the above samples as follows
\begin{multline*}
a{x^2} + bx + c = \frac{{{{(r + 1)}^2}}}{{{r^2}}}\,{{\mathop{\rm cov}} _1}\left( {a{x^2} + bx + c,{P_0}(x,1;\,{x^r});{x^r}} \right)\\
 + 3\frac{{{{(r + 1)}^2}}}{{{r^2}}}\,{{\mathop{\rm cov}} _1}\left( {a{x^2} + bx + c,{P_1}(x,1;\,{x^r});{x^r}} \right)\left( { - 2\frac{{r(r + 2)}}{{{r^2} - 1}}x + 1} \right)\\
 + 5\frac{{{{(r + 1)}^2}}}{{{r^2}}}\,{{\mathop{\rm cov}} _1}\left( {a{x^2} + bx + c,{P_2}(x,1;\,{x^r});{x^r}} \right)\left( {6\frac{{r(r + 3)}}{{(r - 2)(r + 1)}}{x^2} - 6\frac{{r(r + 2)}}{{{r^2} - 1}}\,\,x + 1} \right),
\end{multline*}
\end{remark}
where
\begin{multline*}
{{\mathop{\rm cov}} _1}\left( {a{x^2} + bx + c,{P_0}(x,1;\,{x^r});{x^r}} \right) = \frac{1}{3}\frac{{r(r - 2)}}{{(r + 1)(r + 3)}}a + \frac{1}{2}\frac{{r(r - 1)}}{{(r + 1)(r + 2)}}b + \frac{{{r^2}}}{{{{(r + 1)}^2}}}c,\\
{{\mathop{\rm cov}} _1}\left( {a{x^2} + bx + c,{P_1}(x,1;\,{x^r});{x^r}} \right) =  - \frac{1}{6}\frac{{r(r - 2)}}{{(r + 1)(r + 3)}}a - \frac{1}{6}\frac{{r(r - 1)}}{{(r + 1)(r + 2)}}b,
\end{multline*}
and
\[{{\mathop{\rm cov}} _1}\left( {a{x^2} + bx + c,{P_2}(x,1;\,{x^r});{x^r}} \right) = \frac{1}{{30}}\frac{{r(r - 2)}}{{(r + 1)(r + 3)}}a.\]
According to remark \ref{rem5.1}, since the polynomials \eqref{eq6.12} are now explicitly known, a new sequence of orthogonal functions can be constructed as follows
\[{G_n}(x,1;r) = {}_4{F_3}\left( {\left. {\begin{array}{*{20}{c}}
{ - n,\,\,n + 1,\, - r,\,\,r + 2}\\
{1,\,\, - r + 1,\,\,r + 1}
\end{array}\,} \right|\,x} \right) - \frac{{2r + 1}}{{r + 1}}\frac{{{{(1 + r)}_n}}}{{{{(1 - r)}_n}}}\,{x^r},\]
satisfying the orthogonality relation
\[\int_{\,0}^1 {{G_n}(x,1;r){G_m}(x,1;r)\,dx}  = \frac{{{r^2}}}{{{{(r + 1)}^2}}}\frac{1}{{2n + 1}}{\delta _{n,m}}.\]
The above orthogonality is valid for every $ n $ if and only if $ r >  - 1/2 $ and $ r \notin {{\mathbb{Z}}^ + } $. However, if $ r = m \in\mathbb{N} $, the polynomial set $ \{ {G_n}(x,1;m)\} _{n = 0}^{m - 1} $ will be finitely orthogonal on $ [0,1] $ according to remark \ref{rem6.3}. For instance, for $ m=3 $ the elements of the finite set
\[\{ {G_n}(x,1;3)\} _{n = 0}^2 = \{  - \frac{7}{4}{x^3} + 1,\,\,\frac{7}{2}{x^3} - \frac{{15}}{4}x + 1,\,\, - \frac{{35}}{2}{x^3} + 27{x^2} - \frac{{45}}{4}x + 1\} ,\]
are mutually orthogonal with respect to the weight function $ w(x)=1 $ on $ [0,1] $ having the norm square value $ \frac{9}{{16}}\lbrace\frac{1}{{2n + 1}}\rbrace_{n=0}^{2} $.

\section{On the ordinary case of p-covariances}\label{sec7}
When $ Z = c^{*}  \ne 0 $ and $ p=1 $ we have $ {{\mathop{\rm cov}} _1}\,(X,Y;c^{*} ) = {\mathop{\rm cov}} \,(X,Y) $. Replacing these assumptions in \eqref{eq3.12} as
\begin{equation}\label{eq7.1}
{X_n} = 
\begin{vmatrix}
{{\mathop{\rm var}} \,({V_0})} & {{\mathop{\rm cov}} \,({V_0},{V_1})} &\ldots & {{\mathop{\rm cov}} \,({V_0},{V_n})} \\
{{\mathop{\rm cov}} \,({V_1},{V_0})} & {{\mathop{\rm var}} \,({V_1})} & \ldots & {{\mathop{\rm cov}} \,({V_1},{V_n})} \\
 \vdots  & \vdots & \vdots & \vdots \\
 {{\mathop{\rm cov}} \,({V_{n - 1}},{V_0})} & {{\mathop{\rm cov}} \,({V_{n - 1}},{V_1})} & \ldots & {{\mathop{\rm cov}} \,({V_{n-1}},{V_n})}\\
 V_{0} & V_{1} & \ldots & V_{n} 
\end{vmatrix},
\end{equation}
generates an ordinary uncorrelated element. Note that the leading coefficient in \eqref{eq3.12} has no effect on being uncorrelated and we have therefore ignored it in the above determinant.

There are totally two cases for the initial value $ V_{0} $ in \eqref{eq7.1}, i.e. when $ V_{0} $ is a constant value (preferably $ V_{0}=1 $) or it is not constant. As we have shown in section \ref{subsubsec2.3.2}, the first case leads to the approved result \eqref{eq2.14}, whereas the non-constant case contains some novel results. The next subsection describes one of them.

\subsection{Another uncorrelated sequence of hypergeometric polynomials of $ _{4}F_{3} $ type}\label{subsec7.1}

Let $ w(x)=1 $ be defined on $ [0,1] $ and consider the initial data $ \{ {V_k} = {x^{r + k}}\} _{k = 0}^n $ for $ r\ne 0 $ because $ V_{0}\ne 1 $. In this case, the components corresponding to \eqref{eq7.1} are computed as
\begin{multline}\label{eq7.2}
{\mathop{\rm cov}} ({x^{r + i}},{x^{r + j}}) = \frac{{(r + i)(r + j)}}{{(r + i + 1)(r + j + 1)(2r + i + j + 1)}}\\
{\rm{for}}\,\,\,i = 0,1,...,n - 1\,\,{\rm{and}}\,j = 0,1,...,n.
\end{multline}
Since
\[{\mathop{\rm var}} ({x^{r + j}}) = \frac{{{{(r + j)}^2}}}{{{{(r + j + 1)}^2}(2r + 1 + 2j)}} > 0,\]
for $ j=0 $ we must also add the condition $ 2r+1>0 $.

If \eqref{eq7.2} is directly substituted into \eqref{eq7.1}, with the aid of advanced mathematical software, we can again find the explicit form of the functions, say $ {X_n} = {f_{r + n}}(x;\,r) $, as
\begin{equation}\label{eq7.3}
{f_{r + n}}(x;\,r) = {x^r}{Q_n}(x;\,r) = {x^r}{}_4{F_3}\left( {\left. {\begin{array}{*{20}{c}}
{ - n,\,\,n + 2r + 1,\,r,\,\,r + 2}\\
{2r + 1,\,\,r + 1,\,\,r + 1}
\end{array}\,} \right|\,x} \right),
\end{equation}
satisfying the ordinary uncorrelatedness condition
\begin{multline*}
\int_{\,0}^1 {{f_{r + n}}(x;\,r){f_{r + m}}(x;\,r)\,dx}  - \int_{\,0}^1 {{f_{r + n}}(x;\,r)dx} \,\int_{\,0}^1 {{f_{r + m}}(x;\,r)dx} \\
 = \left( {\int_{\,0}^1 {f_{r + n}^2(x;\,r)\,dx}  - {{\left( {\int_{\,0}^1 {{f_{r + n}}(x;\,r)dx} } \right)}^2}} \right){\delta _{m,n}},
\end{multline*}
which is equivalent to
\begin{multline}\label{eq7.4}
\int_{\,0}^1 {{x^{2r}}{Q_n}(x;\,r){Q_m}(x;\,r)\,dx}  - \int_{\,0}^1 {{x^r}{Q_n}(x;\,r)\,dx} \,\int_{\,0}^1 {{x^r}{Q_m}(x;\,r)\,dx} \\
= \left( {\int_{\,0}^1 {{x^{2r}}Q_n^2(x;\,r)\,dx}  - {{\left( {\int_{\,0}^1 {{x^r}{Q_n}(x;\,r)\,dx} } \right)}^2}} \right){\delta _{m,n}}.
\end{multline}

\begin{remark}\label{rem7.1.1}
If in general $ {P_r}(X = x) = u(x)/\int_{\,a}^b {u(x)\,dx}  $, one can find a relationship between ordinary covariances and 1-covariances as follows
\begin{multline}\label{eq7.5}
\left( {\int_a^b {w(x)\,dx} } \right)\,{\left. {{\mathop{\rm cov}} \left( {z(x)f(x),\,z(x)g(x)} \right)\,} \right|_{\,u(x) = w(x)}}\\
 = \left( {\int_a^b {w(x)\,{z^2}(x)\,dx} } \right){\left. {{{{\mathop{\rm cov}} }_1}\left( {f(x),\,g(x);\frac{1}{{z(x)}}} \right)\,} \right|_{\,u(x) = w(x)\,z^2(x)}}.
\end{multline}
Hence, \eqref{eq7.4} shows that if $ z(x)=x^{r} $ and $ w(x)=1 $ are replaced in \eqref{eq7.5}, then
\[{\left. {{\mathop{\rm cov}} \left( {{x^r}{Q_n}(x;\,r),\,{x^r}{Q_m}(x;\,r)} \right)\,} \right|_{u(x) = 1}} = \frac{1}{{2r + 1}}{\left. {{{{\mathop{\rm cov}} }_1}\left( {{Q_n}(x;\,r),\,{Q_m}(x;\,r);{x^{ - r}}} \right)\,} \right|_{u(x) = {x^{2r}}}}.\]
In the sequel, in \eqref{eq7.4},
\begin{align}\label{eq7.6}
\int_{\,0}^1 {{x^r}{Q_n}(x;\,r)\,dx}  &= \sum\limits_{k = 0}^n {\frac{{{{( - n)}_k}{{(n + 2r + 1)}_k}{{(r)}_k}{{(r + 2)}_k}}}{{{{(2r + 1)}_k}{{(r + 1)}_k}{{(r + 1)}_k}k!\,\,(r + 1 + k)}}\,} \\
 &= \frac{1}{{r + 1}}{}_3{F_2}\left( {\left. {\begin{array}{*{20}{c}}
{ - n,\,\,n + 2r + 1,\,r}\\
{2r + 1,\,\,r + 1}
\end{array}\,} \right|\,\,1} \right) = \frac{1}{{r + 1}}\frac{{n!}}{{{{(2r + 1)}_n}}},\nonumber
\end{align}
and
\begin{multline*}
\int_{\,0}^1 {{x^{2r}}{Q_n}(x;\,r){Q_m}(x;\,r)\,dx} \\
 = \frac{1}{{2r + 1}}\sum\limits_{k = 0}^n {\frac{{{{( - n)}_k}{{(n + 2r + 1)}_k}{{(r)}_k}{{(r + 2)}_k}}}{{{{(2r + 2)}_k}{{(r + 1)}_k}{{(r + 1)}_k}k!\,}}{}_5{F_4}\left( {\left. {\begin{array}{*{20}{c}}
{ - m,\,\,m + 2r + 1,\,r,\,\,r + 2,\,2r + 1 + k}\\
{2r + 1,\,\,r + 1,\,\,r + 1,\,\,2r + 2 + k}
\end{array}\,} \right|\,\,1} \right)\,} ,
\end{multline*}
yielding
\begin{multline}\label{eq7.7}
\int_{\,0}^1 {{x^{2r}}{Q_n}(x;\,r){Q_m}(x;\,r)\,dx}  - \int_{\,0}^1 {{x^r}{Q_n}(x;\,r)\,dx} \,\int_{\,0}^1 {{x^r}{Q_m}(x;\,r)\,dx} \\
 = \frac{1}{{2r + 1}}\sum\limits_{k = 0}^n {\frac{{{{( - n)}_k}{{(n + 2r + 1)}_k}{{(r)}_k}{{(r + 2)}_k}}}{{{{(2r + 2)}_k}{{(r + 1)}_k}{{(r + 1)}_k}k!\,}}{}_5{F_4}\left( {\left. {\begin{array}{*{20}{c}}
{ - m,\,\,m + 2r + 1,\,r,\,\,r + 2,\,2r + 1 + k}\\
{2r + 1,\,\,r + 1,\,\,r + 1,\,\,2r + 2 + k}
\end{array}\,} \right|\,\,1} \right)\,} \\
 - \frac{1}{{{{(r + 1)}^2}}}\frac{{n!\,\,m!}}{{{{(2r + 1)}_n}{{(2r + 1)}_m}}}.
\end{multline}
To derive \eqref{eq7.6}, we have used a similar computational technique as follows. Since
\[A(n) = {}_3{F_2}\left( {\left. {\begin{array}{*{20}{c}}
{ - n,\,\,n + 2r + 1,\,r}\\
{2r + 1,\,\,r + 1}
\end{array}\,} \right|\,\,1} \right),\]
satisfies the first order equation
\[A(n + 1) = \frac{{n + 1}}{{n + 1 + 2r}}A(n),\]
so
\[{}_3{F_2}\left( {\left. {\begin{array}{*{20}{c}}
{ - n,\,\,n + 2r + 1,\,r}\\
{2r + 1,\,\,r + 1}
\end{array}\,} \right|\,\,1} \right) = \frac{{n!}}{{{{(2r + 1)}_n}}}.\]
Also, to compute the hypergeometric term in \eqref{eq7.7}, since
\[B(m) = {}_5{F_4}\left( {\left. {\begin{array}{*{20}{c}}
{ - m,\,\,m + 2r + 1,\,r,\,\,r + 2,\,2r + 1 + k}\\
{2r + 1,\,\,r + 1,\,\,r + 1,\,\,2r + 2 + k}
\end{array}\,} \right|\,\,1} \right)\,,\]
satisfies the recurrence relation
\begin{multline*}
 - (m + 2r + 1)(m + 2r + 2)(m + 2r + 3 + k)\,B(m + 2)\\
 + (m + 2)(2m + 2r + 3)(m + 2r + 1)\,B(m + 1) + (m + 1)(m + 2)(k - m)\,B(m) = 0,
\end{multline*}
having two independent solutions
\[{B_1}(m) = \frac{{m!}}{{{{(2r + 1)}_m}}}\,\,\,\,\,\,{\rm{and}}\,\,\,\,\,\,{B_2}(m) = \frac{{m!\,\,{{( - k)}_m}}}{{{{(2r + 1)}_m}{{(2r + k + 2)}_m}}},\]
so
\[B(m) = {c_1}\,{B_1}(m) + {c_2}\,{B_2}(m),\]
which finally results in
\begin{multline}\label{eq7.8}
{}_5{F_4}\left( {\left. {\begin{array}{*{20}{c}}
{ - m,\,\,m + 2r + 1,\,r,\,\,r + 2,\,2r + 1 + k}\\
{2r + 1,\,\,r + 1,\,\,r + 1,\,\,2r + 2 + k}
\end{array}\,} \right|\,\,1} \right)\\
 = \frac{1}{{(r + 1)(r + 1 + k)}}\frac{{m!}}{{{{(2r + 1)}_m}}}\left( {2r + 1 + k + r(r + k)\frac{{{{( - k)}_m}}}{{{{(2r + 2 + k)}_m}}}} \right).
\end{multline}
In order to compute
\[{\mathop{\rm var}} \left( {{f_{r + n}}(x,r)} \right) = \int_{\,0}^1 {{x^{2r}}Q_n^2(x;\,r)\,dx}  - {\left( {\int_{\,0}^1 {{x^r}{Q_n}(x;\,r)\,dx} } \right)^2},\]
it is enough to take $ m=n $ in \eqref{eq7.7} and then use \eqref{eq7.8} to arrive at
\begin{multline*}
{B^*} = \sum\limits_{k = 0}^n {\frac{{{{( - n)}_k}{{(n + 2r + 1)}_k}{{(r)}_k}{{(r + 2)}_k}}}{{{{(2r + 2)}_k}{{(r + 1)}_k}{{(r + 1)}_k}k!\,}}{}_5{F_4}\left( {\left. {\begin{array}{*{20}{c}}
{ - n,\,\,n + 2r + 1,\,r,\,\,r + 2,\,2r + 1 + k}\\
{2r + 1,\,\,r + 1,\,\,r + 1,\,\,2r + 2 + k}
\end{array}\,} \right|\,\,1} \right)\,} \\
\,\,\,\,\,\,\, = \frac{{n!}}{{{{(2r + 1)}_n}}}\frac{{2r + 1}}{{{{(r + 1)}^2}}}\sum\limits_{k = 0}^n {\frac{{{{( - n)}_k}{{(n + 2r + 1)}_k}{{(r)}_k}}}{{{{(2)}_k}{{( - r + 1)}_k}{{(r + 1)}_k}k!\,}}} \\
\,\,\,\,\,\,\, + \frac{{n!}}{{{{(2r + 1)}_n}}}\frac{{{r^2}}}{{{{(r + 1)}^2}{{(2r + 2)}_n}}}\sum\limits_{k = 0}^n {\frac{{{{( - n)}_k}{{(n + 2r + 1)}_k}{{( - k)}_n}}}{{{{(n + 2r + 2)}_k}k!\,}}} ,
\end{multline*}
which is finally simplified as
\[{B^*} = \frac{{2r + 1}}{{{{(r + 1)}^2}}}\frac{{{{(n!)}^2}}}{{{{(2r + 1)}_n}{{(2r + 1)}_n}}} + \frac{{{r^2}}}{{{{(r + 1)}^2}\,}}\frac{{{{(n!)}^2}{{(n + 2r + 1)}_n}}}{{{{(2r + 1)}_n}{{(2r + 2)}_n}{{(n + 2r + 2)}_n}}}.\]
\end{remark}
\begin{corollary}\label{coro7.1.2}
If $ r>-1/2 $ and $ r\ne 0 $, then
\[\int_{\,0}^1 {{x^{2r}}Q_n^2(x;\,r)\,dx}  - {\left( {\int_{\,0}^1 {{x^r}{Q_n}(x;\,r)\,dx} } \right)^2} = \frac{1}{{2n + 2r + 1}}{\left( {\frac{{r\,n!}}{{(r + 1){{(2r + 1)}_n}}}} \right)^2}.\]
Moreover, \eqref{eq7.7} is simplified as
\begin{multline*}
\int_{\,0}^1 {{x^{2r}}{Q_n}(x;\,r){Q_m}(x;\,r)\,dx}  - \int_{\,0}^1 {{x^r}{Q_n}(x;\,r)\,dx} \,\int_{\,0}^1 {{x^r}{Q_m}(x;\,r)\,dx} \\
 = \frac{{{r^2}m!}}{{{{(r + 1)}^2}(2r + 1){{(2r + 1)}_m}{{(2r + 2)}_m}\,}}\sum\limits_{k = 0}^n {\frac{{{{(n + 2r + 1)}_k}{{( - n)}_k}}}{{{{(m + 2r + 2)}_k}}}\frac{{{{( - k)}_m}}}{{k!}}}  = 0 \Leftrightarrow m \ne n.
\end{multline*}
Therefore
\begin{equation}\label{eq7.9}
\sum\limits_{k = m}^n {\frac{{{{(n + 2r + 1)}_k}{{( - n)}_k}}}{{{{(m + 2r + 2)}_k}}}\frac{{{{( - k)}_m}}}{{k!}}}  = n!\,\,\frac{{n + 2r + 1}}{{2n + 2r + 1}}{\delta _{n,m}},
\end{equation}
which is a generalization of the second result of corollary \ref{coro6.6.1} for $ r=0 $.
\end{corollary}
In section \ref{subsec9.1}, we will show that both polynomials obtained in sections \ref{sec6} and \ref{sec7} are particular cases of a two-parametric sequence of uncorrelated polynomials. 

\section{A class of uncorrelated polynomials based on a predetermined orthogonal polynomial}\label{sec8}

In this section, we introduce a class of uncorrelated polynomials which is constructed by a predetermined sequence of orthogonal polynomials and then study its general properties in detail. In order to clarify the subject, we also present two hypergeometric examples based on Jacobi and Laguerre polynomials. First of all, we should raise two important points.

Let $ {\{ {\Phi _n}(x)\} _{n = 0}} $ (with $ {\Phi _0}(x) = 1 $ valid especially for polynomial sequences) be a sequence of real functions orthogonal with respect to $ w(x) $ on $ [a,b] $, i.e.
\begin{equation}\label{eq8.1}
{\left. {E\left( {{\Phi _m}(x){\Phi _n}(x)} \right)\,} \right|_{w(x)}} = {\left. {E\left( {\Phi _n^2(x)} \right)\,} \right|_{w(x)}}\,{\delta _{m,n}},
\end{equation}
where $ {\left. {(.)\,} \right|_{w(x)}} $ means $ {P_r}(X = x) = \frac{{w(x)}}{{\int_{\,a}^b {w(x)\,dx} }} $ as before.

The first point is that if the mentioned sequence has the special form 
\[ {\Phi _n}(x) = \theta (x)\,{g_n}(x) + {\beta _n} \quad  (\text{with} \,\,{\Phi _0}(x) = 1),\] 
in which $  \theta (x) $ is a function independent of $ n $ and  $ \{ {\beta _n}\}  $ is an arbitrary numeric sequence, then noting \eqref{eq8.1} and \eqref{eq7.5} in remark \ref{rem7.1.1} we have
\begin{multline}\label{eq8.2}
{\left. {{\mathop{\rm cov}} \left( {{\Phi _m}(x),{\Phi _n}(x)} \right)\,} \right|_{\,w(x)}} = {\left. {E\left( {{\Phi _m}(x){\Phi _n}(x)} \right)\,} \right|_{\,w(x)}} - {\left. {E\left( {{\Phi _m}(x){\Phi _0}(x)} \right)\,} \right|_{\,w(x)}}{\left. {E\left( {{\Phi _n}(x){\Phi _0}(x)} \right)\,} \right|_{\,w(x)}}\\
 = {\left. {E\left( {\Phi _n^2(x)} \right)\,} \right|_{\,w(x)}}{\delta _{m,n}} = {\left. {{\mathop{\rm cov}} \,\left( {\theta (x)\,{g_m}(x) + {\beta _m},\,\,\theta (x)\,{g_n}(x) + {\beta _n}} \right)\,} \right|_{\,w(x)}}\\
 = {\left. {{\mathop{\rm cov}} \,\left( {\theta (x)\,{g_m}(x),\theta (x)\,{g_n}(x)} \right)\,} \right|_{\,w(x)}} = \frac{{\int_a^b {w(x)\,{\theta ^2}(x)\,dx} }}{{\int_a^b {w(x)\,dx} }}{\left. {{{{\mathop{\rm cov}} }_1}\left( {{g_m}(x),\,{g_n}(x);\frac{1}{{\theta (x)}}} \right)\,} \right|_{\,w(x)\,{\theta ^2}(x)}},
\end{multline}
leading to the following biorthogonality relation according to the sub-section \ref{subsec4.2},
\[{\left. {E\left( {{g_m}(x)\left( {{g_n}(x) - \frac{{E\left( {{g_n}(x)/\theta (x)} \right)}}{{E\left( {1/{\theta ^2}(x)} \right)}}\frac{1}{{\theta (x)}}} \right)} \right)\,} \right|_{\,w(x)\,{\theta ^2}(x)}} = \frac{{\int_a^b {w(x)\,dx} }}{{\int_a^b {w(x)\,{\theta ^2}(x)\,dx} }}{\left. {E\left( {\Phi _n^2(x)} \right)\,} \right|_{\,w(x)}}{\delta _{m,n}}.\]
Also, if it satisfies a second order equation as
\[a(x)\,{\Phi ''_n}(x) + b(x)\,{\Phi '_n}(x) + {\lambda _n}u(x)\,{\Phi _n}(x) = 0,\]
then $ {\{ {g_n}(x)\} _{n = 0}} $ will satisfy the equation
\begin{multline}\label{eq8.3}
\left( {a(x)\theta (x)} \right)\,{g''_n}(x) + \left( {2a(x)\theta '(x) + b(x)\theta (x)} \right)\,{g'_n}(x)\\
+ \left( {a(x)\theta ''(x) + b(x)\theta '(x) + {\lambda _n}u(x)\theta (x)} \right)\,{g_n}(x) =  - {\lambda _n}{\beta _n}u(x).
\end{multline}
In this sense, the second point is that the relation
\[\frac{{\,\int_{\,a}^b {w(x)\,\theta (x)\,{g_n}(x)\,dx} }}{{\,\int_{\,a}^b {w(x)\,dx} }} = \frac{{\,\int_{\,a}^b {w(x)\,\left( {{\Phi _n}(x) - {\beta _n}} \right)\,dx} }}{{\,\int_{\,a}^b {w(x)\,dx} }} =  - {\beta _n},\]
will change equation \eqref{eq8.3} to
\begin{multline*}
\left( {a(x)\theta (x)} \right)\,{g''_n}(x) + \left( {2a(x)\theta '(x) + b(x)\theta (x)} \right)\,{g'_n}(x)\\
 + \left( {a(x)\theta ''(x) + b(x)\theta '(x) + {\lambda _n}u(x)\theta (x)} \right)\,{g_n}(x) = {\lambda _n}u(x)\frac{{\,\int_{\,a}^b {w(x)\,\theta (x)\,{g_n}(x)\,dx} }}{{\,\int_{\,a}^b {w(x)\,dx} }},
\end{multline*}
which is a particular case of equation \eqref{eq5.3} in theorem \ref{thm5.2}.

Now, noting the two above points, for a real parameter $ \lambda $ let $ {\left\{ {{P_n}(x;\lambda ) = \sum\limits_{k = 0}^n {a_k^{(n)}{{(x - \lambda )}^k}} } \right\}_{n = 0}} $ be a sequence of polynomials orthogonal with respect to $ w(x) $ on $ [a,b] $ as
\begin{equation}\label{eq8.4}
{\left. {E\left( {{P_m}(x;\lambda ){P_n}(x;\lambda )} \right)\,} \right|_{w(x)}} = {\left. {E\left( {P_n^2(x;\lambda )} \right)\,} \right|_{w(x)}}\,{\delta _{m,n}}.
\end{equation}
It can be verified that the sequence
\begin{equation}\label{eq8.5}
{Q_n}(x;\lambda ) = \frac{{{P_{n + 1}}(x;\lambda ) - {P_{n + 1}}(\lambda ;\lambda )}}{{x - \lambda }} = \sum\limits_{k = 0}^n {a_{k + 1}^{(n + 1)}{{(x - \lambda )}^k}} ,
\end{equation}
is also a polynomial of degree $ n $.

With reference to \eqref{eq8.2} and \eqref{eq8.4}, the following equalities hold for the polynomial sequence \eqref{eq8.5},
\begin{align}\label{eq8.6}
&\frac{{\int_a^b {w(x)\,{{(x - \lambda )}^2}\,dx} }}{{\int_a^b {w(x)\,dx} }}{\left. {{{{\mathop{\rm cov}} }_1}\left( {{Q_m}(x;\lambda ),\,{Q_n}(x;\lambda );\frac{1}{{x - \lambda }}} \right)\,} \right|_{\,w(x)\,{{(x - \lambda )}^2}}}\\
&\quad = {\left. {{\mathop{\rm cov}} \left( {(x - \lambda ){Q_m}(x;\lambda ),(x - \lambda ){Q_n}(x;\lambda )} \right)\,} \right|_{\,w(x)}}\nonumber\\
&\quad  = {\left. {{\mathop{\rm cov}} \left( {{P_{m + 1}}(x;\lambda ) - {P_{m + 1}}(\lambda ;\lambda ),{P_{n + 1}}(x;\lambda ) - {P_{n + 1}}(\lambda ;\lambda )} \right)\,} \right|_{\,w(x)}}\nonumber\\
&\quad  = {\left. {{\mathop{\rm cov}} \left( {{P_{m + 1}}(x;\lambda ),{P_{n + 1}}(x;\lambda )} \right)\,} \right|_{\,w(x)}} = {\left. {E\left( {P_{n + 1}^2(x;\lambda )} \right)\,} \right|_{\,w(x)}}{\delta _{m,n}}.\nonumber
\end{align}
\begin{corollary}\label{coro8.1}
From \eqref{eq8.6}, the relation
\[{\left. {{{{\mathop{\rm cov}} }_1}\left( {{Q_m}(x;\lambda ),\,{Q_n}(x;\lambda );\frac{1}{{x - \lambda }}} \right)\,} \right|_{\,w(x)\,{{(x - \lambda )}^2}}} = \frac{{\int_a^b {w(x)\,P_{n + 1}^2(x;\lambda )\,dx} }}{{\int_a^b {w(x)\,{{(x - \lambda )}^2}\,dx} }}\,\,{\delta _{m,n}},\]
shows that the polynomial set $ {\left\{ {{Q_n}(x;\lambda )  } \right\}_{n = 0}} $ is a complete uncorrelated sequence with respect to the fixed function $ z(x) = \frac{1}{{x - \lambda }} $ and the probability function $ {P_r}(X = x) = \frac{{w(x)\,{{(x - \lambda )}^2}}}{{\int_{\,a}^b {w(x)\,{{(x - \lambda )}^2}dx} }} $ respectively. Also, relation \eqref{eq8.6} shows that the two defined sequences $ {\left\{ {{P_n}(x;\lambda ) = \sum\limits_{k = 0}^n {a_k^{(n)}{{(x - \lambda )}^k}} } \right\}_{n = 0}} $ and $ {\left\{ {{Q_n}(x;\lambda ) = \sum\limits_{k = 0}^n {a_{k + 1}^{(n + 1)}{{(x - \lambda )}^k}} } \right\}_{n = 0}} $ are biorthogonal with respect to the weight function $ (x - \lambda )\,w(x) $ on $ [a,b] $, as we have
\begin{multline*}
{\left. {{{{\mathop{\rm cov}} }_1}\left( {{Q_m}(x;\lambda ),\,{Q_n}(x;\lambda );\frac{1}{{x - \lambda }}} \right)\,} \right|_{\,w(x)\,{{(x - \lambda )}^2}}}\\
 = E{\left. {\left( {{Q_m}(x;\lambda )\left( {{Q_n}(x;\lambda ) - \frac{{E\left( {{Q_n}(x;\lambda )/(x - \lambda )} \right)}}{{E\left( {1/{{(x - \lambda )}^2}} \right)}}\frac{1}{{x - \lambda }}} \right)} \right)\,} \right|_{\,w(x)\,{{(x - \lambda )}^2}}},
\end{multline*}
where
\[{Q_n}(x;\lambda ) - \frac{{E\left( {{Q_n}(x;\lambda )/(x - \lambda )} \right)}}{{E\left( {1/{{(x - \lambda )}^2}} \right)}}\frac{1}{{x - \lambda }} = \frac{{{P_{n + 1}}(x;\lambda )}}{{x - \lambda }}.\]
\end{corollary}
There is a direct proof for this conclusion, too. If we suppose
\[{\left\langle {{P_m}(x;\lambda ),{P_n}(x;\lambda )} \right\rangle _{w(x)}} = {\left\langle {{P_n}(x;\lambda ),{P_n}(x;\lambda )} \right\rangle _{w(x)}}{\delta _{m,n}},\]
then for every $ m\in\mathbb{N}  $ we have
\begin{align}\label{eq8.7}
&{\left\langle {{Q_n}(x;\lambda ),{P_m}(x;\lambda )} \right\rangle _{(x - \lambda )\,w(x)}} = {\left\langle {\frac{{{P_{n + 1}}(x;\lambda ) - {P_{n + 1}}(\lambda ;\lambda )}}{{x - \lambda }},{P_m}(x;\lambda )} \right\rangle _{(x - \lambda )\,w(x)}}\\[3mm]
&\qquad\qquad = {\left\langle {{P_{n + 1}}(x;\lambda ),{P_m}(x;\lambda )} \right\rangle _{w(x)}} - {P_{n + 1}}(\lambda ;\lambda ){\left\langle {1,{P_m}(x;\lambda )} \right\rangle _{w(x)}}\notag\\[3mm]
&\qquad\qquad  = {\left\langle {{P_{n + 1}}(x;\lambda ),{P_{n + 1}}(x;\lambda )} \right\rangle _{w(x)}}{\delta _{n + 1,m}}.\notag
\end{align}
Using \eqref{eq8.7}, one can create a biorthogonal approximation (or expansion) for any appropriate function, say $ f(x) $, in terms of the uncorrelated polynomials $ {\left\{ {{Q_k}(x;\lambda )} \right\}_{k = 0}} $ as follows
\[f(x) \cong \sum\limits_{k = 0}^{n \to \infty } {{c_k}{Q_k}(x;\lambda )}  = \sum\limits_{k = 0}^{n \to \infty } {\frac{{{{\left\langle {f(x),{P_{k + 1}}(x;\lambda )} \right\rangle }_{(x - \lambda )\,w(x)}}}}{{{{\left\langle {{P_{k + 1}}(x;\lambda ),{P_{k + 1}}(x;\lambda )} \right\rangle }_{w(x)}}}}\frac{{{P_{k + 1}}(x;\lambda ) - {P_{k + 1}}(\lambda ;\lambda )}}{{x - \lambda }}} ,\]
whose error is clearly minimized with respect to the fixed function $ z(x) = \frac{1}{{x - \lambda }} $ in the sense of least 1-variances.

In the sequel, since $ {\left\{ {{P_n}(x;\lambda )} \right\}_{n = 0}} $ was assumed to be orthogonal, its monic type must satisfy a three term recurrence relation \cite{ref4} of the form
\begin{equation}\label{eq8.8}
{\bar P_{n + 1}}(x;\lambda ) = (x - {B_n}){\bar P_n}(x;\lambda ) - {C_n}{\bar P_{n - 1}}(x;\lambda )\,\,\,\,\,\,\text{with}\,\,\,\,\,{\bar P_0}(x;\lambda ) = 1\,\,\,\,{\rm{and}}\,\,\,{\bar P_1}(x;\lambda ) = x - {B_1}.
\end{equation}
After doing some computations in hand, substituting \eqref{eq8.5} into \eqref{eq8.8} gives
\begin{equation}\label{eq8.9}
{\bar Q_{n + 1}}(x;\lambda ) = (x - {B_{n + 1}}){\bar Q_n}(x;\lambda ) - {C_{n + 1}}{\bar Q_{n - 1}}(x;\lambda ) + {\bar P_{n + 1}}(\lambda ;\lambda )\,\,\,\,\,\,\text{with}\,\,\,\,\,{\bar Q_0}(x;\lambda ) = 1.
\end{equation}
This type of recurrence relation in \eqref{eq8.9} helps us obtain an analogue of the well-known Christoffel-Darboux identity \cite{ref4} as follows. We have respectively in \eqref{eq8.9},
\begin{multline*}
\left( {x{{\bar Q}_n}(x;\lambda ) + {{\bar P}_{n + 1}}(\lambda ;\lambda )} \right){{\bar Q}_n}(t;\lambda ) = {{\bar Q}_{n + 1}}(x;\lambda ){{\bar Q}_n}(t;\lambda )\\
 + {B_{n + 1}}{{\bar Q}_n}(x;\lambda ){{\bar Q}_n}(t;\lambda ) + {C_{n + 1}}{{\bar Q}_{n - 1}}(x;\lambda ){{\bar Q}_n}(t;\lambda ),
\end{multline*}
and
\begin{multline*}
\left( {t\,{{\bar Q}_n}(t;\lambda ) + {{\bar P}_{n + 1}}(\lambda ;\lambda )} \right){{\bar Q}_n}(x;\lambda ) = {{\bar Q}_{n + 1}}(t;\lambda ){{\bar Q}_n}(x;\lambda )\\
 + {B_{n + 1}}{{\bar Q}_n}(t;\lambda ){{\bar Q}_n}(x;\lambda ) + {C_{n + 1}}{{\bar Q}_{n - 1}}(t;\lambda ){{\bar Q}_n}(x;\lambda ).
\end{multline*}
Therefore, by defining the kernel
\[{G_n}(x,t) = \frac{1}{{\prod\limits_{j = 1}^{n + 1} {{C_j}} }}\frac{{{{\bar Q}_{n + 1}}(x;\lambda ){{\bar Q}_n}(t;\lambda ) - {{\bar Q}_{n + 1}}(t;\lambda ){{\bar Q}_n}(x;\lambda )}}{{x - t}},\]
we eventually obtain
\begin{multline}\label{eq8.10}
\sum\limits_{n = 0}^m {\frac{1}{{\prod\limits_{j = 1}^{n + 1} {{C_j}} }}\left( {{{\bar Q}_n}(x;\lambda ){{\bar Q}_n}(t;\lambda ) - {{\bar P}_{n + 1}}(\lambda ;\lambda )\frac{{{{\bar Q}_n}(x;\lambda ) - {{\bar Q}_n}(t;\lambda )}}{{x - t}}} \right)} \\
 = \sum\limits_{n = 0}^m {{G_n}(x,t) - {G_{n - 1}}(x,t)}  = \frac{1}{{\prod\limits_{j = 1}^{m + 1} {{C_j}} }}\frac{{{{\bar Q}_{m + 1}}(x;\lambda ){{\bar Q}_m}(t;\lambda ) - {{\bar Q}_{m + 1}}(t;\lambda ){{\bar Q}_m}(x;\lambda )}}{{x - t}}.
\end{multline}
Let us introduce two uncorrelated polynomials of hypergeometric type here which are built based on Jacobi and Laguerre polynomials and then apply all above-mentioned results on them.

\subsection{An uncorrelated sequence of hypergeometric polynomials of $ _{3}F_{2} $ type}\label{subsec8.2}
As is known, the monic Jacobi polynomials \cite{ref21}
\begin{equation}\label{eq8.11}
\bar P_n^{(\alpha ,\beta )}(x) = \frac{{{2^n}{{(\alpha  + 1)}_n}}}{{{{(n + \alpha  + \beta  + 1)}_n}}}{}_2{F_1}\left( {\left. {\begin{array}{*{20}{c}}
{ - n,\,\,n + \alpha  + \beta  + 1}\\
{\alpha  + 1}
\end{array}\,} \right|\,\frac{{1 - x}}{2}} \right),
\end{equation}
satisfy the equation
\begin{equation}\label{eq8.12}
(1 - {x^2})\frac{{{d^2}}}{{d{x^2}}}\bar P_n^{(\alpha ,\beta )}(x) - \left( {(\alpha  + \beta  + 2)x + \alpha  - \beta } \right)\frac{d}{{dx}}\bar P_n^{(\alpha ,\beta )}(x) + n(n + \alpha  + \beta  + 1)\bar P_n^{(\alpha ,\beta )}(x) = 0\,,
\end{equation}
and are orthogonal with respect to the weight function $ {(1 - x)^\alpha }{(1 + x)^\beta } $ on $ [-1,1] $ as
\begin{multline}\label{eq8.13}
\int_{ - 1}^1 {{{(1 - x)}^\alpha }{{(1 + x)}^\beta }\bar P_m^{(\alpha ,\beta )}(x)\bar P_n^{(\alpha ,\beta )}(x)\,dx} \\
 = n!\,{2^{2n + \alpha  + \beta  + 1}}\frac{{\Gamma (n + \alpha  + \beta  + 1)\Gamma (n + \alpha  + 1)\Gamma (n + \beta  + 1)}}{{\Gamma (2n + \alpha  + \beta  + 1)\Gamma (2n + \alpha  + \beta  + 2)}}{\delta _{n,m}}\,.
\end{multline}
Since $ \bar P_n^{(\alpha ,\beta )}( - x) = {( - 1)^n}\bar P_n^{(\beta ,\alpha )}(x) $, another representation is as
\begin{equation}\label{eq8.14}
\bar P_n^{(\alpha ,\beta )}(x) = \frac{{{{( - 1)}^n}{2^n}{{(\beta  + 1)}_n}}}{{{{(n + \alpha  + \beta  + 1)}_n}}}{}_2{F_1}\left( {\left. {\begin{array}{*{20}{c}}
{ - n,\,\,n + \alpha  + \beta  + 1}\\
{\beta  + 1}
\end{array}\,} \right|\,\frac{{1 + x}}{2}} \right).
\end{equation}
Also, they satisfy a three term recurrence relation in the form
\begin{multline}\label{eq8.15}
\bar P_{n + 1}^{(\alpha ,\beta )}(x) = \left( {x - \frac{{{\beta ^2} - {\alpha ^2}}}{{(2n + \alpha  + \beta )(2n + \alpha  + \beta  + 2)}}} \right)\,\bar P_n^{(\alpha ,\beta )}(x)\\
\quad - 4\frac{{n(n + \alpha )(n + \beta )(n + \alpha  + \beta )}}{{(2n + \alpha  + \beta  + 1){{(2n + \alpha  + \beta )}^2}(2n + \alpha  + \beta  - 1)}}\bar P_{n - 1}^{(\alpha ,\beta )}(x)\,.
\end{multline}
Representations \eqref{eq8.11} and \eqref{eq8.14} show that there are two specific values for $ \lambda $ in \eqref{eq8.5}, i.e. $ \lambda=1 $ and $ \lambda=-1 $. Noting that
\[\bar P_n^{(\alpha ,\beta )}(1) = \frac{{{2^n}{{(\alpha  + 1)}_n}}}{{{{(n + \alpha  + \beta  + 1)}_n}}},\]
the first kind of uncorrelated polynomials is defined as
\begin{align}\label{eq8.16}
\bar Q_n^{(\alpha ,\beta )}(x;1) &= \frac{{\bar P_{n + 1}^{(\alpha ,\beta )}(1) - \bar P_{n + 1}^{(\alpha ,\beta )}(x)}}{{1 - x}}\\
& = \frac{{(n + 1)\,{2^n}{{(\alpha  + 2)}_n}}}{{{{(n + \alpha  + \beta  + 3)}_n}}}{}_3{F_2}\left( {\left. {\begin{array}{*{20}{c}}
{ - n,\,\,n + \alpha  + \beta  + 3,\,\,1}\\
{\alpha  + 2,\,\,2}
\end{array}\,} \right|\,\frac{{1 - x}}{2}} \right).\notag
\end{align}
Also, for $ \lambda=-1 $ the second kind  is defined by
\begin{multline}\label{eq8.17}
\bar Q_n^{(\alpha ,\beta )}(x; - 1) = \frac{{\bar P_{n + 1}^{(\alpha ,\beta )}(x) - \bar P_{n + 1}^{(\alpha ,\beta )}( - 1)}}{{x + 1}} = \frac{{{{( - 1)}^{n + 1}}\bar P_{n + 1}^{(\beta ,\alpha )}( - x) - {{( - 1)}^{n + 1}}\bar P_{n + 1}^{(\beta ,\alpha )}(1)}}{{1 - ( - x)}}\\
 = {( - 1)^n}\bar Q_n^{(\beta ,\alpha )}( - x;1) = \frac{{(n + 1){{( - 2)}^n}{{(\beta  + 2)}_n}}}{{{{(n + \alpha  + \beta  + 3)}_n}}}{}_3{F_2}\left( {\left. {\begin{array}{*{20}{c}}
{ - n,\,\,n + \alpha  + \beta  + 3,\,\,1}\\
{\beta  + 2,\,\,2}
\end{array}\,} \right|\,\frac{{1 + x}}{2}} \right).
\end{multline}
Relation \eqref{eq8.17} shows that we shall deal with only one value, i.e. $ \lambda=1 $.

If in \eqref{eq8.16}, $ \bar P_{n + 1}^{(\alpha ,\beta )}(x) = (x - 1)\,\bar Q_n^{(\alpha ,\beta )}(x;1) + \bar P_{n + 1}^{(\alpha ,\beta )}(1) $ is substituted into the differential  equation \eqref{eq8.12}, we obtain
\begin{multline}\label{eq8.18}
{(1 - x)^2}(1 + x)\frac{{{d^2}}}{{d{x^2}}}\bar Q_n^{(\alpha ,\beta )}(x;1) - (1 - x)\left( {(\alpha  + \beta  + 4)x + \alpha  - \beta  + 2} \right)\frac{d}{{dx}}\,\bar Q_n^{(\alpha ,\beta )}(x;1)\\
 + \left( {n(n + \alpha  + \beta  + 3)(1 - x) + 2\alpha  + 2} \right)\,\bar Q_n^{(\alpha ,\beta )}(x;1)\\
 = (n + 1)(n + \alpha  + \beta  + 2)\bar P_{n + 1}^{(\alpha ,\beta )}(1) = \frac{{(n + 1)(n + \alpha  + \beta  + 2){2^n}{{(\alpha  + 1)}_n}}}{{{{(n + \alpha  + \beta  + 1)}_n}}}.
\end{multline}
On the other hand,
\begin{multline*}
\int_{ - 1}^1 {{{(1 - x)}^{\alpha  + 1}}{{(1 + x)}^\beta }\bar Q_n^{(\alpha ,\beta )}(x;1)\,dx}  = \int_{ - 1}^1 {{{(1 - x)}^\alpha }{{(1 + x)}^\beta }\left( {\bar P_{n + 1}^{(\alpha ,\beta )}(1) - \bar P_{n + 1}^{(\alpha ,\beta )}(x)} \right)\,dx} \\
\,\,\,\,\,\,\,\,\,\,\,\,\,\,\,\,\,\,\,\,\,\,\,\,\,\,\, = \bar P_{n + 1}^{(\alpha ,\beta )}(1)\,\int_{ - 1}^1 {{{(1 - x)}^\alpha }{{(1 + x)}^\beta }dx}  = \bar P_{n + 1}^{(\alpha ,\beta )}(1)\,{2^{\alpha  + \beta  + 1}}\frac{{\Gamma (\alpha  + 1)\Gamma (\beta  + 1)}}{{\Gamma (\alpha  + \beta  + 2)}},
\end{multline*}
changes equation \eqref{eq8.18} to
\begin{multline}\label{eq8.19}
{(1 - x)^2}(1 + x)\frac{{{d^2}}}{{d{x^2}}}\bar Q_n^{(\alpha ,\beta )}(x;1) - (1 - x)\left( {(\alpha  + \beta  + 4)x + \alpha  - \beta  + 2} \right)\frac{d}{{dx}}\,\bar Q_n^{(\alpha ,\beta )}(x;1)\\
 + \left( {\gamma _n^*(1 - x) + 2\alpha  + 2} \right)\,\bar Q_n^{(\alpha ,\beta )}(x;1)\\
 = \frac{{\Gamma (\alpha  + \beta  + 2)(\gamma _n^* + \alpha  + \beta  + 2)}}{{\,{2^{\alpha  + \beta  + 1}}\Gamma (\alpha  + 1)\Gamma (\beta  + 1)}}\int_{ - 1}^1 {{{(1 - x)}^{\alpha  + 1}}{{(1 + x)}^\beta }\bar Q_n^{(\alpha ,\beta )}(x;1)\,dx} \,,
\end{multline}
in which $ \gamma _n^* = n(n + \alpha  + \beta  + 3) $.

\begin{theorem}\label{thm8.2.1}
For every $ \alpha ,\,\beta  >  - 1 $, we have
\begin{multline*}
{\left. {{{{\mathop{\rm cov}} }_1}\left( {\bar Q_m^{(\alpha ,\beta )}(x;1),\,\bar Q_n^{(\alpha ,\beta )}(x;1);\frac{1}{{1 - x}}} \right)\,} \right|_{{{(1 - x)}^{\alpha  + 2}}{{(1 + x)}^\beta }}}\\
 = n!\,{2^{2n - 2}}\frac{{\Gamma (\alpha  + \beta  + 4)\Gamma (n + \alpha  + \beta  + 1)\Gamma (n + \alpha  + 1)\Gamma (n + \beta  + 1)}}{{\Gamma (\alpha  + 3)\Gamma (\beta  + 1)\Gamma (2n + \alpha  + \beta  + 1)\Gamma (2n + \alpha  + \beta  + 2)}}\,{\delta _{m,n}}.
\end{multline*}
\end{theorem}
\begin{proof}
We would like to prove this theorem via differential equation \eqref{eq8.19} so that if it is written in the self adjoint form
\begin{multline*}
\frac{d}{{dx}}\left( {{{(1 - x)}^{\alpha  + 3}}{{(1 + x)}^{\beta  + 1}}\frac{d}{{dx}}\bar Q_n^{(\alpha ,\beta )}(x;1)} \right)\\
 + \left( {\gamma _n^*{{(1 - x)}^{\alpha  + 2}}{{(1 + x)}^\beta } + \left( {2\alpha  + 2} \right){{(1 - x)}^{\alpha  + 1}}{{(1 + x)}^\beta }} \right)\,\bar Q_n^{(\alpha ,\beta )}(x;1)\\
 = \frac{{\Gamma (\alpha  + \beta  + 2)(\gamma _n^* + \alpha  + \beta  + 2)}}{{\,{2^{\alpha  + \beta  + 1}}\Gamma (\alpha  + 1)\Gamma (\beta  + 1)}}\left( {\int_{ - 1}^1 {{{(1 - x)}^{\alpha  + 1}}{{(1 + x)}^\beta }\bar Q_n^{(\alpha ,\beta )}(x;1)\,dx} } \right)\,{(1 - x)^{\alpha  + 1}}{(1 + x)^\beta },
\end{multline*}
then
\begin{multline*}
\left[ {{{(1 - x)}^{\alpha  + 3}}{{(1 + x)}^{\beta  + 1}}\left( {\bar Q_m^{(\alpha ,\beta )}(x;1)\frac{d}{{dx}}\bar Q_n^{(\alpha ,\beta )}(x;1) - \bar Q_n^{(\alpha ,\beta )}(x;1)\frac{d}{{dx}}\bar Q_m^{(\alpha ,\beta )}(x;1)} \right)} \right]_{\, - 1}^{\,1}\\
 + \left( {\gamma _n^* - \gamma _m^*} \right)\int_{\, - 1}^1 {{{(1 - x)}^{\alpha  + 2}}{{(1 + x)}^\beta }\,\bar Q_n^{(\alpha ,\beta )}(x;1)\,\bar Q_m^{(\alpha ,\beta )}(x;1)\,dx} \\
  = \frac{{\Gamma (\alpha  + \beta  + 2)}}{{\,{2^{\alpha  + \beta  + 1}}\Gamma (\alpha  + 1)\Gamma (\beta  + 1)}}\left( {\gamma _n^* - \gamma _m^*} \right)\\
\,\,\,\,\,\,\,\,\,\,\,\,\, \times \left( {\int_{ - 1}^1 {{{(1 - x)}^{\alpha  + 1}}{{(1 + x)}^\beta }\bar Q_n^{(\alpha ,\beta )}(x;1)\,dx} } \right)\left( {\int_{ - 1}^1 {{{(1 - x)}^{\alpha  + 1}}{{(1 + x)}^\beta }\bar Q_m^{(\alpha ,\beta )}(x;1)\,dx} } \right),
\end{multline*}
leading to the result
\begin{multline}\label{eq8.20}
\int_{\, - 1}^1 {{{(1 - x)}^{\alpha  + 2}}{{(1 + x)}^\beta }\,\bar Q_n^{(\alpha ,\beta )}(x;1)\,\bar Q_m^{(\alpha ,\beta )}(x;1)\,dx} \\
 = \frac{{\left( {\int_{ - 1}^1 {{{(1 - x)}^{\alpha  + 1}}{{(1 + x)}^\beta }\bar Q_n^{(\alpha ,\beta )}(x;1)\,dx} } \right)\left( {\int_{ - 1}^1 {{{(1 - x)}^{\alpha  + 1}}{{(1 + x)}^\beta }\bar Q_m^{(\alpha ,\beta )}(x;1)\,dx} } \right)}}{{\int_{ - 1}^1 {{{(1 - x)}^\alpha }{{(1 + x)}^\beta }\,dx} }}\\
 = \bar P_{n + 1}^{(\alpha ,\beta )}(1)\,\bar P_{m + 1}^{(\alpha ,\beta )}(1)\int_{ - 1}^1 {{{(1 - x)}^\alpha }{{(1 + x)}^\beta }dx} \\
 = {2^{n + m + \alpha  + \beta  + 1}}\frac{{\Gamma (\alpha  + 1)\Gamma (\beta  + 1)}}{{\Gamma (\alpha  + \beta  + 2)}}\frac{{{{(\alpha  + 1)}_n}{{(\alpha  + 1)}_m}}}{{{{(n + \alpha  + \beta  + 1)}_n}{{(m + \alpha  + \beta  + 1)}_m}}} \Leftrightarrow n \ne m,
\end{multline}
which proves the first part. To obtain the variance value, i.e. for $ n=m $ in \eqref{eq8.20}, it is enough to refer to corollary \ref{coro8.1} and then apply relation \eqref{eq8.13}.
\end{proof}
Since in (8.15),
\[{C_j} = \frac{1}{4}\frac{{j(j + \alpha )(j + \beta )(j + \alpha  + \beta )}}{{(j + \frac{{\alpha  + \beta  + 1}}{2}){{(j + \frac{{\alpha  + \beta }}{2})}^2}(j + \frac{{\alpha  + \beta  - 1}}{2})}},\]
we have
\[\prod\limits_{j = 1}^{m + 1} {{C_j}}  = \frac{1}{{{4^{m + 1}}}}\frac{{{{(1)}_{m + 1}}{{(\alpha  + 1)}_{m + 1}}{{(\beta  + 1)}_{m + 1}}{{(\alpha  + \beta  + 1)}_{m + 1}}}}{{{{(\frac{{\alpha  + \beta  + 3}}{2})}_{m + 1}}(\frac{{\alpha  + \beta  + 2}}{2})_{m + 1}^2{{(\frac{{\alpha  + \beta  + 1}}{2})}_{m + 1}}}}.\]
Therefore, the identity \eqref{eq8.9} for the polynomials \eqref{eq8.16} takes the form
\begin{multline*}
\sum\limits_{n = 0}^m \begin{array}{l}
\frac{{{4^{n + 1}}{{(\frac{{\alpha  + \beta  + 3}}{2})}_{n + 1}}(\frac{{\alpha  + \beta  + 2}}{2})_{n + 1}^2{{(\frac{{\alpha  + \beta  + 1}}{2})}_{n + 1}}}}{{(n + 1)!\,\,{{(\alpha  + 1)}_{n + 1}}{{(\beta  + 1)}_{n + 1}}{{(\alpha  + \beta  + 1)}_{n + 1}}}}\\
\,\,\,\,\,\,\,\,\,\,\,\,\,\, \times \left( {\bar Q_n^{(\alpha ,\beta )}(x;1)\,\bar Q_n^{(\alpha ,\beta )}(t;1)\, - \frac{{{2^{n + 1}}{{(\alpha  + 1)}_{n + 1}}}}{{{{(n + \alpha  + \beta  + 2)}_{n + 1}}}}\frac{{\bar Q_n^{(\alpha ,\beta )}(x;1)\, - \bar Q_n^{(\alpha ,\beta )}(t;1)}}{{x - t}}} \right)
\end{array} \\
 = \frac{{{4^{m + 1}}{{(\frac{{\alpha  + \beta  + 3}}{2})}_{m + 1}}(\frac{{\alpha  + \beta  + 2}}{2})_{m + 1}^2{{(\frac{{\alpha  + \beta  + 1}}{2})}_{m + 1}}}}{{(m + 1)!\,\,{{(\alpha  + 1)}_{m + 1}}{{(\beta  + 1)}_{m + 1}}{{(\alpha  + \beta  + 1)}_{m + 1}}}}\\
 \times \,\,\frac{{\bar Q_{m + 1}^{(\alpha ,\beta )}(x;1)\,\bar Q_m^{(\alpha ,\beta )}(t;1) - \bar Q_{m + 1}^{(\alpha ,\beta )}(t;1)\,\bar Q_m^{(\alpha ,\beta )}(x;1)}}{{x - t}}.
\end{multline*}

\subsubsection{Some particular trigonometric cases}\label{subsubsec8.2.2}
There are four trigonometric cases of Jacobi polynomials which are known in the literature as the Chebyshev polynomials of first, second, third and fourth kind. The main advantage of these polynomials is that their roots are explicitly known \cite{ref4}, see also \cite{ref17}. Their monic forms are represented as
\begin{align}
{{\bar T}_n}(x) &= \bar P_n^{( - \frac{1}{2}, - \frac{1}{2})}(x) = \frac{1}{{{2^{n - 1}}}}\,\cos \left( {n\arccos x} \right) = \prod\limits_{k = 1}^n {(x - \cos \frac{{(2k - 1)\pi }}{{2n}})} ,\label{eq8.21}\\
{{\bar U}_n}(x) &= \bar P_n^{(\frac{1}{2},\frac{1}{2})}(x) = \frac{1}{{{2^n}\sqrt {1 - {x^2}} }}\,\sin \left( {(n + 1)\arccos x} \right) = \prod\limits_{k = 1}^n {(x - \cos \frac{{k\pi }}{{n + 1}})} ,\nonumber\\
{{\bar V}_n}(x) &= \bar P_n^{( - \frac{1}{2},\frac{1}{2})}(x) = \frac{1}{{{2^n}}}\sqrt {\frac{2}{{1 + x}}} \,\cos ((n + \frac{1}{2})\arccos x) = \prod\limits_{k = 1}^n {(x - \cos \frac{{(2k - 1)\pi }}{{2n + 1}})} ,\nonumber\\
{{\bar W}_n}(x) &= \bar P_n^{(\frac{1}{2}, - \frac{1}{2})}(x) = \frac{1}{{{2^n}}}\sqrt {\frac{2}{{1 - x}}} \,\sin ((n + \frac{1}{2})\arccos x) = \prod\limits_{k = 1}^n {(x - \cos \frac{{2k\pi }}{{2n + 1}})} .\nonumber
\end{align}
Noting that
\[{T_n}(x) = {2^{n - 1}}{\bar T_n}(x),\,\,\,\,{U_n}(x) = {2^n}{\bar U_n}(x),\,\,\,\,{V_n}(x) = {2^n}{\bar V_n}(x)\,\,\,\,\text{and}\,\,\,\,{W_n}(x) = {2^n}{\bar W_n}(x),\]
they satisfy the following orthogonality relations
\begin{align}\label{eq8.22}
&\int_{ - 1}^1 {{T_n}(x){T_m}(x)\frac{1}{{\sqrt {1 - {x^2}} }}\,dx}  = \left\{ \begin{array}{l}
\frac{\pi }{2}\,{\delta _{n,m}},\\
\pi \,\,\,\,{\rm{if}}\,\,\,n = m = 0,
\end{array} \right.\\
&\int_{ - 1}^1 {{U_n}(x){U_m}(x)\,\sqrt {1 - {x^2}} \,dx}  = \frac{\pi }{2}{\delta _{n,m}},\nonumber\\
&\int_{ - 1}^1 {{V_n}(x)\,{V_m}(x)\sqrt {\frac{{1 + x}}{{1 - x}}} \,dx}  = \pi \,{\delta _{n,m}},\nonumber\\
&\int_{ - 1}^1 {{W_n}(x)\,{W_m}(x)\sqrt {\frac{{1 - x}}{{1 + x}}} \,dx}  = \pi \,{\delta _{n,m}}.\nonumber
\end{align}
Now, we can use relations \eqref{eq8.21} and define four trigonometric uncorrelated sequences, according to the main definition \eqref{eq8.16} as follows
\begin{multline}\label{eq8.23}
{\bar T_n}(x;1) = \bar Q_n^{( - \frac{1}{2}, - \frac{1}{2})}(x;1) = \frac{{1 - {{\bar T}_{n + 1}}(x)}}{{1 - x}} = \frac{{(n + 1)\,{2^n}{{(3/2)}_n}}}{{{{(n + 2)}_n}}}{}_3{F_2}\left( {\left. {\begin{array}{*{20}{c}}
{ - n,\,\,n + 2,\,\,1}\\
{3/2,\,\,2}
\end{array}\,} \right|\,\frac{{1 - x}}{2}} \right) \\
= \prod\limits_{k = 1}^n {(x - \cos \frac{{2k\pi }}{{n + 1}}} ),
\end{multline}
\[{\bar U_n}(x;1) = \bar Q_n^{(\frac{1}{2},\frac{1}{2})}(x;1) = \frac{{n + 2 - {{\bar U}_{n + 1}}(x)}}{{1 - x}} = \frac{{(n + 1)\,{2^n}{{(5/2)}_n}}}{{{{(n + 4)}_n}}}{}_3{F_2}\left( {\left. {\begin{array}{*{20}{c}}
{ - n,\,\,n + 4,\,\,1}\\
{5/2,\,\,2}
\end{array}\,} \right|\,\frac{{1 - x}}{2}} \right),\]
\begin{multline}\label{eq8.24}
{\bar V_n}(x;1) = \bar Q_n^{( - \frac{1}{2},\frac{1}{2})}(x;1) = \frac{{1 - {{\bar V}_{n + 1}}(x)}}{{1 - x}} = \frac{{(n + 1)\,{2^n}{{(3/2)}_n}}}{{{{(n + 3)}_n}}}{}_3{F_2}\left( {\left. {\begin{array}{*{20}{c}}
{ - n,\,\,n + 3,\,\,1}\\
{3/2,\,\,2}
\end{array}\,} \right|\,\frac{{1 - x}}{2}} \right) \\
= \prod\limits_{k = 1}^n {(x - \cos \frac{{2k\pi }}{n}} ),
\end{multline}
\[{\bar W_n}(x;1) = \bar Q_n^{(\frac{1}{2}, - \frac{1}{2})}(x;1) = \frac{{2n + 3 - {{\bar W}_{n + 1}}(x)}}{{1 - x}} = \frac{{(n + 1)\,{2^n}{{(5/2)}_n}}}{{{{(n + 3)}_n}}}{}_3{F_2}\left( {\left. {\begin{array}{*{20}{c}}
{ - n,\,\,n + 3,\,\,1}\\
{5/2,\,\,2}
\end{array}\,} \right|\,\frac{{1 - x}}{2}} \right).\]
As we observe, only relations \eqref{eq8.23} and \eqref{eq8.24} are decomposable with multiple roots. In this direction, it is worth mentioning that there is a generic decomposable sequence as
\[{T_n}(x;\lambda ) = \frac{{{T_{n + 1}}(x) - {T_{n + 1}}(\lambda )}}{{x - \lambda }} = {2^n}\prod\limits_{k = 1}^n {x - \left( {\lambda \cos \frac{{2k\pi }}{{n + 1}} - \sqrt {1 - {\lambda ^2}} \sin \frac{{2k\pi }}{{n + 1}}} \right)} \,,\]
satisfying the non-homogenous differential equation
\begin{multline*}
(x - \lambda )(1 - {x^2})\frac{{{d^2}}}{{d{x^2}}}{T_n}(x;\lambda ) + \left( { - 3{x^2} + \lambda x + 2} \right)\frac{d}{{dx}}\,{T_n}(x;\lambda )\\
 + \left( {{{(n + 1)}^2}(x - \lambda ) - x} \right)\,{T_n}(x;\lambda ) =  - {(n + 1)^2}\cos ((n + 1)\arccos \lambda ),
\end{multline*}
and the recurrence relation
\[{T_{n + 1}}(x;\lambda ) = 2x{T_n}(x;\lambda ) - {T_{n - 1}}(x;\lambda ) + 2\cos (n\arccos \lambda ).\]
For instance
\[{T_n}(x;0) = \frac{{{T_{n + 1}}(x) - {T_{n + 1}}(0)}}{x} = {2^n}\prod\limits_{k = 1}^n {(x + \sin \frac{{2k\pi }}{{n + 1}}} ),\]
is an uncorrelated polynomial with respect to the fixed function $ z(x)=x^{-1} $ so that according to corollary \ref{coro8.1} and relation \eqref{eq8.22} we have
\[\int_{ - 1}^1 {\frac{{T_{n + 1}^2(x)}}{{\sqrt {1 - {x^2}} }}\,dx}  = \int_{ - 1}^1 {\frac{{{x^2}}}{{\sqrt {1 - {x^2}} }}\,dx}  = \frac{\pi }{2},\]
and as a result
\begin{multline*}
{\left. {{{{\mathop{\rm cov}} }_1}\left( {{T_m}(x;0),\,{T_n}(x;0);\frac{1}{x}} \right)\,} \right|_{\frac{{{x^2}}}{{\sqrt {1 - {x^2}} }}}} = \,\int_{ - 1}^1 {\frac{{{x^2}}}{{\sqrt {1 - {x^2}} }}\,{T_m}(x;0){T_n}(x;0)\,dx} \\
- \frac{1}{\pi }\,\int_{ - 1}^1 {\frac{x}{{\sqrt {1 - {x^2}} }}\,{T_m}(x;0)\,dx} \,\int_{ - 1}^1 {\frac{x}{{\sqrt {1 - {x^2}} }}\,{T_n}(x;0)\,dx}  = {\delta _{m,n}}.
\end{multline*}
\begin{remark}\label{rem8.2.3}
If in relation \eqref{eq8.21}, $ \arccos x = \theta  $, the Chebyshev polynomials will be transformed to four trigonometric sequences orthogonal with respect to the constant weight function on $ [0,\pi] $ and are respectively represented as $ \{ \cos n\theta \} _{n = 0} $, $ \{ \sin (n + 1)\theta \} _{n = 0} $, $ \{ \cos (n + \frac{1}{2})\theta \} _{n = 0} $ and $ \{ \sin (n + \frac{1}{2})\theta \} _{n = 0} $ satisfying the following orthogonality relations, according to the relations \eqref{eq8.22},
\begin{align*}
& \int_0^\pi  {\cos n\theta \,\cos m\theta \,d\theta }  = \left\{ \begin{array}{l}
\frac{\pi }{2}\,{\delta _{n,m}},\\
\pi \,\,\,\,{\rm{if}}\,\,\,n = m = 0,
\end{array} \right.\\
& \int_0^\pi  {\sin (n + 1)\theta \,\sin (m + 1)\theta \,d\theta }  = \frac{\pi }{2}\,{\delta _{n,m}},\\
& \int_0^\pi  {\cos (n + \frac{1}{2})\theta \,\cos (m + \frac{1}{2})\theta \,dx}  = \frac{\pi }{2}\,{\delta _{n,m}},
\end{align*}
and
\[\int_0^\pi  {\sin (n + \frac{1}{2})\theta \,\sin (m + \frac{1}{2})\theta \,dx}  = \frac{\pi }{2}\,{\delta _{n,m}}.\]
Here our goal is to consider such orthogonal sequences as initial data corresponding to determinants \eqref{eq7.1} to see what the subsequent uncorrelated functions look like.

Since the probability density function for all above-mentioned sequences is $ w(\theta ) = \frac{1}{\pi } $ on $ [0,\pi] $, for the first sequence we obtain
\[{\mathop{\rm cov}} \,\left( {\cos k\theta \,,\cos j\theta } \right) = E\left( {\cos k\theta \cos j\theta } \right) - E\left( {\cos k\theta } \right)E\left( {\cos j\theta } \right) = 0 \Leftrightarrow k \ne j,\]
which reveals that the initial data corresponding to the first orthogonal sequence do not cause to generate uncorrelated trigonometric functions because $ {V_0} = \cos 0 = 1 $. But, the story is somewhat different for the second sequence $ \{ {V_k} = \sin (k + 1)\theta \} _{k = 0}^n $ as we have
\[{\mathop{\rm cov}} \,\left( {\sin (k + 1)\theta \,,\sin (j + 1)\theta } \right) =  - \frac{1}{{{\pi ^2}}}\frac{{(1 + {{( - 1)}^k})(1 + {{( - 1)}^j})}}{{(k + 1)(j + 1)\,}} + \frac{1}{2}{\delta _{k,j}},\]
and for $ k=j $,
\[{\mathop{\rm var}} \left( {\sin (k + 1)\theta } \right) =  - \frac{2}{{{\pi ^2}}}\frac{{1 + {{( - 1)}^k}}}{{{{(k + 1)}^2}\,}} + \frac{1}{2}\,.\]
Substituting the above data into \eqref{eq7.1}, the monic type of the elements e.g. for $ \{ {\bar X_k} = {\bar \Phi _{k + 1}}(\theta )\} _{k = 0}^5 $ are derived as
\begin{align*}
{{\bar \Phi }_1}(\theta ) &= \sin \theta ,\\
{{\bar \Phi }_2}(\theta ) &= \sin 2\theta ,\\
{{\bar \Phi }_3}(\theta ) &= \sin 3\theta  + \frac{8}{3}\frac{1}{{{\pi ^2} - 8}}\sin \theta ,\\
{{\bar \Phi }_4}(\theta ) &= \sin 4\theta ,\\
{{\bar \Phi }_5}(\theta ) &= \sin 5\theta  + \frac{{24}}{5}\frac{1}{{9{\pi ^2} - 80}}\sin 3\theta  + \frac{8}{5}\frac{{9{\pi ^2} - 88}}{{(9{\pi ^2} - 80)({\pi ^2} - 8)}}\sin \theta ,\\
{{\bar \Phi }_6}(\theta ) &= \sin 6\theta ,
\end{align*}
satisfying the uncorrelatedness condition
\begin{equation}\label{eq8.25}
\int_{\,0}^\pi  {{{\bar \Phi }_n}(\theta )\,{{\bar \Phi }_m}(\theta )\,d\theta }  = \frac{1}{\pi }\int_{\,0}^\pi  {{{\bar \Phi }_n}(\theta )\,d\theta } \int_{\,0}^\pi  {{{\bar \Phi }_m}(\theta )\,d\theta }  \Leftrightarrow n \ne m.
\end{equation}
The samples show that the general structure of $ \{ {\bar \Phi _k}(\theta )\} _{k = 1}^n $ is as follows
\[{\bar \Phi _{2k}}(\theta ) = \sin 2k\theta \quad \text{and} \quad {\bar \Phi _{2k + 1}}(\theta ) = \sum\limits_{j = 0}^k {{a_{2j + 1}}\sin (2j + 1)\theta }\quad \text{with}\quad {a_{2k + 1}} = 1. \]
Also, if the change of variable $ \theta  = \arccos t $ is applied in \eqref{eq8.25}, then
\[\int_{\, - 1}^1 {\frac{{{\Phi _n}(\arccos t)\,{\Phi _m}(\arccos t)}}{{\sqrt {1 - {t^2}} }}\,dt}  = \frac{1}{\pi }\int_{\, - 1}^1 {{\Phi _n}(\arccos t)\,dt} \int_{\, - 1}^1 {{\Phi _m}(\arccos t)\,dt}  \Leftrightarrow n \ne m,\]
where
\[{\Phi _{2k}}(\arccos t) = \sqrt {1 - {t^2}} {U_{2k - 1}}(t)\quad \text{and}\quad {\Phi _{2k + 1}}(\arccos t) = \sqrt {1 - {t^2}} \sum\limits_{j = 0}^k {{a_{2j + 1}}{U_{2j}}(t)} .\]
The procedure for deriving two other sequences $ \{ {V_k} = \cos (k + \frac{1}{2})\theta \} _{k = 0}^n $ and $ \{ {V_k} = \sin (k + \frac{1}{2})\theta \} _{k = 0}^n $ is similar. For instance, we have
\[{\mathop{\rm cov}} \,\left( {\sin (k + \frac{1}{2})\theta \,,\sin (j + \frac{1}{2})\theta } \right) =  - \frac{4}{{{\pi ^2}}}\frac{1}{{(2k + 1)(2j + 1)\,}} + \frac{1}{2}{\delta _{k,j}}.\]
\end{remark}

\subsection{An uncorrelated sequence of hypergeometric polynomials of $ _{2}F_{2} $ type}\label{subsec8.3}
This turn, consider the monic type of the (generalized) Laguerre polynomials \cite{ref4, ref21}
\[\bar L_n^{(\alpha )}(x) = {( - 1)^n}{(\alpha  + 1)_n}{}_1{F_1}\left( {\left. {\begin{array}{*{20}{c}}
{ - n}\\
{\alpha  + 1}
\end{array}\,} \right|\,x} \right),\]
satisfying the equation
\[x\frac{{{d^2}}}{{d{x^2}}}\bar L_n^{(\alpha )}(x) + \left( {\alpha  + 1 - x} \right)\frac{d}{{dx}}\bar L_n^{(\alpha )}(x) + n\,\bar L_n^{(\alpha )}(x) = 0\,,\]
and orthogonal with respect to the weight function $ {x^\alpha }{e^{ - x}} $ on $ [0,\infty) $ as
\[\int_0^\infty  {{x^\alpha }{e^{ - x}}\bar L_m^{(\alpha )}(x)\,\bar L_n^{(\alpha )}(x)\,dx}  = n!\,\Gamma (n + \alpha  + 1)\,{\delta _{n,m}}\,.\]
They also satisfy the recurrence relation
\begin{equation}\label{eq8.26}
\bar L_{n + 1}^{(\alpha )}(x) = \left( {x - 2n - \alpha  - 1} \right)\,\bar L_n^{(\alpha )}(x) - n(n + \alpha )\bar L_{n - 1}^{(\alpha )}(x)\,.
\end{equation}
Noting that
\[\bar L_n^{(\alpha )}(0) = {( - 1)^n}{(\alpha  + 1)_n},\]
the uncorrelated polynomials based on Laguerre polynomials is defined as
\begin{equation}\label{eq8.27}
\bar Q_n^{(\alpha )}(x;0) = \frac{{\bar L_{n + 1}^{(\alpha )}(x) - \bar L_{n + 1}^{(\alpha )}(0)}}{x} = {( - 1)^n}(n + 1){(\alpha  + 2)_n}\,{}_2{F_2}\left( {\left. {\begin{array}{*{20}{c}}
{ - n,\,\,1}\\
{\alpha  + 2,\,\,2}
\end{array}\,} \right|\,x} \right).
\end{equation}
Similar to the previous example, for $ \alpha>-1 $ we can prove that
\[{\left. {{{{\mathop{\rm cov}} }_1}\left( {\bar Q_m^{(\alpha )}(x;0),\,\bar Q_n^{(\alpha )}(x;0);\frac{1}{x}} \right)\,} \right|_{{x^{\alpha  + 2}}{e^{ - x}}}} = \frac{{n!\,\,\Gamma (n + \alpha  + 1)}}{{\Gamma (\alpha  + 3)}}\,{\delta _{m,n}}.\]
Also, since in \eqref{eq8.26},
\[{C_j} = j(j + \alpha )\,\,\,\,\text{and}\,\,\,\,\prod\limits_{j = 1}^{m + 1} {{C_j}}  = {(1)_{m + 1}}{(\alpha  + 1)_{m + 1}} = (m + 1)!\,\frac{{\Gamma (\alpha  + m + 2)}}{{\Gamma (\alpha  + 1)}},\]
the identity \eqref{eq8.9} for the monic polynomials \eqref{eq8.27} is derived as
\begin{multline*}
\sum\limits_{n = 0}^m {(n + 1)!\,{{(\alpha  + 1)}_{n + 1}}\left( {\bar Q_n^{(\alpha )}(x;0)\,\bar Q_n^{(\alpha )}(t;0)\, - {{( - 1)}^n}{{(\alpha  + 1)}_n}\frac{{\bar Q_n^{(\alpha )}(x;0)\, - \bar Q_n^{(\alpha )}(t;0)}}{{x - t}}} \right)} \\
 = (m + 1)!\,{(\alpha  + 1)_{m + 1}}\frac{{\bar Q_{m + 1}^{(\alpha )}(x;0)\,\bar Q_m^{(\alpha )}(t;0) - \bar Q_{m + 1}^{(\alpha )}(t;0)\,\bar Q_m^{(\alpha )}(x;0)}}{{x - t}}.
\end{multline*}

\section{A unified approach for the polynomials obtained in sections \ref{sec6}, \ref{sec7} and \ref{sec8}}\label{sec9}

According to the distributions given in table \ref{tab1}, only beta and gamma weight functions can be considered for the non-symmetric infinite cases of uncorrelated polynomials, as the normal distribution is somehow connected to a special case of gamma distribution. In this direction, the general properties of two polynomials \eqref{eq6.12} and \eqref{eq7.3} reveal that the most general case of complete uncorrelated polynomials relevant to the beta weight function is when $ w(x) = {x^a}{(1 - x)^b} $ and $ w(x)\,z(x) = {x^c}{(1 - x)^d} $, i.e. $  z(x) = {x^{c - a}}{(1 - x)^{d - b}} $ where $ a,b,c,d\in\mathbb{R} $ and $ x\in[0,1] $. Hence, if the corresponding uncorrelated polynomial is indicated as $ {{\bf{P}}_n}(x;a,b,c,d) $, we have
\begin{align}\label{eq9.1}
&\int_{\,0}^1 {{x^a}{{(1 - x)}^b}{{\bf{P}}_n}(x;a,b,c,d)\,{{\bf{P}}_m}(x;a,b,c,d)\,dx} \\
& - \frac{{\Gamma (2c + 2d - a - b + 2)}}{{\Gamma (2c - a + 1)\Gamma (2d - b + 1)}}\int_{\,0}^1 {{x^c}{{(1 - x)}^d}{{\bf{P}}_n}(x;a,b,c,d)\,dx} \,\int_{\,0}^1 {{x^c}{{(1 - x)}^d}{{\bf{P}}_m}(x;a,b,c,d)\,dx}\notag \\
&\qquad = \left( \int_{\,0}^1 {{x^a}{{(1 - x)}^b}{\bf{P}}_n^2(x;a,b,c,d)\,dx} \right.\notag \\
&\qquad\quad \left. - \frac{{\Gamma (2c + 2d - a - b + 2)}}{{\Gamma (2c - a + 1)\Gamma (2d - b + 1)}}{{\left( {\int_{\,0}^1 {{x^c}{{(1 - x)}^d}{{\bf{P}}_n}(x;a,b,c,d)\,dx} } \right)}^2}\right){\delta _{m,n}},\notag
\end{align}
provided that
\[2c - a + 1 > 0,\,\,\,2d - b + 1 > 0\,\,\,\,\,{\rm{and}}\,\,\,\,b,d >  - 1.\]
The components of the determinant \eqref{eq6.7} corresponding to this generic polynomial are computed as
\begin{multline}\label{eq9.2}
{\left. {{{{\mathop{\rm cov}} }_1}\,\left( {{x^i},{x^j};{x^{c - a}}{{(1 - x)}^{d - b}}} \right)\,} \right|_{w(x) = {x^a}{{(1 - x)}^b}}} = \\
\Gamma (b + 1)\,\frac{{\Gamma (a + i + j + 1)}}{{\Gamma (a + b + i + j + 2)}} - \frac{{\Gamma (2c + 2d - a - b + 2){\Gamma ^2}(d + 1)}}{{\Gamma (2c - a + 1)\Gamma (2d - b + 1)}}\,\frac{{\Gamma (c + i + 1)\Gamma (c + j + 1)}}{{\Gamma (c + d + i + 2)\Gamma (c + d + j + 2)}},
\end{multline}
in which $ 2c - a + 1 > 0,\,\,\,2d - b + 1 > 0 $ and $ b,d >  - 1 $.

According to the preceding information, the polynomials \eqref{eq6.12} can be represented as 
\begin{equation}\label{eq9.3}
{}_4{F_3}\left( {\left. {\begin{array}{*{20}{c}}
{ - n,\,\,n + 1,\, - r,\,\,r + 2}\\
{1,\,\, - r + 1,\,\,r + 1}
\end{array}\,} \right|\,x} \right)\, = {{\bf{P}}_n}(x;0,0,r,0),
\end{equation}
and the polynomials \eqref{eq7.3} as
\begin{equation}\label{eq9.4}
{}_4{F_3}\left( {\left. {\begin{array}{*{20}{c}}
{ - n,\,\,n + 2r + 1,\,r,\,\,r + 2}\\
{2r + 1,\,\,r + 1,\,\,r + 1}
\end{array}\,} \right|\,x} \right) = {{\bf{P}}_n}(x;2r,0,r,0),
\end{equation}
and since in \eqref{eq8.20},
\begin{multline*}
\int_{\,0}^1 {{x^{\alpha  + 2}}{{(1 - x)}^\beta }{}_3{F_2}\left( {\left. {\begin{array}{*{20}{c}}
{ - n,\,\,n + \alpha  + \beta  + 3,\,\,1}\\
{\alpha  + 2,\,\,2}
\end{array}\,} \right|\,x} \right){}_3{F_2}\left( {\left. {\begin{array}{*{20}{c}}
{ - m,\,\,m + \alpha  + \beta  + 3,\,\,1}\\
{\alpha  + 2,\,\,2}
\end{array}\,} \right|\,x} \right)\,dx} \\
 =\frac{1}{{\int_{\,0}^1 {{x^\alpha }{{(1 - x)}^\beta }dx} }}(\left( {\int_{\,0}^1 {{x^{\alpha  + 1}}{{(1 - x)}^\beta }{}_3{F_2}\left( {\left. {\begin{array}{*{20}{c}}
{ - n,\,\,n + \alpha  + \beta  + 3,\,\,1}\\
{\alpha  + 2,\,\,2}
\end{array}\,} \right|\,x} \right)\,dx} } \right) \\
\times \left( {\int_{\,0}^1 {{x^{\alpha  + 1}}{{(1 - x)}^\beta }{}_3{F_2}\left( {\left. {\begin{array}{*{20}{c}}
{ - m,\,\,m + \alpha  + \beta  + 3,\,\,1}\\
{\alpha  + 2,\,\,2}
\end{array}\,} \right|\,x} \right)\,dx} } \right)\quad \Leftrightarrow \,\, n \ne m, 
\end{multline*}
the shifted polynomials \eqref{eq8.17} on $ [0,1] $ are represented as
\[{}_3{F_2}\left( {\left. {\begin{array}{*{20}{c}}
{ - n,\,\,n + \alpha  + \beta  + 3,\,\,1}\\
{\alpha  + 2,\,\,2}
\end{array}\,} \right|\,x} \right) = {{\bf{P}}_n}(x;\alpha  + 2,\beta ,\alpha  + 1,\beta ).\]
Similarly, the most general case of complete uncorrelated polynomials relevant to the gamma weight function is when $ w(x) = {x^a}{e^{ - bx}} $ and $ w(x)\,z(x) = {x^c}{e^{ - d\,x}} $, i.e. $ z(x) = {x^{c - a}}{e^{ - (d - b)\,x}} $ where $ a,b,c,d \in\mathbb{R} $ and $ x\in\mathbb{R}^{+} $. Therefore, if the corresponding uncorrelated polynomial is indicated as $ {{\bf{Q}}_n}(x;a,b,c,d) $, then
\begin{multline}\label{eq9.5}
\int_{\,0}^\infty  {{x^a}{e^{ - bx}}\,{{\bf{Q}}_n}(x;a,b,c,d)\,{{\bf{Q}}_m}(x;a,b,c,d)\,dx} \\
 - \frac{{\Gamma (2c - a + 1)}}{{{{(2d - b)}^{2c - a + 1}}}}\int_{\,0}^\infty  {{x^c}{e^{ - d\,x}}\,{{\bf{Q}}_n}(x;a,b,c,d)\,dx} \,\int_{\,0}^\infty  {{x^c}{e^{ - d\,x}}\,{{\bf{Q}}_m}(x;a,b,c,d)\,dx} \\
 = \left( {\int_{\,0}^\infty  {{x^a}{e^{ - bx}}\,{\bf{Q}}_n^2(x;a,b,c,d)\,dx}  - \frac{{\Gamma (2c - a + 1)}}{{{{(2d - b)}^{2c - a + 1}}}}{{\left( {\int_{\,0}^\infty  {{x^c}{e^{ - d\,x}}\,{{\bf{Q}}_n}(x;a,b,c,d)\,dx} } \right)}^2}} \right){\delta _{m,n}},
\end{multline}
provided that
\[2c - a + 1 > 0,\,\,\,2d - b > 0\,\,\,\,{\rm{and}}\,\,\,\,b,d > 0.\]
The components of the determinant \eqref{eq6.7} corresponding to this second generic polynomial are computed as
\begin{multline}\label{eq9.6}
{\left. {{{{\mathop{\rm cov}} }_1}\,\left( {{x^i},{x^j};{x^{c - a}}{e^{ - (d\, - b)x}}} \right)\,} \right|_{w(x) = {x^a}{e^{ - bx}}}}\\
= \frac{{\Gamma (a + i + j + 1)}}{{{b^{a + i + j + 1}}}} - \frac{{\Gamma (2c - a + 1)}}{{{{(2d - b)}^{2c - a + 1}}}}\,\frac{{\Gamma (c + i + 1)\Gamma (c + j + 1)}}{{{d^{2c + i + j + 2}}}},
\end{multline}
in which $ 2c - a + 1 > 0,\,\,\,2d - b > 0\,\,\,\,{\rm{and}}\,\,\,\,b,d > 0 $.

As a sample, the polynomials \eqref{eq8.27} can be represented as
\[{}_2{F_2}\left( {\left. {\begin{array}{*{20}{c}}
{ - n,\,\,1}\\
{\alpha  + 2,\,\,2}
\end{array}\,} \right|\,x} \right) = {{\bf{Q}}_n}(x;\alpha  + 2,1,\alpha  + 1,1)\,.\]

\subsection{A basic example of uncorrelated hypergeometric polynomials of $ _{4}F_{3} $ type}\label{subsec9.1}
In this section, we are going to obtain the explicit form of $ {{\bf{P}}_n}(x;a,0,c,0) $ using an interesting technique. Since $ b=d=0 $, so $ w(x)=x^{a} $ and $ z(x)=x^{c-a} $ defined on $ [0,1] $. In the first step, we can simplify \eqref{eq9.2} for $ b=d=0 $ as
\[{\left. {{{{\mathop{\rm cov}} }_1}\,\left( {{x^i},{x^j};{x^{c - a}}} \right)\,} \right|_{w(x) = {x^a}}} = \frac{{(c - a - i)(c - a - j)}}{{(i + c + 1)(j + c + 1)(i + j + a + 1)}}\,,\]
where $2c - a + 1 > 0,\,\,\,a \ne c\,\,\,\,{\rm{and}}\,\,\,\,a,c >  - 1$.

Referring to the results \eqref{eq9.3} and \eqref{eq9.4} and this fact that $ {{\bf{P}}_n}(x;a,0,c,0) $ is a generalization of both of them, we can imagine that it is of $ _{4}F_{3} $ type without loss of generality.

Since in a general $ _{4}F_{3} $ polynomial of the form \[{}_4{F_3}\left( {\left. {\begin{array}{*{20}{c}}
{ - n,\,\,{a_2},\,{a_3},\,\,{a_4}}\\
{{b_1},\,\,{b_2},\,\,{b_3}}
\end{array}\,} \right|\,x} \right) = \sum\limits_{k = 0}^n {{u_k}\,{x^k}} \,\,\,\,{\rm{where}}\,\,\,\,\,{u_k} = \frac{{{{( - n)}_k}{{({a_2})}_k}{{({a_3})}_k}{{({a_4})}_k}}}{{{{({b_1})}_k}{{({b_2})}_k}{{({b_3})}_k}k!}},\]
we have
\begin{equation}\label{eq9.7}
\frac{{{u_{n - 1}}}}{{{u_n}}} =  - \frac{{n\,(n + {b_1} - 1)(n + {b_2} - 1)(n + {b_3} - 1)}}{{(n + {a_2} - 1)(n + {a_3} - 1)(n + {a_4} - 1)}},
\end{equation}
which is to be equal to the minus of the coefficient of $ x^{n-1} $ in the determinant of the monic polynomial $ {{\bf{\bar P}}_n}(x;a,0,c,0) $ in \eqref{eq6.7}, if for simplicity we set
\[{a_{i,j}} = \frac{{(c - a - i)(c - a - j)}}{{(i + c + 1)(j + c + 1)(i + j + a + 1)}}\,\,\,\,\,\,{\rm{for}}\,\,\,i = 0,1,...,n - 1\,\,\,{\rm{and}}\,\,\,j = 0,1,...,n,\]
to achieve our goal, we should therefore compute the two following determinants (according to the main determinant \eqref{eq6.7}),
\[M_{n}=\begin{vmatrix}
{{a_{0,0}}} & {{a_{0,1}}} & \cdots & {a_{0,n - 2}} & {a_{0,n - 1}} \\   
{{a_{1,0}}} & {{a_{1,1}}} &\cdots & {a_{1,n - 2}} & {a_{1,n - 1}} \\
 \vdots & \vdots & \vdots &\vdots & \vdots \\
{{a_{n - 1,0}}} & {{a_{n - 1,1}}} & \cdots &  {a_{n - 1,n - 2}} & {a_{n - 1,n - 1}}
\end{vmatrix},\]
and
\[N_{n}=\begin{vmatrix}
{{a_{0,0}}} & {{a_{0,1}}} & \cdots & {a_{0,n - 2}} & {a_{0,n}} \\   
{{a_{1,0}}} & {{a_{1,1}}} &\cdots & {a_{1,n - 2}} & {a_{1,n }} \\
 \vdots & \vdots & \vdots &\vdots & \vdots \\
{{a_{n - 1,0}}} & {{a_{n - 1,1}}} & \cdots &  {a_{n - 1,n - 2}} & {a_{n - 1,n }}
\end{vmatrix},\]
and then obtain the quotient $  - {N_n}/{M_n} $ (with the aid of advanced mathematical software) as 
\[ - \frac{{{N_n}}}{{{M_n}}} =  - \frac{{n\,(n + a)(n + a - c)(n + c)}}{{(2n + a)(n + a - c - 1)(n + c + 1)}},\]
and compare it with \eqref{eq9.7} to finally obtain the explicit form of the polynomials as
\begin{equation}\label{eq9.8}
{{\bf{P}}_n}(x;a,0,c,0) = {}_4{F_3}\left( {\left. {\begin{array}{*{20}{c}}
{ - n,\,\,n + a + 1,\,a - c,\,\,c + 2}\\
{a + 1,\,\,a - c + 1,\,\,c + 1}
\end{array}\,} \right|\,x} \right)\,,
\end{equation}
satisfying the uncorrelatedness condition \eqref{eq9.1} with $ b=d=0 $, i.e.
\begin{multline}\label{eq9.9}
\int_{\,0}^1 {{x^a}\,{{\bf{P}}_n}(x;a,0,c,0)\,{{\bf{P}}_m}(x;a,0,c,0)\,dx} \\
 - (2c - a + 1)\,\int_{\,0}^1 {{x^c}\,{{\bf{P}}_n}(x;a,0,c,0)\,dx} \,\int_{\,0}^1 {{x^c}\,{{\bf{P}}_m}(x;a,0,c,0)\,dx} \\
 = \left( {\int_{\,0}^1 {{x^a}\,{\bf{P}}_n^2(x;a,0,c,0)\,dx}  - (2c - a + 1){{\left( {\int_{\,0}^1 {{x^c}\,{{\bf{P}}_n}(x;a,0,c,0)\,dx} } \right)}^2}} \right){\delta _{m,n}}\\
  \Leftrightarrow 2c - a + 1 > 0,\,\,a \ne c\,\,\,{\rm{and}}\,\,\,a,c >  - 1\,.
\end{multline}
For the limit case $ 2c - a + 1 = 0 $ in \eqref{eq9.9}, the polynomials \eqref{eq9.8} would reduce to the well-known special case of shifted Jacobi polynomials defined on $ [0,1] $, i.e.
\[{{\bf{P}}_n}(x;a,0,\frac{{a - 1}}{2},0) = P_{n, + }^{(0,a)}(x) = {}_2{F_1}\left( {\left. {\begin{array}{*{20}{c}}
{ - n,\,\,n + a + 1}\\
{a + 1}
\end{array}\,} \right|\,x} \right)\,\,\,\,\,\,\,\,\,(a >  - 1).\]
Also, we have
\begin{align*}
\mathop {\lim }\limits_{c \to \infty } {{\bf{P}}_n}(x;a,0,c,0) &= \mathop {\lim }\limits_{c \to \infty } {}_4{F_3}\left( {\left. {\begin{array}{*{20}{c}}
{ - n,\,\,n + a + 1,\,a - c,\,\,c + 2}\\
{a + 1,\,\,a - c + 1,\,\,c + 1}
\end{array}\,} \right|\,x} \right)\, \\
&= {}_2{F_1}\left( {\left. {\begin{array}{*{20}{c}}
{ - n,\,\,n + a + 1}\\
{a + 1}
\end{array}\,} \right|\,x} \right) = P_{n, + }^{(0,a)}(x),
\end{align*}
and
\[\mathop {\lim }\limits_{a \to \infty } {{\bf{P}}_n}(x;a,0,c,0) = {}_2{F_1}\left( {\left. {\begin{array}{*{20}{c}}
{ - n,\,\,c + 2}\\
{c + 1}
\end{array}\,} \right|\,x} \right)\,\,\,\,\,\,\,\,\,(c >  - 1).\]
It seems that the results are straightforward for $ a=c=0 $ and $ b=d=s $ in \eqref{eq9.5} and \eqref{eq9.6}, as they take the Laplace transform form
\[L\left( {f(x)} \right) = \int_{\,0}^\infty  {{e^{ - sx}}f(x)\,dx} \,\,\,\,\,\,(s > 0),\]
and \eqref{eq9.5} becomes
\begin{multline*}
L\Big( {{{\bf{Q}}_n}(x;0,s,0,s)\,{{\bf{Q}}_m}(x;0,s,0,s)} \Big) - \frac{1}{s}L\Big( {{{\bf{Q}}_n}(x;0,s,0,s)} \Big)L\Big( {{{\bf{Q}}_m}(x;0,s,0,s)} \Big)\\
 = \left( {L\Big( {{\bf{Q}}_n^2(x;0,s,0,s)} \Big) - \frac{1}{s}{L^2}\Big( {{{\bf{Q}}_n}(x;0,s,0,s)} \Big)} \right){\delta _{m,n}}.
\end{multline*}
As we mentioned, the polynomials \eqref{eq9.8} are a generalization of \eqref{eq6.12} or \eqref{eq9.3} for $ (a,c)=(0,r) $ and a generalization of \eqref{eq7.3} or \eqref{eq9.4} for $ (a,c)=(2r,r) $.

In order to evaluate the existing integrals in \eqref{eq9.9}, we first have
\begin{multline}\label{eq9.10}
\int_{\,0}^1 {{x^c}\,{{\bf{P}}_n}(x;a,0,c,0)\,dx}  = \sum\limits_{k = 0}^n {\frac{{{{( - n)}_k}{{(n + a + 1)}_k}{{(a - c)}_k}{{(c + 2)}_k}}}{{{{(a + 1)}_k}{{(a - c + 1)}_k}{{(c + 1)}_k}k!\,\,(c + 1 + k)}}\,} \\
\,\,\,\,\,\,\,\,\,\,\,\,\, = \frac{1}{{c + 1}}{}_3{F_2}\left( {\left. {\begin{array}{*{20}{c}}
{ - n,\,\,n + a + 1,\,a - c}\\
{a + 1,\,\,a - c + 1}
\end{array}\,} \right|\,\,1} \right) = \frac{1}{{c + 1}}\frac{{n!\,\,{{(c + 1)}_n}}}{{{{(a + 1)}_n}\,{{(a + 1 - c)}_n}}},
\end{multline}
and
\begin{multline}\label{eq9.11}
\int_{\,0}^1 {{x^a}\,{{\bf{P}}_n}(x;a,0,c,0)\,{{\bf{P}}_m}(x;a,0,c,0)\,dx} \\
 = \frac{1}{{a + 1}}\sum\limits_{k = 0}^n \frac{{{{( - n)}_k}{{(n + a + 1)}_k}{{(a - c)}_k}{{(c + 2)}_k}}}{{{{(a + 2)}_k}{{(a - c + 1)}_k}{{(c + 1)}_k}k!\,}}\\
 \times {}_5{F_4}\left( {\left. {\begin{array}{*{20}{c}}
{ - m,\,\,m + a + 1,\,a - c,\,\,c + 2,\,a + 1 + k}\\
{a + 1,\,\,a - c + 1,\,c + 1,\,\,a + 2 + k}
\end{array}\,} \right|\,\,1} \right)\, ,
\end{multline}
which simplifies the left side of \eqref{eq9.9} as
\begin{align}\label{eq9.12}
&\int_{\,0}^1 {{x^a}\,{{\bf{P}}_n}(x;a,0,c,0)\,{{\bf{P}}_m}(x;a,0,c,0)\,dx}  \\
&\qquad - (2c - a + 1)\,\int_{\,0}^1 {{x^c}\,{{\bf{P}}_n}(x;a,0,c,0)\,dx} \,\int_{\,0}^1 {{x^c}\,{{\bf{P}}_m}(x;a,0,c,0)\,dx} \notag\\
&\quad = \frac{1}{{a + 1}}\sum\limits_{k = 0}^n \frac{{{{( - n)}_k}{{(n + a + 1)}_k}{{(a - c)}_k}{{(c + 2)}_k}}}{{{{(a + 2)}_k}{{(a - c + 1)}_k}{{(c + 1)}_k}k!\,}}\notag\\
&\qquad\times {}_5{F_4}\left( {\left. {\begin{array}{*{20}{c}}
{ - m,\,\,m + a + 1,\,a - c,\,\,c + 2,\,a + 1 + k}\\
{a + 1,\,\,a - c + 1,\,c + 1,\,\,a + 2 + k}
\end{array}\,} \right|\,\,1} \right)\, \notag\\
&\qquad - \frac{{2c - a + 1}}{{{{(c + 1)}^2}}}\frac{{n!\,\,{{(c + 1)}_n}}}{{{{(a + 1)}_n}\,{{(a + 1 - c)}_n}}}\frac{{m!\,\,{{(c + 1)}_m}}}{{{{(a + 1)}_m}\,{{(a + 1 - c)}_m}}}.\notag
\end{align}
To calculate \eqref{eq9.10}, we have again used the recurrence relations technique. Since
\[N(n) = {}_3{F_2}\left( {\left. {\begin{array}{*{20}{c}}
{ - n,\,\,n + a + 1,\,a - c}\\
{a + 1,\,\,a - c + 1}
\end{array}\,} \right|\,\,1} \right),\]
satisfies the first order relation
\[N(n + 1) = \frac{{(n + 1)(n + c + 1)}}{{(n + a + 1)(n + a - c + 1)}}N(n),\]
so
\[{}_3{F_2}\left( {\left. {\begin{array}{*{20}{c}}
{ - n,\,\,n + a + 1,\,a - c}\\
{a + 1,\,\,a - c + 1}
\end{array}\,} \right|\,\,1} \right) = \frac{{n!\,\,{{(c + 1)}_n}}}{{{{(a + 1)}_n}\,{{(a + 1 - c)}_n}}}.\]
However, note that all results obtained in \eqref{eq9.10}, \eqref{eq7.6} and \eqref{eq6.14} could also be derived by the Saalschutz theorem \cite{ref11}, which says if $ c^{*} $ is a negative integer and $ {a^*} + {b^*} + {c^*} + 1 = {d^*} + {e^*} $ then
\[_3{F_2}\left( {\left. {\begin{array}{*{20}{c}}
{\begin{array}{*{20}{c}}
{{a^*,}}&{{b^*,}}&{{c^*}}
\end{array}}\\
{\begin{array}{*{20}{c}}
{{d^*,}}&{{e^*}}
\end{array}}
\end{array}\,} \right|\,\,1} \right) = \frac{{{{({d^*} - {a^*})}_{|{c^*}|}}{{({d^*} - {b^*})}_{|{c^*}|}}}}{{{{({d^*})}_{|{c^*}|}}{{({d^*} - {a^*} - {b^*})}_{|{c^*}|}}}}.\]
Similarly, to calculate $ {}_5{F_4}(.) $ in \eqref{eq9.11}, we can apply a recurrence technique as follows. Since
\[M(m) = {}_5{F_4}\left( {\left. {\begin{array}{*{20}{c}}
{ - m,\,\,m + a + 1,\,a - c,\,\,c + 2,\,a + 1 + k}\\
{a + 1,\,\,a - c + 1,\,c + 1,\,\,a + 2 + k}
\end{array}\,} \right|\,\,1} \right),\]
satisfies the second order equation
\begin{multline*}
(2m + a + 2)(m + a + 2)(m + a + 1)(m + a + 2 - c)(m + a + 3 + k)\,M(m + 2)\\
 - (2m + a + 3)(m + a + 1)(m + 2)\left( {2m(m + a + 3) + (2c - a)k + (a + 2)(c + 2)} \right)\,M(m + 1)\\
 + (2m + a + 4)(m + 2)(m + 1)(m + 1 + c)(m - k)\,M(m) = 0,
\end{multline*}
having two independent solutions
\[{M_1}(m) = \frac{{m!\,\,{{(c + 1)}_m}}}{{{{(a + 1)}_m}{{(a + 1 - c)}_m}}},\]
and
\[{M_2}(m) = \frac{{m!\,\,{{( - k)}_m}}}{{{{(a + 1)}_m}{{(a + 2 + k)}_m}}},\]
so
\begin{align}\label{eq9.13}
&{}_5{F_4}\left( {\left. {\begin{array}{*{20}{c}}
{ - m,\,\,m + a + 1,\,a - c,\,\,c + 2,\,a + 1 + k}\\
{a + 1,\,\,a - c + 1,\,c + 1,\,\,a + 2 + k}
\end{array}\,} \right|\,\,1} \right)\\
&\quad = \frac{{m!\,}}{{(c + 1)(c + 1 + k){{(a + 1)}_m}}} \notag \\
&\qquad \times \left( {(2c - a + 1)(a + 1 + k)\frac{{{{(c + 1)}_m}}}{{{{(a + 1 - c)}_m}}} + \,(a - c)(a - c + k)\frac{{{{( - k)}_m}}}{{{{(a + 2 + k)}_m}}}} \right).\notag
\end{align}
Hence, in order to compute
\begin{multline*}
{{\mathop{\rm var}} _1}\left( {{{\bf{P}}_n}(x;a,0,c,0);\,{x^{c - a}}} \right) \\
= \int_{\,0}^1 {{x^a}\,{\bf{P}}_n^2(x;a,0,c,0)\,dx}  - (2c - a + 1){\left( {\int_{\,0}^1 {{x^c}\,{{\bf{P}}_n}(x;a,0,c,0)\,dx} } \right)^2},
\end{multline*}
we first suppose in \eqref{eq9.12} that $ m=n $ and then refer to \eqref{eq9.13} to arrive at
\begin{align*}
\left( {a + 1} \right){M^*} &= \sum\limits_{k = 0}^n \frac{{{{( - n)}_k}{{(n + a + 1)}_k}{{(a - c)}_k}{{(c + 2)}_k}}}{{{{(a + 2)}_k}{{(a - c + 1)}_k}{{(c + 1)}_k}k!\,}}\\
&\qquad\times {}_5{F_4}\left( {\left. {\begin{array}{*{20}{c}}
{ - n,\,\,n + a + 1,\,a - c,\,\,c + 2,\,a + 1 + k}\\
{a + 1,\,\,a - c + 1,\,c + 1,\,\,a + 2 + k}
\end{array}\,} \right|\,\,1} \right) \\
 &= \frac{{(a + 1)(2c - a + 1)}}{{{{(c + 1)}^2}}}\frac{{{{(n!)}^2}(c + 1)_n^2}}{{(a + 1)_n^2(a - c + 1)_n^2}}\\
  &\quad + \frac{{(a + 1){{(a - c)}^2}}}{{{{(c + 1)}^2}}}\frac{{n!}}{{(n + a + 1)(a + 1)_n^2}}\sum\limits_{k = 0}^n {\frac{{{{( - n)}_k}{{(n + a + 1)}_k}{{( - k)}_n}}}{{{{(n + a + 2)}_k}k!\,}}} .
\end{align*}
Again, noting that $ {( - k)_n} = 0 $ for any $ k<n $, the following final result will be derived.

\begin{corollary}\label{coro9.1.1}
If $ 2c - a + 1 > 0,\,\,a \ne c\,\,\,{\rm{and}}\,\,\,a,c >  - 1 $, then
\begin{align*}
&\int_{\,0}^1 {{x^a}\,{\bf{P}}_n^2(x;a,0,c,0)\,dx}  - (2c - a + 1){\left( {\int_{\,0}^1 {{x^c}\,{{\bf{P}}_n}(x;a,0,c,0)\,dx} } \right)^2} \\
&\qquad\qquad\qquad\qquad\qquad= \frac{1}{{2n + a + 1}}\left( {\frac{{(a - c)\,n!}}{{(c + 1)\,{{(a + 1)}_n}}}} \right)^2.
\end{align*}
Moreover, \eqref{eq9.12} is simplified as
\begin{multline*}
\int_{\,0}^1 {{x^a}\,{{\bf{P}}_n}(x;a,0,c,0)\,{{\bf{P}}_m}(x;a,0,c,0)\,dx}  \\
- (2c - a + 1)\,\int_{\,0}^1 {{x^c}\,{{\bf{P}}_n}(x;a,0,c,0)\,dx} \,\int_{\,0}^1 {{x^c}\,{{\bf{P}}_m}(x;a,0,c,0)\,dx} \\
\,\,\,\,\,\,\,\,\,\,\,\,\,\,\,\, = \frac{{{{(a - c)}^2}}}{{{{(c + 1)}^2}}}\frac{{m!}}{{(m + a + 1)(a + 1)_m^2\,}}\sum\limits_{k = 0}^n {\frac{{{{(n + a + 1)}_k}{{( - n)}_k}}}{{{{(m + a + 2)}_k}}}\frac{{{{( - k)}_m}}}{{k!}}}  = 0 \Leftrightarrow m \ne n,
\end{multline*}
leading to the same as relation \eqref{eq7.9} for $ 2r=\alpha>-1 $.
\end{corollary}

\begin{corollary}\label{coro9.1.2}
Using the latter corollary, one can now construct an optimized polynomial approximation (or expansion) for $ f(x) $ whose error 1-variance is minimized as follows
\begin{equation}\label{eq9.14}
f(x) = \sum\limits_{k = 0}^{n \to \infty } {\,\frac{{{A_k}}}{{{B_k}}}} {}_4{F_3}\left( {\left. {\begin{array}{*{20}{c}}
{ - k,\,\,k + a + 1,\,a - c,\,\,c + 2}\\
{a + 1,\,\,a - c + 1,\,\,c + 1}
\end{array}\,} \right|\,x} \right)\,,
\end{equation}
in which
\begin{align}\label{eq9.15}
{A_k} &= \int_{\,0}^1 {{x^a}\,f(x)\,{}_4{F_3}\left( {\left. {\begin{array}{*{20}{c}}
{ - k,\,\,k + a + 1,\,a - c,\,\,c + 2}\\
{a + 1,\,\,a - c + 1,\,\,c + 1}
\end{array}\,} \right|\,x} \right)\,dx} \\
&\quad - \left( {2c - a + 1} \right)\int_{\,0}^1 {{x^c}f(x)\,dx} \,\int_{\,0}^1 {{x^c}{}_4{F_3}\left( {\left. {\begin{array}{*{20}{c}}
{ - k,\,\,k + a + 1,\,a - c,\,\,c + 2}\\
{a + 1,\,\,a - c + 1,\,\,c + 1}
\end{array}\,} \right|\,x} \right)\,dx} \notag\\
& = \sum\limits_{j = 0}^k {\frac{{{{( - k)}_j}{{(k + a + 1)}_j}{{(a - c)}_j}{{(c + 2)}_j}}}{{{{(a + 1)}_j}{{(a - c + 1)}_j}{{(c + 1)}_j}\,j!}}\int_{\,0}^1 {{x^{a + j}}f(x)\,dx} }  \notag \\
&\qquad - \frac{{2c - a + 1}}{{{{(c + 1)}^2}}}\frac{{k!\,\,{{(c + 1)}_k}}}{{{{(a + 1)}_k}\,{{(a + 1 - c)}_k}}}\,\int_{\,0}^1 {{x^c}f(x)\,dx} ,\notag
\end{align}
and
\begin{equation}\label{eq9.16}
{B_k} = \frac{1}{{2k + a + 1}}{\left( {\frac{{(a - c)\,k!}}{{(c + 1)\,{{(a + 1)}_k}}}} \right)^2}.
\end{equation}
We clearly observe in the polynomial type approximation \eqref{eq9.14} that its basis do not satisfy an orthogonal condition but a complete uncorrelatedness condition. However, if we define the non-polynomial sequence
\begin{align}\label{eq9.17}
&{{\bf{G}}_n}(x;a,c) = {}_4{F_3}\left( {\left. {\begin{array}{*{20}{c}}
{ - n,\,\,n + a + 1,\,a - c,\,\,c + 2}\\
{a + 1,\,\,a - c + 1,\,\,c + 1}
\end{array}\,} \right|\,x} \right)\\
&\qquad\qquad\qquad- \frac{{2c - a + 1}}{{{{(c + 1)}^2}}}\frac{{n!\,\,{{(c + 1)}_n}}}{{{{(a + 1)}_n}\,{{(a + 1 - c)}_n}}}{x^{c - a}},\notag
\end{align}
satisfying the orthogonality condition
\begin{multline}\label{eq9.18}
\int_{\,0}^1 {{x^a}\,{{\bf{G}}_n}(x;a,c)\,{{\bf{G}}_m}(x;a,c)\,dx}  = \frac{1}{{2n + a + 1}}{\left( {\frac{{(a - c)\,n!}}{{(c + 1)\,{{(a + 1)}_n}}}} \right)^2}{\delta _{m,n}}\\ \Leftrightarrow 2c - a + 1 > 0,\,\,a \ne c\,\,\,{\rm{and}}\,\,\,a,c >  - 1,
\end{multline}
we then obtain a non-polynomial type approximation (or expansion) as follows
\begin{equation}\label{eq9.19}
f(x) = \sum\limits_{k = 0}^{n \to \infty } {\,\frac{{A_k^*}}{{B_k^*}}} \,{{\bf{G}}_k}(x;a,c)\,,
\end{equation}
in which $ A_k^* = {A_k}\,\,\,\,{\rm{and}}\,\,\,B_k^* = {B_k} $ according to the remark \ref{rem4.1.1}, i.e. as the same forms as \eqref{eq9.15} and \eqref{eq9.16}.

Once again, the orthogonal sequence \eqref{eq9.17} can be generated only if we have previously obtained the uncorrelated polynomial sequence \eqref{eq9.8}. Also, the non-polynomial approximation \eqref{eq9.19} is optimized in the sense of ordinary least squares while the polynomial type approximation \eqref{eq9.14} is optimized in the sense of least 1-variances. 
\end{corollary}

\subsubsection{A special property for $ {{\bf{P}}_n}(x;a,0,c,0) $}\label{subsubsec9.1.3}

Let us study the polynomials \eqref{eq9.8} from a differential equations point of view. In general, it is known that the hypergeometric series
\[y(x) = {}_4{F_3}\left( {\left. {\begin{array}{*{20}{c}}
{{a_1},\,\,{a_2},{a_3},\,\,{a_4}}\\
{{b_1},\,{b_2},\,\,{b_3}}
\end{array}\,} \right|\,x} \right),\]
satisfies a fourth order equation as
\begin{multline}\label{eq9.20}
{x^3}(1 - x)\,{y^{(4)}}(x) + {x^2}\Big( {{s_1} + 3 - ({s_2} + 6)x} \Big){y^{(3)}}(x) \\
+ x\Big( {{s_1} + 1 + {s_3} - \left( {3\,{s_2} + 7 + \,{s_4}} \right)x} \Big)y''(x)\\
 + \Big( {{b_1}{b_2}{b_3} - \,\left( {{s_2} + 1 + {s_4} + {s_5}} \right)x} \Big)y'(x) - {a_1}\,{a_2}\,{a_3}\,{a_4}\,y(x) = 0,
\end{multline}
in which
\begin{align*}
{s_1} &= {b_1} + {b_2} + {b_3},\\
{s_2} &= {a_1} + {a_2} + {a_3} + {a_4},\\
{s_3} &= {b_1}{b_2} + {b_1}{b_3} + {b_2}{b_3},\\
{s_4} &= {a_1}{a_2} + {a_1}{a_3} + {a_1}{a_4} + {a_2}{a_3} + {a_2}{a_4} + {a_3}{a_4},\\
{s_5} &= {a_1}{a_2}{a_3} + {a_1}{a_2}{a_4} + {a_1}{a_3}{a_4} + {a_2}{a_3}{a_4},
\end{align*}
with a general solution in the special form 
\begin{multline}\label{eq9.21}
y(x) = {c_1}\,{x^{ - {a_1}}}{}_4{F_3}\left( {\left. {\begin{array}{*{20}{c}}
{{a_1},\,\,{a_1} - {b_1} + 1,{a_1} - {b_2} + 1,\,\,{a_1} - {b_3} + 1}\\
{{a_1} - {a_2} + 1,{a_1} - {a_3} + 1,\,\,{a_1} - {a_4} + 1}
\end{array}\,} \right|\,\frac{1}{x}} \right)\\
 + {c_2}\,{x^{ - {a_2}}}{}_4{F_3}\left( {\left. {\begin{array}{*{20}{c}}
{{a_2},\,\,{a_2} - {b_1} + 1,{a_2} - {b_2} + 1,\,\,{a_2} - {b_3} + 1}\\
{{a_2} - {a_1} + 1,{a_2} - {a_3} + 1,\,\,{a_2} - {a_4} + 1}
\end{array}\,} \right|\,\frac{1}{x}} \right)\\
 + {c_3}\,{x^{ - {a_3}}}{}_4{F_3}\left( {\left. {\begin{array}{*{20}{c}}
{{a_3},\,\,{a_3} - {b_1} + 1,{a_3} - {b_2} + 1,\,\,{a_3} - {b_3} + 1}\\
{{a_3} - {a_1} + 1,{a_3} - {a_2} + 1,\,\,{a_3} - {a_4} + 1}
\end{array}\,} \right|\,\frac{1}{x}} \right)\\
 + {c_4}\,{x^{ - {a_4}}}{}_4{F_3}\left( {\left. {\begin{array}{*{20}{c}}
{{a_4},\,\,{a_4} - {b_1} + 1,{a_4} - {b_2} + 1,\,\,{a_4} - {b_3} + 1}\\
{{a_4} - {a_1} + 1,{a_4} - {a_2} + 1,\,\,{a_4} - {a_3} + 1}
\end{array}\,} \right|\,\frac{1}{x}} \right),
\end{multline}
provided that 
\[{a_1} - {a_2},\,\,{a_1} - {a_3},\,\,{a_1} - {a_4},\,\,{a_2} - {a_3},\,\,{a_2} - {a_4}\,\,{\rm{and}}\,\,\,{a_3} - {a_4} \notin\mathbb{Z}.\]
The above information can be found in mathematical sites, e.g. mathworld.wolfram.com.
According to \eqref{eq9.20} and \eqref{eq9.21}, the polynomials $ y = {{\bf{P}}_n}(x;a,0,c,0) $ satisfy the differential equation
\begin{multline}\label{eq9.22}
{x^3}(1 - x)\,{y^{(4)}}(x) + {x^2}\Big( {2a + 6 - (2a + 9)x} \Big){y^{(3)}}(x)\\
 + x\Big( {{a^2} + (c + 6)a + 7 - {c^2} + \big( {n(n + a + 1) + {c^2} + (2 - a)c - {a^2} - 11a - 18} \big)x} \Big)y''(x)\\
 + \Big( {(a + 1)(c + 1)(a - c + 1) + \big( {n(n + a + 1)(a + 3) - (a + 2)(c + 3)(a - c + 1)} \big)x} \Big)y'(x)\\
 + n(n + a + 1)\,(a - c)(c + 2)\,y(x) = 0,
\end{multline}
with the general solution
\begin{multline}\label{eq9.23}
y(x) = {c_1}\,{}_4{F_3}\left( {\left. {\begin{array}{*{20}{c}}
{ - n,\,\,n + a + 1,\,a - c,\,\,c + 2}\\
{a + 1,\,\,a - c + 1,\,\,c + 1}
\end{array}\,} \right|\,x} \right) + {c_3}\,{x^{c - a}}\\
\,\,\,\,\,\,\,\,\,\,\, + {c_2}\,{x^{ - n - a - 1}}{}_4{F_3}\left( {\left. {\begin{array}{*{20}{c}}
{n + a + 1,\,\,n + 1,\,n + c + 1,\,\,n + a - c + 1}\\
{2n + a + 2,\,\,n + c + 2,\,\,n + a - c}
\end{array}\,} \right|\,\frac{1}{x}} \right)\\
\,\,\,\,\,\,\,\,\,\,\, + {c_4}\,{x^{ - c - 2}}{}_4{F_3}\left( {\left. {\begin{array}{*{20}{c}}
{c + 2,\,\,c - a + 2,\,2c - a + 2,\,\,2}\\
{n + c + 3,\,\, - n + c - a + 2,\,\,2c - a + 3}
\end{array}\,} \right|\,\frac{1}{x}} \right),
\end{multline}
where the fixed function $ z(x) = {x^{c - a}} $ and polynomials \eqref{eq9.8} both appeared in the basis solutions. Note in \eqref{eq9.23} and subsequently \eqref{eq9.21} that
\begin{multline*}
{}_4{\bar F_3}\left( {\left. {\begin{array}{*{20}{c}}
{ - n,\,\,n + a + 1,\,a - c,\,\,c + 2}\\
{a + 1,\,\,a - c + 1,\,\,c + 1}
\end{array}\,} \right|\,x} \right) \\
= {x^n}{}_4{\bar F_3}\left( {\left. {\begin{array}{*{20}{c}}
{ - n,\,\, - n - a,\, - n - a + c,\,\, - n - c}\\
{ - 2n - a,\,\, - n - a + c + 1,\,\, - n - c - 1}
\end{array}\,} \right|\,\frac{1}{x}} \right),
\end{multline*}
and
\[{}_4{F_3}\left( {\left. {\begin{array}{*{20}{c}}
{a - c,\,\, - c,\,0,\,\,a - 2c}\\
{n + a - c + 1,\,\, - n - c,\,\,a - 2c - 1}
\end{array}\,} \right|\,\frac{1}{x}} \right) = 1.\]
Now if in \eqref{eq9.23}, 
\[{c_1} = 1,\,\,{c_3} =  - \frac{{2c - a + 1}}{{{{(c + 1)}^2}}}\frac{{n!\,\,{{(c + 1)}_n}}}{{{{(a + 1)}_n}\,{{(a + 1 - c)}_n}}} \quad\text{and}\quad {c_2} = {c_4} = 0 ,\]
the orthogonal sequence $ {{\bf{G}}_n}(x;a,c) $ in \eqref{eq9.17} will  appear that clearly satisfies the same as differential equation \eqref{eq9.22}. The important point is that for $ y = {{\bf{G}}_n}(x;a,c) $ equation \eqref{eq9.22} can be rewritten as
\begin{align}\label{eq9.24}
&\quad{x^2}\frac{{{d^2}}}{{d{x^2}}}\Big( x(1 - x)\,{\bf{G}}''_n(x;a,c) + (a + 1 - (a + 2)x)\,{\bf{G}}'_n(x;a,c)\\
 &\qquad+ \left( {n(n + a + 1) + (2c - a + 1){x^{ - 1}}} \right){{\bf{G}}_n}(x;a,c) \Big)\notag\\
 &\qquad+ (a + 3)x\frac{d}{{dx}}\Big( x(1 - x)\,{\bf{G}}''_n(x;a,c) + (a + 1 - (a + 2)x)\,{\bf{G}}'_n(x;a,c)\notag \\
 &\qquad+ \left( {n(n + a + 1) + (2c - a + 1){x^{ - 1}}} \right){\bf{G}_n}(x;a,c) \Big)\notag\\
&\qquad+(a - c)(c + 2)\Big( x(1 - x)\,{\bf{G}}''_n(x;a,c) + (a + 1 - (a + 2)x)\,{\bf{G}}'_n(x;a,c)\notag \\
&\qquad+ \left( {n(n + a + 1) + (2c - a + 1){x^{ - 1}}} \right){{\bf{G}}_n}(x;a,c) \Big)\notag\\
& = (2c - a + 1)(a - c - 1)(c + 1)\,{x^{ - 1}}\,{{\bf{G}}_n}(x;a,c)\,.\notag
\end{align}
If for simplicity we assume that
\begin{multline}\label{eq9.25}
x(1 - x)\,{\bf{G}}''_n(x;a,c) + (a + 1 - (a + 2)x)\,{\bf{G}}'_n(x;a,c)\\
 + \left( {n(n + a + 1) + (2c - a + 1){x^{ - 1}}} \right){{\bf{G}}_n}(x;a,c) = {{\bf{R}}_n}(x;a,c),
\end{multline}
then \eqref{eq9.24} changes to
\begin{multline}\label{eq9.26}
{x^2}\,{\bf{R}}''_n(x;a,c) + (a + 3)x\,{\bf{R}}'_n(x;a,c) + (a - c)(c + 2)\,{{\bf{R}}_n}(x;a,c)\\
 = (2c - a + 1)(a - c - 1)(c + 1)\,{x^{ - 1}}\,{{\bf{G}}_n}(x;a,c)\,,
\end{multline}
which is a non-homogenous second order linear equation with the analytical solution
\begin{multline}\label{eq9.27}
{{\bf{R}}_n}(x;a,c) = \frac{{(2c - a + 1)(a - c - 1)(c + 1)}}{{(2c - a + 2)}}\\
\times \left( {{x^{c - a}}\int {{x^{a - c - 2}}{{\bf{G}}_n}(x;a,c)\,dx}  - {x^{ - c - 2}}\int {{x^c}{{\bf{G}}_n}(x;a,c)\,dx} } \right),
\end{multline}
where ${y_1}(x) = {x^{c - a}}\,\,\,{\rm{and}}\,\,\,{y_2}(x) = {x^{ - c - 2}}$ are two basis solutions of equation \eqref{eq9.26}.

The two integrals in \eqref{eq9.27} can be simplified as follows
\begin{align*}
\int {{x^{a - c - 2}}{{\bf{G}}_n}(x;a,c)\,dx}  &= \int {{x^{a - c - 2}}{}_4{F_3}\left( {\left. {\begin{array}{*{20}{c}}
{ - n,\,\,n + a + 1,\,a - c,\,\,c + 2}\\
{a + 1,\,\,a - c + 1,\,\,c + 1}
\end{array}\,} \right|\,x} \right)\,dx} \\
&\quad - \frac{{2c - a + 1}}{{{{(c + 1)}^2}}}\frac{{n!\,\,{{(c + 1)}_n}}}{{{{(a + 1)}_n}\,{{(a + 1 - c)}_n}}}\int {{x^{ - 2}}\,dx} \\
 &= \frac{{{x^{a - c - 1}}}}{{a - c - 1}}{}_4{F_3}\left( {\left. {\begin{array}{*{20}{c}}
{ - n,\,\,n + a + 1,\,a - c - 1,\,\,c + 2}\\
{a + 1,\,\,a - c + 1,\,\,c + 1}
\end{array}\,} \right|\,x} \right) \\
&\quad + \frac{{2c - a + 1}}{{{{(c + 1)}^2}}}\frac{{n!\,\,{{(c + 1)}_n}}}{{{{(a + 1)}_n}\,{{(a + 1 - c)}_n}}}{x^{ - 1}},
\end{align*}
and
\begin{multline*}
\int {{x^c}{{\bf{G}}_n}(x;a,c)\,dx}  = \int {{x^c}{}_4{F_3}\left( {\left. {\begin{array}{*{20}{c}}
{ - n,\,\,n + a + 1,\,a - c,\,\,c + 2}\\
{a + 1,\,\,a - c + 1,\,\,c + 1}
\end{array}\,} \right|\,x} \right)\,dx} \\
\,\,\,\,\,\,\,\,\,\,\,\,\,\,\,\,\,\,\,\,\,\,\,\,\,\,\,\,\,\,\,\,\,\,\,\,\, - \frac{{2c - a + 1}}{{{{(c + 1)}^2}}}\frac{{n!\,\,{{(c + 1)}_n}}}{{{{(a + 1)}_n}\,{{(a + 1 - c)}_n}}}\int {{x^{2c - a}}\,dx} \\
 = \frac{{{x^{c + 1}}}}{{c + 1}}{}_3{F_2}\left( {\left. {\begin{array}{*{20}{c}}
{ - n,\,\,n + a + 1,\,a - c}\\
{a + 1,\,\,a - c + 1}
\end{array}\,} \right|\,x} \right) - \frac{1}{{{{(c + 1)}^2}}}\frac{{n!\,\,{{(c + 1)}_n}}}{{{{(a + 1)}_n}\,{{(a + 1 - c)}_n}}}{x^{2c - a + 1}}.
\end{multline*}
Therefore
\begin{multline}\label{eq9.28}
{{\bf{R}}_n}(x;a,c) = (2c - a + 1)\,{x^{ - 1}}{}_3{F_2}\left( {\left. {\begin{array}{*{20}{c}}
{ - n,\,\,n + a + 1,\,a - c - 1}\\
{a + 1,\,\,a - c + 1}
\end{array}\,} \right|\,x} \right)\\
\,\,\,\,\,\,\,\,\,\,\,\,\,\,\,\,\,\,\,\,\,\,\,\, + \frac{{(2c - a + 1)(a - c - 1)}}{{c + 1}}\frac{{n!\,\,{{(c + 1)}_n}}}{{{{(a + 1)}_n}\,{{(a + 1 - c)}_n}}}{x^{c - a - 1}}.
\end{multline}
Relations (9.27), (9.28) and (9.25) now show that in addition to (9.24), $ {{\bf{G}}_n}(x;a,c) $ satisfies the equation
\begin{align}\label{eq9.29}
& x(1 - x)\,{\bf{G}}''_n(x;a,c) + (a + 1 - (a + 2)x)\,{\bf{G}}'_n(x;a,c) + \left( {n(n + a + 1) + (2c - a + 1){x^{ - 1}}} \right){{\bf{G}}_n}(x;a,c)\\
 &\qquad = \frac{{(2c - a + 1)(a - c - 1)(c + 1)}}{{(2c - a + 2)}}\notag \\
&\qquad\qquad \times \left( {{x^{c - a}}\int {{x^{a - c - 2}}{{\bf{G}}_n}(x;a,c)\,dx}  - {x^{ - c - 2}}\int {{x^c}{{\bf{G}}_n}(x;a,c)\,dx} } \right)\notag \\
& \qquad= (2c - a + 1)\,{x^{ - 1}}{}_3{F_2}\left( {\left. {\begin{array}{*{20}{c}}
{ - n,\,\,n + a + 1,\,a - c - 1}\\
{a + 1,\,\,a - c + 1}
\end{array}\,} \right|\,x} \right)\notag \\
&\qquad\qquad + \frac{{(2c - a + 1)(a - c - 1)}}{{c + 1}}\frac{{n!\,\,{{(c + 1)}_n}}}{{{{(a + 1)}_n}\,{{(a + 1 - c)}_n}}}{x^{c - a - 1}}.\notag
\end{align}
Let us write equation \eqref{eq9.29} in a self-adjoint form as
\begin{multline}\label{eq9.30}
{\left( {{x^{a + 1}}(1 - x)\,{\bf{G}}'_n(x;a,c)} \right)^\prime } + \left( {n(n + a + 1){x^a} + (2c - a + 1){x^{a - 1}}} \right){{\bf{G}}_n}(x;a,c)\\
 = {x^a}\,{{\bf{R}}_n}(x;a,c),
\end{multline}
and for the index $ m $ as
\begin{multline}\label{eq9.31}
{\left( {{x^{a + 1}}(1 - x)\,{\bf{G}}'_m(x;a,c)} \right)^\prime } + \left( {m(m + a + 1){x^a} + (2c - a + 1){x^{a - 1}}} \right){{\bf{G}}_m}(x;a,c) \\
= {x^a}\,{{\bf{R}}_m}(x;a,c).
\end{multline}
On multiplying by $ {{\bf{G}}_m}(x;a,c) $ in \eqref{eq9.30} and $ {{\bf{G}}_n}(x;a,c) $ in \eqref{eq9.31} and subtracting we get
\begin{multline}\label{eq9.32}
\Big[{x^{a + 1}}(1 - x)\left( {\bf{G}}_m(x;a,c){\bf{G}}'_n(x;a,c) - {{\bf{G}}_n}(x;a,c){\bf{G}}'_m(x;a,c) \right) \Big]_0^1\\
 + \big( {n(n + a + 1) - m(m + a + 1)} \big)\int_0^1 {{x^a}\,{{\bf{G}}_n}(x;a,c){{\bf{G}}_m}(x;a,c)\,dx} \\
 = \int_0^1 {{x^a}\,\Big( {{{\bf{G}}_m}(x;a,c){{\bf{R}}_n}(x;a,c) - {{\bf{G}}_n}(x;a,c){{\bf{R}}_m}(x;a,c)} \Big)\,dx} \,.
\end{multline}
Since in \eqref{eq9.32},
\begin{multline*}
\int_0^1 {{x^a}\,{{\bf{G}}_n}(x;a,c){{\bf{G}}_m}(x;a,c)\,dx}  = \\
\Big[ {{x^{a + 1}}(1 - x)\left( {{{\bf{G}}_m}(x;a,c){\bf{G}}'_n(x;a,c) - {{\bf{G}}_n}(x;a,c){\bf{G}}'_m(x;a,c)\,} \right)} \Big]_0^1 = 0 \Leftrightarrow n \ne m,
\end{multline*}
if we take
\[{\bf{A}}(n,m) = \int_0^1 {{x^a}\,{{\bf{G}}_n}(x;a,c)\,{{\bf{R}}_m}(x;a,c)\,dx} \,,\]
relation \eqref{eq9.32} shows that for any $ n\ne m $ we finally have
\[{\bf{A}}(m,n) = {\bf{A}}(n,m).\]

\section{p-uncorrelated vectors with respect to a fixed vector}\label{sec10}
The concept of p-uncorrelatedness can also be employed in vector spaces. Let $ {\vec A_m} = ({a_1},{a_2},...,{a_m}) $ and $ {\vec B_m} = ({b_1},{b_2},...,{b_m}) $ be two arbitrary vectors and $ {\vec I_m} = (1,1,...,1) $ denote a unit vector. Also let $ {\vec Z_m} = ({z_1},{z_2},...,{z_m}) $ be a fixed and predetermined vector. Recalling the definition of the inner product of two vectors as
\[{\vec A_m}.{\vec B_m} = \sum\limits_{k = 1}^m {{a_k}{b_k}} ,\]
it is not difficult to verify that
\begin{align}\label{eq10.1}
&\left( {{{\vec A}_m} - (1 - \sqrt {1 - p} )\frac{{{{\vec A}_m}.{{\vec Z}_m}}}{{{{\vec Z}_m}.{{\vec Z}_m}}}{{\vec Z}_m}} \right).\left( {{{\vec B}_m} - (1 - \sqrt {1 - p} )\frac{{{{\vec B}_m}.{{\vec Z}_m}}}{{{{\vec Z}_m}.{{\vec Z}_m}}}{{\vec Z}_m}} \right) \\
&\qquad\qquad\qquad\qquad\qquad= {\vec A_m}.{\vec B_m} - p\frac{{({{\vec A}_m}.{{\vec Z}_m})({{\vec B}_m}.{{\vec Z}_m})}}{{{{\vec Z}_m}.{{\vec Z}_m}}}.\notag
\end{align}
For instance, if $ {\vec Z_m} = {\vec I_m} $, then
\begin{align*}
&\left( {{{\vec A}_m} - (1 - \sqrt {1 - p} )\frac{{{{\vec A}_m}.{{\vec I}_m}}}{m}{{\vec I}_m}} \right).\left( {{{\vec B}_m} - (1 - \sqrt {1 - p} )\frac{{{{\vec B}_m}.{{\vec I}_m}}}{m}{{\vec I}_m}} \right) \\
&\qquad\qquad\qquad\qquad\qquad= {\vec A_m}.{\vec B_m} - \frac{p}{m}({\vec A_m}.{\vec I_m})({\vec B_m}.{\vec I_m}),
\end{align*}
where $ {\vec I_m}.{\vec I_m} = m $ and $ p\in[0,1] $.

Relation \eqref{eq10.1} shows that the two vectors $ {\vec A_m}\,\,\,{\rm{and}}\,\,\,{\vec B_m} $ are p-uncorrelated with respect to the fixed vector $ {\vec Z_m} $  if
\[p\,({\vec A_m}.{\vec Z_m})({\vec B_m}.{\vec Z_m}) = ({\vec A_m}.{\vec B_m})({\vec Z_m}.{\vec Z_m})\,.\]
Also, the notions of p-covariance and p-variance can be defined for these two vectors (with respect to $ {\vec Z_m} $) as follows
\begin{equation}\label{eq10.2}
{{\mathop{\rm cov}} _p}({\vec A_m},{\vec B_m};\,{\vec Z_m}) = \frac{1}{m}\left( {{{\vec A}_m}.{{\vec B}_m} - p\frac{{({{\vec A}_m}.{{\vec Z}_m})({{\vec B}_m}.{{\vec Z}_m})}}{{{{\vec Z}_m}.{{\vec Z}_m}}}} \right),
\end{equation}
and
\begin{equation}\label{eq10.3}
{{\mathop{\rm var}} _p}({\vec A_m};\,{\vec Z_m}) = \frac{1}{m}\left( {{{\vec A}_m}.{{\vec A}_m} - p\frac{{{{({{\vec A}_m}.{{\vec Z}_m})}^2}}}{{{{\vec Z}_m}.{{\vec Z}_m}}}} \right) \ge 0.
\end{equation}
Referring to the basic representation \eqref{eq3.12} and definitions \eqref{eq10.2} and \eqref{eq10.3}, we can establish a set of p-uncorrelated vectors in terms of the parameter $ p\in[0,1] $ if and only if $ m=n+1 $ and the finite set of initial vectors are linearly independent. Under such conditions we have
\begin{multline}\label{eq10.4}
\Delta _{n - 1}^{(p)}\left( {\{ {{\vec V}_{k,m}}\} _{k = 0}^{n - 1};{{\vec Z}_m}} \right){\vec X_{n,m}}(p) \\
= \begin{vmatrix}
{{{{\mathop{\rm var}} }_p}\,({{\vec V}_{0,m}};{{\vec Z}_m})} &  {{{{\mathop{\rm cov}} }_p}\,({{\vec V}_{0,m}},{{\vec V}_{1,m}};{{\vec Z}_m})} & \cdots & {{{{\mathop{\rm cov}} }_p}\,({{\vec V}_{0,m}},{{\vec V}_{n,m}};{{\vec Z}_m})} \\
{{{{\mathop{\rm cov}} }_p}\,({{\vec V}_{1,m}},{{\vec V}_{0,m}};{{\vec Z}_m})} & {{{{\mathop{\rm var}} }_p}\,({{\vec V}_{1,m}};{{\vec Z}_m})} & \cdots & {{{{\mathop{\rm cov}} }_p}\,({{\vec V}_{1,m}},{{\vec V}_{n,m}};{{\vec Z}_m})} \\
 \vdots &  \vdots  & \vdots & \vdots   \\
{{{{\mathop{\rm cov}} }_p}\,({{\vec V}_{n - 1,m}},{{\vec V}_{0,m}};{{\vec Z}_m})} & {{{{\mathop{\rm cov}} }_p}\,({{\vec V}_{n - 1,m}},{{\vec V}_{1,m}};{{\vec Z}_m})} & \cdots & {{{{\mathop{\rm cov}} }_p}\,({{\vec V}_{n - 1,m}},{{\vec V}_{n,m}};{{\vec Z}_m})}\\
{{\vec V}_{0,m}} & {{\vec V}_{1,m}} & \cdots  & {{\vec V}_{n,m}}
\end{vmatrix},
\end{multline}
where
\begin{multline}\label{eq10.5}
\Delta _{n - 1}^{(p)}\left( {\{ {{\vec V}_{k,m}}\} _{k = 0}^{n - 1};{{\vec Z}_m}} \right) \\
=\begin{vmatrix}
{{{{\mathop{\rm var}} }_p}\,({{\vec V}_{0,m}};{{\vec Z}_m})} &  {{{{\mathop{\rm cov}} }_p}\,({{\vec V}_{0,m}},{{\vec V}_{1,m}};{{\vec Z}_m})} & \cdots & {{{{\mathop{\rm cov}} }_p}\,({{\vec V}_{0,m}},{{\vec V}_{n,m}};{{\vec Z}_m})} \\
{{{{\mathop{\rm cov}} }_p}\,({{\vec V}_{1,m}},{{\vec V}_{0,m}};{{\vec Z}_m})} & {{{{\mathop{\rm var}} }_p}\,({{\vec V}_{1,m}};{{\vec Z}_m})} & \cdots & {{{{\mathop{\rm cov}} }_p}\,({{\vec V}_{1,m}},{{\vec V}_{n,m}};{{\vec Z}_m})} \\
 \vdots &  \vdots  & \vdots & \vdots   \\
{{{{\mathop{\rm cov}} }_p}\,({{\vec V}_{n - 1,m}},{{\vec V}_{0,m}};{{\vec Z}_m})} & {{{{\mathop{\rm cov}} }_p}\,({{\vec V}_{n - 1,m}},{{\vec V}_{1,m}};{{\vec Z}_m})} & \cdots & {{{{\mathop{\rm var}} }_p}\,({{\vec V}_{n - 1,m}};{{\vec Z}_m})}
\end{vmatrix},
\end{multline}
and $ \Delta _{ - 1}^{(p)}(.) = 1 $. To better clarify the issue, here we consider a specific example.
\begin{example}\label{ex9.1}
For $ m=3 $, given three initial orthogonal vectors
\[{\vec V_{0,3}}(1,0,0),\,\,{\vec V_{1,3}}(0,1,0),\,\,{\vec V_{2,3}}(0,0,1),\]
together with the fixed vector $ {\vec Z_3}(1,2,3) $. Substituting them into \eqref{eq10.4} and \eqref{eq10.5} eventually yields
\begin{align}
{{\vec X}_{0,3}}(p) &= (1,0,0),\nonumber\\
{{\vec X}_{1,3}}(p) &= (\frac{{2p}}{{14 - p}},1,0),\label{eq10.6}\\
{{\vec X}_{2,3}}(p) &= (\frac{{3p}}{{14 - 5p}},\frac{{6p}}{{14 - 5p}},1),\nonumber
\end{align}
which satisfy the conditions
\[{{\mathop{\rm cov}} _p}\Big({\vec X_{0,3}}(p),{\vec X_{1,3}}(p);{\vec Z_3}\Big) = {{\mathop{\rm cov}} _p}\Big({\vec X_{0,3}}(p),{\vec X_{2,3}}(p);{\vec Z_3}\Big) = {{\mathop{\rm cov}} _p}\Big({\vec X_{1,3}}(p),{\vec X_{2,3}}(p);{\vec Z_3}\Big) = 0.\]
According to theorem \ref{thm4.1}, every arbitrary vector of dimension 3 can be expanded in terms of the above vectors so that we have
\begin{multline}\label{eq10.7}
\vec A = (a,b,c) = \frac{{{{{\mathop{\rm cov}} }_p}\,({{\vec X}_{0,3}}(p),\vec A;{{\vec Z}_3})}}{{{{{\mathop{\rm var}} }_p}\,({{\vec X}_{0,3}}(p);{{\vec Z}_3})}}(1,0,0) + \frac{{{{{\mathop{\rm cov}} }_p}\,({{\vec X}_{1,3}}(p),\vec A;{{\vec Z}_3})}}{{{{{\mathop{\rm var}} }_p}\,({{\vec X}_{1,3}}(p);{{\vec Z}_3})}}(\frac{{2p}}{{14 - p}},1,0)\\
\,\,\,\,\,\,\,\,\,\,\,\,\,\,\,\,\,\,\,\,\,\,\,\,\,\, + \frac{{{{{\mathop{\rm cov}} }_p}\,({{\vec X}_{2,3}}(p),\vec A;{{\vec Z}_3})}}{{{{{\mathop{\rm var}} }_p}\,({{\vec X}_{2,3}}(p);{{\vec Z}_3})}}(\frac{{3p}}{{14 - 5p}},\frac{{6p}}{{14 - 5p}},1),
\end{multline}
in which
\begin{align}\label{eq10.8}
& \frac{{{{{\mathop{\rm cov}} }_p}\,({{\vec X}_{0,3}}(p),\vec A;{{\vec Z}_3})}}{{{{{\mathop{\rm var}} }_p}\,({{\vec X}_{0,3}}(p);{{\vec Z}_3})}} = a - \frac{{2p}}{{14 - p}}b - \frac{{3p}}{{14 - p}}c,\nonumber\\
& \frac{{{{{\mathop{\rm cov}} }_p}\,({{\vec X}_{1,3}}(p),\vec A;{{\vec Z}_3})}}{{{{{\mathop{\rm var}} }_p}\,({{\vec X}_{1,3}}(p);{{\vec Z}_3})}} = b - \frac{{6p}}{{14 - 5p}}c,\\
& \frac{{{{{\mathop{\rm cov}} }_p}\,({{\vec X}_{2,3}}(p),\vec A;{{\vec Z}_3})}}{{{{{\mathop{\rm var}} }_p}\,({{\vec X}_{2,3}}(p);{{\vec Z}_3})}} = c.\nonumber
\end{align}
For $ p=0 $, the finite expansion \eqref{eq10.7} would reduce to an ordinary orthogonal expansion, while there is an important point for the case $ p=1 $. We observe that replacing $ p=1 $ in the last item of \eqref{eq10.6} gives
\[{\vec X_{2,3}}(1) = \frac{1}{3}(1,2,3) = \frac{1}{3}{\vec Z_3}.\]
Hence, in \eqref{eq10.7} (and subsequently \eqref{eq10.8}),
\[{{\mathop{\rm cov}} _1}\,({\vec X_{2,3}}(1),\vec A;{\vec Z_3}) = {{\mathop{\rm cov}} _1}\,({\vec X_{2,3}}(1),\vec A;3{\vec X_{2,3}}(1)) = 0,\]
and
\[{{\mathop{\rm var}} _1}\,({\vec X_{2,3}}(1);{\vec Z_3}) = {{\mathop{\rm var}} _1}\,(\frac{1}{3}{\vec Z_3};{\vec Z_3}) = 0,\]
which shows that the expansion \eqref{eq10.7} is not valid for the sole case $ p=1 $, although it is  valid for any other $ p\in[0,1) $. This is one of the reasons why we have considered this theory for every arbitrary parameter $ p\in[0,1] $. For a better analysis, see figure \ref{fig1} again. In general, we conjecture that
\begin{equation}\label{eq10.9}
{\vec X_{n,m}}(1) = \frac{{\det \,({{\vec V}_{0,m}},{{\vec V}_{1,m}},...,{{\vec V}_{n - 1,m}},{{\vec V}_{n,m}})}}{{\det \,({{\vec V}_{0,m}},{{\vec V}_{1,m}},...,{{\vec V}_{n - 1,m}},{{\vec Z}_m})}}\,\,{\vec Z_m}\,\,\,\,\,\,\,\,\,\,\,(m = n + 1),
\end{equation}
as for the above-mentioned example 
\[\det \,({\vec V_{0,3}},{\vec V_{1,3}},{\vec V_{2,3}}) = \left| {\,\begin{array}{*{20}{c}}
1&0&0\\
0&1&0\\
0&0&1
\end{array}\,} \right| = 1\,\,\,\,\,\,\,{\rm{and}}\,\,\,\,\,\,\,\det \,({\vec V_{0,3}},{\vec V_{1,3}},{\vec Z_3}) = \left| {\,\begin{array}{*{20}{c}}
1&0&1\\
0&1&2\\
0&0&3
\end{array}\,} \right| = 3.\]
We symbolically examined such a conjecture \eqref{eq10.9} for the particular cases $ n=2,3 $ and the results were true.

Now that the explicit forms of the p-uncorrelated vectors \eqref{eq10.6} are available, we can make three parametric orthogonal vectors based on them as follows
\begin{align}\label{eq10.10}
{{\vec V}_{0,3}}(p) &= {{\vec X}_{0,3}}(p) - (1 - \sqrt {1 - p} )\frac{{{{\vec X}_{0,3}}(p).{{\vec Z}_3}}}{{{{\vec Z}_3}.{{\vec Z}_3}}}{{\vec Z}_3} \\
&= \frac{1}{{14}}(13 + \sqrt {1 - p} , - 2 + 2\sqrt {1 - p} , - 3 + 3\sqrt {1 - p} ),\notag\\
{{\vec V}_{1,3}}(p)& = \frac{1}{{14 - p}}(2p - 2 + 2\sqrt {1 - p} , - p + 10 + 4\sqrt {1 - p} , - 6 + 6\sqrt {1 - p} ),\nonumber\\
{{\vec V}_{2,3}}(p) &= \frac{1}{{14 - 5p}}(3p - 3 + 3\sqrt {1 - p} ,6p - 6 + 6\sqrt {1 - p} , - 5p + 5 + 9\sqrt {1 - p} ).\nonumber
\end{align}
Note in the above vectors that
\[{\vec V_{0,3}}(0) = {\vec V_{0,3}},\,\,\,{\vec V_{1,3}}(0) = {\vec V_{1,3}}\,\,\,\,{\rm{and}}\,\,\,\,\,{\vec V_{2,3}}(0) = {\vec V_{2,3}},\]
while for $ p=1 $ we have
\[{\vec V_{0,3}}(1) = \frac{1}{{14}}(13, - 2, - 3),\,\,\,\,\,\,{\vec V_{1,3}}(1) = \frac{3}{{13}}(0,3, - 2)\,\,\,\,\,\,\text{and}\,\,\,\,\,{\vec V_{2,3}}(1) = (0,0,0),\]
which confirms that it is not valid for choosing in the orthogonal vectors \eqref{eq10.10}.
\end{example}

\section{An upper bound for 1-covariances}\label{sec11}
As the optimized case is $ p=1 $, in this part we are going to obtain an upper bound for $ {{\mathop{\rm cov}} _1}\,(X,Y;Z) $.

Let $ {m_X},{m_Y},{M_X} $ and $ M_{Y} $ be real numbers such that
\begin{equation}\label{eq11.1}
{m_X}\,Z \le X \le {M_X}\,Z\quad \text{and} \quad {m_Y}\,Z \le Y \le {M_Y}\,Z.
\end{equation}
It can be verified that the following identity holds true
\begin{multline}\label{eq11.2}
{{\mathop{\rm var}} _1}(X;Z)= \\
\frac{1}{{E({Z^2})}}\Big( {{M_X}E({Z^2}) - E(XZ)} \Big)\Big( {E(XZ) - {m_X}E({Z^2})} \Big) - E\Big( {\left( {{M_X}Z - X} \right)\left( {X - {m_X}Z} \right)} \Big).
\end{multline}
Noting the conditions \eqref{eq11.1} and this fact that
\[E\Big( {\left( {{M_X}\,Z - X} \right)\left( {X - {m_X}\,Z} \right)} \Big)\, \ge 0\,,\]
equality \eqref{eq11.2} leads to the inequality
\begin{multline}\label{eq11.3}
{{\mathop{\rm var}} _1}(X;Z) \le \frac{1}{{E({Z^2})}}\Big( {{M_X}E({Z^2}) - E(XZ)} \Big)\Big( {E(XZ) - {m_X}E({Z^2})} \Big)\\
 \le \frac{1}{{4E({Z^2})}}{\Big( {{M_X}E({Z^2}) - {m_X}E({Z^2})} \Big)^2} = \frac{{E({Z^2})}}{4}{\left( {{M_X} - {m_X}} \right)^2}.
\end{multline}
On the other side, following \eqref{eq11.3} in the well-known inequality
\[{\mathop{\rm cov}} _1^2(X,Y;Z) \le {{\mathop{\rm var}} _1}\,(X;Z)\,\,{\rm var}_{1}(Y;Z),\]
gives
\begin{align}\label{eq11.4}
{\mathop{\rm cov}} _1^2(X,Y;Z) & \le \frac{1}{{{E^2}({Z^2})}}\Big( {{M_X}E({Z^2}) - E(XZ)} \Big)\Big( {E(XZ) - {m_X}E({Z^2})} \Big)\\
 &\times \,\,\Big( {{M_Y}E({Z^2}) - E(YZ)} \Big)\Big( {E(YZ) - {m_Y}E({Z^2})} \Big)\notag\\
 &\le \frac{{{E^2}({Z^2})}}{{16}}{\left( {{M_X} - {m_X}} \right)^2}{\left( {{M_Y} - {m_Y}} \right)^2}.\notag
\end{align}
One of the direct consequences of \eqref{eq11.4} is that
\begin{equation}\label{eq11.5}
\left| {\,{{{\mathop{\rm cov}} }_1}(X,Y;Z)\,} \right| \le \frac{{E({Z^2})}}{4}\left( {{M_X} - {m_X}} \right)\left( {{M_Y} - {m_Y}} \right),
\end{equation}
where the constant $ 1/4 $ in \eqref{eq11.5} is the best possible number in the sense that it cannot be replaced by a smaller quantity.

As a particular case, if in \eqref{eq11.5} we take
\[{P_r}(X = x) = \frac{{w(x)}}{{\int_\alpha ^\beta  {w(x)\,dx} }},\quad X = f(x),\,\,Y = g(x)\,\,\,{\rm{and}}\,\,\,\,Z = z(x) = 1,\]
it will reduce to the weighted Gr\"{u}ss inequality \cite{ref19, ref20}
\[\left| {\frac{{\int_\alpha ^\beta  {w(x)f(x)g(x)\,dx} }}{{\int_\alpha ^\beta  {w(x)\,dx} }} - \frac{{\left( {\int_\alpha ^\beta  {w(x)f(x)\,dx} } \right)\,\left( {\int_\alpha ^\beta  {w(x)g(x)\,dx} } \right)}}{{{{(\int_\alpha ^\beta  {w(x)\,dx} )}^2}}}\,} \right| \le \frac{1}{4}({M_f} - {m_f})({M_g} - {m_g}),\]
in which
\[{m_f} \le f(x) \le {M_f}\quad \text{and}\quad {m_g} \le g(x) \le {M_g} \quad\text{for all}\quad x \in [\alpha ,\beta ].\]

\section{An approximation for p-variances using quadrature rules}\label{sec12}

We begin this section with a general $ n $-point (weighted) quadrature rule as 
\begin{equation}\label{eq12.1}
\int_{\,a}^{\,b} {w(x)\,f(x)\,dx}  = \sum\limits_{k = 1}^n {\,{w_k}\,f({x_k})}  + {R_n}[f],
\end{equation}
in which $ w(x) $ is positive on $ [a,b] $, $ \{ {x_k}\} _{k = 1}^n $ and $ \{ {w_k}\} _{k = 1}^n $ are respectively nodes and weight coefficients and $ {R_n}[f] $ is the corresponding error, see e.g. \cite{ref14, ref16}.

If $ {\mathbf{\Pi}}_{d} $ denotes the set of all algebraic polynomials of degree at most $ d $, the rule \eqref{eq12.1} has degree of exactness $ d $ if for every $ p\in{\mathbf{\Pi}}_{d} $ we have $ {R_n}[p] = 0 $. Moreover, if $ {R_n}[p] \ne 0 $ for some $ p\in{\mathbf{\Pi}}_{d+1} $, formula \eqref{eq12.1} has precise degree of exactness $ d $.

It is well known that for given $ n $ mutually different nodes $ \{ {x_k}\} _{k = 1}^n $ we can always achieve a degree of exactness $ d=n-1 $ by interpolating at these nodes and integrating the interpolated polynomial instead of $ f $. Namely, taking the node polynomial
\[{N_n}(x) = \prod\limits_{k = 1}^n {(x - {x_k})}, \]
and integrating from the Lagrange interpolation formula
\begin{equation}\label{eq12.2}
f(x) = \sum\limits_{k = 1}^n {f({x_k})\,L(x\,;\,{x_k})}  + \frac{1}{{n!}}{f^{(n)}}({\xi _x})\,{N_n}(x)\,,
\end{equation}
where
\[L(x\,;\,{x_k}) = \frac{{{N_n}(x)}}{{{{N'}_n}({x_k})(x - {x_k})}}\,,\]
we obtain \eqref{eq12.1}, with
\[{w_k} = \frac{1}{{{{N'}_n}({x_k})}}\,\,\int_{\,a}^{\,b} {\frac{{{N_n}(x)\,w(x)}}{{x - {x_k}}}\,dx} \,,\]
and
\begin{equation}\label{eq12.3}
{R_n}[f] = \,\frac{1}{{n!}}\int_{\,a}^{\,b} {{f^{(n)}}({\xi _x})\,{N_n}(x)\,w(x)\,dx} \,.
\end{equation}
It is clear in \eqref{eq12.3} that if $ f\in{\mathbf{\Pi}}_{n-1} $ then $ {R_n}[f]=0 $.

To approximate the p-variance values
\begin{equation}\label{eq12.4}
{{\mathop{\rm var}} _p}(f(x);z(x)) = \frac{{\int_{\,a}^b {w(x)\,{f^2}(x)\,dx} }}{{\int_{\,a}^b {w(x)\,dx} }} - p\frac{{{{(\int_{\,a}^b {w(x)\,f(x)\,z(x)\,dx} )}^2}\,}}{{\,\int_{\,a}^b {w(x)\,dx} \,\int_{\,a}^b {w(x)\,{z^2}(x)\,dx} }},
\end{equation}
we can similarly follow the above-mentioned approach.

For this purpose, it is just sufficient to apply $ {{\mathop{\rm cov}} _p}\left( {f(x),\,(.);z(x)} \right) $ on both sides of \eqref{eq12.2} to get
\begin{align*}
&{{\mathop{\rm var}} _p}\left( {f(x);z(x)} \right) = {{\mathop{\rm cov}} _p}\left( {f(x),\sum\limits_{k = 1}^n {f({x_k})\,L(x\,;\,{x_k})} ;z(x)} \right) \\
&\qquad\qquad\qquad\qquad+ {{\mathop{\rm cov}} _p}\left( {f(x),\frac{1}{{n!}}{f^{(n)}}({\xi _x})\,{N_n}(x);z(x)} \right)\\
& = \sum\limits_{k = 1}^n {f({x_k}){{{\mathop{\rm cov}} }_p}\left( {f(x),L(x\,;\,{x_k});z(x)} \right)\,}  + \frac{1}{{n!}}{{\mathop{\rm cov}} _p}\left( {f(x),{f^{(n)}}({\xi _x})\,{N_n}(x);z(x)} \right)\,,
\end{align*}
which gives the right side of quadrature formula \eqref{eq12.1} with
\begin{equation}\label{eq12.5}
{w_k} = {{\mathop{\rm cov}} _p}\left( {f(x),L(x\,;\,{x_k});z(x)} \right)\,,
\end{equation}
and
\begin{equation}\label{eq12.6}
{R_n}[f] = \frac{1}{{n!}}\,{{\mathop{\rm cov}} _p}\left( {f(x),{f^{(n)}}({\xi _x})\,{N_n}(x);z(x)} \right)\,.
\end{equation}
We observe in \eqref{eq12.6} that for any $ f\in{\mathbf{\Pi}}_{n-1} $ we automatically have $ {R_n}[f]=0 $.

Of course, another method to obtain the coefficients \eqref{eq12.5} is to use the undetermined coefficient method via solving the linear system
\begin{equation}\label{eq12.7}
\left[ {\begin{array}{*{20}{c}}
{\begin{array}{*{20}{c}}
1\\
{{x_1}}\\
 \vdots \\
{x_1^{n - 1}}
\end{array}}&{\begin{array}{*{20}{c}}
1\\
{{x_2}}\\
 \vdots \\
{x_2^{n - 1}}
\end{array}}&{\begin{array}{*{20}{c}}
 \cdots \\
 \cdots \\
 \vdots \\
 \cdots 
\end{array}}&{\begin{array}{*{20}{c}}
1\\
{{x_n}}\\
 \vdots \\
{x_n^{n - 1}}
\end{array}}
\end{array}} \right]\left[ {\begin{array}{*{20}{c}}
{{w_1}}\\
{{w_2}}\\
 \vdots \\
{{w_n}}
\end{array}} \right] = \left[ {\begin{array}{*{20}{c}}
{{{{\mathop{\rm var}} }_p}\,(1;z(x))}\\
{{{{\mathop{\rm var}} }_p}\,(x;z(x))}\\
 \vdots \\
{{{{\mathop{\rm var}} }_p}\,({x^{n - 1}};z(x))}
\end{array}} \right].
\end{equation}
For example, the two-point approximate formula for evaluating \eqref{eq12.4} via \eqref{eq12.7} is
\begin{align*}
{{\mathop{\rm var}} _p}(f(x);z(x)) & \cong \frac{{{x_2}{{{\mathop{\rm var}} }_p}(1;z(x)) - {{{\mathop{\rm var}} }_p}(x;z(x))}}{{{x_2} - {x_1}}}\,f({x_1}) \\
&\,\,\,- \frac{{{x_1}{{{\mathop{\rm var}} }_p}(1;z(x)) - {{{\mathop{\rm var}} }_p}(x;z(x))}}{{{x_2} - {x_1}}}\,f({x_2})\,.
\end{align*}

\section{On improving the approximate solutions of over-determined systems}\label{sec13}

For $ n>m $, consider the linear system of equations
\begin{equation}\label{eq13.1}
\sum\limits_{j = 1}^m {{a_{i,j}}\,{x_j}}  = {b_i}\,\,\,\,\,\,(i = 1,2,...,n),
\end{equation}
whose matrix representation is $ {A_{n \times m}}{X_{m \times 1}} = {B_{n \times 1}} $ where
\[A = \left[ {\begin{array}{*{20}{c}}
{\begin{array}{*{20}{c}}
{{a_{11}}}\\
{{a_{21}}}\\
 \vdots \\
{{a_{n1}}}
\end{array}}&{\begin{array}{*{20}{c}}
{{a_{12}}}\\
{{a_{22}}}\\
 \vdots \\
{{a_{n1}}}
\end{array}}&{\begin{array}{*{20}{c}}
 \cdots \\
 \cdots \\
 \vdots \\
 \cdots 
\end{array}}&{\begin{array}{*{20}{c}}
{{a_{1m}}}\\
{{a_{2m}}}\\
 \vdots \\
{{a_{nm}}}
\end{array}}
\end{array}} \right],\,\,\,\,X = \left[ {\begin{array}{*{20}{c}}
{{x_1}}\\
{{x_2}}\\
 \vdots \\
{{x_m}}
\end{array}} \right]\,\,\,\,\,{\rm{and}}\,\,\,\,\,B = \,\left[ {\begin{array}{*{20}{c}}
{{b_1}}\\
{{b_2}}\\
 \vdots \\
{{b_n}}
\end{array}} \right].\]
As we mentioned in the introduction part, the linear system \eqref{eq13.1} is called an over-determined system since the number of equations is more than the number of unknowns.

Such systems usually have no exact solution and the goal is instead to find an approximate solution for the unknowns $ \{ {x_j}\} _{j = 1}^m $ which fit equations in the sense of solving the problem
\begin{equation}\label{eq13.2}
\mathop {\min }\limits_{\{ {x_j}\} } {E_{m,n}}({x_1},...,{x_m}) = \mathop {\min }\limits_{\{ {x_j}\} } \,\sum\limits_{i = 1}^n \Big(\sum\limits_{j = 1}^m {{a_{i,j}}\,{x_j}}  - {b_i}\Big)^2 .
\end{equation}
It has been proved \cite{ref2} that the minimization problem (13.2) has a unique vector solution provided that the $ m $ columns of the matrix $ A $ are linearly independent, given by solving the normal equations
\[{A^T}A\,\tilde X\, = {A^T}b,\]
where $ A^{T} $ indicates the matrix transpose of $ A $ and $ \tilde{X} $ is the approximate solution of the least squares type expressed by
\[\tilde X = {({A^T}A)^{ - 1}}{A^T}b.\]
Instead of considering the problem \eqref{eq12.2}, we would like now to consider the minimization problem
\begin{equation}\label{eq13.3}
\mathop {\min }\limits_{\{ {x_j}\} } {V_{m,n}}(\left. {{x_1},...,{x_m}} \right|\{ {z_i}\} _{i = 1}^n) = \mathop {\min }\limits_{\{ {x_j}\} } \,\,\sum\limits_{i = 1}^n \Big(\sum\limits_{j = 1}^m {{a_{i,j}}\,{x_j}}  - {b_i}\Big)^2  - \frac{p}{{\sum\limits_{i = 1}^n {z_i^2} }}{\left( {\sum\limits_{i = 1}^n {{z_i}\,(\sum\limits_{j = 1}^m {{a_{i,j}}\,{x_j}}  - {b_i})} } \right)^2},
\end{equation}
based on the fixed vector $ Z_{1 \times n}^T = [{z_1},{z_2},...,{z_n}] $, where the quantity \eqref{eq13.3} is clearly smaller than the quantity \eqref{eq13.2} for any arbitrary selection of $ \{ {z_i}\} _{i = 1}^n $. In this direction, 
\begin{multline*}
\frac{{\partial {V_{m,n}}(\left. {{x_1},...,{x_m}} \right|\{ {z_i}\} _{i = 1}^n)}}{{\partial {x_k}}} \\
= 2\sum\limits_{i = 1}^n {{a_{i,k}}(\sum\limits_{j = 1}^m {{a_{i,j}}\,{x_j}}  - {b_i})}  - \frac{{2p}}{{\sum\limits_{i = 1}^n {z_i^2} }}\left( {\sum\limits_{i = 1}^n {{a_{i,k}}{z_i}} } \right)\left( {\sum\limits_{i = 1}^n {{z_i}\,(\sum\limits_{j = 1}^m {{a_{i,j}}\,{x_j}}  - {b_i})} } \right) = 0,
\end{multline*}
leads to the linear system
\begin{multline}\label{eq13.4}
\sum\limits_{j = 1}^m {\left( {\sum\limits_{i = 1}^n {{a_{i,k}}{a_{i,j}}}  - p\frac{{\sum\limits_{i = 1}^n {{a_{i,k}}{z_i}} \sum\limits_{i = 1}^n {{a_{i,j}}\,{z_i}} }}{{\sum\limits_{i = 1}^n {z_i^2} }}} \right)\,{x_j}}\\
  = \sum\limits_{i = 1}^n {{a_{i,k}}{b_i}}  - p\frac{{\sum\limits_{i = 1}^n {{a_{i,k}}{z_i}} \sum\limits_{i = 1}^n {{b_i}\,{z_i}} }}{{\sum\limits_{i = 1}^n {z_i^2} }}\qquad(k = 1,2,...,m),
\end{multline}
which can also be represented as the matrix form
\[\left( {{A^T}A - \frac{p}{{{Z^T}Z}}{A^T}Z\,{Z^T}A} \right)\,{\tilde X_{p,Z}} = {A^T}B - \frac{p}{{{Z^T}Z}}{A^T}Z\,{Z^T}B,\]
with the solution
\begin{equation}\label{eq13.5}
{\tilde X_{p,Z}} = {\left( {{A^T}A - \frac{p}{{{Z^T}Z}}{A^T}Z\,{Z^T}A} \right)^{ - 1}}\left( {{A^T}B - \frac{p}{{{Z^T}Z}}{A^T}Z\,{Z^T}B} \right).
\end{equation}
A simple case of the approximate solution \eqref{eq13.5} is when $ Z_{1 \times n}^T = {I_n} = [1,1,...,1] $ and $ p=1 $, i.e. an ordinary least variance problem. In this case, \eqref{eq13.5} becomes
\[{\tilde X_{1,{I_n}}} = {\left( {{A^T}A - \frac{1}{n}{A^T}I_n^T{I_n}\,A} \right)^{ - 1}}\left( {{A^T}B - \frac{1}{n}{A^T}I_n^T{I_n}\,B} \right).\]
Let us consider a numeric example for the ordinary variances case.
\begin{example}\label{ex13.1}
Suppose $ m=2 $, $ Z_{1 \times n}^T = {I_n} = [1,1,...,1] $ and $ p=1 $. Then, the corresponding over-determined system takes the simple form
\[{a_{i,1}}\,{x_1} + {a_{i,2}}\,{x_2} = {b_i}\,\,\,\,\,\,(i = 1,2,...,n > 2),\]
and the problem \eqref{eq13.3} reduces to
\begin{equation}\label{eq13.6}
\mathop {\min }\limits_{\{ {x_1},{x_2}\} } {V_{2,n}}(\left. {{x_1},{x_2}} \right|{I_n}) = \mathop {\min }\limits_{\{ {x_1},{x_2}\} } \,\,\sum\limits_{i = 1}^n {{{({a_{i,1}}\,{x_1} + {a_{i,2}}\,{x_2} - {b_i})}^2}}  - \frac{1}{n}{\left( {\sum\limits_{i = 1}^n {({a_{i,1}}\,{x_1} + {a_{i,2}}\,{x_2} - {b_i})} } \right)^2}.
\end{equation}
Hence, the explicit solutions of the system \eqref{eq13.4}, i.e. 
\[\left\{ \begin{array}{l}
\left( {\sum\limits_{i = 1}^n {{{({a_{i,1}})}^2}}  - \dfrac{1}{n}{\Big(\sum\limits_{i = 1}^n {{a_{i,1}}}\Big)^2}} \right){x_1} + \left( {\sum\limits_{i = 1}^n {{a_{i,1}}{a_{i,2}}}  - \dfrac{1}{n}\sum\limits_{i = 1}^n {{a_{i,1}}} \sum\limits_{i = 1}^n {{a_{i,2}}} } \right){x_2} \\[3mm]
= \sum\limits_{i = 1}^n {{a_{i,1}}{b_i}}  - \dfrac{1}{n}\sum\limits_{i = 1}^n {{a_{i,1}}} \sum\limits_{i = 1}^n {{b_i}} ,\\[3mm]
\left( {\sum\limits_{i = 1}^n {{a_{i,2}}{a_{i,1}}}  - \dfrac{1}{n}\sum\limits_{i = 1}^n {{a_{i,2}}} \sum\limits_{i = 1}^n {{a_{i,1}}} } \right){x_1} + \left( {\sum\limits_{i = 1}^n {{{({a_{i,2}})}^2}}  - \dfrac{1}{n}{\Big(\sum\limits_{i = 1}^n {{a_{i,2}}}\Big)^2}} \right){x_2}\\[3mm]
 = \sum\limits_{i = 1}^n {{a_{i,2}}{b_i}}  - \dfrac{1}{n}\sum\limits_{i = 1}^n {{a_{i,2}}} \sum\limits_{i = 1}^n {{b_i}} ,
\end{array} \right.\]
are respectively
\[{x_1} = \frac{\begin{array}{l}
\left( {\sum\limits_{i = 1}^n {{a_{i,1}}{b_i}}  - \frac{1}{n}\sum\limits_{i = 1}^n {{a_{i,1}}} \sum\limits_{i = 1}^n {{b_i}} } \right)\left( {\sum\limits_{i = 1}^n {{{({a_{i,2}})}^2}}  - \frac{1}{n}{\Big(\sum\limits_{i = 1}^n {{a_{i,2}}}\Big)^2}} \right) \\
- \left( {\sum\limits_{i = 1}^n {{a_{i,2}}{b_i}}  - \frac{1}{n}\sum\limits_{i = 1}^n {{a_{i,2}}} \sum\limits_{i = 1}^n {{b_i}} } \right)\left( {\sum\limits_{i = 1}^n {{a_{i,1}}{a_{i,2}}}  - \frac{1}{n}\sum\limits_{i = 1}^n {{a_{i,1}}} \sum\limits_{i = 1}^n {{a_{i,2}}} } \right)
\end{array}}{\begin{array}{l}
\left( {\sum\limits_{i = 1}^n {{{({a_{i,1}})}^2}}  - \frac{1}{n}{\Big(\sum\limits_{i = 1}^n {{a_{i,1}}} \Big)^2}} \right)\left( {\sum\limits_{i = 1}^n {{{({a_{i,2}})}^2}}  - \frac{1}{n}{\Big(\sum\limits_{i = 1}^n {{a_{i,2}}} \Big)}^2} \right)\\
 - {{\left( {\sum\limits_{i = 1}^n {{a_{i,1}}{a_{i,2}}}  - \frac{1}{n}\sum\limits_{i = 1}^n {{a_{i,1}}} \sum\limits_{i = 1}^n {{a_{i,2}}} } \right)}^2}
\end{array}},
\]
and
\[{x_2} = \frac{\begin{array}{l}
\left( {\sum\limits_{i = 1}^n {{a_{i,2}}{b_i}}  - \frac{1}{n}\sum\limits_{i = 1}^n {{a_{i,2}}} \sum\limits_{i = 1}^n {{b_i}} } \right)\left( {\sum\limits_{i = 1}^n {{{({a_{i,1}})}^2}}  - \frac{1}{n}{\Big(\sum\limits_{i = 1}^n {{a_{i,1}}} \Big)^2}} \right)\\
 - \left( {\sum\limits_{i = 1}^n {{a_{i,1}}{b_i}}  - \frac{1}{n}\sum\limits_{i = 1}^n {{a_{i,1}}} \sum\limits_{i = 1}^n {{b_i}} } \right)\left( {\sum\limits_{i = 1}^n {{a_{i,1}}{a_{i,2}}}  - \frac{1}{n}\sum\limits_{i = 1}^n {{a_{i,1}}} \sum\limits_{i = 1}^n {{a_{i,2}}} } \right)
\end{array}}{\begin{array}{l}
\left( {\sum\limits_{i = 1}^n {{{({a_{i,1}})}^2}}  - \frac{1}{n}{\Big(\sum\limits_{i = 1}^n {{a_{i,1}}} \Big)^2}} \right)\left( {\sum\limits_{i = 1}^n {{{({a_{i,2}})}^2}}  - \frac{1}{n}{\Big(\sum\limits_{i = 1}^n {{a_{i,2}}} \Big)^2}} \right) \\
- {{\left( {\sum\limits_{i = 1}^n {{a_{i,1}}{a_{i,2}}}  - \frac{1}{n}\sum\limits_{i = 1}^n {{a_{i,1}}} \sum\limits_{i = 1}^n {{a_{i,2}}} } \right)}^2}
\end{array}},\]
while the approximate solutions corresponding to the well-known problem \eqref{eq13.2} are
\[{\tilde x_1} = \frac{{(\sum\limits_{i = 1}^n {{a_{i,1}}{b_i}} )\sum\limits_{i = 1}^n {{{({a_{i,2}})}^2}}  - (\sum\limits_{i = 1}^n {{a_{i,2}}{b_i}} )\sum\limits_{i = 1}^n {{a_{i,1}}{a_{i,2}}} }}{{\sum\limits_{i = 1}^n {{{({a_{i,1}})}^2}} \sum\limits_{i = 1}^n {{{({a_{i,2}})}^2}}  - {\Big(\sum\limits_{i = 1}^n {{a_{i,1}}{a_{i,2}}} \Big)^2}}},\]
and
\[{\tilde x_2} = \frac{{(\sum\limits_{i = 1}^n {{a_{i,2}}{b_i}} )\sum\limits_{i = 1}^n {{{({a_{i,1}})}^2}}  - (\sum\limits_{i = 1}^n {{a_{i,1}}{b_i}} )\sum\limits_{i = 1}^n {{a_{i,1}}{a_{i,2}}} }}{{\sum\limits_{i = 1}^n {{{({a_{i,1}})}^2}} \sum\limits_{i = 1}^n {{{({a_{i,2}})}^2}}  - {\Big(\sum\limits_{i = 1}^n {{a_{i,1}}{a_{i,2}}} \Big)^2}}}.\]
Let us compare these solutions for a particular numerical case. If for example
\[{A_{4 \times 2}} = \left[ {\begin{array}{*{20}{c}}
{\begin{array}{*{20}{c}}
{ - 1}\\
2\\
1\\
{ - 1}
\end{array}}&{\begin{array}{*{20}{c}}
1\\
{ - 1}\\
{ - 2}\\
2
\end{array}}
\end{array}} \right],\,\,\,\,{X_{2 \times 1}} = \left[ {\begin{array}{*{20}{c}}
{{x_1}}\\
{{x_2}}
\end{array}} \right]\,\,\,\,\,{\rm{and}}\,\,\,\,\,{B_{4 \times 1}} = \,\left[ {\begin{array}{*{20}{c}}
1\\
2\\
3\\
4
\end{array}} \right],\]
then the solutions of the ordinary least squares problem are
\[({\tilde x_1},{\tilde x_2}) = (\frac{9}{7},\,1),\]
while the solutions corresponding to the minimization problem \eqref{eq13.6} are
\[\left( {{x_1}(p = 1;{I_4}),{x_2}(p = 1;{I_4})} \right) = (\frac{8}{{74}},\,\frac{{13}}{{74}}).\]
By substituting such values into the remaining term
\[{V_{2,4}}(\left. {{x_1},{x_2}} \right|{I_4}) = \,\sum\limits_{i = 1}^4 {{{({a_{i,1}}\,{x_1} + {a_{i,2}}\,{x_2} - {b_i})}^2}}  - \frac{1}{4}{\left( {\sum\limits_{i = 1}^4 {({a_{i,1}}\,{x_1} + {a_{i,2}}\,{x_2} - {b_i})} } \right)^2},\]
we observe that
\[{V_{2,4}}\left( {\left. {\frac{9}{7},\,1\,} \right|{I_4}} \right) = \frac{{53983}}{{7252}} \cong 7.4438,\]
whereas
\[{V_{2,4}}\left( {\left. {\frac{8}{{74}},\,\frac{{13}}{{74}}\,} \right|{I_4}} \right) = \frac{{35378}}{{7252}} \cong 4.8783.\]
On the other hand, for the well-known remaining term
\[{E_{2,4}}({x_1},{x_2}) = \,\sum\limits_{i = 1}^4 {{{({a_{i,1}}\,{x_1} + {a_{i,2}}\,{x_2} - {b_i})}^2}} ,\]
we observe that
\[{E_{2,4}}\left( {\frac{9}{7},\,1} \right) = \frac{{185}}{7} \cong 26.4285,\]
whereas
\[{E_{2,4}}\left( {\frac{8}{{74}},\,\frac{{13}}{{74}}} \right) = \frac{{80335}}{{2738}} \cong 29.3407.\]
In conclusion,
\[{V_{2,4}}\left( {\left. {\frac{8}{{74}},\,\frac{{13}}{{74}}\,} \right|{I_4}} \right) < {V_{2,4}}\left( {\left. {\frac{9}{7},\,1\,} \right|{I_4}} \right) < {E_{2,4}}\left( {\frac{9}{7},\,1} \right),\]
which confirms inequality \eqref{eq2.4}.
\end{example}

\section{On improving the Bessel inequality and Parseval identity}\label{sec14}
Two cases can be considered for the aforesaid purpose.

\subsection{First type of improvement}\label{sec14.1}
Let $ \{ {\Phi _k}(x)\} _{k = 0}^\infty  $ be a sequence of continuous functions which are p-uncorrelated with respect to the fixed function $ z(x) $ and the probability density function $ w(x)/\int_a^b {w(x)\,dx}  $ on $ [a,b] $ as before. Then, according to \eqref{eq4.1},
\begin{equation}\label{eq14.1}
f(x) \sim \sum\limits_{k = 0}^\infty  {\frac{{{{{\mathop{\rm cov}} }_p}\,\left( {{\Phi _k}(x),f(x);z(x)} \right)}}{{{{{\mathop{\rm var}} }_p}\left( {{\Phi _k}(x);z(x)} \right)}}{\Phi _k}(x)} \,,
\end{equation}
denotes a p-uncorrelated expansion for $ f(x) $ in which 
\begin{multline*}
\frac{{{{{\mathop{\rm cov}} }_p}\,\left( {{\Phi _k}(x),f(x);z(x)} \right)}}{{{{{\mathop{\rm var}} }_p}\left( {{\Phi _k}(x);z(x)} \right)}} = \\
\frac{{\int_{\,a}^b {w(x)\,{z^2}(x)\,dx} \int_{\,a}^b {w(x)\,{\Phi _k}(x)f(x)\,dx}  - p\int_{\,a}^b {w(x)\,{\Phi _k}(x)\,z(x)\,dx} \,\int_{\,a}^b {w(x)f(x)\,z(x)\,dx} }}{{\int_{\,a}^b {w(x)\,{z^2}(x)\,dx} \int_{\,a}^b {w(x)\,\Phi _k^2(x)\,dx}  - p{{\left( {\int_{\,a}^b {w(x)\,{\Phi _k}(x)\,z(x)\,dx} } \right)}^2}\,}}.
\end{multline*}
Referring to corollary \ref{coro4.5} and relation \eqref{eq4.16}, the following inequality holds for the expansion \eqref{eq14.1}:
\begin{equation}\label{eq14.2}
0 \le \sum\limits_{k = 0}^\infty  {\frac{{{\mathop{\rm cov}} _p^2\,\left( {{\Phi _k}(x),f(x);z(x)} \right)}}{{{{{\mathop{\rm var}} }_p}\left( {{\Phi _k}(x);z(x)} \right)}}}  \le {{\mathop{\rm var}} _p}\left( {f(x);z(x)} \right).
\end{equation}
Also, according to the definition of convergence in p-variance, inequality \eqref{eq14.2} will be transformed to an equality if
\begin{equation}\label{eq14.3}
\mathop {\lim }\limits_{n \to \infty } \,\,{{\mathop{\rm var}} _p}\left( {f(x) - \sum\limits_{k = 0}^n {\frac{{{{{\mathop{\rm cov}} }_p}\,\left( {{\Phi _k}(x),f(x);z(x)} \right)}}{{{{{\mathop{\rm var}} }_p}\left( {{\Phi _k}(x);z(x)} \right)}}{\Phi _k}(x)} \,} \right) = 0,
\end{equation}
which results in
\begin{equation}\label{eq14.4}
\sum\limits_{k = 0}^\infty  {\frac{{{\mathop{\rm cov}} _p^2\,\left( {{\Phi _k}(x),f(x);z(x)} \right)}}{{{{{\mathop{\rm var}} }_p}\left( {{\Phi _k}(x);z(x)} \right)}}}  = {{\mathop{\rm var}} _p}\left( {f(x);z(x)} \right).
\end{equation}
If \eqref{eq14.3} or equivalently \eqref{eq14.4} is satisfied, the p-uncorrelated sequence $ \{ {\Phi _k}(x)\} _{k = 0}^\infty  $ is ``complete" with respect to the fixed function $ z(x) $ and the symbol ``$ \sim $" in \eqref{eq14.1} will change to the equality.

Noting the above comments, now let $ f, g $ be two expandable functions of type \eqref{eq14.1} and $ \{ {\Phi _k}(x)\} _{k = 0}^\infty  $ be a ``complete" p-uncorrelated sequence. Since
\[f(x) = \sum\limits_{k = 0}^\infty  {\frac{{{{{\mathop{\rm cov}} }_p}\,\left( {{\Phi _k}(x),f(x);z(x)} \right)}}{{{{{\mathop{\rm var}} }_p}\left( {{\Phi _k}(x);z(x)} \right)}}{\Phi _k}(x)} \,,\]
and
\[g(x) = \sum\limits_{k = 0}^\infty  {\frac{{{{{\mathop{\rm cov}} }_p}\,\left( {{\Phi _k}(x),g(x);z(x)} \right)}}{{{{{\mathop{\rm var}} }_p}\left( {{\Phi _k}(x);z(x)} \right)}}{\Phi _k}(x)} \,,\]
thanks to the general identity
\[{{\mathop{\rm cov}} _p}\left( {\sum\limits_{k = 0}^n {{a_k}{\Phi _k}(x)} ,\,\sum\limits_{j = 0}^m {{b_j}{\Phi _j}(x)} ;z(x)} \right) = \sum\limits_{k = 0}^n {\sum\limits_{j = 0}^m {{a_k}{b_j}\,{{{\mathop{\rm cov}} }_p}\,\left( {{\Phi _k}(x),{\Phi _j}(x);z(x)} \right)} } ,\]
and this fact that
\[{{\mathop{\rm cov}} _p}\,\left( {{\Phi _k}(x),{\Phi _j}(x);z(x)} \right) = {{\mathop{\rm var}} _p}\left( {{\Phi _k}(x);z(x)} \right)\,{\delta _{k,j}},\]
we obtain
\begin{multline}\label{eq14.5}
{{\mathop{\rm cov}} _p}\,\left( {f(x),g(x);z(x)} \right)=\\
{{\mathop{\rm cov}} _p}\left( {\left( {\sum\limits_{k = 0}^\infty  {\frac{{{{{\mathop{\rm cov}} }_p}\,\left( {{\Phi _k}(x),f(x);z(x)} \right)}}{{{{{\mathop{\rm var}} }_p}\left( {{\Phi _k}(x);z(x)} \right)}}{\Phi _k}(x)} } \right),\,\left( {\sum\limits_{k = 0}^\infty  {\frac{{{{{\mathop{\rm cov}} }_p}\,\left( {{\Phi _k}(x),g(x);z(x)} \right)}}{{{{{\mathop{\rm var}} }_p}\left( {{\Phi _k}(x);z(x)} \right)}}{\Phi _k}(x)} } \right);z(x)} \right)\\
\,\,\,\,\,\,\,\,\,\,\,\,\,\,\,\,\,\,\,\,\,\,\,\,\,\,\,\,\,\,\,\,\, = \sum\limits_{k = 0}^\infty  {\frac{{{{{\mathop{\rm cov}} }_p}\,\left( {{\Phi _k}(x),f(x);z(x)} \right)\,\,{{{\mathop{\rm cov}} }_p}\,\left( {{\Phi _k}(x),g(x);z(x)} \right)}}{{{{{\mathop{\rm var}} }_p}\left( {{\Phi _k}(x);z(x)} \right)}}} .
\end{multline}
which is an extension of the identity \eqref{eq14.4} for $ f(x) = g(x) $. Also, for $ p=0 $, this important identity leads to the generalized Parseval identity \cite{ref5}
\begin{equation}\label{eq14.6}
E\left( {f(x)g(x)} \right) = \sum\limits_{k = 0}^\infty  {\frac{{E\left( {f(x)\,{\Phi _k}(x)} \right)\,\,E\left( {g(x)\,{\Phi _k}(x)} \right)}}{{E\left( {\Phi _k^2(x)} \right)}}} .
\end{equation}
The finite type of \eqref{eq14.5} is when $ f, g $ and $ \{ {\Phi _k}(x)\} _{k = 0}^\infty  $ are all polynomial functions. For example, let $ {\Phi _k}(x) = {{\bf{P}}_k}(x;a,0,c,0) $ denote the same as polynomials \eqref{eq9.8} satisfying
\begin{multline*}
\int_{\,0}^1 {{x^a}\,{{\bf{P}}_n}(x;a,0,c,0)\,{{\bf{P}}_m}(x;a,0,c,0)\,dx}  \\
- (2c - a + 1)\,\int_{\,0}^1 {{x^c}\,{{\bf{P}}_n}(x;a,0,c,0)\,dx} \,\int_{\,0}^1 {{x^c}\,{{\bf{P}}_m}(x;a,0,c,0)\,dx} \\
 = \frac{1}{{2n + a + 1}}{\left( {\frac{{(a - c)\,n!}}{{(c + 1)\,{{(a + 1)}_n}}}} \right)^2}{\delta _{m,n}}\quad \Leftrightarrow \quad 2c - a + 1 > 0,\,\,a \ne c\,\,\,{\rm{and}}\,\,\,a,c >  - 1\,.
\end{multline*}
Also let
\[{Q_m}(x) = \sum\limits_{k = 0}^m {{q_k}{x^k}} \,\,\,{\rm{and}}\,\,\,{R_m}(x) = \sum\limits_{k = 0}^m {{r_k}{x^k}} ,\]
be two arbitrary polynomials of the same degree. Since
\begin{multline*}
{Q_m}(x) = \sum\limits_{k = 0}^m (2k + a + 1){{\left( {\frac{{(c + 1)\,{{(a + 1)}_k}}}{{(a - c)\,k!}}} \right)}^2}\\
\times {{{\mathop{\rm cov}} }_1}\,\left( {{{\bf{P}}_k}(x;a,0,c,0),{Q_m}(x);{x^{c - a}}} \right){{\bf{P}}_k}(x;a,0,c,0) \,,
\end{multline*}
and
\begin{multline*}
{R_m}(x) = \sum\limits_{k = 0}^m (2k + a + 1){{\left( {\frac{{(c + 1)\,{{(a + 1)}_k}}}{{(a - c)\,k!}}} \right)}^2}\\
\times {{{\mathop{\rm cov}} }_1}\,\left( {{{\bf{P}}_k}(x;a,0,c,0),{R_m}(x);{x^{c - a}}} \right){{\bf{P}}_k}(x;a,0,c,0) \,,
\end{multline*}
according to \eqref{eq14.5} we have
\begin{multline*}
{{\mathop{\rm cov}} _1}\,\left( {{Q_m}(x),{R_m}(x);{x^{c - a}}} \right) =
\sum\limits_{k = 0}^m (2k + a + 1){{\left( {\frac{{(c + 1)\,{{(a + 1)}_k}}}{{(a - c)\,k!}}} \right)}^2}\\
\times {{{\mathop{\rm cov}} }_1}\,\left( {{{\bf{P}}_k}(x;a,0,c,0),{Q_m}(x);{x^{c - a}}} \right){{{\mathop{\rm cov}} }_1}\,\left( {{{\bf{P}}_k}(x;a,0,c,0),{R_m}(x);{x^{c - a}}} \right) .
\end{multline*}

\subsection{Second type of improvement}\label{subsec14.2}
As inequality \eqref{eq4.14} is valid for any arbitrary selection of the coefficients $ \{ {\alpha _k}\} _{k = 0}^n $, i.e.
\begin{equation}\label{eq14.7}
0 \le {{\mathop{\rm var}} _p}\,\left( {Y - \sum\limits_{k = 0}^n {{\alpha _k}{X_k}} \,;Z} \right) \le E\,\left( {{\Big(Y - \sum\limits_{k = 0}^n {{\alpha _k}{X_k}} \Big)^2}} \right),
\end{equation}
such kind of inequalities can be applied for orthogonal expansions. Suppose that $ \{ {\Phi _k}(x)\} _{k = 0}^\infty  $ is a sequence of continuous functions orthogonal with respect to the weight function $ w(x) $ on $ [a,b] $. If $ f(x) $ is a piecewise continuous function, then
\[f(x) \sim \sum\limits_{k = 0}^\infty  {{\alpha _k}{\Phi _k}(x)} \,\,\,\,\,\,\,\text{with}\,\,\,\,\,\,{\alpha _k} = \frac{{{{\left\langle {f,{\Phi _k}} \right\rangle }_w}}}{{{{\left\langle {{\Phi _k},{\Phi _k}} \right\rangle }_w}}},\]
is known as its corresponding orthogonal expansion in which
\[{\left\langle {f,g} \right\rangle _w} = \int_{\,a}^b {w(x)\,f(x)g(x)\,dx} \,.\]
The positive quantity
\begin{equation}\label{eq14.8}
{S_n} = \int_{\,a}^b {w(x)\,{{\left( {\sum\limits_{k = 0}^n {{\alpha _k}{\Phi _k}(x)}  - f(x)} \right)}^2}dx} \,,
\end{equation}
will eventually lead to the Bessel inequality \cite{ref26}
\[0 \le {S_n} = {\left\langle {f,f} \right\rangle _w} - \sum\limits_{k = 0}^n {\frac{{\left\langle {f,{\Phi _k}} \right\rangle _w^2}}{{{{\left\langle {{\Phi _k},{\Phi _k}} \right\rangle }_w}}}} \,.\]
Now, noting \eqref{eq14.7} and \eqref{eq14.8}, instead of $ S_{n} $ we define the following positive quantity
\begin{multline*}
{V_n}(p;z(x)) = {S_n} - R_n^2(p;z(x)) = \int_{\,a}^b {w(x)\,{{\left( {\sum\limits_{k = 0}^n {{\alpha _k}{\Phi _k}(x)}  - f(x)} \right)}^2}dx} \\
- \frac{p}{{\int_{\,a}^b {w(x)\,{z^2}(x)dx} }}{\left( {\int_{\,a}^b {w(x)\,z(x)\,\left( {\sum\limits_{k = 0}^n {{\alpha _k}{\Phi _k}(x)}  - f(x)} \right)dx} } \right)^2}.
\end{multline*}
It is clear that
\begin{equation}\label{eq14.9}
0 \le {V_n}(p;z(x)) \le {S_n}\,.
\end{equation}
Therefore
\begin{multline*}
0 \le {V_n}(p;z(x)) = {\left\langle {f,f} \right\rangle _w} - \sum\limits_{k = 0}^n {\frac{{\left\langle {f,{\Phi _k}} \right\rangle _w^2}}{{{{\left\langle {{\Phi _k},{\Phi _k}} \right\rangle }_w}}}} \,\\
 - \frac{p}{{{{\left\langle {z,z} \right\rangle }_w}}}\left( {{{\left( {\sum\limits_{k = 0}^n {\frac{{{{\left\langle {f,{\Phi _k}} \right\rangle }_w}{{\left\langle {z,{\Phi _k}} \right\rangle }_w}}}{{{{\left\langle {{\Phi _k},{\Phi _k}} \right\rangle }_w}}}} } \right)}^2} + \left\langle {f,z} \right\rangle _w^2 - 2{{\left\langle {f,z} \right\rangle }_w}\sum\limits_{k = 0}^n {\frac{{{{\left\langle {f,{\Phi _k}} \right\rangle }_w}{{\left\langle {z,{\Phi _k}} \right\rangle }_w}}}{{{{\left\langle {{\Phi _k},{\Phi _k}} \right\rangle }_w}}}} } \right)\,,
\end{multline*}
can be re-written as
\begin{multline}\label{eq14.10}
{\left\langle {z,z} \right\rangle _w}{\left\langle {f,f} \right\rangle _w} - p\left\langle {f,z} \right\rangle _w^2 \ge {\left\langle {z,z} \right\rangle _w}\sum\limits_{k = 0}^n {\frac{{\left\langle {f,{\Phi _k}} \right\rangle _w^2}}{{{{\left\langle {{\Phi _k},{\Phi _k}} \right\rangle }_w}}}}  + p\,{\left( {\sum\limits_{k = 0}^n {\frac{{{{\left\langle {f,{\Phi _k}} \right\rangle }_w}{{\left\langle {z,{\Phi _k}} \right\rangle }_w}}}{{{{\left\langle {{\Phi _k},{\Phi _k}} \right\rangle }_w}}}} } \right)^2}\\
\,\,\,\,\,\,\,\,\,\,\,\,\,\,\,\,\,\,\,\,\,\,\,\,\,\,\,\,\,\,\,\,\,\,\,\,\,\,\,\,\,\,\,\,\,\,\,\, - 2p{\left\langle {f,z} \right\rangle _w}\sum\limits_{k = 0}^n {\frac{{{{\left\langle {f,{\Phi _k}} \right\rangle }_w}{{\left\langle {z,{\Phi _k}} \right\rangle }_w}}}{{{{\left\langle {{\Phi _k},{\Phi _k}} \right\rangle }_w}}}} \,.
\end{multline}
Inequality \eqref{eq14.10} is an improvement of the well-known Bessel inequality for every $ p\in[0,1] $ with respect to the fixed function $ z(x) $.

For example, if $ {\Phi _k}(x) = \sin (k + 1)x $ for $ x\in[0,\pi] $ and $ w(x)=z(x)=1 $ are replaced in \eqref{eq14.10}, the Bessel inequality of the Fourier sine expansion will be improved as follows
\begin{multline*}
\int_0^\pi  {{f^2}(x)\,dx}  - \frac{p}{\pi }{\left( {\int_0^\pi  {f(x)\,dx} } \right)^2} \ge \frac{2}{\pi }\sum\limits_{k = 0}^n {{{\left( {\int_0^\pi  {f(x)\sin (k + 1)x\,dx} } \right)}^2}} \\
\,\,\,\,\,\,\,\,\,\,\,\,\,\,\,\,\,\,\,\,\,\,\,\,\,\,\,\,\,\,\,\,\,\,\,\,\,\,\,\,\,\,\,\,\,\,\,\,\,\,\,\,\,\,\,\,\,\,\,\,\,\,\,\,\, + \frac{{4p}}{{{\pi ^3}}}\,{\left( {\sum\limits_{k = 0}^n {\frac{{1 + {{( - 1)}^k}}}{{k + 1}}\int_0^\pi  {f(x)\sin (k + 1)x\,dx} } } \right)^2}\\
\,\,\,\,\,\,\,\,\,\,\,\,\,\,\,\,\,\,\,\,\,\,\,\,\,\,\,\,\,\,\,\,\,\,\,\,\,\,\,\,\,\,\,\,\,\,\,\,\,\,\,\,\,\,\,\,\,\,\,\,\,\,\,\,\,\, - \frac{{4p}}{{{\pi ^2}}}\left( {\int_0^\pi  {f(x)\,dx} } \right)\sum\limits_{k = 0}^n {\frac{{1 + {{( - 1)}^k}}}{{k + 1}}\int_0^\pi  {f(x)\sin (k + 1)x\,dx} } \,.
\end{multline*}
Obviously \eqref{eq14.10} will remain an inequality if the sequence $ \{ {\Phi _k}(x)\} _{k = 0} $ does not form a complete orthogonal system. On the other side, if they are a complete orthogonal set, inequality \eqref{eq14.10} becomes an equality when $ n\rightarrow\infty $. In other words, suppose $ \{ {\Phi _k}(x)\} _{k = 0} $ is a complete orthogonal set. Since
\[\mathop {\lim }\limits_{n \to \infty } {S_n} = \mathop {\lim }\limits_{n \to \infty } \int_{\,a}^b {w(x)\,{{\left( {\sum\limits_{k = 0}^n {{\alpha _k}{\Phi _k}(x)}  - f(x)} \right)}^2}dx} \, = 0,\]
we directly conclude from \eqref{eq14.9} that
\[0 \le \mathop {\lim }\limits_{n \to \infty } {V_n}(p;z(x)) \le \mathop {\lim }\limits_{n \to \infty } {S_n} = 0,\]
and therefore
\[\mathop {\lim }\limits_{n \to \infty } \,\,{S_n} - R_n^2(p;z(x)) = \mathop {\lim }\limits_{n \to \infty } \,{\left( {\int_{\,a}^b {w(x)\,z(x)\,\left( {\sum\limits_{k = 0}^n {{\alpha _k}{\Phi _k}(x)}  - f(x)} \right)dx} } \right)^2} = 0,\]
eventually yields
\[\sum\limits_{k = 0}^\infty  {\frac{{{{\left\langle {f,{\Phi _k}} \right\rangle }_w}{{\left\langle {z,{\Phi _k}} \right\rangle }_w}}}{{{{\left\langle {{\Phi _k},{\Phi _k}} \right\rangle }_w}}}}  = {\left\langle {f,\,z} \right\rangle _w},\]
which is known in the literature as the inner product form of the generalized Parseval identity \eqref{eq14.6}

In the next section, we will refer to the above-mentioned results in order to extend the presented theory in terms of a set of fixed mutually orthogonal variables.

\section{Least p-variances with respect to fixed orthogonal variables}\label{sec15}
Since the parts of this section are somewhat similar to the previous sections, we just state basic concepts and related theorems without proof.

Suppose $ x, y $ and $ \{ {z_k}\} _{k = 1}^m $ are elements of an inner product space $ \mathbf{S} $ such that $ \{ {z_k}\} _{k = 1}^m $ are mutually orthogonal as
\begin{equation}\label{eq15.1}
\left\langle {{z_i}\,,\,\,{z_j}} \right\rangle  = \left\langle {{z_j}\,,\,\,{z_j}} \right\rangle \,{\delta _{i,j}}.
\end{equation}
Due to the orthogonality property \eqref{eq15.1}, the following identity holds true
\begin{align}\label{eq15.2}
&\left\langle {x - (1 - \sqrt {1 - p} )\sum\limits_{k = 1}^m {\frac{{\left\langle {x,{z_k}} \right\rangle }}{{\left\langle {{z_k},{z_k}} \right\rangle }}\,{z_k}} \,,\,\,y - (1 - \sqrt {1 - p} )\sum\limits_{k = 1}^m {\frac{{\left\langle {y,{z_k}} \right\rangle }}{{\left\langle {{z_k},{z_k}} \right\rangle }}\,{z_k}} } \right\rangle  \\
&\qquad\qquad\qquad = \left\langle {x,y} \right\rangle  - p\sum\limits_{k = 1}^m {\frac{{\left\langle {x,{z_k}} \right\rangle \left\langle {y,{z_k}} \right\rangle }}{{\left\langle {{z_k},{z_k}} \right\rangle }}} ,\notag
\end{align}
and for $ y=x $, gives
\begin{equation}\label{eq15.3}
\left\langle {x,x} \right\rangle  - p\sum\limits_{k = 1}^m {\frac{{{{\left\langle {x,{z_k}} \right\rangle }^2}}}{{\left\langle {{z_k},{z_k}} \right\rangle }}}  \ge 0\,\,\,\,\,\,\,\,\,\,\,\forall p \in [0,1].
\end{equation}
The identity \eqref{eq15.2} and inequality \eqref{eq15.3} can again be employed in mathematical statistics.

\begin{definition}\label{def15.1}
Let $ X, Y $ and $ \{ {Z_k}\} _{k = 1}^m $ be arbitrary random variables such that $ \{ {Z_k}\} _{k = 1}^m $ are mutually orthogonal, i.e.
\begin{equation}\label{eq15.4}
E({Z_i}\,{Z_j}) = E(Z_j^2)\,{\delta _{i,j}},\,\,\,\,\,\,\,\,\,\,i,j = 1,2,...,m.
\end{equation}
Corresponding to \eqref{eq15.2} we define
\begin{align}\label{eq15.5}
&{{\mathop{\rm cov}} _p}(X,Y;\{ {Z_k}\} _{k = 1}^m)\\ 
&\quad = E\left( {\left( {X - (1 - \sqrt {1 - p} )\sum\limits_{k = 1}^m {\frac{{E(X{Z_k})}}{{E(Z_k^2)}}\,{Z_k}} } \right)\left( {Y - (1 - \sqrt {1 - p} )\sum\limits_{k = 1}^m {\frac{{E(Y{Z_k})}}{{E(Z_k^2)}}\,{Z_k}} } \right)} \right)\notag\\
 & \quad= E(XY) - p\sum\limits_{k = 1}^m {\frac{{E(X{Z_k})E(Y{Z_k})}}{{E(Z_k^2)}}} \,,\notag 
\end{align}
and call it "p-covariance of $ X $ and $ Y $ with respect to the fixed orthogonal variables $ \{ {Z_k}\} _{k = 1}^m $".
\end{definition}
For $ Y=X $, \eqref{eq15.5} changes to
\begin{align}\label{eq15.6}
{{\mathop{\rm var}} _p}(X;\{ {Z_k}\} _{k = 1}^m)&= E\left( {{{\left( {X - (1 - \sqrt {1 - p} )\sum\limits_{k = 1}^m {\frac{{E(X{Z_k})}}{{E(Z_k^2)}}\,{Z_k}} } \right)}^2}} \right)\\
&= E({X^2}) - p\sum\limits_{k = 1}^m {\frac{{{E^2}(X{Z_k})}}{{E(Z_k^2)}}}  \ge 0\,,\notag
\end{align}
where $ \{ E(Z_k^2)\} _{k = 1}^m $ are all positive.

Note in \eqref{eq15.5} that
\[\sum\limits_{k = 1}^m {\frac{{E(X{Z_k})}}{{E(Z_k^2)}}\,{Z_k}}  = \sum\limits_{k = 1}^m {{\rm{pro}}{{\rm{j}}_{\,{Z_k}}}X} ,\]
and therefore e.g. for $ p=1 $,
\[{{\mathop{\rm cov}} _1}(X,Y;\{ {Z_k}\} _{k = 1}^m) = E\left( {(X - \sum\limits_{k = 1}^m {{\rm{pro}}{{\rm{j}}_{\,{Z_k}}}X} )\,(Y - \sum\limits_{k = 1}^m {{\rm{pro}}{{\rm{j}}_{\,{Z_k}}}Y} )} \right).\]
Moreover, for orthogonal variables $ \{ {Z_k}\} _{k = 1}^m $ and $ p \in [0,1] $, we have
\begin{equation}\label{eq15.7}
0 \le {{\mathop{\rm var}} _1}(X;\{ {Z_k}\} _{k = 1}^m) \le {{\mathop{\rm var}} _p}(X;\{ {Z_k}\} _{k = 1}^m) \le {{\mathop{\rm var}} _0}(X;\{ {Z_k}\} _{k = 1}^m) = E({X^2}).
\end{equation}
A remarkable point in \eqref{eq15.7} is that if $ m,n $ are two natural numbers such that $ n>m $, then
\[0 \le {{\mathop{\rm var}} _p}(X;\{ {Z_k}\} _{k = 1}^n) \le {{\mathop{\rm var}} _p}(X;\{ {Z_k}\} _{k = 1}^m),\]
which can be proved directly via \eqref{eq15.6}. For instance, if $ n=2 $ and $ m=1 $, then
\[0 \le {{\mathop{\rm var}} _1}(X;{Z_1},{Z_2}) \le {{\mathop{\rm var}} _p}(X;{Z_1},{Z_2}) \le {{\mathop{\rm var}} _p}(X;{Z_1}) \le {{\mathop{\rm var}} _0}(X;{Z_1}) = E({X^2}),\]
in which
\[E({Z_1}\,{Z_2}) = 0.\]
The following properties hold true for definitions \eqref{eq15.5} and \eqref{eq15.6} provided that the orthogonal condition \eqref{eq15.4} is satisfied:
\begin{align*}
b1) &\qquad \qquad\qquad {{\mathop{\rm cov}} _p}(X,Y;\{ {Z_k}\} _{k = 1}^m) = {{\mathop{\rm cov}} _p}(Y,X;\{ {Z_k}\} _{k = 1}^m).\\
b2) & \qquad \qquad {{\mathop{\rm cov}} _p}\,(\alpha X,\beta Y;\{ {Z_k}\} _{k = 1}^m) = \alpha \beta \,\,{{\mathop{\rm cov}} _p}(X,Y;\{ {Z_k}\} _{k = 1}^m)\,\,\,\,\,\,\,\,\,\,\,(\alpha ,\beta  \in\mathbb{R}).\\
b3) & \begin{array}{l}
{{\mathop{\rm cov}} _p}\,(X + \alpha ,Y + \beta ;\{ {Z_k}\} _{k = 1}^m) = \,\,{{\mathop{\rm cov}} _p}\,(X,Y;\{ {Z_k}\} _{k = 1}^m) + \alpha \,{{\mathop{\rm cov}} _p}\,(1,Y;\{ {Z_k}\} _{k = 1}^m)\\[3mm]
\,\,\,\,\,\,\,\,\,\,\,\,\,\,\,\,\,\,\,\,\,\,\,\,\,\,\,\,\,\,\,\,\,\,\,\,\,\,\,\,\,\,\,\,\,\,\,\,\,\,\,\,\,\,\,\,\,\,\,\,\qquad \qquad + \beta \,\,{{\mathop{\rm cov}} _p}\,(X,1;\{ {Z_k}\} _{k = 1}^m) + \alpha \beta \,\,{{\mathop{\rm cov}} _p}\,(1,1;\{ {Z_k}\} _{k = 1}^m).
\end{array}\\
b4) &\quad {{\mathop{\rm cov}} _p}\,\left( {\sum\limits_{k = 0}^n {{c_k}{X_k}} ,{X_j};\{ {Z_k}\} _{k = 1}^m} \right) = \sum\limits_{k = 0}^n {{c_k}{{{\mathop{\rm cov}} }_p}\,({X_k},{X_j};\{ {Z_k}\} _{k = 1}^m)} \,\,\,\,\,\,\,(\{ {c_k}\} _{k = 0}^n \in\mathbb{R}).\\
b5) & \begin{array}{l}
{\rm var}_{p}(\alpha X + \beta Y;\{ {Z_k}\} _{k = 1}^m) = {\alpha ^2}\,{{\mathop{\rm var}} _p}\,(X;\{ {Z_k}\} _{k = 1}^m) + {\beta ^2}{{\mathop{\rm var}} _p}\,(Y;\{ {Z_k}\} _{k = 1}^m)\\[3mm]
\,\,\,\,\,\,\,\,\,\,\,\,\,\,\,\,\,\,\,\,\,\,\,\,\,\,\,\,\,\,\,\,\,\,\,\,\,\,\,\,\,\,\,\,\,\,\,\,\,\,\,\qquad\qquad + 2\alpha \beta \,\,{{\mathop{\rm cov}} _p}\,(X,Y;\{ {Z_k}\} _{k = 1}^m).
\end{array}
\end{align*}
\begin{definition}\label{def15.2}
Based on definitions \eqref{eq15.5} and \eqref{eq15.6} and the orthogonality condition \eqref{eq15.4}, we define
\begin{equation}\label{eq15.8}
{\rho _p}\,(X,Y;\{ {Z_k}\} _{k = 1}^m) = \frac{{{{{\mathop{\rm cov}} }_p}\,(X,Y;\{ {Z_k}\} _{k = 1}^m)}}{{\sqrt {{{{\mathop{\rm var}} }_p}\,(X;\{ {Z_k}\} _{k = 1}^m){\rm var}_{p}(Y;\{ {Z_k}\} _{k = 1}^m)} }},
\end{equation}
and call it ``p-correlation coefficient of $ X $ and $ Y $ with respect to the fixed orthogonal variables $ \{ {Z_k}\} _{k = 1}^m $".
\end{definition}
Clearly
\[{\rho _p}\,(X,Y;\{ {Z_k}\} _{k = 1}^m) \in [ - 1,1],\]
because if
\[U = X - (1 - \sqrt {1 - p} )\sum\limits_{k = 1}^m {\frac{{E(X{Z_k})}}{{E(Z_k^2)}}\,{Z_k}} \,\,\,\,\,\,\,{\rm{and}}\,\,\,\,\,\,V = Y - (1 - \sqrt {1 - p} )\sum\limits_{k = 1}^m {\frac{{E(Y{Z_k})}}{{E(Z_k^2)}}\,{Z_k}} ,\]
are replaced in the Cauchy-Schwarz inequality \eqref{eq1.11}, then
\[{\mathop{\rm cov}} _p^2(X,Y;\{ {Z_k}\} _{k = 1}^m) \le {{\mathop{\rm var}} _p}\,(X;\{ {Z_k}\} _{k = 1}^m)\,\,{\rm var}_{p}(Y;\{ {Z_k}\} _{k = 1}^m).\]
\begin{definition}\label{def15.3}
If $ {\rho _p}\,(X,Y;\{ {Z_k}\} _{k = 1}^m) = 0 $ in \eqref{eq15.8}, we say that $ X $ and $ Y $ are p-uncorrelated with respect to $ \{ {Z_k}\} _{k = 1}^m $ and we have 
\[E(XY) = p\sum\limits_{k = 1}^m {\frac{{E(X{Z_k})E(Y{Z_k})}}{{E(Z_k^2)}}\,} .\]
\end{definition}

\subsection{Least p-variance approximations based on fixed orthogonal variables}\label{subsec15.4}

Again, consider the approximation \eqref{eq2.1},
\[Y \cong \sum\limits_{k = 0}^n {{c_k}{X_k}} \,,\]
and define the p-variance of the remaining term
\[R({c_0},{c_1},...,{c_n}) = \sum\limits_{k = 0}^n {{c_k}{X_k}}  - Y\,,\]
with respect to the orthogonal variables $ \{ {Z_k}\} _{k = 1}^m $ as follows
\begin{multline}\label{eq15.9}
{{\mathop{\rm var}} _p}\,\left( {R({c_0},...,{c_n});\{ {Z_k}\} _{k = 1}^m} \right) \\
= E\left( {{{\left( {R({c_0},...,{c_n}) - (1 - \sqrt {1 - p} )\sum\limits_{k = 1}^m {\frac{{E({Z_k}R({c_0},...,{c_n}))}}{{E(Z_k^2)}}\,{Z_k}} } \right)}^2}} \right).
\end{multline}
To minimize \eqref{eq15.9}, the relations
\[\frac{{\partial {{{\mathop{\rm var}} }_p}\,\left( {R({c_0},...,{c_n});\{ {Z_k}\} _{k = 1}^m} \right)}}{{\partial {c_j}}} = 0\,\,\,\,\,{\rm{for}}\,\,\,j = 0,1,\,...\,,\,n,\]
eventually lead to the following linear system
\begin{align}\label{eq15.10}
&\left[ {\begin{array}{*{20}{c}}
{\begin{array}{*{20}{c}}
{{{{\mathop{\rm var}} }_p}\,({X_0};\{ {Z_k}\} _{k = 1}^m)}\\
{{{{\mathop{\rm cov}} }_p}\,({X_1},{X_0};\{ {Z_k}\} _{k = 1}^m)}\\
 \vdots \\
{{{{\mathop{\rm cov}} }_p}\,({X_n},{X_0};\{ {Z_k}\} _{k = 1}^m)}
\end{array}}&{\begin{array}{*{20}{c}}
{{{{\mathop{\rm cov}} }_p}\,({X_0},{X_1};\{ {Z_k}\} _{k = 1}^m)}\\
{{{{\mathop{\rm var}} }_p}\,({X_1};\{ {Z_k}\} _{k = 1}^m)}\\
 \vdots \\
{{{{\mathop{\rm cov}} }_p}\,({X_n},{X_1};\{ {Z_k}\} _{k = 1}^m)}
\end{array}}&{\begin{array}{*{20}{c}}
 \cdots \\
 \cdots \\
 \vdots \\
 \cdots 
\end{array}}&{\begin{array}{*{20}{c}}
{{{{\mathop{\rm cov}} }_p}\,({X_0},{X_n};\{ {Z_k}\} _{k = 1}^m)}\\
{{{{\mathop{\rm cov}} }_p}\,({X_1},{X_n};\{ {Z_k}\} _{k = 1}^m)}\\
 \vdots \\
{{{{\mathop{\rm var}} }_p}\,({X_n};\{ {Z_k}\} _{k = 1}^m)}
\end{array}}
\end{array}} \right]\\
&\qquad\qquad\qquad\qquad\qquad\qquad \times \left[ {\begin{array}{*{20}{c}}
{{c_0}}\\
{{c_1}}\\
 \vdots \\
{{c_n}}
\end{array}} \right]
= \left[ {\begin{array}{*{20}{c}}
{{{{\mathop{\rm cov}} }_p}\,({X_0},Y;\{ {Z_k}\} _{k = 1}^m)}\\
{{{{\mathop{\rm cov}} }_p}\,({X_1},Y;\{ {Z_k}\} _{k = 1}^m)}\\
 \vdots \\
{{{{\mathop{\rm cov}} }_p}\,({X_n},Y;\{ {Z_k}\} _{k = 1}^m)}
\end{array}} \right].\notag
\end{align}
Two continuous and discrete spaces can be considered for the system \eqref{eq15.10}.

\subsubsection{First case.}\label{subsubsec15.4.1}
If $ \left\{ {{X_k} = {\Phi _k}(x)} \right\}_{k = 0}^n $, $ Y=f(x) $ and the orthogonal set $ \left\{ {{Z_k} = {z_k}(x)} \right\}_{k = 1}^m $ are defined in a continuous space with a probability density function as
\[{P_r}\left( {X = x} \right) = \frac{{w(x)}}{{\int_{\,a}^b {w(x)\,dx} }},\]
the elements of the system \eqref{eq15.10} appear as
\begin{multline*}
{{\mathop{\rm cov}} _p}\,\left( {{\Phi _i}(x),{\Phi _j}(x)\,;\left\{ {{z_k}(x)} \right\}_{k = 1}^m} \right) = \\
\frac{1}{{\int_{\,a}^b {w(x)\,dx} }}\Big(\int_{\,a}^b {w(x)\,{\Phi _i}(x)\,{\Phi _j}(x)\,dx}  - p\sum\limits_{k = 1}^m {\frac{{\int_{\,a}^b {w(x)\,{\Phi _i}(x)\,{z_k}(x)\,dx} \,\int_{\,a}^b {w(x)\,{\Phi _j}(x)\,{z_k}(x)\,dx} }}{{\,\int_{\,a}^b {w(x)\,z_k^2(x)\,dx} }}} \Big),
\end{multline*}
only if
\[\int_{\,a}^b {w(x)\,{z_i}(x)\,{z_j}(x)\,dx}  = \left( {\int_{\,a}^b {w(x)\,z_j^2(x)\,dx} } \right){\delta _{i,j}}.\]

\subsubsection{Second case.} \label{subsubsec15.4.2}
If the above-mentioned variables are defined on a counter set, say $ {A^*} = \{ {x_k}\} _{k = 0}^m $, with a discrete probability density function as
\[{P_r}\left( {X = x} \right) = \frac{{j(x)}}{{\sum\limits_{x \in {A^*}} {j(x)} }},\]
then
\begin{multline*}
\,\,\,\,\,\,\,\,\,\,\,\,\,\,\,\,\,\,\,\,\,\,\,\,\,\,\,\,\,\,\,\,\,\,{{\mathop{\rm cov}} _p}\,\left( {{\Phi _i}(x),{\Phi _j}(x)\,;\{ {z_k}(x)\} _{k = 1}^m} \right) = \\
\frac{1}{{\sum\limits_{x \in {A^*}} {j(x)} }}\Big(\sum\limits_{x \in {A^*}} {j(x){\Phi _i}(x)\,{\Phi _j}(x)}  - p\sum\limits_{k = 1}^m {\frac{{\sum\limits_{x \in {A^*}} {j(x){\Phi _i}(x)\,{z_k}(x)} \,\sum\limits_{x \in {A^*}} {j(x){\Phi _j}(x)\,{z_k}(x)} }}{{\,\sum\limits_{x \in {A^*}} {j(x)\,z_k^2(x)} }}} \Big),
\end{multline*}
only if
\[\sum\limits_{x \in {A^*}} {j(x){z_i}(x)\,{z_j}(x)}  = \left( {\sum\limits_{x \in {A^*}} {j(x)z_i^2(x)} } \right){\delta _{i,j}}.\]
For example, suppose $ m=2 $, $ x \in [0,\pi ] $ and $ {P_r}\left( {X = x} \right) = \frac{1}{\pi } $. In this case, it is well known that if $ {z_1}(x) = \sin x $ and $ {z_2}(x) = \cos x $, then
\[E\left( {{z_1}(x){z_2}(x)} \right) = \frac{1}{\pi }\int_{\,0}^\pi  {\sin x\,\cos x\,dx}  = 0,\]
and therefore
\begin{multline*}
\,{{\mathop{\rm cov}} _p}\,\left( {{\Phi _i}(x),{\Phi _j}(x)\,;\sin x,\,\cos x} \right) = \frac{1}{\pi }\int_{\,0}^\pi  {{\Phi _i}(x)\,{\Phi _j}(x)\,dx} \\
 - \frac{{2p}}{{{\pi ^2}}}\left( {\int_{\,0}^\pi  {{\Phi _i}(x)\,\sin x\,dx} \int_{\,0}^\pi  {{\Phi _j}(x)\,\sin x\,dx}  + \int_{\,0}^\pi  {{\Phi _i}(x)\,\cos x\,dx} \int_{\,0}^\pi  {{\Phi _j}(x)\,\cos x\,dx} } \right).
\end{multline*}
The weighted version of the above example can be considered in various cases. For example, if $ w(x) = {e^x}, z_{1}(x)=e^{-x}\sin x $ and $ z_{2}(x)=\cos x $ all defined on $ [0,\pi] $, then
\[E\left( {{z_1}(x){z_2}(x)} \right) = \frac{1}{{{e^\pi } - 1}}\int_{\,0}^\pi  {\sin x\,\cos x\,dx}  = 0,\]
and consequently the corresponding space is defined by
\begin{multline*}
\,\left( {{e^\pi } - 1} \right)\,\,{{\mathop{\rm cov}} _p}\,\left( {{\Phi _i}(x),{\Phi _j}(x)\,;\sin x,\,\cos x} \right) = \int_{\,0}^\pi  {{e^x}{\Phi _i}(x)\,{\Phi _j}(x)\,dx} \\
 - p\left( {\frac{{\int_{\,0}^\pi  {{\Phi _i}(x)\,\sin x\,dx} \int_{\,0}^\pi  {{\Phi _j}(x)\,\sin x\,dx} }}{{\int_{\,0}^\pi  {{e^{ - x}}{{\sin }^2}x\,dx} }} + \frac{{\int_{\,0}^\pi  {{\Phi _i}(x)\,{e^x}\cos x\,dx} \int_{\,0}^\pi  {{\Phi _j}(x)\,{e^x}\cos x\,dx} }}{{\int_{\,0}^\pi  {{e^x}{{\cos }^2}x\,dx} }}} \right),
\end{multline*}
where
\[\int_{\,0}^\pi  {{e^{ - x}}{{\sin }^2}x\,dx}  = \,\frac{2}{5}(1 - {e^{ - \pi }})\,\,\,\,\,\,\,{\rm{and}}\,\,\,\,\,\,\int_{\,0}^\pi  {{e^x}{{\cos }^2}x\,dx}  = \frac{3}{5}({e^\pi } - 1).\]
In the sequel, applying the uncorrelatedness condition
\begin{equation}\label{eq15.11}
{{\mathop{\rm cov}} _p}\,({X_i},{X_j};\{ {Z_k}\} _{k = 1}^m) = {{\mathop{\rm var}} _p}\,({X_j};\{ {Z_k}\} _{k = 1}^m)\,{\delta _{i,j}}\,\,\,{\rm{for}}\,{\rm{every}}\,\,\,i,j = 0,1,...,n,
\end{equation}
on the elements of the linear system \eqref{eq15.10}, one can obtain the unknown coefficients as
\[{c_k} = \frac{{{{{\mathop{\rm cov}} }_p}\,({X_k},Y;\{ {Z_k}\} _{k = 1}^m)}}{{{{{\mathop{\rm var}} }_p}\,({X_k};\{ {Z_k}\} _{k = 1}^m)}}.\]
In this case
\[Y \cong \sum\limits_{k = 0}^n {\frac{{{{{\mathop{\rm cov}} }_p}({X_k},Y;\{ {Z_k}\} _{k = 1}^m)}}{{{{{\mathop{\rm var}} }_p}\,({X_k};\{ {Z_k}\} _{k = 1}^m)}}{X_k}} \,,\]
is the best approximation in the sense of least p-variance of the error with respect to the fixed orthogonal variables $ \{ {Z_k}\} _{k = 1}^m $.

\begin{theorem}\label{thm15.5}
Any finite set of random variables satisfying the condition \eqref{eq15.11} is linearly independent.
\end{theorem}

\begin{theorem}\label{thm15.6}
Let $ \{ {V_k}\} _{k = 0} $ be a finite or infinite sequence of random variables such that any finite number of elements $ \{ {V_k}\} _{k = 0}^{n} $ are linearly independent. One can find constants $ \{ {a_{i,j}}\}  $ such that the elements
\begin{align*}
{X_0} &= {V_0},\\
{X_1} &= {V_1} + {a_{12}}{V_0},\\
{X_2} &= {V_2} + {a_{22}}{V_1} + {a_{23}}{V_0},\\
&\,\,\,\,\,\,\,\,\,\,\,\,\,\,\,\,\,\,\,\,\,\,\,\,\, \vdots \\
{X_n} &= {V_n} + {a_{n2}}{V_{n - 1}} + ... + {a_{n,n + 1}}{V_0} = {V_n} - \sum\limits_{k = 0}^{n - 1} {\frac{{{{{\mathop{\rm cov}} }_p}({X_k},Y;\{ {Z_k}\} _{k = 1}^m)}}{{{{{\mathop{\rm var}} }_p}\,({X_k};\{ {Z_k}\} _{k = 1}^m)}}{X_k}} ,
\end{align*}
are mutually p-uncorrelated with respect to the fixed orthogonal variables $ \{ {Z_k}\} _{k = 1}^m $.

Also, there are constants $ \{ {b_{i,j}}\}  $ such that
\begin{align*}
{V_0} &= {X_0},\\
{V_1} &= {X_1} + {b_{12}}{X_0},\\
{V_2} &= {X_2} + {b_{22}}{X_1} + {b_{23}}{X_0},\\
\,\,\,\,\,\,\,\,\,\,\,\,\,\,\,\,\,\,\,\,\,\,\,\,\, \vdots \\
{V_n} &= {X_n} + {b_{n2}}{X_{n - 1}} + ... + {b_{n,n + 1}}{X_0},
\end{align*}
and
\[{{\mathop{\rm cov}} _p}\,({X_n},{V_k};\{ {Z_k}\} _{k = 1}^m) = 0\,\,\,\,\,\,\,\,{\rm{for}}\,\,\,k = 0,1,...,n - 1,\]
provided that
\[{{\mathop{\rm cov}} _p}\,({X_i},{X_j};\{ {Z_k}\} _{k = 1}^m) = 0\,\,\,\,\,\,\,{\rm{for}}\,\,\,\,\,\,i \ne j.\]
\end{theorem}

\subsection{A general representation for p-uncorrelated variables with respect to the fixed orthogonal variables}\label{subsec15.7}

Let $ \{ {V_k}\} _{k = 0}^{} $ be a finite or infinite sequence of random variables such that any finite number of its elements are linearly independent. As before, if $ \{ {Z_k}\} _{k = 1}^{m} $ are orthogonal satisfying the condition \eqref{eq15.4}, the following non-monic type determinant shows a general representation for the generic p-uncorrelated variable $ X_{n} $:
\begin{equation}\label{eq15.12}
{X_n} =\begin{vmatrix}
{{{{\mathop{\rm var}} }_p}\,({V_0};\{ {Z_k}\} _{k = 1}^m)} & {{{{\mathop{\rm cov}} }_p}\,({V_0},{V_1};\{ {Z_k}\} _{k = 1}^m)} & \ldots & {{{{\mathop{\rm cov}} }_p}\,({V_0},{V_n};\{ {Z_k}\} _{k = 1}^m)}\\
{{{{\mathop{\rm cov}} }_p}\,({V_1},{V_0};\{ {Z_k}\} _{k = 1}^m)} & {{{{\mathop{\rm var}} }_p}\,({V_1};\{ {Z_k}\} _{k = 1}^m)} & \ldots & {{{{\mathop{\rm cov}} }_p}\,({V_1},{V_n};\{ {Z_k}\} _{k = 1}^m)}\\
\vdots & \vdots & \vdots & \vdots \\
{{{{\mathop{\rm cov}} }_p}\,({V_{n - 1}},{V_0};\{ {Z_k}\} _{k = 1}^m)} & {{{{\mathop{\rm cov}} }_p}\,({V_{n - 1}},{V_1};\{ {Z_k}\} _{k = 1}^m)} & \ldots & {{{{\mathop{\rm cov}} }_p}\,({V_{n - 1}},{V_n};\{ {Z_k}\} _{k = 1}^m)}\\
{V_0} & {V_1} & \ldots & {V_n}
\end{vmatrix}.
\end{equation}

\begin{theorem}\label{thm15.8}
Let $ \{ {V_k}\} _{k = 0}^n $ be linearly independent variables and $ \{ {X_k}\} _{k = 0}^n $ be their corresponding p-uncorrelated elements generated by $ \{ {V_k}\} _{k = 0}^n $ in \eqref{eq15.12}. If $ \sum\limits_{k = 0}^n {{a_k}{V_k}}  = {W_n} $ then
\[{W_n} = \sum\limits_{k = 0}^n {\frac{{{{{\mathop{\rm cov}} }_p}\,({X_k},{W_n};\{ {Z_k}\} _{k = 1}^m)}}{{{{{\mathop{\rm var}} }_p}\,({X_k};Z)}}{X_k}} \,.\]
\end{theorem}

\begin{theorem}\label{thm15.9}
Let $ \{ {X_k}\} _{k = 0}^n $ be p-uncorrelated variables satisfying the condition \eqref{eq15.11} and let $ Y $ be arbitrary. Then
\[{{\mathop{\rm var}} _p}\,\left( {Y - \sum\limits_{k = 0}^n {\frac{{{{{\mathop{\rm cov}} }_p}\,({X_k},Y;\{ {Z_k}\} _{k = 1}^m)}}{{{{{\mathop{\rm var}} }_p}\,({X_k};\{ {Z_k}\} _{k = 1}^m)}}{X_k}} \,;\{ {Z_k}\} _{k = 1}^m} \right) \le {{\mathop{\rm var}} _p}\,\left( {Y - \sum\limits_{k = 0}^n {{\alpha _k}{X_k}} \,;\{ {Z_k}\} _{k = 1}^m} \right),\]
for any selection of constants $ \{ {\alpha _k}\} _{k = 0}^n $.
\end{theorem}

\begin{corollary}\label{coro15.10}
Under the conditions of the above theorem we have the following equality
\begin{align*}
&{{\mathop{\rm var}} _p}\,\left( {Y - \sum\limits_{k = 0}^n {\frac{{{{{\mathop{\rm cov}} }_p}\,({X_k},Y;\{ {Z_k}\} _{k = 1}^m)}}{{{{{\mathop{\rm var}} }_p}\,({X_k};\{ {Z_k}\} _{k = 1}^m)}}{X_k}} \,;\{ {Z_k}\} _{k = 1}^m} \right) \\
&\qquad\qquad\qquad\qquad= {{\mathop{\rm var}} _p}\,(Y;\{ {Z_k}\} _{k = 1}^m) - \sum\limits_{k = 0}^n {\frac{{{\mathop{\rm cov}} _p^2({X_k},Y;\{ {Z_k}\} _{k = 1}^m)}}{{{{{\mathop{\rm var}} }_p}\,({X_k};\{ {Z_k}\} _{k = 1}^m)}}} \\
&\qquad\qquad\qquad\qquad= {{\mathop{\rm var}} _p}\,(Y;\{ {Z_k}\} _{k = 1}^m)\left( {1 - \sum\limits_{k = 0}^n {\rho _p^2\,({X_k},Y;\{ {Z_k}\} _{k = 1}^m)} } \right),
\end{align*}
and as a result
\[\sum\limits_{k = 0}^n {\rho _p^2\,({X_k},Y;\{ {Z_k}\} _{k = 1}^m)}  \le 1.\]
\end{corollary}

\begin{theorem}\label{thm15.11}
Let $ \{ {V_k}\} _{k = 0}^n $ be linearly independent variables and $ \{ {X_k}\} _{k = 0}^n $ be their corresponding p-uncorrelated elements generated by $ \{ {V_k}\} _{k = 0}^n $ in \eqref{eq15.12}. Then, for any selection of constants $ \{ {\lambda _k}\} _{k = 0}^{n - 1} $ we have
\[{{\mathop{\rm var}} _p}\,\left( {{X_n}\,;\{ {Z_k}\} _{k = 1}^m} \right) \le {{\mathop{\rm var}} _p}\,\left( {{V_n} + \sum\limits_{k = 0}^{n - 1} {{\lambda _k}{V_k}} \,;\{ {Z_k}\} _{k = 1}^m} \right).\]
\end{theorem}

\begin{corollary}\label{coro15.12}
Let $ \{ {X_k}\} _{k = 0}^n $ be p-uncorrelated variables with respect to the fixed orthogonal variables $ \{ {Z_k}\} _{k = 1}^m $. Then, for any variable $ Y $ and every $ j = 0,1,...,n $ we have
\[{{\mathop{\rm cov}} _p}\,\left( {Y - \sum\limits_{k = 0}^n {\frac{{{{{\mathop{\rm cov}} }_p}\,({X_k},Y;\{ {Z_k}\} _{k = 1}^m)}}{{{{{\mathop{\rm var}} }_p}\,({X_k};\{ {Z_k}\} _{k = 1}^m)}}{X_k}} ,{X_j}\,;\{ {Z_k}\} _{k = 1}^m} \right) = 0.\]

\subsection{p-uncorrelated functions with respect to the fixed orthogonal functions}\label{subsec15.13}

Let $ \left\{ {{X_k} = {\Phi _k}(x)} \right\}_{k = 0}^n $, $ Y=f(x) $ and the orthogonal set $ \left\{ {{Z_k} = {z_k}(x)} \right\}_{k = 1}^m $ be defined in a continuous space with a probability density function as
\[{P_r}\left( {X = x} \right) = \frac{{w(x)}}{{\int_{\,a}^b {w(x)\,dx} }}.\]
We say that $ \left\{ {{\Phi _k}(x)} \right\}_{k = 0}^\infty  $ are (weighted) p-uncorrelated functions with respect to the fixed orthogonal functions $\left\{ {{z _k}(x)} \right\}_{k = 1}^m $ if they satisfy the condition
\begin{multline*}
\int_{\,a}^b {w(x)\,{\Phi _i}(x)\,{\Phi _j}(x)\,dx}  - p\sum\limits_{k = 1}^m {\frac{{\int_{\,a}^b {w(x)\,{\Phi _i}(x)\,{z_k}(x)\,dx} \,\int_{\,a}^b {w(x)\,{\Phi _j}(x)\,{z_k}(x)\,dx} }}{{\,\int_{\,a}^b {w(x)\,z_k^2(x)\,dx} }}} \\
\,\,\,\,\,\,\,\,\,\,\,\,\,\,\,\,\,\,\,\,\,\,\,\,\,\,\,\,\,\,\,\,\,\,\,\,\,\,\,\, = \left( {\int_{\,a}^b {w(x)\,\Phi _j^2(x)\,dx}  - p\sum\limits_{k = 1}^m {\frac{{{{\left( {\int_{\,a}^b {w(x)\,{\Phi _j}(x)\,{z_k}(x)\,dx} } \right)}^2}}}{{\,\int_{\,a}^b {w(x)\,z_k^2(x)\,dx} }}} } \right)\,{\delta _{i,j}},
\end{multline*}
provided that
\[\int_{\,a}^b {w(x)\,{z_i}(x)\,{z_j}(x)\,dx}  = \left( {\int_{\,a}^b {w(x)\,z_j^2(x)\,dx} } \right){\delta _{i,j}}.\]
\end{corollary}

\begin{remark}\label{rem15.14}
After deriving the above-mentioned p-uncorrelated functions, one will be able to define a sequence of orthogonal functions in the form
\begin{equation*}
{\Phi _n}(x) - (1 - \sqrt {1 - p} )\sum\limits_{k = 1}^m {\frac{{\,\int_{\,a}^b {w(x)\,{\Phi _n}(x)\,{z_k}(x)\,dx} }}{{\,\int_{\,a}^b {w(x)\,z_k^2(x)\,dx} }}\,{z_k}(x)}  = {G_n}(x;p,\{ {z_k}(x)\} _{k = 1}^m)\,,
\end{equation*}
having the orthogonality property 
\begin{multline*}
\int_{\,a}^b {w(x)\,{G_i}(x;p,\{ {z_k}(x)\} _{k = 1}^m){G_j}(x;p,\{ {z_k}(x)\} _{k = 1}^m)\,dx}  \\
= \left( {\int_{\,a}^b {w(x)\,dx} } \right)\Big( {{{{\mathop{\rm var}} }_p}({\Phi _j}(x);\,\{ {z_k}(x)\} _{k = 1}^m)} \Big){\delta _{i,j}}.
\end{multline*}
\end{remark}

\subsection{p-uncorrelated vectors with respect to fixed orthogonal vectors}\label{subsec15.15}

Let $ {\vec A_m} = ({a_1},{a_2},...,{a_m}) $ and $ {\vec B_m} = ({b_1},{b_2},...,{b_m}) $ be two arbitrary vectors and $ \left\{ {{{\vec Z}_{k,m}} = ({z_{k,1}},{z_{k,2}},...,{z_{k,m}})} \right\}_{k = 1}^l $ be a set of fixed orthogonal vectors. It can be verified that
\begin{align*}
&\left( {{{\vec A}_m} - (1 - \sqrt {1 - p} )\sum\limits_{k = 1}^l {\frac{{{{\vec A}_m}.{{\vec Z}_{k,m}}}}{{{{\vec Z}_{k,m}}.{{\vec Z}_{k,m}}}}{{\vec Z}_{k,m}}} } \right).\left( {{{\vec B}_m} - (1 - \sqrt {1 - p} )\sum\limits_{k = 1}^l {\frac{{{{\vec B}_m}.{{\vec Z}_{k,m}}}}{{{{\vec Z}_{k,m}}.{{\vec Z}_{k,m}}}}{{\vec Z}_{k,m}}} } \right)\\
&\qquad\qquad\,\,\,\,\,\,\,\,\,\,\,\,\,\,\,\,\,\,\,\,\,\,\,\,\,\,\,\,\,\,\,\,\,\,\,\,\,\,\,\,\,\,\,\,\,\,\,\, = {{\vec A}_m}.{{\vec B}_m} - p\sum\limits_{k = 1}^l {\frac{{({{\vec A}_m}.{{\vec Z}_{k,m}})({{\vec B}_m}.{{\vec Z}_{k,m}})}}{{{{\vec Z}_{k,m}}.{{\vec Z}_{k,m}}}}} ,
\end{align*}
only if
\[{\vec Z_{k,m}}.\,{\vec Z_{j,m}} = \left( {{{\vec Z}_{j,m}}.\,{{\vec Z}_{j,m}}} \right)\,{\delta _{k,j}}.\]

\bigskip

\noindent
{\textbf{Final Remark.}} As we observed, this theory was established based on an algebraic inequality of type \eqref{eq1.6}, i.e. 
\[0 \le {{\mathop{\rm var}} _1}\,(X;Z) \le {{\mathop{\rm var}} _p}\,(X;Z) \le {{\mathop{\rm var}} _0}\,(X;Z) = E({X^2}).\]
A notable point is that there is still another algebraic inequality that refines the above inequality and consequently improves the theory presented in this work. However, it contains various new definitions, concepts and calculations and therefore needs a much time to be completed. We hope inshallah to finish it soon.

\section*{Acknowledgments}
This work has been supported by the Alexander von Humboldt Foundation under the grant number: Ref 3.4 - IRN - 1128637 - GF-E.

\end{document}